\documentclass[11pt]{article} 
\usepackage[utf8]{inputenc} 

\usepackage[margin=1in]{geometry} 
\usepackage{graphicx} 
\usepackage{nicefrac}
\usepackage{anyfontsize}

\usepackage{booktabs} 
\usepackage{array} 
\usepackage{paralist} 
\usepackage{verbatim} 
\usepackage{mathrsfs}
\usepackage{amssymb}
\usepackage{amsthm}
\usepackage{amsmath,amsfonts,amssymb}
\usepackage{esint}
\usepackage{graphics}
\usepackage{enumerate}
\usepackage{mathtools}
\usepackage{xfrac}
\usepackage{bbm}
\usepackage{subcaption}
\usepackage{stmaryrd}
 \usepackage{mathabx}

\usepackage[usenames,dvipsnames]{xcolor}
\usepackage[colorlinks=true, pdfstartview=FitV, linkcolor=blue, citecolor=blue, urlcolor=blue]{hyperref}
\usepackage[normalem]{ulem}

\usepackage{tikz}
\usetikzlibrary{calc}
\usepackage{pgf}
\usetikzlibrary{external}
\numberwithin{equation}{section}
\numberwithin{figure}{section}

\newtheorem{theorem}{Theorem}[section]
\newtheorem{lemma}{Lemma}[section]

\newtheorem{corollary}[lemma]{Corollary}
\newtheorem{proposition}[lemma]{Proposition}
\theoremstyle{definition}

\newtheorem{remark}[lemma]{Remark}

\DeclarePairedDelimiter{\snorm}{\llbracket}{\rrbracket}
\DeclarePairedDelimiter{\norm}{\lVert}{\rVert}
\newcommand{\abs}[1]{\left|#1\right|}

\newcommand{\average}[1]{\left\langle#1\right\rangle}

\newcommand*{\supp}{\ensuremath{\mathrm{supp\,}}}
\newcommand*{\Id}{\ensuremath{\mathrm{I}_d}}
\newcommand*{\Itwo}{\ensuremath{\mathrm{I}_2}}

\renewcommand*{\div}{\ensuremath{\mathrm{div\,}}}

\newcommand*{\NN}{\ensuremath{\mathbb{N}}}
\newcommand*{\N}{\ensuremath{\mathbb{N}}}

\newcommand*{\TT}{\ensuremath{\mathbb{T}}}
\newcommand*{\ZZ}{\ensuremath{\mathbb{Z}}}
\newcommand*{\Z}{\ensuremath{\mathbb{Z}}}
\newcommand*{\RR}{\ensuremath{\mathbb{R}}}
\newcommand*{\R}{\ensuremath{\mathbb{R}}}
\newcommand*{\Zd}{\ensuremath{\mathbb{Z}^d}}
\newcommand*{\Rd}{\ensuremath{\mathbb{R}^d}}

\newcommand*{\DD}{\mathscr{D}}
\newcommand{\eps}{\varepsilon}
\newcommand{\OO}{\mathcal O}

\renewcommand*{\tilde}{\widetilde}

\newcommand{\Proj}{\ensuremath{\mathbb{P}}}

\renewcommand{\b}{\ensuremath{\mathbf{b}}}
\newcommand{\vv}{\ensuremath{\mathbf{v}}}

\newcommand{\f}{\mathbf{f}}
\newcommand{\J}{\mathbf{J}}
\newcommand{\g}{\mathbf{g}}
\newcommand{\h}{\widetilde{H}}
\newcommand{\s}{\mathbf{s}}

\newcommand{\Reynolds}{\mathring{\mathbf{R}}}
\newcommand{\ep}{\eps}

\DeclareMathOperator{\dist}{dist}

\newcommand{\q}{q}

\DeclareSymbolFont{boldoperators}{OT1}{cmr}{bx}{n}
\SetSymbolFont{boldoperators}{bold}{OT1}{cmr}{bx}{n}
\usepackage{accents}
\newcommand\thickbar[1]{\accentset{\rule{.45em}{.6pt}}{#1}}
\renewcommand{\bar}{\thickbar}
\renewcommand{\a}{\mathbf{a}}

\newcommand{\K}{\mathbf{K}}
\newcommand{\Khom}{\overline{\K}}
\newcommand{\ahom}{\bar{\a}}

\newcommand{\data}{\mathrm{data}}

\def\Xint#1{\mathchoice
{\XXint\displaystyle\textstyle{#1}}%
{\XXint\textstyle\scriptstyle{#1}}%
{\XXint\scriptstyle\scriptscriptstyle{#1}}%
{\XXint\scriptscriptstyle\scriptscriptstyle{#1}}%
\!\int}
\def\XXint#1#2#3{{\setbox0=\hbox{$#1{#2#3}{\int}$}
\vcenter{\hbox{$#2#3$}}\kern-.5\wd0}}

\def\fint{\Xint-}

\let\originalleft\left
\let\originalright\right
\renewcommand{\left}{\mathopen{}\mathclose\bgroup\originalleft}
\renewcommand{\right}{\aftergroup\egroup\originalright}

\def\kk{{\boldsymbol{k}}}
\def\aa{{\boldsymbol{\alpha}}}
\def\bb{{\boldsymbol{\beta}}}
\def\xx{\boldsymbol{x}}
\def\yy{\boldsymbol{y}}
\def\gb{\boldsymbol{g}}
\def\lll{{\boldsymbol{\ell}}}

\def\Chi{{\boldsymbol{\chi}}}

\newcommand{\flow}{X}
\newcommand{\backflow}{\flow^{-1}}
\newcommand{\stream}{\psi}
\newcommand{\streamr}[2]{\psi_{{#1},{#2}}}
\newcommand{\corr}{\chi}

\newcommand{\indc}{{\mathbf{1}}}

\newcommand{\e}{{\mathbf{e}}}

\makeatletter
\pgfmathdeclarefunction{erf}{1}{%
  \begingroup
    \pgfmathparse{#1 > 0 ? 1 : -1}%
    \edef\sign{\pgfmathresult}%
    \pgfmathparse{abs(#1)}%
    \edef\x{\pgfmathresult}%
    \pgfmathparse{1/(1+0.3275911*\x)}%
    \edef\t{\pgfmathresult}%
    \pgfmathparse{%
      1 - (((((1.061405429*\t -1.453152027)*\t) + 1.421413741)*\t 
      -0.284496736)*\t + 0.254829592)*\t*exp(-(\x*\x))}%
    \edef\y{\pgfmathresult}%
    \pgfmathparse{(\sign)*\y}%
    \pgfmath@smuggleone\pgfmathresult%
  \endgroup
}
\makeatother

\usepackage{titlesec}

\newcommand{\addperiod}[1]{#1.}
\titleformat{\section}
   {\centering\normalfont\Large}{\thesection.}{0.5em}{}
\titleformat*{\subsection}{\bfseries}
\titleformat{\subsubsection}[runin]
  {\normalfont\bfseries}
  {\thesubsubsection.}
  {0.5em}
  {\addperiod}
\titleformat*{\subsubsection}{\normalfont\itshape}
\titleformat*{\paragraph}{\bfseries}
\titleformat*{\subparagraph}{\large\bfseries}

\title{Anomalous diffusion by fractal homogenization}

\author{Scott Armstrong
\thanks{Courant Institute of Mathematical Sciences, New York University.
{\footnotesize \href{mailto:scotta@cims.nyu.edu}{scotta@cims.nyu.edu}.}
}
\and 
Vlad Vicol
\thanks{Courant Institute of Mathematical Sciences,  New York University.
{\footnotesize \href{mailto:vicol@cims.nyu.edu}{vicol@cims.nyu.edu}.}
}
}

\date{October 9, 2024} 
 
\usepackage[nottoc,notlot,notlof]{tocbibind}

\begin{document}

\maketitle

\begin{abstract}
For every~$\alpha < \nicefrac13$, we construct an explicit divergence-free vector field~$\b(t,x)$ which is periodic in space and time and belongs to~$C^0_t C^{\alpha}_x \cap C^{\alpha}_t C^0_x$ such that the corresponding scalar advection-diffusion equation 
$$\partial_t \theta^\kappa  + \b \cdot \nabla \theta^\kappa - \kappa \Delta \theta^\kappa = 0$$ 
exhibits anomalous dissipation of scalar variance for arbitrary~$H^1$ initial data:
$$\limsup_{\kappa \to 0} \int_0^{1}  \int_{\TT^d} \kappa \bigl| \nabla \theta^\kappa (t,x) \bigr|^2 \,dx\,dt >0.$$  
The vector field is deterministic and has a fractal structure, with periodic shear flows  alternating in time between different directions serving as the base fractal. These shear flows are repeatedly inserted at infinitely many scales in suitable Lagrangian coordinates. 
Using an argument based on ideas from quantitative homogenization, the corresponding advection-diffusion equation with small~$\kappa$ is progressively renormalized, one scale at a time, starting from the (very small) length scale determined by the molecular diffusivity up to the macroscopic (unit) scale. At each renormalization step, the effective diffusivity is enhanced by the influence of advection on that scale. By iterating this procedure across many scales, the effective diffusivity on the macroscopic scale is shown to be of order one. 
\end{abstract}

\setcounter{tocdepth}{2}
\tableofcontents

\section{Introduction}
\label{s.intro}

We consider the Cauchy problem for the linear advection-diffusion equation 
\begin{equation}
\label{e.passive.scalar}
\left\{
\begin{aligned}
& \partial_t \theta^\kappa  
+ \b \cdot \nabla \theta^\kappa
- \kappa \Delta \theta^\kappa
= 0
&  \mbox{in} & \ (0,\infty) \times \TT^d\,, \\
& \theta^\kappa (0,\cdot) = \theta_0 &  \mbox{on} & \ \TT^d
\,.
\end{aligned}
\right.
\end{equation}
The initial data~$\theta_0$ is assumed to belong to~$L^2(\TT^d)$ and have zero mean; it can also be assumed to be smooth. 
The vector field~$\b(t,x)$ in~\eqref{e.passive.scalar} is assumed to be \emph{incompressible}, that is, divergence-free:
\begin{equation}
\label{e.divfree}
\nabla \cdot \b(t,\cdot) = 0\,, \quad \forall t\in (0,\infty)
\,.
\end{equation}
Physically, the solution~$\theta^\kappa$ represents a scalar quantity, such as temperature or the concentration of a pollutant in a fluid, which is ``passive'' in the sense of having a negligible effect on the flow itself. For this reason, the equation in~\eqref{e.passive.scalar} is often called the \emph{passive scalar equation}. 
We are interested in the case in which the parameter~$\kappa>0$ is very small and the vector field~$\b(t,x)$, although continuous in~$(t,x)$, is still quite rough---possessing certain properties characteristic of turbulent flows.

\smallskip

The main result of this paper is the construction of an explicit vector field~$\b(t,x)$ for which the variance of the corresponding passive scalar~$\theta^\kappa$ exhibits anomalous dissipation.

\begin{theorem}[Anomalous dissipation of scalar variance]
\label{t.anomalous.diffusion}
Let~$d\geq 2$ and~$\alpha \in (0,\nicefrac13)$. 
There exists a vector field
\begin{equation}
\label{e.b.reg}
\mathbf{b} \in 
C^0_t C^{0,\alpha}_x \bigl ([0,1]\times \TT^d\bigr )  
\cap 
C^{0,\alpha}_t C^0_x\bigl ([0,1]\times \TT^d\bigr ) 
\end{equation}
which satisfies~\eqref{e.divfree}
such that,
for every mean-zero initial datum~$\theta_0 \in H^1(\TT^d)$, the family of unique solutions~$\{ \theta^\kappa\}_{\kappa>0} \in C([0,1];L^2(\TT^d))$ of the advection-diffusion equation~\eqref{e.passive.scalar}
satisfy
\begin{equation}
\label{e.anomalous.diffusion}
\limsup_{\kappa \to 0} \,
\kappa \|\nabla\theta^\kappa \|_{L^2( (0,1) \times \TT^d)}^2
\geq 
 \varrho^2   \| \theta_0    \|_{L^2(\TT^d)}^2
\,,
\end{equation}
for some constant~$\varrho=\varrho(d,\theta_0)\in(0,1]$ which depends only on~$d$ and the ratio~$\| \theta_0 \|_{L^2(\TT^d)} / \| \nabla \theta_0 \|_{L^2(\TT^d)}$.
\end{theorem}

\smallskip

The initial-value problem~\eqref{e.passive.scalar} has a unique global solution for every~$\kappa>0$ provided that the vector field~$\b(t,x)$ belongs to $L^\infty_t L^d_x$. By the incompressibility condition~\eqref{e.divfree}, this solution satisfies the energy balance relation
\begin{equation}
\label{e.energy.balance}
\norm{\theta_0(\cdot)}_{L^2(\TT^d)}^2
-
\norm{\theta^\kappa(1,\cdot)}_{L^2(\TT^d)}^2
= 2 \kappa \|\nabla\theta^\kappa \|_{L^2( (0,1) \times \TT^d)}^2\,.
\end{equation}
The quantity on the right side of~\eqref{e.energy.balance} is therefore called the \emph{dissipation of scalar variance}. 
While norms of~$\b(t,x)$ do not appear explicitly in \eqref{e.energy.balance}, the solution~$\theta^\kappa$ of course depends on the vector field in a very complicated and nonlinear way.

\smallskip

The family~$\{ \theta^\kappa\}_{\kappa > 0}$ in Theorem~\ref{t.anomalous.diffusion} are actually classical solutions of~\eqref{e.passive.scalar}.
Indeed, the incompressibility condition~\eqref{e.divfree} allows us to write the drift term as part of the second-order diffusion term, using a stream matrix which, in view of~\eqref{e.b.reg}, belongs to~$C^{1,\alpha}$. The standard Schauder estimates therefore imply that, for each~$\kappa>0$, the solution of~\eqref{e.passive.scalar} belongs to~$C^{0}_t C^{2,\alpha}_x \cap C^{1,\nicefrac\alpha2}_t C^0_x$ for positive times. 

\smallskip

As we will see in the proof, the parameter~$\varrho >0$ in Theorem~\ref{t.anomalous.diffusion} can be taken to be 
\begin{equation}
\label{e.varrho.def}
\varrho(d,\theta_0)
=
c  
\biggl( 
\frac{\| \theta_0  \|_{L^2(\TT^d)}}{\| \theta_0  \|_{H^{1}(\TT^d)} }
\biggr)^{\!\frac{1+\alpha}{1-\alpha}+\ep} 
\,,
\end{equation}
where~$\ep>0$ is any positive constant, and~$c = c(d,\ep)>0$ is a positive constant. In particular,~$\varrho(d,\theta_0)$ depends only on~$d$ and a lower bound for the length scale~$\| \theta_0 \|_{L^2(\TT^d)} / \| \nabla \theta_0 \|_{L^2(\TT^d)}$. 

\smallskip

The proof of Theorem~\ref{t.anomalous.diffusion} gives more  information concerning the scalar~$\theta^{\kappa_j}$ along the subsequence~$\kappa_j\downarrow 0$ than what appears in the statement. For instance, it yields a tiny exponent~$\mu>0$ depending only on~$\alpha$ such that~$\| \theta^{\kappa_j} \|_{C^{0,\mu}_t L^2_x([0,1]\times \TT^d)}$ stays bounded along the subsequence~$\kappa_j\downarrow 0$. 
The mapping~$t \mapsto \int_{\TT^d} | \theta^{\kappa_j} (t,x)|^2\,dx$ is therefore  uniformly continuous on~$[0,1]$: see Remark~\ref{r.LeBron.2}. In particular, the diffusive anomaly does not happen at any single ``blow-up time.'' The uniform regularity of the scalar~$\theta^{\kappa_j}$ will be greatly improved in a forthcoming paper~\cite{ARV}. 

\smallskip

Our arguments exhibit an explicit subsequence~$\kappa_j \to 0$ along which the lower bound in~\eqref{e.anomalous.diffusion} is realized, which, in particular, does not depend on~$\theta_0$. 
In fact, we construct a sequence of disjoint intervals~$I_j:= [ \tfrac12 \kappa_j, 2\kappa_j]$ with~$\kappa_j \to 0$ such that, if~$I= \cup_{j\in\N} I_j$, then  
\begin{equation*}
\inf_{\kappa \in I} \, 
\kappa \|\nabla\theta^\kappa \|_{L^2( (0,1) \times \TT^d)}^2
\geq 
\varrho^2  \| \theta_0  \|_{L^2(\TT^d)}^2
\,.
\end{equation*}
The position of~$\kappa$ within the interval~$I_j$ determines, up to an error which can be made arbitrarily small, the value of~$\kappa \|\nabla\theta^\kappa \|_{L^2( (0,1) \times \TT^d)}^2$, with the left and right endpoints of~$I_j$ giving rise to significantly different values, at least for carefully chosen initial data. We therefore demonstrate the lack of a selection principle for vanishing diffusivity limits to the solutions of the transport equation ($\kappa=0$ in \eqref{e.passive.scalar}), which evidently possesses non-unique bounded weak solutions. For a precise statement see Proposition~\ref{p.no.selection.principle}, which is established in Section~\ref{ss.proof.selection}.

\smallskip

The vector field~$\b(t,x)$ appearing in the statement of Theorem~\ref{t.anomalous.diffusion} has an explicit construction using deterministic ingredients, namely periodic shear flows with directions that alternate in time. An infinite sequence of copies of these shear flows are embedded in the vector field, each with a different wave number, with the sequence of wave numbers tending to infinity at a super-geometric rate. 
The proof of Theorem~\ref{t.anomalous.diffusion} is based on a renormalization of effective diffusivities, in which each active scale in the vector field is homogenized, one-by-one. Each homogenization step enhances the effective diffusivity of the equation. After an iteration up the scales, this reveals an effective diffusivity of order one on the macroscopic scale, which implies anomalous diffusion. 
In Section~\ref{ss.outline}, below, we review the motivation for the construction of the vector field and give an outline of the reiterated homogenization method used to prove Theorem~\ref{t.anomalous.diffusion}. In Section~\ref{ss.uniform.estimates} we discuss future extensions of this result. 

\smallskip

Theorem~\ref{t.anomalous.diffusion} and its proof provide only an example of a vector field $\b(t,x)$ such that the advection-diffusion equation displays anomalous diffusion. 
Examples---the more physically realistic the better---are certainly useful for building intuition about very complex phenomena. However, we believe that the main value of this work is in the proof strategy, which is a demonstration of the possibility of rigorously proving anomalous diffusion by analyzing the backwards cascade of eddy diffusivities via quantitative homogenization techniques.
We expect this point of view to be robust and of independent interest to the broader area of rigorous hydrodynamic turbulence.

\subsection{Motivation and prior results on anomalous dissipation of scalar variance}
\label{ss.litreview}

If the vector field~$\b(t,x)$ has significantly more spatial regularity than~\eqref{e.b.reg}---for example, if it belongs to~$L^1_t C^{0,1}_x$---then the flows determined by the vector field are well-defined and the corresponding transport equation is well-posed (by standard ODE theory), which then must be the equation satisfied by the limit as~$\kappa \to 0$ of the solutions~$\theta^\kappa$. 
Consequently, as the flows are measure-preserving by~\eqref{e.divfree}, we deduce that
\begin{equation}
\label{e.no.L2.decay}
\lim_{\kappa \to 0}
\| \theta^\kappa(t,\cdot) \|_{L^2(\TT^d)} = \| \theta_0 \|_{L^2(\TT^d)} 
\,, \quad \forall t\in (0,\infty)\,.
\end{equation}
In view of~\eqref{e.energy.balance}, 
this limit is equivalent to 
\begin{equation}
\label{e.lim}
\lim_{\kappa \to 0} 
\kappa \|\nabla\theta^\kappa\|_{L^2((0,t)\times \TT^d)}^2
=0\,, 
\quad \forall t\in (0,\infty)\,,
\end{equation}
which is evidently in contrast to the conclusion of Theorem~\ref{t.anomalous.diffusion}. 

\smallskip

If the limit in~\eqref{e.lim} does not hold, then we speak of \emph{anomalous dissipation of scalar variance} or, alternatively, \emph{anomalous diffusion}. 
It is widely expected that solutions of~\eqref{e.passive.scalar} with vector fields~$\b(t,x)$ which are rougher than Lipschitz in space (but still H\"older continuous) may exhibit anomalous dissipation of scalar variance. This prediction was first discussed by Obukhov in~\cite{Obukhov}.

Indeed, anomalous diffusion is presumed to occur for vector fields describing the velocity of a turbulent fluid, and is a basic assumption in phenomenological theories of scalar turbulence in the physics literature. This remarkable prediction that the rate of dissipation is independent of~$\kappa$, when~$\b(t,x)$ describes a turbulent flow, is backed by very strong experimental and numerical evidence~\cite{shraiman,Warhaft,DSY}.

\smallskip

The reason that anomalous diffusion is expected to hold for ``turbulent'' velocity fields is explained in the physics literature roughly as follows.
A characteristic of a turbulent velocity field~$\b(t,x)$ is that it exhibits activity across a large range of length scales. Advection by the velocity field rearranges the level sets of the scalar~$\theta^\kappa$, creating wiggles on smaller length scales, which are then mixed by the features of the velocity field on those smaller scales. This process continues across a large number of length scales, called the \emph{inertial-convection range}, with smaller and smaller spatial oscillations created. Finally, the wiggles in the scalar reach down to the very small scale at which the molecular diffusivity dominates advection, at which point they are dissipated away. 

\smallskip

A rigorous theoretical explanation of this phenomenon is still elusive. In fact, the mathematical analysis seems to lag the phenomenological theories by so much that not even a \emph{satisfactory example} of anomalous dissipation for passive scalars is available (the few available results are discussed in~Section~\ref{ss.litreview}, below). 

\smallskip

It is not hard to see why this is so: the physicists' explanation is the only way anomalous dissipation can happen. 
For very small~$\kappa$, the diffusion term~$\kappa \Delta$ essentially acts only on very small length scales---otherwise its effect is negligible and the advection term dominates. But it is clear from the identity~\eqref{e.energy.balance} that the diffusion term is the only thing can be responsible for dissipation. If anomalous dissipation is observed, it must be the vector field that is responsible for pushing the oscillations of the scalar into smaller and smaller scales. Since wiggles in the vector field interact with those of the scalar only if their wave numbers are separated by at most an order of magnitude,\footnote{This is due to the incompressibility condition~\eqref{e.divfree}, and the implicit assumption that~$\b(t,x)$ is continuous. If the constraint~\eqref{e.divfree} is dropped, then it is easy to make examples, for instance by creating a vector field which pushes all particles into a small neighborhood of the origin before suddenly pushing them away in radial directions. If the vector field is allowed to be very rough in time, then small scales can also be created fairly easily, as discussed below.} 
it follows that both the vector field and the scalar must have a large number of active scales whose interactions span the range from the macroscopic scale to the ``inertial'' scale on which the diffusion is felt. Since the~$\theta^\kappa$ depends on~$\b(t,x)$ in a highly nontrivial, nonlinear fashion, it is very challenging to analyze such a situation---even if one is permitted to construct the vector field.

\smallskip
 
There are essentially only two known classes of examples which exhibit anomalous dissipation of scalar variance. The first is a stochastic model which is very rough in time (the Kraichnan model), and the second is a class of deterministic vector fields which are ``quasi self-similar'' and have only one active scale at each time (the singularly focusing alternating shear flows).

\subsubsection*{\bf The Kraichnan model}
Kraichnan introduced in~\cite{K1} a simplified model for passive scalar turbulence, one of the early examples of ``synthetic turbulence''. He proposed that~$\b= \b^\nu$ is taken to be a realization of a statistically homogeneous, isotropic, stationary Gaussian random field, which has zero mean, is very rough in time (it has white-noise correlation), and is colored in space (with a Kolmogorov-type scaling of increments in space, above a certain scale).\footnote{More precisely, $\b = \b^\nu$ (here $\nu$ denotes an inverse Reynolds number) has covariance $\langle (\b_i^\nu(x,t)-\b_i^\nu(y,t))(\b_j^\nu(x,s)-\b_j^\nu(y,s))\rangle = \delta(t-s) D_{ij}(x-y)$, where the matrix $D_{ij}$ is symmetric, it has incompressible rows $\partial_{j} D_{ij} = 0$, and most importantly, the diagonal entries satisfy $D_{ij}(z) = A |z|^{2\alpha}$ for $\ell_\nu \ll |z| \ll 1$, and $D_{ij}(z) = B |z|^2$ for $|z|\ll \ell_\nu$. Here $\alpha \in (0,1]$ measures the space H\"older regularity of the field in the inertial range, and $\ell_\nu$ is the dissipative scale.  The infinite Reynolds number limit corresponds to $\ell_\nu \to 0$ as $\nu\to 0$. See e.g.~\cite{kupiainen2003},~\cite{gawedzki2008},~\cite[(2.26)--(2.27)]{DrivasEyink17}.} Then one is to study the statistics of the field $\theta^\kappa$ solving \eqref{e.passive.scalar} (understood in the Stratonovich sense,  $d\theta^\kappa - \kappa \Delta \theta^\kappa dt = d \b^\nu \circ \nabla \theta^\kappa$). The main result concerning anomalous diffusion~\eqref{e.anomalous.diffusion} in the joint $\nu,\kappa \to 0$ limit, was established by  Bernard, Gawedzki, and Kupiainen~\cite{Bernard}. See also~\cite{gawedzki2000,evanden2000,weinan2001,lejan2002} for further results and refinements. Moreover, the Lagrangian flows $\xi^{\nu,\kappa}$ become non-unique and stochastic in the $\nu,\kappa \to 0$ limit, for a fixed initial particle position and a fixed velocity realization $\b^\nu$. This phenomenon is called \emph{spontaneous stochasticity}. In fact, it was shown by Drivas and Eyink~\cite{DrivasEyink17} that spontaneous stochasticity is equivalent to anomalous dissipation, not just for the Kraichnan model, but for any passive scalar transport of the type \eqref{e.passive.scalar} (even in the presence of boundaries). We refer to~\cite{falkovich2001,kupiainen2003,gawedzki2008,DrivasEyink17} for excellent discussions about the Kraichnan model. 

The main drawback of this model stems from the white-noise temporal correlation of the vector field $\b(t,x)$, which is indeed so rough that it is probably responsible for the anomalous diffusivity. At the experimental level, a consequence of this roughness was already noted by Sreenivasan and Schumacher~\cite{sreenivasan2010}: there are several differences between the predictions of the Kraichnan model and the behavior of a passive scalar in Navier-Stokes turbulence.  At the mathematical level, the  white-noise temporal correlation allows for a certain explicit and exact computation of the statistics of the solution. Namely, one may obtain closed expressions for the correlation functions of the scalar $\theta^\kappa$; therein, the assumed white in time correlation structure of the velocity field plays a crucial role. As a consequence, the ``exact analysis'' developed for the Kraichnan model is not robust, and it did not allow the fluid dynamics community to build sturdy tools for understanding the energy cascade in Navier-Stokes turbulence.  

Nonetheless, as noted by Majda and Kramer in~\cite{MK}, exactly solvable models provide excellent test problems for assessing the strengths and weaknesses of  approximate closure theories in turbulence. Besides the Kraichnan model discussed here, an exact mathematical analysis of diffusion (enhancement and anomalies) is also available for the ``Simple Shear Models'' of Avellaneda and Majda~\cite{AvM1,AvM2, AvM2a}, which generalize an earlier model of Kubo~\cite{Kubo}. These examples emphasize how randomly fluctuating velocity fields act as effective diffusion processes, on large scales and long times. The vector fields in~\cite{AvM1,AvM2, AvM2a} are of a shear flow type $\b(t,x) = (w(t),v(t,x_1))$, where the spatially uniform sweeping component $w(t)$ is taken as a stationary random process with possibly nonzero mean, and the shearing component $v(x_1,t)$ is taken as a homogeneous and stationary, mean zero random field, whose statistics can be fine tuned to match the statistically stationary turbulent flows. Using exactly solvable renormalization group theories and Lagrangian renormalized perturbation theories (available for these simple shear flows), Avellaneda and Majda are able to identify several distinct regimes, as indexed by the mean of $w$, the strength of the infrared divergence in $v$, and the decorrelation time of long-wave portions of the statistical velocity spectrum. Note however that anomalous diffusion~\eqref{e.anomalous.diffusion}, is  {\em not} available in the  ``Simple Shear Models'' of~\cite{AvM1,AvM2, AvM2a}.

\subsubsection*{\bf Singularly-focusing alternating shear flows} To the best of our knowledge, the first example of a deterministic vector field $\b(t,x)$, for which the anomalous dissipation of scalar variance~\eqref{e.anomalous.diffusion} is established rigorously, was recently constructed by Drivas, Elgindi, Iyer, and Jeong~\cite{DEIJ}. In~\cite[Theorem 1]{DEIJ}, it is shown that for any $\alpha \in [0,1)$ and $d\geq 2$, there exists a vector field~$\b \in L^1([0,1];C^\alpha(\TT^d)) \cap L^\infty([0,1];L^\infty(\TT^d))$, such that the following holds: $\b(t,\cdot)$ is smooth for any $t<1$; for any initial data with $\theta_0 \in H^2$ which is sufficiently close (in $L^2$) to a an eigenfunction of the Laplacian, anomalous diffusion~\eqref{e.anomalous.diffusion} holds for some $\varrho \in (0,1)$; and the scalar field $\theta^\kappa$ remains uniformly bounded in $L^\infty([0,1];L^\infty(\TT^d))$ as $\kappa \to 0$.
The above result is sharp in the sense that if\footnote{The Lipschitz regularity may be replaced with merely the integrability of $\nabla \b$. Indeed, it follows from the Di Perna-Lions theory~\cite{diperna1989} that as soon as $\b \in L^1_t W^{1,1}_x$ is divergence free, all bounded weak solutions of the transport equation $\partial_t \theta + \div(\b \, \theta)= 0$ are renormalized, and thus they conserve the energy $\tfrac 12 \| \theta(t,\cdot)\|_{L^2}^2$. See also the work of Ambrosio~\cite{Ambrosio04} for $\b \in L^1_t {\rm BV}_x$, divergence free. Then, as $\kappa \to 0^+$ the a priori (subsequential) weak convergence of $\theta^\kappa$ to a weak solution $\theta$ of the transport equation, is in fact strong (due to the energy balance \eqref{e.energy.balance} and lower-semicontinuity), implying that there is no dissipation anomaly.} $\b \in L^1_t W^{1,\infty}_x$ (corresponding to $\alpha = 1$), then trivially one has $\lim_{\kappa \to 0^+} \kappa \langle |\nabla \theta^\kappa|^2 \rangle = 0$, as discussed in the first paragraph of Section~\ref{ss.litreview}. 

\smallskip

In essence, the construction of the vector field $\b(t,x)$ in~\cite{DEIJ} alternates shear flows with stream function\footnote{The sinusoidal shear velocity profiles are replaced by a smoothened sawtooth function, which makes the computations easier, and in fact almost explicit.} $\sin(2^{(1+\alpha)j} x_{1+j\mod d})$, on intervals $[t_{j-1}, t_j)$, for $\{t_j = 1- 2^{-j} \colon j\geq 1\}$. This construction is on the one hand inspired by the earlier work of Pierrehumbert~\cite{pierrehumbert1994}, who proposed an alternating shear flow of a single frequency, but with random i.i.d. phase shifts, to construct a ``universal mixer'' for the transport equation.\footnote{The proof that the Pierrehumbert construction indeed an universal exponential mixer was recently obtained by Blumenthal, Coti Zelati, and Gvalani~\cite{blumenthal2022}, using a random dynamical systems based perspective.} On the other hand, the idea of a quasi self-similar evolution on  $[t_{j-1},t_j)\times \TT^d$ which singularly focuses as $j\to \infty$ all the ``action'' towards the final time slice $\{t=1\}\times \TT^d$---where all the anomalous dissipation of scalar variance occurs---is inspired by earlier works of Aizenman~\cite{Aizenman78} and Depauw~\cite{Depauw03} concerning the uniqueness of the transport equation below, and Alberti, Bianchini, and Crippa~\cite{ABC} respectively Alberti, Crippa, and Mazzucato~\cite{ACM1,ACM2} regarding mixing for the transport equation.\footnote{The deterministic theory of \emph{mixing} for the linear transport ($\kappa=0$) and of \emph{enhancement of diffusion} for the drift-diffusion ($\kappa>0$) equation~\eqref{e.passive.scalar} is too vast to review here.  Usually these theories consider vector fields $\b$ whose regularity is at least $L^1_t W^{1,p}_x$ for $p\geq 1$, so that the Di~Perna-Lions theory applies to bounded solutions of the scalar linear transport. The questions typically asked are: when $\kappa = 0$, to describe the decreasing function $\varrho(t)$ and the timescale $t_0$ such that  $\|\theta^0(t,\cdot)\|_{H^{-1}} \leq \varrho(t-t_0) \|\theta^0(t_0,\cdot)\|_{H^1}$, see e.g.~\cite{CC08,IKX,Seis,ElgindiUniversal}. In other works, the loss of regularity and nonuniqueness of weak solutions to the continuity equation is discussed~\cite{ABC,Jabin16,YZ,ACM1,ACM2,crippa2022growth} and~\cite{modena2018,brue2021,cheskidov2021}. For $\kappa>0$, it is well-known that some enhancement of diffusion takes place due to mixing properties of the underlying flow of  $\b$~\cite{CKRZ,BCZ,FengIyer,zelati2020relation,zelati2019stochastic}. For such diffusion enhancing flows, the challenge is to quantify the optimal rate $r(\kappa) \gg \kappa$ and the timescale $t_\kappa\ll \kappa^{-1}$ such that $\|\theta^\kappa(t,\cdot)\|_{L^2}^2 \leq C e^{-r(\kappa) t} \|\theta_0\|_{L^2}^2$, for all $t\geq t_\kappa$~\cite{zelati2020relation,zelati2019stochastic,brue2021advection,Elgindi2023}. The dissipation anomaly considered in this paper is an extreme form of enhancement of diffusion, with $r(\kappa) = \OO(1)$ uniformly in $\kappa$ as $\kappa\to 0$.} With $\b(t,x)$ constructed as such, the proof of~\cite{DEIJ} hinges on comparing the family of solutions $\{\theta^\kappa\}_{\kappa>0}$ to a solution $\theta^0$ of the transport equation ($\kappa=0$) which satisfies $\lim_{t\to 1^-} \|\nabla \theta^0\|_{L^2((0,t) \times \TT^d)} = + \infty$ and for which a significant amount of energy travels to higher and higher frequencies as $t\to 1^-$, either as inviscid mixing or as a balanced growth of Sobolev norms, resulting in a lack of compactness at time $t=1$; see the abstract criterion for anomalous dissipation in~\cite[Corollary~1.5]{DEIJ}.

\smallskip

Alternating shear flows which focus in a singular and quasi self-similar way onto a final time slice have also been recently considered by Brue and De Lellis~\cite{BrueDeLellis22}, Colombo, Crippa, and Sorella~\cite{CCS22}, and jointly in~\cite{BCCDLS22}, to give examples of anomalous dissipation of energy for solutions of the \emph{forced} 3D Navier-Stokes equations~\cite{BrueDeLellis22,BCCDLS22}, and to establish anomalous diffusion for the drift-diffusion equation together with \emph{uniform-in-diffusivity H\"older regularity} for the associated passive scalar. At the core of all these works is the anomalous dissipation of scalar variance for the drift-diffusion equation~\eqref{e.anomalous.diffusion}.

Indeed, it is well-known that for $2\frac 12$-dimensional solutions of the 3D Navier-Stokes equations the vertical component of the flow satisfies the linear advection-diffusion equation~\eqref{e.passive.scalar}. More precisely, if $u_H = (u_1,u_2)(x_1,x_2,t)\colon \TT^2 \times \RR \to \RR^2$ and $u_3  = (x_1,x_2,t) \colon \TT^2 \times \RR \to \RR$ satisfy $\partial_t u_H + (u_H \cdot \nabla_H) u_H + \nabla_H p - \nu \Delta_H u_H = f_H$, respectively $\partial_t u_3 + (u_H \cdot \nabla_H) u_3 - \nu \Delta_H u_3 = 0$, where $f_H$ is a horizontal body force, $p$ is a scalar pressure ensuring $\nabla_H \cdot u_H = 0$, and we denote ``horizontal'' differential operators by $\nabla_H= (\partial_{x_1},\partial_{x_2})$ and $\Delta_H=\partial_{x_1 x_1} + \partial_{x_2x_2}$, then the vector field $u = (u_H,u_3)$ solves the 3D Navier-Stokes equations with viscosity $\nu$, pressure $p$, and body force $(f_H,0)$. 
Then, inspired by the constructions in~\cite{Aizenman78,Depauw03,ABC,ACM1,ACM2} the papers~\cite{BrueDeLellis22,CCS22,BCCDLS22} construct both initial data for the scalar $u_3$ (essentially a $\pm 1$ checkerboard at unit scale) and a two-dimensional vector field $u_H$---which is essentially a sequence of alternating shear flows which are quasi self-similar on intervals of the type $[t_{j-1},t_j)$ with amplitudes $a_j$ and frequencies $\lambda_j$, where $t_j \to 0^+$, $a_j, \lambda_j \to \infty$ as $j\to \infty$---such that the the inviscid  transport equation $\partial_t u_3 + (u_H \cdot \nabla_H) u_3 =0$ mixes perfectly as $t \to 1^-$, i.e. $u_3(t,\cdot) \rightharpoonup 0$ as $t\to 1^{-}$. To incorporate the effect of a vanishing sequence of diffusions $\nu_j \to 0^+$ as $j\to \infty$, these authors smooth out the aforementioned vector field at a specific $(a_j,\lambda_j,t_j,\nu_j)$-dependent scale, and then either appeal to the abstract criterion from~\cite{DEIJ} or directly measure the variance of the associated stochastic process, to show that the drift-diffusion equation may be viewed as a perturbation of the transport equation, and hence exhibits anomalous diffusion. The term $f_H$ is then just the remainder obtained by inserting the constructed vector $u_H$ into the horizontal part of the 3D Navier-Stokes equations. A clever fine-tuning of the parameters $(a_j,\lambda_j,t_j,\nu_j)$ in the construction attains both the uniform H\"older regularity of the sequence $\{ u^{\nu_j} \}_{j\geq 1}$ (in the full range strictly below $L^1_t W^{1,1}_x$), and fact that $\limsup_{j\to \infty} \nu_j \|\nabla u^{\nu_j}\|_{L^2((0,1)\times\TT^3)}^2 > 0$. As in~\cite{DEIJ}, in these constructions the anomalous dissipation occurs only on the time slice $\{t=1\} \times \TT^d$. We also note that by adding an extra space dimension to replace time, Johansson and Sorella~\cite{JohanssonSorella2023} have obtained similar results for the advection-diffusion equation in dimensions larger than $3$, for a vector field which is \emph{autonomous}; here, the quasi self-similar singular focusing is achieved on a ``last space slice'' instead of a ``last time slice'' (see also~\cite[Figure~3]{Aizenman78} for a closely related idea).

\smallskip
 
The main drawbacks of the aforementioned constructions of~\cite{DEIJ} and of~\cite{BrueDeLellis22,CCS22,BCCDLS22} are as follows: (i) all the energy that can be dissipated anomalously is dissipated at only one instant in time, (ii) the vector field $\b(t,x)$ has  only one active scale at each time $t \in[0,1)$, (iii)  the drift-diffusion equation is treated as a perturbation of the transport equation, and (iv) the vector field $\b(t,x)$ and the initial datum $\theta_0$ are not constructed independently of each other, and the diffusive anomaly is not proved for all smooth initial data.

\smallskip

Regarding point (i), we note that the existence of a single time (e.g.~$t=1$  for~\cite{DEIJ,BrueDeLellis22,CCS22,BCCDLS22}) at which all of the anomalous diffusion occurs, is incompatible with the (statistical) \emph{stationarity} of the turbulent vector fields, for which anomalous diffusion has been robustly observed in practice. In contrast, the vector field $\b(t,x)$ which we construct in Theorem~\ref{t.anomalous.diffusion} does not distinguish any special times, and for $t_1,t_2 \in [0,1]$ chosen at random, $\b(t_1,\cdot)$ and $\b(t_2,\cdot)$ have the same regularity, are macroscopically undistinguishable. This means in particular that our vector field does not quasi self-similarly focus the dynamics onto a single time slice, leading us to point (ii). In the previous examples of anomalous diffusion~\cite{DEIJ,BrueDeLellis22,CCS22,BCCDLS22} at each instance of time $t \in [0,1)$ only one shear flow is active (at a suitable spatial frequency), which in turn necessitates singular focusing in time for the passive scalar to witness infinitesimally small scales in $\b(t,x)$. This picture is inconsistent with the observed power spectra of turbulent flows in statistical equilibrium~\cite{Frisch}. The vector field $\b(t,x)$ from Theorem~\ref{t.anomalous.diffusion} does not have this property: at a.e.~$t\in [0,1]$ the vector field $\b(t,\cdot)$ contains infinitely many shear flows of diverging frequencies, which are twisted by the Lagrangian flows induced by the sum of the flows at all scales ``above'' that of the shear being considered. At first sight, one may think that this ``feature'' of $\b(t,x)$ comes with a ``bug'': the underlying transport equation is severely ill-posed, leading us to point (iii). At the heart of the proofs in~\cite{DEIJ,BrueDeLellis22,CCS22,BCCDLS22}, the transport equation does the heavy lifting, in a quasi self-similar fashion as $t\to 1^-$. In a sense, it is shown that the non-diffusive picture  is stable in $L^2$ under diffusive perturbations. Our work presents a fundamental difference, as we do not  view \eqref{e.passive.scalar} as a perturbation of the transport equation $(\partial_t + \b \cdot \nabla) \theta = 0$. In fact, the diffusion is used in a fundamental way in the proof (see Section~\ref{ss.outline}). At each scale larger than the smallest active scale (determined by $\kappa$) the advection part of the operator is in balance with a renormalized diffusion operator. A welcome consequence of this perspective and of this proof strategy is that in our analysis the vector field $\b(t,x)$ and the initial data $\theta_0$ are independent of each other, with \eqref{e.anomalous.diffusion} holding for every~$\theta_0 \in \dot{H}^1(\TT^d)$. This ``universality'' was however not present in any of the earlier works on this subject, as mentioned in point (iv) above. In~\cite{BrueDeLellis22,CCS22,BCCDLS22} the main results establish the existence of both a vector field~$\b(t,x)$ and of an initial datum~$\theta_0$ (a~$\pm 1$ checkerboard) for which~\eqref{e.anomalous.diffusion} holds, while in~\cite[Theorem 1]{DEIJ} the initial datum needs to be sufficiently close (with respect to the~$H^2$ topology) to an eigenfunction of the Laplacian on~$\TT^d$.\footnote{This limitation also applies to the constructions based on intermittent convex integration schemes~\cite{Buck2020} applied to the transport and drift diffusion equations e.g.~in~\cite{modena2018,brue2021,cheskidov2021,pitcho2021}: all of these construct the vector field at the same time as the scalar.} This is of course not the physically motivated problem since the turbulent vector field~$\b(t,x)$ should be given in advance (as a solution of, say, 3D Navier-Stokes), and then the passive scalar is to be advected and diffused in this flow. The reason why Theorem~\ref{t.anomalous.diffusion} yields anomalous diffusion for all~$H^1$ initial data of zero mean is not the construction of the vector field~$\b(t,x)$ per se, it is the proof strategy, which shows that the quantity~$\kappa \|\nabla \theta^\kappa\|_{L^2((0,1)\times \TT^d)}^2$ is close (in a $\kappa$-independent sense) to the rate of diffusion experienced by (essentially) a heat equation with the same initial datum, and unit-size diffusivity coefficient.

\subsection{An outline of the proof: fractal homogenization}
\label{ss.outline}

We present the proof of Theorem~\ref{t.anomalous.diffusion} only in dimension~$d=2$ rather than a general dimension~$d\geq 2$ for convenience and readability. The argument in higher dimensions has only notational differences. 

\smallskip

As mentioned above, the proof of Theorem~\ref{t.anomalous.diffusion} is based on the idea that anomalous diffusivity is the consequence of a ``homogenization cascade'' of ``eddy diffusivities,'' which goes from small scales to large scales. We think of each homogenization step as modifying the equation by removing the fastest wiggles in the vector field and---due the enhancement of diffusivity caused by these wiggles---increasing the diffusivity parameter~$\kappa$. The ``effective diffusivities'' thereby increase as we zoom out to larger scales, until finally, at the macroscopic scale, the vector field has no remaining wiggles and the effective diffusivity is of order one. This strategy, which is a renormalization group-type approach, has a very long history dating back to the 19th century (see~\cite[Chapter 9]{Frisch}).

\smallskip

In this subsection, we will give a complete overview of the main ideas behind the construction of the vector field~$\b(t,x)$ and the proof of anomalous dissipation of scalar variance. The full proof is very lengthy, as the justifications of many of the intuitions here require long computations and many estimates.

\subsubsection*{\bf Advection-enhanced diffusion and homogenization}
We briefly review the phenomenon of advection-enhanced diffusion, from the point of view of classical homogenization. Consider a~$\Z\times \Zd$-periodic, mean-zero, incompressible vector field~$\mathbf{u}(t,x)$ and the advection-diffusion equation
\begin{equation}
\label{e.advec.u}
\partial_t \theta_\ep - \kappa \Delta \theta_\ep + \tfrac1\ep \mathbf{u}(\tfrac{t} {\ep^2},\tfrac{x}{\ep} ) \cdot \nabla \theta_\ep  = 0 
\,.
\end{equation}
The advection term may be expressed as a second-order term:
\begin{equation*}
\tfrac1\ep\mathbf{u}(\tfrac{t} {\ep^2},\tfrac{x}{\ep} ) \cdot 
\nabla \theta_\ep = 
- \nabla \cdot \bigl( \s(\tfrac{\cdot} {\ep^2},\tfrac{\cdot}{\ep} )\nabla \theta_\ep \bigr), 
\end{equation*}
where~$\s$ is a \emph{stream matrix} for~$\mathbf{u}$, that is, an anti-symmetric matrix such that~$-\nabla\cdot \s = \mathbf{u}$. This allows us to write~\eqref{e.advec.u} as
\begin{equation}
\label{e.advec.s}
\partial_t \theta_\ep - \nabla \cdot \bigl( \kappa \Id + \s\bigl( \tfrac{\cdot} {\ep^2},\tfrac \cdot \ep \bigr) \bigr) \nabla \theta_\ep  = 0 \,.
\end{equation}
Here there are two scales: the small scale~$\ep>0$ on which the stream matrix oscillates, and the macroscopic scale which is of order one. Classical homogenization theory says that~\eqref{e.advec.s} homogenizes to the \emph{effective equation}
\begin{equation*}
\partial_t \overline{\theta} - \nabla \cdot \ahom \nabla \overline{\theta}  = 0 \,,
\end{equation*}
in the sense that, roughly speaking, solutions of the former converge in~$L^2$, as~$\ep \to 0$, to those of the latter. 
The effective diffusion matrix~$\ahom$ is given by the formula
\begin{equation*}
\ahom e = \bigl\langle\!\!\bigl\langle \bigl( \kappa \Id + \s \bigr) (e + \nabla \chi_e )
\bigl\rangle\!\!\bigl\rangle \,, \quad e \in\Rd\,,
\end{equation*}
where~$\langle\hspace{-2.5pt}\langle\cdot \rangle\hspace{-2.5pt}\rangle$ denotes the (space-time) average of a~$\Z\times\Zd$--periodic function and~$\chi_e$ is the \emph{corrector} with slope~$e$, that is, the unique periodic (in space and time), mean-zero solution of the cell problem
\begin{equation*}
\partial_t \chi_e -\nabla \cdot \bigl( \kappa \Id + \s \bigr)(e+\nabla \chi_e ) = 0 
\,.
\end{equation*}
The symmetric part of~$\ahom$ is given by 
\begin{equation*}
\tfrac12 (\ahom + \ahom^t) _{ij} 
=
\kappa \delta_{ij} 
+ 
\kappa \, \bigl\langle\!\!\bigl\langle \nabla \chi_{e_i} \cdot \nabla \chi_{e_j}\bigl\rangle\!\!\bigl\rangle\,.
\end{equation*}
The second term is positive (in the ordering of nonnegative definite matrices), and therefore the symmetric part of~$\ahom$ is larger than the original diffusion matrix~$\kappa \Id$. This effect is called the enhancement of diffusivity due to advection. 

\smallskip

The enhancement of diffusivity from the point of view of homogenization has been well-studied over the past four decades. There are too many works to cite here, so we refer the reader to~\cite{FP1,MK} and the references therein. 
The proposal to use homogenization methods to turbulence models has however received a great deal of skepticism, primarily due to the lack of asymptotic scale separation. The homogenization limit requires sending the parameter~$\ep$, representing the ratio of the two scales, to zero. 
Such criticisms can be found in~\cite[page 225]{Frisch} and~\cite[page 304]{MK}.

\smallskip

Indeed, the vector field~$\b(t,x)$ we will construct will have certain active scales and the ratio of any pair of these active scales is fixed and not parametrized by a parameter being sent to zero. Moreover, we have infinitely many active scales, and not only two as in the simple setup described above. These issues pose serious analytic challenges, and we will address them using \emph{quantitative} homogenization methods. Rather than reason in terms of asymptotic limits, we need to precisely quantify the length scales and time scales on which homogenization occurs. We will next consider this question in the context of a simple shear flow.

\subsubsection*{\bf Homogenization of a simple shear flow}

The vector field~$\b(t,x)$ will have a fractal-like structure, and so we need to introduce the ``base'' of the fractal, that is, the pattern which links two different scales and will be repeated infinitely many times. This role will be served by a simple alternating shear flow. 

\smallskip

Given parameters~$a,\ep>0$, consider the simple time-independent shear flow~$\mathbf{u}(x)$ defined by
\begin{equation*}
\mathbf{u}(x) = \begin{pmatrix}
0 \\ 
2\pi a \ep \cos(\frac{2\pi x_1}{\ep}) 
\end{pmatrix}
\,.
\end{equation*}
The length scale on which the shear flow varies is~$\ep$, and the parameter~$a$ represents the size of the Lipschitz norm of the vector field. The stream function for $\mathbf{u}$ is~$\psi_{\ep}(x) = a \ep^2 \sin(\frac{2\pi x_1}{\ep})$; in other words, $\mathbf{u} = \nabla^\perp \psi_\ep$. We stress here that~$\ep$ is not a parameter to be sent to zero, it just represents the inverse wave number of~$\mathbf{u}$. 

\smallskip

The equation for a passive scalar~$\theta$ advected by~$\mathbf{u}$ with diffusivity~$\kappa>0$
can be written as 
\begin{align}
\label{e.heur.aep}
\partial_t \theta 
-\nabla \cdot \bigl( \mathbf{K}^\ep( x) \nabla \theta \bigr)
 = 0
\end{align}
where~$\mathbf{K}^\ep$ is the non-symmetric matrix  
\begin{align}
\label{e.aep.def}
\mathbf{K}^\ep(x) 
= \bigl( \kappa \Itwo + \psi_\ep(x) \sigma\bigr) 
=
\begin{pmatrix}
\kappa & - a \ep^2 \sin(\frac{2\pi x_1}{\ep}) \\
a \ep^2 \sin(\frac{2\pi x_1}{\ep}) & \kappa
\end{pmatrix}
\,,
\end{align}
and~$\sigma$ is defined in~\eqref{e.sigma}, below. 
Due to homogenization, we expect that~\eqref{e.heur.aep} should be close, on large enough scales, to its effective equation
\begin{equation*}
\partial_t \theta - \nabla \cdot \Khom_2 \nabla \theta = 0,
\end{equation*}
where the effective diffusivity matrix~$\Khom_2$ in this case can be computed explicitly. It is:
\begin{equation}
\label{e.Khom.int}
\Khom_{2} 
=
\begin{pmatrix}
\kappa &0 \\
0 & \kappa + \frac{a^2\ep^4}{\kappa}
\end{pmatrix}
\,.
\end{equation}
We again stress that we are not sending~$\ep \to 0$ here, the homogenization is with respect to a large-scale limit.

The formula~\eqref{e.Khom.int} was derived more than 70 years ago by Taylor~\cite{Taylor} and is sometimes called the \emph{Taylor dispersion formula} (specialized to the shear flow); see also~\cite{Aris}. As he showed, homogenization is observed on length scales much larger than~$a\ep^3\kappa^{-1}$ and time scales much larger than~$\ep^2\kappa^{-1}$. This can be proved analytically from estimates on the correctors, which can be computed explicitly in this case. 

\smallskip

There is another way to think about this, in terms of the particle trajectories. The diffusion~$Y_t$ process corresponding to~\eqref{e.heur.aep} satisfies the SDE
\begin{equation}
\label{e.SDE}
dY_t = \mathbf{u}(t,Y_t)dt +  \sqrt{2\kappa} dW_t
\,.
\end{equation}
The particle evolving according to these dynamics will move with speed of order~$a\ep$ in the~$x_2$ direction, changing its direction (up or down) and its magnitude on time scales of order $\ep^2 \kappa^{-1}$, which is the time it takes the diffusion to alter its $x_1$ coordinate on the order of~$\ep$. The vector field has typical size $a\ep$, therefore in this time the particle will have travelled a distance of order~$a\ep \cdot \ep^2 \kappa^{-1} = a\ep^3\kappa^{-1}$. If we zoom out and observe the motion of the particle on length scales much larger than~$a\ep^3\kappa^{-1}$ and time scales much larger than~$\ep^2 \kappa^{-1}$, then what we see (roughly) is that the~$x_2$ coordinate of the particle is performing a random walk with steps of size~$a\ep^3\kappa^{-1}$, with~$\ep^2 \kappa^{-1}$ units of time between steps. This leads to a diffusivity in the $x_2$ direction of order 
\begin{align*}
\frac{\bigl( a\ep^3\kappa^{-1}\bigr)^2 }{\ep^2 \kappa^{-1}}
=
\frac{a^2\ep^4}{\kappa}
\,.
\end{align*}
Of course, this diffusive effect caused by the advection should be in addition to the molecular diffusion, so we expect to find an effective diffusion of order 
\begin{equation*}
\mbox{effective diffusivity in~$x_2$ direction} 
=
O\biggl(
\kappa + \frac{a^2 \ep^4}{\kappa}
\biggr)
\,.
\end{equation*}
This rough intuition is in agreement the more precise formula~\eqref{e.Khom.int}. 

\smallskip

The dimensionless quantity~$a \ep^2\kappa^{-1}$, recognized as representing the (square root of the) ellipticity contrast in the matrix~$\mathbf{K}^\ep$ defined in~\eqref{e.heur.aep}, is a measure of the strength of the shear flow term relative to the molecular diffusion. It determines the multiple of the small scale~$\ep$ on which homogenization occurs. 

\smallskip

Now consider a vector field which alternates between shear flows in the~$x_1$ direction and shear flows in the~$x_2$ direction with frequency~$\tau^{-1}$:
\begin{equation}
\label{e.vepatau}
\mathbf{v}_{\ep,a,\tau} (t,x):= 
2\pi a \ep\sum_{k\in \Z} 
\indc_{[k\tau,(k+1)\tau)}(t) 
\biggl( 
\begin{pmatrix}
0 \\ 
\cos(\frac{2\pi x_1}{\ep}) 
\end{pmatrix}
\indc_{\{k \in 2\Z\}} 
+
\begin{pmatrix}
- \cos(\frac{2\pi x_2}{\ep})  \\ 
0
\end{pmatrix}
\indc_{\{k \in 2\Z+1\}} 
\biggr)\,.
\end{equation}
If we require that~$\tau \gg \ep^2\kappa^{-1}$, so that the shear flows have enough time to homogenize, then the corresponding advection-diffusion equation homogenizes to the average of~$\Khom_2$ and the analogous matrix~$\Khom_1$ with the diagonal entries swapped, which is conveniently isotropic.
We find that, on length scales much larger than~$a\ep^3\kappa^{-1}$ and time scales much larger than~$\tau$, the equation with alternating shear flows will homogenize to 
\begin{equation*}
\partial_t \theta -  \overline{\kappa} \Delta \theta = 0,
\end{equation*}
where the effective diffusivity is given by Taylor's formula,~$\overline{\kappa} = \bigl( 1 + \frac{a^2\ep^4}{2\kappa^2} \bigr)\kappa$.

\subsubsection*{\bf The construction of the multiscale vector field~$\b(t,x)$}

The above discussion suggests an idea for setting up a ``homogenization cascade'' by constructing a vector field~$\b(t,x)$ with many copies of the alternating shear flows on different scales. 
We look for a decreasing sequence of length scales~$\ep_m\to 0$, of time scales~$\tau_m\to 0$, and of diffusivities~$\kappa_m\to 0$ and an increasing sequence of parameters~$a_m\to \infty$ which satisfy the following relations:
\begin{equation}
\label{e.four.contraints}
\left\{
\begin{aligned}
& \kappa_{m-1} = \Bigl( 1 + \frac{a_{m}^2\ep_{m} ^4}{2\kappa_{m}^2} \Bigr) \kappa_{m} \,,
\\ & 
a_m = \ep_m^{\alpha-1}\,,
\\ &
\tau_{m} \gg \frac{\ep_m^2}{ \kappa_m}\,,
\\ & 
\ep_{m-1} \gg 
\Bigl(\frac{a_m\ep_m^2}{\kappa_m} \Bigr)
\ep_{m}\,.
\end{aligned}
\right.
\end{equation}
The last condition is to ensure that the wiggles we put in the vector field at scale~$\ep_{m-1}$ do not interfere with the homogenization of those at scale~$\ep_m$. 
The condition on~$a_m$ is because we want the vector field to be H\"older continuous with exponent~$\alpha\in(0,1)$.
We would then like to define a vector field~$\b(t,x)$ in a recursive way, by adding shear flows at each scale~$\ep_m$, roughly as follows: set~$\b_0 = 0$, and then define
\begin{equation*}
\b_{m}(t,x) := \b_{m-1}(t,x)
+
\mathbf{v}_{\ep_m,a_m,\tau_m} (t,x)\,.
\end{equation*}
The idea is that the vector field~$\b_{m-1}$ is ``macroscopic'' from the point of view of~$\mathbf{v}_{\ep_m,a_m,\tau_m}$, which will homogenize before spatial or temporal variations in~$\b_{m-1}$ are noticed. 
We will then define~$\b:= \lim_{m\to \infty} \b_{m}$. Note that this limit makes sense, due to the fact that the supremum of~$|\mathbf{v}_{\ep_m,a_m,\tau_m}|$ is of order~$a_m \ep_m = \ep_m^\alpha$, which is small and can be summed up (since the scales will be at least geometrically separated).
The hope is then that we have set up the parameters in such a way that the advection-diffusion equation with diffusivity~$\kappa_{m}$ and vector field~$\b_m$ will homogenize to the one with diffusivity~$\kappa_{m-1}$ and vector field~$\b_{m-1}$.

\smallskip

This however will not work without another crucial modification. The presence of the ``macroscopic'' vector field~$\b_{m-1}$ actually interferes with the homogenization of the wiggles represented by~$\mathbf{v}_{\ep_m,a_m,\tau_m}$. Indeed, this ``macroscopic'' term is essentially a constant from the point of view of the much faster field~$\mathbf{v}_{\ep_m,a_m,\tau_m}$, and a large constant drift added to a shear flow essentially destroys the shear flow structure, with its very long streamlines, and consequently removes most of the enhancement of the diffusivity. In fact, a large constant background drift will destroy the enhancement of diffusivity of any time-independent flow: see Appendix~\ref{a.nosweep} for details.

\smallskip

This presents a serious obstacle to building examples of continuous vector fields which exhibit of anomalous diffusion, since continuity implies that larger wave numbers should have larger amplitudes. The solution to this problem is to force the small scale shear flow~$\mathbf{v}_{\ep_m,a_m,\tau_m}$ to be swept by the vector field~$\b_{m-1}$, so that they appear to be stationary in the moving reference frame of a particle advected by~$\b_{m-1}$. We do this by modifying the definition of~$\b_m$ as follows:
\begin{equation*}
\b_{m}(t,x) := \b_{m-1}(t,x)
+
\mathbf{v}_{\ep_m,a_m,\tau_m}( t, X_{m-1}^{-1}(t,x) ) \,,
\end{equation*}
where~$X_{m-1}$ is the flow for the vector field~$\b_{m-1}$, that is, the solution of~$\partial_t X = \b_{m-1}(t,X)$.
If changed into Lagrangian coordinates, then the~$\b_{m-1}$ term would disappear, and the Laplacian term would only be slightly distorted. In this way, the vector field~$\mathbf{v}_{\ep_m,a_m,\tau_m}$ can be homogenized without disturbing~$\b_{m-1}$. This ``self-advection'' property---arising here naturally from the renormalization perspective as a way to gain a sufficient enhancement of diffusivity between two scales---is a property that real fluids have.\footnote{This property is not shared by many other models of synthetic turbulence, to our knowledge. See~\cite{TD,EB} for a discussion of this point, and the apparent disparities between models of synthetic turbulence and real turbulent flows. Summarizing~\cite{TD}, the authors of~\cite{EB} write ``The key point is that large-scale eddies in real turbulence advect both particles and smaller scale eddies, while large-scale eddies in synthetic turbulence advect only particle pairs and not smaller eddies.'' We remark that a related difficulty was faced by Kraichnan in his attempts to build his ``DIA" model of turbulence: see~\cite[Section II.A]{EF}, in particular the discussion of random Galilean invariance.}

\smallskip

This introduces a new complexity to our construction, because the inverse flows must be renewed on a time scale which is much less than the inverse of the Lipschitz constant of~$\b_{m-1}$, which is of order~$a_m$. Otherwise the distortion due to the flows becomes intractable. 
Therefore we modify our definition again by introducing~$\tau_m^{\prime\prime} > \tau$ and defining
\begin{equation}
\label{e.bm.intr}
\b_{m}(t,x) := \b_{m-1}(t,x)
+
\sum_{l\in\Z} 
\indc_{[ l \tau_m^{\prime\prime}, (l+1)\tau_m^{\prime\prime})}(t)
\mathbf{v}_{\ep_m,a_m,\tau_m}( t, X_{m-1}^{-1}(t,x,l\tau_m^{\prime\prime}) ) \,,
\end{equation}
where~$X_{m-1}(t,x,s)$ is the flow for~$\b_{m-1}$ with~$X_{m-1}(s,x,s)=x$. This gives us two new constraints:
\begin{equation*}
\tau_m \ll \tau_m^{\prime\prime} \ll a_{m-1}^{-1}\,.
\end{equation*}
The second constraint ensures that we have good estimates on the difference between our flows and the identity matrix. The first constraint~$\tau_m^{\prime\prime} \gg \tau_m$ is needed because the periodic renewal of the inverse flows has caused new periodic wiggles (in time) to appear in our vector field, and these must also be homogenized! We need to make sure that the time scale of these wiggles does not interfere with the homogenization problem for~$\mathbf{v}_{\ep_m,a_m,\tau_m}$. 

\smallskip

We next try to see if we can choose the parameters to satisfy all of the constraints. First, 
in order to have~$\kappa_{m-1}$ much larger than~$\kappa_m$, we need that 
\begin{equation*}
\frac{a_m \ep_m^2}{\kappa_m} \gg 1
\,.
\end{equation*}
This suggests that we should try to pick all the parameters so that this number is a negative power, say~$-\gamma<0$, of~$\ep_m$:
\begin{equation*}
\frac{a_m \ep_m^2}{\kappa_m}
\simeq \ep_m^{-\gamma}\,.
\end{equation*}
In other words, we have now chosen~$\kappa_m = a_m \ep_m^{2+\gamma} = \ep_m^{1+\alpha+\gamma}$. The constraints for~$\tau_m$ reduce to~$\ep_m^2 \kappa_m^{-1} \ll a_{m-1}^{-1}$, which can be written in terms of the~$\ep_m$'s as 
\begin{equation*}
\ep_m^{1-\alpha-\gamma}
\ll 
\ep_{m-1}^{1-\alpha}\,.
\end{equation*}
This is a sharper constraint that the one for the~$\ep_m$'s in~\eqref{e.four.contraints}, so it remains to check if it is compatible with the recurrence relation for~$\kappa_m$. This will be the case if and only if
\begin{equation*}
\frac{2(1-\alpha)}{1+\alpha+\gamma} > 1\,.
\end{equation*}
We can therefore pick an appropriate~$\gamma>0$ if and only if~$2(1-\alpha) / (1+\alpha) > 1$, which is equivalent to~$\alpha < \nicefrac 13$. This is the reason for the restriction on~$\alpha$ in the statement of Theorem~\ref{t.anomalous.diffusion}. 

\smallskip

The scales~$\ep_m$ are decreasing supergeometrically, and thus so are the diffusivities~$\kappa_m$. The recurrence for~$\kappa_m$ in~\eqref{e.four.contraints} is very sensitive to the initial choice of~$\kappa$, and for this reason it must be chosen to be within a factor of two of~$\ep_N^{1+\alpha+\gamma}$, for some~$N$. Otherwise the diffusivities will oscillate between very large and very small numbers, and we will lose control of our homogenization estimates. 

\smallskip

Note that the exponents~$1+\alpha+\gamma$ and~$\nicefrac{(1+\alpha+\gamma)}{(1-\alpha)}$ in scaling of the renormalized diffusivities,~$\kappa_m \simeq \ep_m^{1+\alpha+\gamma} \simeq \tau_m^{\frac{1+\alpha+\gamma}{1-\alpha}}$, tend to~$\nicefrac43$ and~$2$, respectively, as~$\alpha \to \nicefrac13$, which is in agreement with Richardson's~$\nicefrac 43$ law. 
In fact, then the variance in the position~$Y_t$ of a particle trajectories at time~$t$ will indeed scale like~$t^{1+\frac{1+\alpha+\gamma}{1-\alpha}}$ (for small~$\kappa$, well-chosen as explained above), and this exponent is close to~$3$ when~$\alpha$ is close to~$\nicefrac13$, as predicted.
Demonstrating this is outside the scope of the present paper, as it requires some uniform estimates for the passive scalar which will be the focus of a forthcoming paper~\cite{ARV}. These are analogous to large-scale regularity estimates in homogenization theory, adapted to the present situation of ``fractal'' homogenization. See below in Section~\ref{ss.uniform.estimates} for more. 

\smallskip

What is described above is a slight simplification of construction of the vector field~$\b(t,x)$ in Section~\ref{s.construction}. In the actual construction, the indicator functions of the time variable appearing in~\eqref{e.vepatau} and~\eqref{e.bm.intr} are replaced by smooth approximations, so that the vector fields~$\b_m$ are smooth in time as well as space, and uniformly~$\alpha$ H\"older continuous in both variables. We also have ``quiet'' time intervals each time we switch the direction of the shear flows, which is convenient for technical reasons. We similarly arrange for the shear flows to pause on time intervals of length~$\tau_m^{\prime}$ around any change of the inverse flows~$X_{m-1}^{-1}$, where~$\tau_m \ll \tau_m^{\prime} \ll \tau_m^{\prime\prime}$ is intermediate between the other two time scales.

\subsubsection*{\bf The homogenization step}

The reader is hopefully convinced that the vector field~$\b(t,x)$ whose construction we have outlined above is a good candidate for exhibiting anomalous diffusion. 

\smallskip

However, analyzing the effect of the complicated fractal-like structure of the vector field~$\b(t,x)$ on the passive scalar~$\theta^\kappa$ is a challenge. 
Periodic homogenization is of course very well-understood, even if there are a large (but finite) number of well-separated scales (a topic referred to as \emph{reiterated} homogenization).
What is not well-understood is the case in which there are essentially infinitely many scales which are not well-separated. This is the situation we encounter here, because even if the ratio~$\ep_{m-1} / \ep_m$  between scales can be made arbitrarily large in our construction of~$\b(t,x)$, once the vector field is constructed it is a fixed finite number. 

\smallskip

This difficulty has not gone unnoticed. Indeed, the idea of renormalization group-type  approach to anomalous diffusion for a passive scalar equation in which ``eddy diffusivities'' are successively renormalized is described very clearly at a heuristic level in~\cite[Section 9.6]{Frisch}. As explained there, this idea has been present since the 19th century, but it is the lack of clear scale separation that is the  primary reason for the limited applicability of homogenization theory to passive scalar turbulence and the reason ``why the concept of eddy viscosity has been regarded by some theoreticians of turbulence as (at best) a pedagogical device.'' Similar remarks can be found in Majda and Kramer~\cite{MK}. 
The present paper is the first work to our knowledge to address this difficulty in a fully rigorous way.\footnote{Here we are thinking of vector fields which are continuous. There have been previous works, such as~\cite{BO,KO}, which use homogenization methods to prove the superdiffusivity of stochastic processes advected by divergence-free vector fields. 
These papers consider a different scaling---there is no small diffusivity parameter~$\kappa$ being sent to zero---and the vector fields considered in these papers have many active scales, with the property that smaller scale wiggles (larger wave numbers) have much larger amplitudes. Note that the latter property is not consistent with continuity, if the vector field were to be rescaled (blown down). Indeed, viewed from the macroscopic scale, these vector fields will belong to negative regularity spaces. Building examples of such vector fields which exhibit superdiffusivity or anomalous diffusion is a much easier task, since one does not have the problem, mentioned above, of the low wave numbers killing the enhancement of the large wave numbers.}

\smallskip

Let~$\theta_m$ be the solution of the equation
\begin{equation*}
\partial_t \theta_{m} 
-\kappa_{m} \Delta \theta_{m} + \b_{m} \cdot \nabla \theta_m = 0 \quad \mbox{in} \ (0,\infty) \times \TT^2
\end{equation*}
with advecting vector field~$\b_m$.
As alluded to above, the main step in the proof of Theorem~\ref{t.anomalous.diffusion} is the demonstration that the equation for~$\theta_m$ homogenizes to the one for~$\theta_{m-1}$. The precise statement is given below in Proposition~\ref{p.indystepdown}, where one finds the estimate
\begin{equation}
\Biggl| \,
\frac{\kappa_{m} \big\| \nabla {\theta}_m \big\|_{L^2((0,1)\times\TT^2)}^2}
{\kappa_{m-1} \left\| \nabla \theta_{m-1} \right\|_{L^2((0,1)\times\TT^2)}^2}
-1 \,
\Biggr| 
\leq C\ep_{m-1}^\delta \,.
\end{equation}
Here~$C$ is a constant depending only on~$\alpha$ and~$\delta>0$ is an explicit exponent. 
An iteration of this estimate yields that the quantity~$\kappa_{m} \big\| \nabla {\theta}_m \big\|_{L^2((0,1)\times\TT^2)}^2$ is nearly independent of~$m$, up to an error which can be made very small, and in particular much smaller than~$\kappa_{0} \big\| \nabla {\theta}_0 \big\|_{L^2((0,1)\times\TT^2)}^2$ which is of order one. 

\smallskip

The basic idea of the proof of the homogenization step is simple and classical. We build an explicit ansatz for the solution of the equation for~$\theta_m$, which we denote by~$\tilde{\theta}_m$. This function is constructed explicitly using ingredients from the equation for~$\theta_{m-1}$, so we have a good understanding of it---we know in particular that the difference~$\tilde{\theta}_m - \theta_{m-1}$ is small in~$L^2$. 
We then plug~$\tilde{\theta}_m$ into the equation for~$\theta_m$ and carefully compute the error. If the error is sufficiently small, then we can deduce that the difference~$\theta_m - \tilde{\theta}_m$ is small from basic energy estimates. 

\smallskip

The definition of the ansatz~$\tilde{\theta}_m$ can be found in~\eqref{e.ansatz}. It is more complicated than usually expected for periodic homogenization.
The classical two-scale ansatz consists of taking the periodic correctors~$\chi_e$ and attaching them to a solution of the macroscopic equation; in our setting, this suggests that we should define the two-scale ansatz by 
\begin{equation*}
\tilde{\theta}_m(t,x)
:=
\theta_{m-1} (t,x)
+
\sum_{i=1}^d 
\ep_{m}
\chi_{e_i} \bigl(\tfrac{x}{\ep_m} \bigr)
\partial_{x_i} \theta_{m-1} (t,x)
\,,
\end{equation*}
where the~$\chi_{e_i}$ represent the correctors corresponding to the shear flows~$\mathbf{v}_{\ep_m,a_m,\tau_m}$ on scale~$\ep_m$ in the definition of~$\b_{m}$. 
Keeping in mind that these shear flows are composed with the inverse flows~$X_{m-1}^{-1}$ in the definition of~$\b_m$, which are very slow compared to~$\mathbf{v}_{\ep_m,a_m,\tau_m}$, it is reasonable to compose these correctors with the inverse flows. Thus we should modify our ansatz to
\begin{equation}
\label{e.ansatz.guess.0}
\tilde{\theta}_m(t,x)
:=
\theta_{m-1} (t,x)
+
\sum_{l\in\Z} 
\indc_{[ l \tau_m^{\prime\prime}, (l+1)\tau_m^{\prime\prime})}(t)
\sum_{i=1}^d 
\ep_{m}
\bigl(\chi_{e_i} \circ 
X_{m-1}^{-1}(t,\cdot,l\tau_m^{\prime\prime})
\bigr) \bigl(\tfrac{x}{\ep_m} \bigr)
\partial_{x_i} \theta_{m-1} (t,x)
\,,
\end{equation}
where as usual the indicator functions of time are actually replaced by a smoother approximation. 

\smallskip

The actual ansatz we make is actually more complicated than~\eqref{e.ansatz.guess.0}, see~\eqref{e.ansatz} in Section~\ref{ss.ansatz} for the precise expression. 
Ultimately, our choice of~$\tilde{\theta}_m$ is  justified by the estimate~\eqref{e.bigbound}, which says~$\tilde{\theta}_m$ is sufficiently close to being a solution of the equation for~$\theta_m$ that we can deduce that the difference~$\theta_m-\tilde{\theta}_m$ is small. Obtaining the bound~\eqref{e.bigbound} turns out to be quite technical and much of the effort in Sections~\ref{s.renorm}--\ref{s.cascade} is devoted to its proof. 

\subsection{Uniform estimates for the scalar}
\label{ss.uniform.estimates}

The predictions made by the phenomenological theories of scalar turbulence go of course much further than  the anomalous dissipation of scalar variance; we refer the reader to~\cite{Frisch,shraiman,Warhaft,falkovich2001,DSY,sreenivasan2010} and references therein   for a detailed account. For example, by drawing direct analogies with the Kolmogorov theory of fluid turbulence, Obukhov~\cite{Obukhov} and Corrsin~\cite{Corrsin} used scaling arguments to predict that if the vector field~$\b(t,x)$ represents a homogenous isotropic velocity field exhibiting K41 ``monofractal'' scaling in the inertial range, with exponent~$\nicefrac 13$, then the scalar field~$\theta^\kappa$ inherits this property---namely ``monofractal'' scaling of structure functions with exponent~$\nicefrac 13$---in the corresponding~$\kappa$-dependent scalar inertial range. This scaling argument can be directly generalized to say that if the structure functions of~$\b(t,x)$ have monofractal scaling with exponent~$\alpha$, then the structure functions of the scalar~$\theta^\kappa$ have monofractal scaling with exponent~$\nicefrac{(1-\alpha)}{2}$, for any~$0< \alpha < 1$, not just~$\nicefrac 13$ (though this is the relevant exponent in fluid turbulence, at least in three spatial dimensions). 

\smallskip

Just as with the Onsager conjecture, one may propose a \emph{mathematical idealization}, corresponding to simultaneously diverging Reynolds and P\'eclet numbers, and postulate a dichotomy:\footnote{See also the discussions in~\cite[Section~5]{DEIJ} and~\cite[Section~1]{CCS22}.} (i) if~$\b \in C^0_t C^\alpha_x$ and if the solutions~$\{\theta^\kappa\}_{\kappa>0}$ of \eqref{e.passive.scalar} are uniformly in~$\kappa$ bounded in~$C^0_t C^{\alpha^\prime}_x$ with~$\alpha' > \nicefrac{(1-\alpha)}{2}$, then~$\lim_{\kappa\to 0} \kappa \|\nabla \theta^\kappa\|_{L^2_t L^2_x}^2 = 0$, and (ii) there exists~$\b \in C^0_t C^\alpha_x$ (presumably with~$\alpha < \nicefrac 13$), such that for~\emph{all} smooth initial conditions~$\theta_0$, the solutions~$\{\theta^\kappa\}_{\kappa>0}$ of~\eqref{e.passive.scalar} are uniformly in~$\kappa$ bounded in~$C^0_t C^{\alpha^\prime}_x$ for any~$\alpha' < \nicefrac{(1-\alpha)}{2}$, and moreover~$\lim_{\kappa\to 0} \kappa \|\nabla \theta^\kappa\|_{L^2_t L^2_x}^2 > 0$. 

\smallskip
Part (i) of this dichotomy is well-known and follows directly from the commutator estimate of Constantin, E, and Titi~\cite{CET}. As stated above, part (ii) of this dichotomy is open.  Theorem~\ref{t.anomalous.diffusion} does not address the H\"older regularity of the family~$\{\theta^\kappa\}_{\kappa>0}$, only the anomalous dissipation of scalar variance. We note however that the paper~\cite{CCS22} establishes a version of this H\"older regularity, but the uniform in~$\kappa$  bounds for~$\| \theta^\kappa(t,\cdot) \|_{C^{\alpha^\prime}_x}$, in the full range~$\alpha' < \nicefrac{(1-\alpha)}{2}$, are only established for a particular initial datum~$\theta_0$, and only in~$L^2$ with respect to the time variable. 

\smallskip
In forthcoming joint work with~Rowan~\cite{ARV}, we will sharpen the statement of Theorems~\ref{t.anomalous.diffusion} in a number of ways. We show in~\cite{ARV} that with $\b(t,x)$ as in Theorem~\ref{t.anomalous.diffusion}, the advection-diffusion equation~\eqref{e.passive.scalar} regularizes the solutions up to~$C^0_t C_x^{\nicefrac{(1-\alpha)}{2}}$ along the subsequences $\{\kappa_j\}$ exhibiting the diffusive anomaly~\eqref{e.anomalous.diffusion}. We show that for all~$t_0 \in (0,1)$, there exists a constant~$C_{t_0} = C_{t_0}(d,\alpha)>0$ such that, for every mean-zero initial datum~$\theta_0\in L^2(\TT^d)$, we have
\begin{equation*}
\sup_{\kappa \in I} 
\| \theta^\kappa \|_{C^0 ([t_0,1],  C^{\nicefrac{(1-\alpha)}{2}}(\TT^d))}
\leq C_{t_0}\| \theta_0 \|_{L^2(\TT^d)}  < \infty\,,
\end{equation*}
where~$I$ is the interval of diffusivities defined in the paragraph below~\eqref{e.varrho.def} above. This is achieved by complementing the argument in this paper with ``large-scale regularity'' techniques developed in quantitative homogenization theory. 
Another consequence of these estimates, which is obtained in~\cite{ARV}, is that the diffusive anomaly $\varrho$ in~\eqref{e.anomalous.diffusion} is uniform in the initial datum: the dependence of~$\varrho(\theta_0)$ on~$\theta_0 \in L^2(\TT^d)$ in Theorem~\ref{t.anomalous.diffusion} can be removed completely; in fact, we can take~$\varrho = \nicefrac12$. Moreover, by obtaining uniform-in-$\kappa$ estimates for the parabolic Green function associated to the drift-diffusion equation~\eqref{e.passive.scalar}, we obtain in~\cite{ARV} estimates for the rate of separation of the squared distance between two realizations of the SDE process $Y_t$ defined in~\eqref{e.SDE}, which are consistent with Richardson's~$\nicefrac43$-law.

\subsection{Notation}

We denote the positive integers by $\NN = \{1,2,\ldots\}$, and the non-negative integers by $\NN_0 = \NN \cup \{0 \} = \{0,1,2,\ldots\}$. 
We denote
\begin{equation}
\label{e.sigma}
\nabla^\perp f : = \sigma \nabla f \,,
\qquad 
\mbox{where}
\qquad
\sigma:= \begin{pmatrix} 0 & -1 \\ 1 & 0 \end{pmatrix}.
\end{equation} 
We use the brackets~$\langle \cdot \rangle$ to denote the mean of a periodic function of space only (not time), that is, $\langle \cdot \rangle = \fint_{\TT^d} (\cdot) dx$. 
Averages of periodic functions in both space and time (or just time) are denoted by~$\langle \hspace{-2.5pt} \langle
\cdot \rangle \hspace{-2.5pt} \rangle$.
We use $\vee$ and $\wedge$ to denote maximum and minimum operations, that is, $a \vee b := \max\{a,b\}$ and $a\wedge b:= \min\{a,b\}$. 
We denote the indicator function of a set $A$ by $\indc_{A}$.
It is convenient to introduce, for every nonnegative integer~$n\in\N_0$ and~$f\in C^\infty(\R^2)$, 
\begin{align}
\label{e.barf}
\snorm{f}_{n,R} =  \frac{(n+1)^2}{ n!R^{n}} \sup_{|\aa| = n} \  \norm{\partial^\aa f}_{L^\infty(\R^2)} 
\qquad \mbox{and} \qquad 
\snorm{f}_{R} = \sup_{n\in\N_0} \ \snorm{f}_{n,R}
\,.
\end{align}
Estimates involving this seminorm are explored in Appendix~\ref{sec:chain:rule}. The prefactor of~$(n+1)^2/n!$ is chosen so that the multiplicative property of the seminorm stated in Lemma~\ref{l.product} is valid. 
Note that $\snorm{\cdot}_{n,R}$ and $\snorm{\cdot}_R$ are monotone decreasing with respect to $R$. 
For a Banach space $X$ with norm $\|\cdot\|_{X}$, for $n\in \NN_0$, and for a sufficiently smooth function $f$, it is convenient to denote the $X$-norm of the $n$th order symmetric tensor $(\partial^{\aa} f)_{|\aa|=n}$ as
\begin{equation}
\|\nabla^{n} f\|_X := \max_{|\aa| = n} \|\partial^{\aa} f\|_X \,.
\label{e.nabla.n}
\end{equation}
We allow constants~$C$ to vary from line to line. If we need to refer to a specific~$C$ from a particular displayed equation, we put the equation number in the subscript of~$C$; e.g.,~$C_{\eqref{e.phi.m.bounds}}$ refers to the constant~$C$ on the right side of~\eqref{e.phi.m.bounds}. 
 
\subsection{Outline of the paper}

In Section~\ref{s.construction}, we construct the vector field~$\b(t,x)$ and give bounds on the derivatives of its approximations~$\b_m$. We also study the regularity of the flows and inverse flows associated to~$\b_m$.
In Section~\ref{s.renorm}, we build the correctors, define the sequence of renormalized diffusivities and other objects which are needed in the homogenization step. The ansatz~$\tilde{\theta}_m$ is introduced in Section~\ref{s.multiscale}, where some important estimates are also proved, using ingredients from Section~\ref{s.renorm}.
The main part of the argument comes in Section~\ref{s.cascade}, where we estimate the error that is made by plugging~$\tilde{\theta}_m$ into the equation for~$\theta_m$. The proof of Theorem~\ref{t.anomalous.diffusion} appears at the end of that section.

\section{The fractal vector field: construction and regularity}
\label{s.construction}

In this section, we construct the periodic, incompressible vector field~$\b(\cdot)$ in Theorem~\ref{t.anomalous.diffusion} and prove that it is H\"older continuous.

\subsection{A list of the ingredients used in the construction}

We present here a list of the objects used in our construction of the incompressible vector field~$\b$. These parameters are fixed throughout the rest of the paper. 

\begin{itemize}

\item We let~$\beta$ be any positive exponent satisfying 
\begin{align}
\label{e.beta.def}
1 < \beta < \frac43
\,.
\end{align}
This represents the regularity of the stream function~$\phi$ obtained in the construction. We typically think of~$\beta$ as very slightly smaller than~$\nicefrac 43$, perhaps~$1.332$. 
The parameter~$\alpha$ in the statement of Theorem~\ref{t.anomalous.diffusion} will be~$\alpha = \beta -1$.

\item We define an exponent $q$ explicitly in terms of~$\beta$ by
\begin{equation}
\label{e.q.def.0}
\q := 
\frac12 \biggl( 
1 + \frac{2-\beta}{2(\beta-1)} \biggr)
\,,
\end{equation}
which prescribes the rate at which the scale separation~$\nicefrac{\ep_{m-1}}{\ep_m}$ between successive scales~$\ep_{m-1}$ and $\ep_m$ becomes larger as~$m$ becomes larger (and the scale~$\ep_m$ becomes smaller): see~\eqref{e.supergeo}, below. Equivalently, this means that 
\begin{equation}
\label{e.beta.def.0}
\beta = \frac 43 \biggl( 1 - \frac{q-1}{4q-1} \biggr)\,.
\end{equation}
The main point is that~$q$ satisfies
\begin{equation}
\label{e.q.def}
1 < \q < \frac{2-\beta}{2(\beta-1)},
\end{equation}
which follows from the inequality~$\beta < \nicefrac 43$ in~\eqref{e.beta.def}. 

\item We fix the small parameter~$\delta\in (0,\nicefrac1{16}]$ defined explicitly by
\begin{equation}
\label{e.delta}
\delta :=
\frac14(q-1) \left( 1 - \frac{2q+1}{2q+2} \beta \right)
=
\frac{(q-1)^2}{4(q+1)(4q-1)}
\,.
\end{equation}

\item We select a large positive integer~$N_*\in \N$ defined by
\begin{equation}
\label{e.N}
N_*:= \biggl\lceil \frac1{\delta^2} + \frac{500}{\delta} \biggr\rceil\,.
\end{equation}
The integer $N_*$ counts the highest number of derivatives we need to track in our argument.

\item We also define the special exponent 
\begin{equation}
\label{e.gamma}
\gamma := \frac{(\q-1)\beta}{\q+1}
\,.
\end{equation}
This exponent is not used in the construction of the vector field in the next subsection, but it appears in Lemma~\ref{l.recurse} as a correction to the exponent for the renormalized diffusivities and subsequently in many of the computations in Section~\ref{s.multiscale}. 

\item We let~$\Lambda \in \N \cap [2^7,\infty)$ be a constant to be chosen later, at the very end of the arguments in Section~\ref{s.cascade}. It is called the \emph{minimal scale separation} and will be chosen to depend only on~$\beta$.

\item We define a sequences of length scales~$\{ \ep_m \}_{m\in\N_0}$ which satisfies
\begin{equation*}
\ep_{m} \ll \ep_{m-1}^{1+\delta} \ll \ep_{m-1}
\,,
\end{equation*}
in the sense that the separation between these length scales is at least a negative power of~$\ep_{m}$. 
These are defined as follows. 
We set~$\ep_0:=1$ and 
\begin{equation}
\label{e.epm.choice}
\ep_m^{-1} := 
\Big\lceil 
\Lambda^{\frac{\q^{m}}{\q-1}}
\Big\rceil
=
\Big\lceil
\exp\bigl(\tfrac{\q^{m}}{\q-1} \log \Lambda \bigr) 
\Big\rceil,
\quad \forall m\in\N 
\,. 
\end{equation}

Since $\ep_1 \leq \Lambda^{-\frac{\q} {\q-1}} \leq \Lambda^{-1}$ and~$
\big\lceil \Lambda^{\frac{\q^{m+1}}{\q-1}} \big\rceil
=
\big\lceil
\Lambda^{\frac{\q^{m}}{\q-1}}
\Lambda^{\q^{m}}
\big\rceil
\geq \Lambda \big\lceil
\Lambda^{\frac{\q^{m}}{\q-1}}
\big\rceil 
$, we have that
\begin{equation}
\label{e.minsep}
\frac{\ep_m}{\ep_{m+1}} 
\geq \Lambda, 
\quad \forall m\in\N_0
\,. 
\end{equation}
In particular, $\ep_m \leq \Lambda^{-m} \leq 2^{-7m}$. In fact, the sequence $\{ \ep_m\}$ decreases at a super-geometric rate with~$\ep_{m+1} \simeq \ep_{m}^q$, as it is routine to check that, for every~$m\geq 1$,  
\begin{equation}
\label{e.supergeo}
\frac 23 \ep_m^q \leq 
(1-10\ep_m) \ep_m^q
\leq
\ep_{m+1} 
\leq 
(1+10\ep_m) \ep_m^q
\leq
\frac43 \ep_{m}^q
\,.
\end{equation}
Finally, we remark that~$\ep_{m}^{-1} \in\N$.

\item We define a sequence $\{ a_m \}_{m\in\N_0}$ of positive constants by
\begin{equation}
\label{e.am.def}
a_m:= \ep_m^{\beta-2}, 
\quad \forall m\in\N\,.
\end{equation}
The constant $a_m$ gives the strength of the shear flows at length scale~$\ep_m$: see~\eqref{e.def.streamr}, below. 

\item We introduce three sequences of time scales:~$\{ \tau_m\}_{m\in\N_0}$, $\{ \tau^\prime_m\}_{m\in\N}$ and~$\{ \tau^{\prime\prime}_m\}_{m\in\N}$. We define these in such a way that, for every $m\in\N$, 
\begin{equation*}
\tau_m \ll \tau^\prime_m 
\ll \tau^{\prime\prime}_m 
\ll a_{m-1}^{-1}
\ll \tau_{m-1} 
\,,
\end{equation*}
in the sense that the separation of the time scales in each of these inequalities is at least a negative power of~$\ep_{m}$. 
They are defined as follows: we set $\tau_0 := 1$ and, for every $m\in\N$, 
\begin{align}
\label{e.taum.def}
& \tau_m := 
\left(
4 
\Bigl \lceil \ep_{m-1}^{-\delta} \Bigr\rceil 
+1
\right)^{-2}
\tau_m^{\prime\prime}
\,
\\ & 
\label{e.taum.prime.def}
\tau^\prime_m := 
\left(
4 
\Bigl \lceil \ep_{m-1}^{-\delta} \Bigr\rceil 
+1
\right)^{-1}
\tau_m^{\prime\prime} \,,
\\ & 
\label{e.taum.primeprime.def}
\tau^{\prime\prime}_m 
:=
2^{-25}
\biggl\lceil
\frac{a_{m-1}}{\ep_{m-1}^{2\delta}}
\biggr\rceil^{-1}
\,.
\end{align}
In particular, we notice that 
\begin{equation}
\label{e.ratios.taum}
\frac{1}{\tau_m}\,,\ 
\frac{1}{\tau^{\prime}_m}\,,\
\frac{1}{\tau^{\prime\prime}_m}
\in 4 \N
\quad  \mbox{and} \quad 
\frac{\tau_m^\prime}{\tau_m}\,,\
\frac{\tau^{\prime\prime}_m}{\tau_m}
\in 4 \N+1, \quad \forall m\in\N
\,,
\end{equation}
and
\begin{equation}
\label{e.taubounds}
2^{-33}\ep_{m-1}^{2-\beta+4\delta}
\leq 
\tau_m 
\leq
2^{-28}
\ep_{m-1}^{2-\beta+4\delta}
\quad \mbox{and} \quad
2^{-25}
\ep_{m-1}^{2-\beta+2\delta}
\leq 
\tau^{\prime\prime}_m 
\leq
2^{-24}
\ep_{m-1}^{2-\beta+2\delta}
\,.
\end{equation}
Thus these sequences, like $\{\ep_m\}$, are decreasing at a super-geometric rate. We also define
\begin{equation}
\label{e.lk.def}
l_k:=
\left\lceil \frac{ k \tau_m + \frac 12 (\tau_m^{\prime\prime}-\tau_m)}{  \tau^{\prime\prime}_m} \right\rceil,
\qquad k\in\Z 
\,,
\end{equation}
so that
\begin{equation}
\label{e.taum.prime.supp}
k \tau_m + \left[ -  \tfrac 12 \tau_m, \tfrac 12 \tau_m \right]
\subseteq 
l_k \tau^{\prime\prime}_m + \left[ -\tfrac12 \tau^{\prime\prime}_m 
, \tfrac12 \tau^{\prime\prime}_m \right]
\,, 
\qquad \forall k\in\Z
\,.
\end{equation}
Notice that there are~$\frac{\tau_m^{\prime\prime}}{\tau_m}$ consecutive values of $k$ that correspond to the same $l_k$.

\item We define stream functions $\stream_{k}$ for each~$k\in\Z$ by 
\begin{equation}
\label{e.def.streamr.0}
\psi_{0,k}(x) := 
\left\{ 
\begin{aligned}
& 
\sin\bigl({2\pi x_i} \bigr) & \mbox{if} & \ k  \in (4 \Z + 2i - 1), \ i \in \{ 1,2\}, \\
& 0 & \mbox{if} & \ k\in 2\Z \,.
\end{aligned}
\right.
\end{equation}
The function~$\psi_{0,k}$ vanishes for even~$k$ and encodes a vertical shear for~$k\in 4\Z+1$ and a horizontal shear for~$k\in 4\Z+3$. 
We scale these stream functions by defining
\begin{equation}
\label{e.def.streamr}
\stream_{m,k} := a_m \ep_m^2 \stream_{0,k} \bigl( \tfrac \cdot{\ep_m}\bigr),
\quad 
\forall
m\in\N, \ k\in\Z.
\end{equation}
Note that the Lipschitz constant of the shear flows~$\nabla^\perp \stream_{m,k}$ is proportional to~$a_m$. 
In fact, recalling that~$\snorm{\cdot}_{R}$ defined in~\eqref{e.barf}, we note that the stream functions~$\psi_{m,k}$ satisfy
\begin{equation}
\label{e.psimk.snorm}
\snorm{\psi_{m,k}}_{\frac{2\pi}{\ep_m}}
\leq 
5a_m\ep_m^2
\,,
\qquad \forall m\in\N,\, \quad k\in\Z\,.
\end{equation}

\item We select an even cutoff function $\zeta\in C^\infty_c(\R )$ of time satisfying, for some constant~$C\in [1,\infty)$ which depends only on~$N_*$ (and thus only on~$\beta$), 
\begin{equation}
\label{e.zetatimecutoff}
0 \leq \zeta \leq \indc_{\left[ -\frac23, \frac 23 \right]},
\quad 
\sum_{k\in\Z} \zeta(\cdot-k) \equiv 1\,,
\quad
\max_{j\in \{0,\ldots,N_*\}}
\norm{ \partial_t^j \zeta }_{L^\infty(\R)} 
\leq C
\,,
\end{equation}  
and 
\begin{equation}
\label{e.weirdo}
\int_{\R} \zeta^2 = \frac 9{10}
\end{equation}
We also scale the cutoff function~$\zeta$ by setting
\begin{equation}
\label{e.zeta.mk.def}
\zeta_{m,k}(t) : = \zeta\left( \frac{t-k\tau_m}{\tau_m}\right), 
\quad
\forall m \in\N, \ k\in \Z.
\end{equation}
Observe that, for some constant $C\in [1,\infty)$ depending only on~$\beta$ (see Remark~\ref{r.constants} below), 
\begin{equation}
\label{e.zeta.mk}
0 \leq \zeta_{m,k} \leq 
\indc_{\left[ (k-\frac23)\tau_m, (k+\frac 23)\tau_m \right]}\,, 
\quad 
\sum_{k\in\Z} \zeta_{m,k} = 1\,,
\quad
\max_{j\in \{ 0,\ldots,N_*\}}
\tau_m^{j}
\norm{ \partial_t^j \zeta_{m,k} }_{L^\infty(\R)} 
\leq  C 
\,.
\end{equation}
The role of these time cutoffs is to enable us to switch between the stream functions~$\psi_{m,k}$ for different~$k$.

\item
We must define a \emph{second} family of time cutoffs~$\{ \xi_{m,k}  \,:\, m\in\N_0,\, k\in\Z\}$ which live on slightly larger intervals than the~$\zeta_{m,k}$'s. We take~$\xi\in C^{\infty}(\R)$ to be a smooth cutoff function of time, satisfying, again for a constant~$C \in [1,\infty)$ depending only on $\beta$, 
\begin{equation}
\label{e.cutoff.xi}
\indc_{[ -\frac34,\frac{3}4]} \leq \xi \leq \indc_{[-\frac54,\frac{5}4]}\,, 
\qquad
\sum_{k\in 2\Z+1} 
\xi(\cdot-k) \equiv 1
\quad \mbox{and} \quad 
\max_{j\in \{0,\ldots,N_*\}}
\norm{ \partial_t^j \xi }_{L^\infty(\R)} 
\leq C
\end{equation}
and then define, for each $k\in4\Z$, 
\begin{equation}
\label{e.xi.mk.def}
\xi_{m,k}(t):= \xi \left( \frac{t - k\tau_m}{\tau_m} \right).
\end{equation}
Observe that the overlap in the periods between two succesive~$\xi_{m,k}$'s is disjoint from the support of the $\zeta_{m,l}$'s when~$l$ is odd (with some extra room). Precisely, we have 
\begin{equation}
\label{e.goodtransition}
\dist \left( \supp \partial_t \xi_{m,k}, \supp \zeta_{m,l} \right) 
\geq \frac1{12} \tau_m, \quad \forall k\in 2\Z+1, \ l\in 2\Z+1.
\end{equation}
See Figure~\ref{fig.small.cutoffs}
The cutoff functions~$\xi_{m,k}$ are not used in the construction of the vector field in the next subsection, but they are needed in the construction of the correctors in Section~\ref{ss.correctors}, see~\eqref{e.Chim}. 

\begin{figure}
\centering
\includegraphics[]{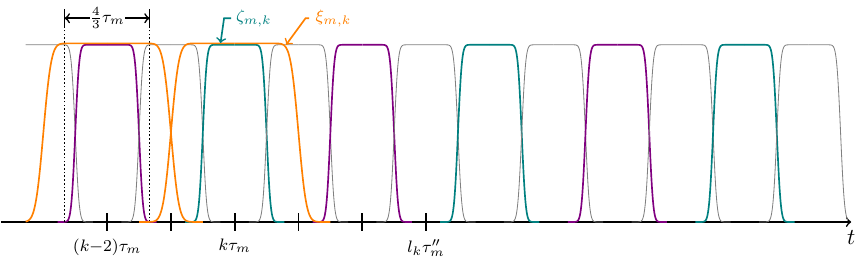}
\caption{
\small{
The families of cutoff functions $\{  {\zeta}_{m,k}\} $ and $\{  {\xi}_{m,k} \}$. These are the small-scale time cutoff functions, and each of them is active on an interval of width close to~$\tau_m$. The~$\zeta_{m,k}$'s corresponding to odd~$k$ are drawn in colors---purple and green corresponding to horizontal and vertical shear flows, respectively---and in grey for even~$k$ as the corresponding vector fields vanish. Only two of the~$\xi_{m,k}$'s are drawn (in orange). These have larger support that the $\zeta_{m,k}$'s and transition between zero and one on intervals in which the latter, for odd $k$, vanish.  
}
}
\label{fig.small.cutoffs}
\end{figure}

\item 
We need to define two more families of time cutoffs
\begin{equation*}
\{ \hat{\zeta}_{m,l} \,:\, m\in\N_0,\, l\in\Z \} 
\quad \text{and} \quad 
\{ \hat{\xi}_{m,l} \,:\, m\in\N_0,\, l\in\Z \}
\,.
\end{equation*}
These have the following properties:
\begin{equation*}
\hat{\zeta}_{m,l} = \hat{\zeta}_{m,0}(\cdot-l \tau_m^{\prime\prime})
\quad \text{and} \quad 
\hat{\xi}_{m,l} = \hat{\xi}_{m,0}(\cdot-l \tau_m^{\prime\prime})
\,,
\end{equation*}
there exists $C \in [1,\infty)$ depending only on $\beta$ such that, for every~$m\in\N_0$ and~$l\in\Z$,
\begin{equation}
\label{e.zeta.prime.ml.fitting}
\left\{
\begin{aligned}
&\indc_{\left[ (l-\frac12) \tau^{\prime\prime}_m +  2\tau_m^\prime, (l+\frac12) \tau^{\prime\prime}_m -2 \tau_m^\prime \right]}
\leq 
\hat{\zeta}_{m,l} 
\leq \indc_{\left[ (l-\frac12) \tau^{\prime\prime}_m +\tau^\prime_m, (l+\frac12) \tau^{\prime\prime}_m - \tau^\prime_m\right]} \,,
\\ 
&
\indc_{\left[ (l-\frac12) \tau^{\prime\prime}_m +\tau^\prime_m, (l+\frac12) \tau^{\prime\prime}_m - \tau^\prime_m\right]}  
\leq 
\hat{\xi}_{m,l} 
\leq 
\indc_{\left[ (l-\frac12) \tau^{\prime\prime}_m- \tau_m' , (l+\frac12) \tau^{\prime\prime}_m+ \tau_m' \right]} \,,
\end{aligned}
\right.
\end{equation}
and
\begin{equation}
\label{e.hatxi.partition}
\sum_{l\in\Z} \hat{\xi}_{m,l} = 1,
\end{equation}
and
\begin{equation}
\label{e.zeta.prime.ml.bounds}
\max_{j\in \{ 0,\ldots,N_*\}}
\bigl(\tau^\prime_m\bigr)^{j}
\Bigl( 
\norm{ \partial_t^j \hat{\zeta}_{m,l} }_{L^\infty(\R)} 
\vee 
\norm{ \partial_t^j \hat{\xi}_{m,l} }_{L^\infty(\R)} 
\Bigr)
\leq  C 
\,.
\end{equation}
Observe that $\{\hat{\zeta}_{m,l}\}_{l\in \ZZ}$ does not form a partition of unity. In view of the definition~\eqref{e.lk.def}, for every $k,l\in\Z$, 
\begin{equation}
\label{e.cutoff.overlaps}
l \neq l_k \implies
\hat{\zeta}_{m,l} \zeta_{m,k} \equiv 0. 
\end{equation}
See Figure~\ref{fig.large.cutoffs}.
The role of~$\hat{\zeta}_{m,l}$ is to smoothly cutoff the vector field near the time at which the flows and inverse flows are refreshed in our construction: see~\eqref{e.psi.recursion}, below. 
The cutoff functions~$\hat{\xi}_{m,k}$ are not used in the construction of the vector field in the next subsection, but they are needed in the construction of the two-scale ansatz in Section~\ref{ss.correctors}, see~\eqref{e.ansatz}.

\begin{figure}
\centering
\includegraphics[]{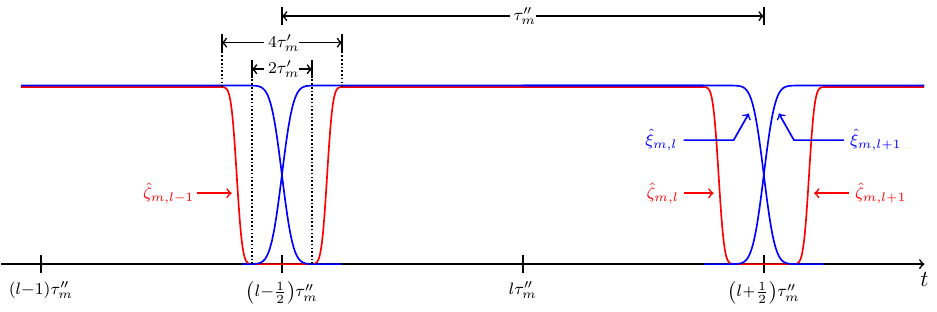}
\caption{
\small{
The families of cutoff functions $\{ \hat{\zeta}_{m,l}\} $ and $\{ \hat{\xi}_{m,l} \}$. These are the large-scale time cutoff functions, and each of them is active on an interval of width close to ~$\tau_m^{\prime\prime}$ and transition between zero and one in intervals of width $\tau_m^\prime \ll \tau_m^{\prime\prime}$. The main difference is that the $\hat{\xi}_{m,l}$'s form a partition of unity, and their transition occurs entirely outside the support of the $\hat{\zeta}_{m,l}$'s. 
}
}
\label{fig.large.cutoffs}
\end{figure}

\end{itemize}

\begin{remark}[Convention for the constants]
\label{r.constants}
Throughout the rest of the paper, we use~$C$ and~$c$ to denote positive constants which depend only on~$\beta$ and may vary in each occurrence. Note that, since the parameters~$q$,~$\delta$ and~$N_*$ are defined explicitly in terms of~$\beta$, our constants may depend on them as well. 
In particular, these constants are understood to never depend on the scale parameter~$m$.
Also, since the parameter~$\Lambda$ will be chosen to be very large at the end of the proof of Theorem~\ref{t.anomalous.diffusion}---precisely to absorb the error terms arising in the proof---it is important that we do not allow the constants~$C$ and~$c$ to depend on an upper bound for~$\Lambda$. Since $\Lambda \geq 32$, they may depend on a lower bound for~$\Lambda$. 

\end{remark}

\subsection{Construction of the vector field}

In this subsection, we construct an incompressible vector field $\b$, which is a sum of rescaled copies of a given family of periodic incompressible vector fields (shear flows) such that each term in the sum is advected by the partial sum of the terms representing larger scales (lower frequencies). 

\smallskip

We proceed by constructing a sequence~$\{ \b_m\}$ of smooth, incompressible vector fields which are~$\Z\times\Z^2$--periodic,   with associated sequences of periodic stream functions~$\{ \phi_m\}$, which satisfy
\begin{equation}
\label{e.phim.bm}
\nabla^\perp \phi_m = \b_m, \quad \langle \phi_m \rangle = 0,
\end{equation}
and their corresponding flows~$\{ \flow_m(\cdot,x,s)\}$ which satisfy, for each~$(x,s) \in \R^2 \times \R$, the ODE \begin{equation}
\label{e.flow.m.def}
\left\{
\begin{aligned}
& \partial_t \flow_m(t,x,s)= \b_m(t,\flow_m(t,x,s)) & \mbox{in} & \ \R, \\
& \flow_m(s,x,s) = x.
\end{aligned}
\right.
\end{equation}
We also denote by~$X^{-1}_m$ the inverse flow, that is,~$X^{-1}_m(t,\cdot,s)$ is the inverse function of~$X_m(t,\cdot,s)$. The construction will be an iterative one, starting with the largest scale shears and progressively building in the smaller scale ones, in the Lagrangian coordinates of the larger scale shears. The vector field~$\b_m$ will have only scales larger than~$\ep_m$ built into it, and we will eventually take~$\b$ as the limit~$m\to \infty$. 

\smallskip

We initialize the construction by setting
\begin{equation}
\label{e.phi0def}
\left\{
\begin{aligned}
&
\phi_0(t,x):= 0 \,,
\\ & 
\b_0(t,x) := \nabla^\perp \phi_0 = 0 \,,
\\ & 
X_0(t,x,s) := x \,.
\end{aligned}
\right.
\end{equation}
Supposing that, for some $m\in\N$, we have defined $\phi_j$, $\b_j$ and $\flow_j$ for every $j\leq m-1$ and that these functions are smooth in all variables and satisfy the properties above. 
We then define~$\phi_m$,~$\b_m$ and~$\flow_m$ as follows. We first define the new stream function~$\phi_m$ by 
\begin{align}
\label{e.psi.recursion}
\phi_m(t,x) 
& :
= 
\phi_{m-1} (t,x)+ 
\sum_{k,l \in\Z} 
\hat{\zeta}_{m,l}(t)
\zeta_{m,k}(t)
\stream_{m,k}\left( \flow_{m-1}^{-1}\left(t, x, l  \tau^{\prime\prime}_m \right) \right)
\notag \\ & \;
= 
\phi_{m-1} (t,x)+ 
\sum_{k \in\Z} 
\hat{\zeta}_{m,l_k}(t)
\zeta_{m,k}(t)
\stream_{m,k}\left( \flow_{m-1}^{-1}\left(t, x, l_k  \tau^{\prime\prime}_m  \right) \right)
\,,
\end{align}
where we recall that~$\psi_{m,k}$ is defined in~\eqref{e.def.streamr}. In the second equality in the display above we have appealed to~\eqref{e.cutoff.overlaps}.
It is clear from induction that~$\phi_m$ is smooth. It also has zero mean by induction, the fact that~$\psi_{m,k}$ has zero mean and the fact that the incompressibility of~$\b_{m-1}$ implies that~$X_{m-1}^{-1}(t,\cdot,s)$ is measure-preserving. We may therefore define~$\b_m(t,x):= \nabla^\perp \phi_m(t,x)$ so that~\eqref{e.phim.bm} is satisfied; and then define~$X_m(\cdot,x,s)$ to be the unique solution of the flow~\eqref{e.flow.m.def}. 
These functions are clearly smooth, so this completes the construction. 

\smallskip

Notice that the~$\psi_{m,k}$'s change each time we increment~$k$, but the inverse flows~$X^{-1}_{m-1}(\cdot,\cdot, l_k  \tau^{\prime\prime}_m)$ depend on~$k$ only through the value of the initial time, namely $l_k\tau_m^{\prime\prime}$. In particular, the inverse flows are the same for~$\nicefrac{\tau_m^{\prime\prime}}{\tau_m} \in 4 \NN + 1$ many consecutive values of~$k$.
To keep the notation short, we define, for every~$m\in\N$ and~$k\in\Z$, 
\begin{align}
\label{e.flowabbrev}
X_{m,l}(t,x) := X_{m}(t,x, l  \tau_m^{\prime\prime}) 
\qquad \mbox{and} \qquad
X^{-1}_{m,l}(t,x):= X_{m}^{-1} (t,x,l  \tau_m^{\prime\prime}).
\end{align}

\smallskip

As far as periodicity is concerned, it is clear from the construction that 
\begin{align*}
\phi_m \quad \mbox{and} \quad \b_{m} \quad \mbox{are  $\Z \times\Z^2$--periodic,}
\end{align*}
that~$x\mapsto \bigl( \flow_{m}(t,x,s) -x \bigr)$ is~$\Z^2$--periodic and~$\flow_{m}$ is $1$--periodic in time, jointly in~$(t,s)$, in the sense that
\begin{equation*}
\flow_m(t+n,x,s+n) = \flow_m(t,x,s), 
\quad \forall 
m\in\N, \,
n\in\Z, \, 
t,s\in\R, \, 
x\in\R^2. 
\end{equation*}
We intend to define the vector field~$\b(t,x)$ by taking a limit:
\begin{align}
\label{e.def.b}
\left\{
\begin{aligned}
&
\phi(t,x) : = \lim_{k\to \infty} \phi_k(t,x),
\\ & 
\b(t,x) : = \nabla^\perp \phi(t,x).
\end{aligned}
\right. 
\end{align}
The limit in the first line of~\eqref{e.def.b} is valid in the sense of $L^\infty([0,1]\times[0,1]^d)$ due to the fact that $|\streamr{k}{m}|$ is bounded by $a_m \ep_m^2 = \ep_m^\beta$, which is summable over~$m\in\N$. That~$\phi$ is regular enough that the second line is valid is less clear at this stage. 

\subsection{Regularity of the stream functions and associated flows}

How regular should we expect the stream function $\phi$ to be? As we will discuss below, the regularity of~$\phi$ is complicated by the composition with the inverse flows in~\eqref{e.psi.recursion}. It turns out that~$X_{m-1}^{-1}(t,\cdot,l_k \tau_m^{\prime\prime})$ is close to the identity map on the support of the cutoff function $\hat{\zeta}_{m,l_k}$, as we will show below.
If we imagine that the inverse flows can be replaced by the identity map in~\eqref{e.psi.recursion},  then we may guess that the spatial regularity of~$\phi_m - \phi_{m-1}$ is similar to a periodic  function with period $\ep_m$ and amplitude~$a_m\ep_m^2 = \ep_m^\beta$. Thus, the  $C^{1,\beta'}$ seminorm of~$\phi_m - \phi_{m-1}$ should be of order~$\ep_m^{\beta-\beta'-1}$, for every $\beta' < \beta-1$. Summing over the scales, this leads us to guess that $\phi_m$ is uniformly bounded in $C^{1,\beta'}$ for every $\beta' < \beta-1$, and thus the limit~$\phi$ should belong to $C^{1,\beta-}$. 
This argument also suggests that $\b$ should belong to $C^{0,\beta-}$.

\smallskip

This guess is correct, although the proof is more subtle than the back-of-the-envelope computation above may lead one to believe. 
Indeed, let us suppose that after the~$m$th step of the construction we have $\phi_m(t,\cdot) \in C^k$, uniformly in~$t$ (with some estimate depending on~$m$ and~$k$). If we try to propagate this bound forward to $\phi_{m+1}$, what we find is that 
\begin{align*}
\phi_m(t,\cdot) \in C^k
& 
\implies 
\b_m(t,\cdot) \in C^{k-1}
&&
\qquad \mbox{(by~\eqref{e.phim.bm}),}
\\ & 
\implies 
x\mapsto \flow_m^{-1}(t,x,s) \ \mbox{ is $C^{k-1}$ }
&&
\qquad \mbox{(regularity of transport equation),}
\\ & 
\implies 
\phi_{m+1}\in C^{k-1} 
&&
\qquad \mbox{(by the first line of~\eqref{e.psi.recursion}).}
\end{align*}
We have lost one derivative!
This suggests that obtaining the desired uniform bounds on $\phi_m$ requires propagating  bounds on all spatial derivatives of $\phi_m$, using the analyticity of~$\streamr{k}{m}$ (and its small size for large~$m$) to close the argument. This is the idea of the argument in the proof of Proposition~\ref{p.SAMS.regularity}, below. 
For this to succeed, every implication in the display above must be carefully quantified. Since the last implication uses the chain rule and must be iterated many times, we need a version of the Fa\'a di Bruno formula (Proposition~\ref{prop:multi:Faa}) and the estimates for derivatives of the composition of two smooth functions that it implies (Proposition~\ref{prop:compose:analytic}). 
The second implication involves the regularity of the inverse flow $X^{-1}_{m-1}$, which solves the transport equation. We need estimates on all the derivatives of $X^{-1}_{m-1}$ which are explicit in their dependence on the order of differentiation. Such estimates are classical, but also difficult to find in the literature in the explicit form we require here; for the reader's convenience, in Appendix~\ref{sec:chain:rule} we present complete statements and proofs of what we need here.  

\smallskip
We recall that by \eqref{e.taubounds} and \eqref{e.zeta.prime.ml.fitting}, if $t \in \supp \hat{\zeta}_{m,l_k}$, then $|t - l_k \tau_m^{\prime\prime}| \leq \frac 12 \tau_m^{\prime\prime} \leq 2^{-25} \ep_{m}^{2-\beta+2\delta} \leq 2^{-25} \ep_{m}^{2-\beta}$. As such each flow $X_{m}(t+s,x,s)$ and backwards flow $X_{m}^{-1}(t+s,x,s)$ needs to be studied for times $t$ which satisfy $|t| \leq 2^{-25} \ep_{m}^{2-\beta}$. This motivates assumption~\eqref{e.Xm.regbounds} below. Also, recall that~$\snorm{\cdot}_{n,R}$ is defined in~\eqref{e.barf}.

\begin{proposition}
\label{p.SAMS.regularity}
For every $m\in\N$,
\begin{align}
\label{e.phi.m.m-1.bounds}
\sup_t \snorm{  (\phi_{m} - \phi_{m-1})(\cdot,t) }_{n, 2^7 \ep_m^{-1}}
\leq 10 \ep_m^{\beta}\,, 
\qquad \forall n\in \N_0
\,,
\end{align}
and, for every $s,t\in\R$ with $|t|\leq 2^{-25} a_m^{-1} = 2^{-25} \ep_m^{2- \beta}$,
\begin{align}
\label{e.Xm.regbounds}
\snorm{ \nabla X^{-1}_{m}(t+s,\cdot,s) - \Itwo}_{n,2^{11} \ep_m^{-1}}
\leq
2^{23}  |t| a_m = 2^{23} |t| \ep_m^{\beta-2}
\,, \quad
\forall n\in\N_0
\,.
\end{align}
Consequently, for any $\beta' \in (0,\beta-1]$ we have 
\begin{equation}
\norm{\phi_m-\phi_{m-1}}_{L^\infty(\R;C^{1,\beta'}(\R^2))} 
\leq 2^{12} \ep_m^{\beta-1-\beta'}\,, 
\label{e.C1beta.phi.m.m-1}
\end{equation}
and there exists a constant $C>0$ which only depends on $\beta$, such that 
\begin{equation}
\norm{ \phi_m-\phi_{m-1}}_{C^{1,\beta'}(\R;L^\infty(\R\times\R^2))} 
\leq C \ep_m^{\beta- 1 - \beta'} 
\,.
\label{e.dt.Cbeta.phi.m.m-1}
\end{equation}
\end{proposition}

\begin{proof}[Proof of Proposition~\ref{p.SAMS.regularity}]
Let~$\{ M_m \}_{m\in\N} \subseteq (0,\infty)$ and~$\{ R_m \}_{m\in\N}\subseteq (0,\infty)$ be two sequences to be defined explicitly below (see \eqref{e.recurrence}). 
Suppose that, for some~$m\in\N$, we have 
\begin{equation}
\label{e.tauprime.indyass}
\tau^{\prime\prime}_m
\leq
(32 M_{m-1})^{-1}
\end{equation}
and 
\begin{equation}
\label{e.indyhyp}
\sup_{t \in \R}
\snorm{\phi_{m-1}}_{n,R_{m-1}} \leq M_{m-1} R_{m-1}^{-2} \frac{(n+2)^2}{(n+1)^3}
\,, 
\qquad \forall n\in \N \,, n\geq 2
\,.
\end{equation}
The assumption~\eqref{e.indyhyp} implies, for every $n\in\N$,
\begin{align}
\sup_{t\in \R} \snorm{\b_{m-1}}_{n,R_{m-1}}
&\leq 
\frac{(n+1)^2}{n! R_{m-1}^n} 
\sup_{t\in \R}  \norm{\nabla^{n+1} \phi_{m-1}}_{L^\infty(\R^2)}
\notag \\ 
&\leq 
\frac{(n+1)^2}{n! R_{m-1}^n} 
\sup_{t \in \R} \snorm{\phi_{m-1}}_{n+1,R_{m-1}} \frac{(n+1)! R_{m-1}^{n+1}}{(n+2)^2}
\notag \\  
&\leq  
M_{m-1} R_{m-1}^{-1}
\label{e.breg.indypath}
\,.
\end{align}
According to the definition \eqref{e.flow.m.def}, the backwards flow $X_{m-1}^{-1}$ solves $(\partial_t + \b_{m-1} \cdot \nabla) X_{m-1}^{-1} = 0$, and so with $Y = X_{m-1}^{-1}(\cdot,x,\cdot) - x$, we are in the setting of Lemma~\ref{l.transport.reg:shift}, with $\f =  -\g = \b_{m-1} $. The previously established estimate \eqref{e.breg.indypath} shows that assumptions \eqref{e.bear.salmon.1} hold with $C_\f = C_\g =  M_{m-1} R_{m-1}^{-1}$, and $R_\f=R_\g = R_{m-1}$. Here we emphasize that the assumption  \eqref{e.bear.salmon.1} is only required to hold for derivative indices $n$ with $n\geq 1$. According to \eqref{e.bear.salmon.2}, we thus define~$\bar{R}_{m-1}(t) := R_{m-1} ( 1 + 8 |t| M_{m-1})$.  By Lemma~\ref{l.transport.reg:shift}, for every $s,t\in \R$ such that $|t|\leq (8 M_{m-1})^{-1}$, and for every $n \in \N$,  
\begin{equation}
\label{e.flowinv2.pre}
\snorm{ X^{-1}_{m-1}(t+s,\cdot,s) - x }_{n,\bar{R}_{m-1}(t)}
\leq
16 |t| M_{m-1} R_{m-1}^{-1}
\,.
\end{equation}
Observe that if $|t| \leq (64 M_{m-1})^{-1}$, then $\bar{R}_{m-1}(t) \leq \frac 98 R_{m-1}$ and~\eqref{e.flowinv2.pre} implies that for all $n\geq 1$ 
\begin{align}
\label{e.flowinv2}
\snorm{ X^{-1}_{m-1}(t+s,\cdot,s) }_{n,\frac 98 R_{m-1}}
&
\leq
\snorm{ X^{-1}_{m-1}(t+s,\cdot,s) - x }_{n,\bar{R}_{m-1}(t)}
+
\snorm{ x }_{n,\frac 98 R_{m-1}}
\notag \\ &
\leq
16 |t| M_{m-1} R_{m-1}^{-1} + 4 \cdot \tfrac 89 R_{m-1}^{-1} 
\leq
5 \cdot \tfrac 89 R_{m-1}^{-1}
\,.
\end{align}
We now have all the necessary ingredients for estimating the second term in~\eqref{e.psi.recursion} using Proposition~\ref{prop:compose:analytic}. With the help of  \eqref{e.psimk.snorm}, \eqref{e.tauprime.indyass}, \eqref{e.flowinv2},~and~\eqref{e.barf.3} we obtain, for every $t \in\R$ and every $n\in \NN$ with $n\geq 2$ that 
\begin{align}
\snorm{( \phi_m-\phi_{m-1})(t,\cdot) }_{n,\frac 98 R_{m-1}+\frac{20 \pi}{\ep_m}}
& 
\leq
2\sup_{k\in\Z}
\sup_{t \in \supp \hat{\zeta}_{m,l_k} \zeta_{m,k}}
\snorm{\psi_{m,k} \circ X_{m-1}^{-1}(t,\cdot,l_k\tau_m^{\prime\prime}) }_{n,\frac 98 R_{m-1}+\frac{20 \pi}{\ep_m}}
\notag \\ & 
\leq
10 a_m\ep_m^2 
\label{e.phimdiff}
\,.
\end{align}
Here we have used that if $t \in \supp \hat{\zeta}_{m,l_k}$, then by \eqref{e.zeta.prime.ml.fitting} we have that  $|t - l_k \tau^{\prime\prime}_m| \leq \frac 12 \tau^{\prime\prime}_m$, and thus~\eqref{e.tauprime.indyass} implies $|t - l_k \tau^{\prime\prime}_m| \leq (64 M_{m-1})^{-1}$.  
Therefore, by the induction assumption~\eqref{e.indyhyp}, the bound \eqref{e.phimdiff}, and the monotonicity of $\snorm{\cdot}_{n,R}$ with respect to $R$, for all $n\geq 2$,  $t\in\R$, and $R_m \geq \frac 98 R_{m-1}+\frac{20 \pi}{\ep_m} > R_{m-1}$, we have 
\begin{align*}
\snorm{\phi_m(t,\cdot)}_{n, R_m}
&
\leq
\snorm{\phi_{m-1}(t,\cdot)}_{n,R_m}
+
\snorm{(\phi_m-\phi_{m-1})(t,\cdot)}_{n,R_m}
\\ & 
\leq 
M_{m-1}  R_{m-1}^{-2} \frac{(n+2)^2}{(n+1)^3} \left(\frac{R_{m-1}}{R_m}\right)^{\!\!n}
+
10a_m\ep_m^2   \left(\frac{\frac 98 R_{m-1}+\frac{20 \pi}{\ep_m}}{R_m}\right)^{\!\!n}
\\ & 
\leq M_m R_m^{-2} \frac{(n+2)^2}{(n+1)^3}
\left(
\frac{M_{m-1}}{M_m} \left(\frac{R_{m-1}}{R_m}\right)^{\!\! n-2}
+
\frac{10a_m\ep_m^2 R_m^2  (n+1)}{M_m}  \left(\frac{\frac 98 R_{m-1}+\frac{20 \pi}{\ep_m}}{R_m}\right)^{\!\! n}
\right)
\,.
\end{align*}
Thus, if we make the choice $R_m \geq 2 (\frac 98 R_{m-1}+\frac{20 \pi}{\ep_m})$, since $(n+1) 2^{-n+2} \leq 3$ for $n\geq 2$, we   arrive at
\begin{align*}
\snorm{\phi_m(t,\cdot)}_{n, R_m}
&
\leq M_m R_m^{-2} \frac{(n+2)^2}{(n+1)^3}
\left(
\frac{M_{m-1}}{M_m}  
+
\frac{30a_m\ep_m^2 (\frac 98 R_{m-1}+\frac{20 \pi}{\ep_m})^2  }{M_m}   \right)
\,.
\end{align*}
With an eye on the induction hypothesis~\eqref{e.indyhyp} with $m$ in place of $m-1$, this motivates the recursion  
\begin{align}
\label{e.recurrence}
\left\{
\begin{aligned}
&R_m:= \tfrac 94 R_{m-1} + 2^7 \ep_m^{-1}  \\
& M_m :=  M_{m-1} + 2^7 a_m \eps_m^2 R_{m-1}^2 + 2^{18} a_m 
\end{aligned}
\right. ,
\end{align}
for all $m\geq 1$.
Here we have used that $40 \pi \leq 2^7$, $60 \cdot (\frac 98)^2 \leq 2^7$, and $60 \cdot (20 \pi)^2 \leq 2^{18}$.
We now take the recurrence~\eqref{e.recurrence} to be the definition of the sequences $\{ M_m\}$ and $\{ R_m \}$, starting from~$M_0:=1$ and $R_0:=1$.

\smallskip

We next analyze the recurrence relation~\eqref{e.recurrence}. Observe that 
\begin{equation*}
\ep_{m+1}R_m =  \Bigl( \frac{9 \ep_{m+1}}{4\ep_m} \Bigr) \ep_m R_{m-1} 
+ 2^7  \Bigl( \frac{\ep_{m+1}}{\ep_m} \Bigr). 
\end{equation*}
Recall from~\eqref{e.minsep} and the fact that $\Lambda \geq 2^7$ that  
\begin{equation}
\label{e.epkcondition}
\frac{\ep_{m+1}}{\ep_m} \leq  \frac1{2^7}, 
\quad \forall m\in\N
\,. 
\end{equation}
We deduce by induction that, for every $m\in\N$ with $m\geq 1$,
\begin{equation}
\label{e.epk1Rk}
\ep_{m+1} R_{m} \leq \tfrac 54.
\end{equation}
Inserting the bound \eqref{e.epk1Rk} back into the first line of~\eqref{e.recurrence} yields 
\begin{equation}
\label{e.Rmbound}
R_m \leq (3 + 2^7) \ep_{m}^{-1}\,. 
\end{equation}
Inserting the bound \eqref{e.epk1Rk} into second line in \eqref{e.recurrence}, we obtain that 
\begin{equation*}
M_{m} \leq M_{m-1} +   (2^8 + 2^{18}) a_m  \,,
\end{equation*} 
and thus by appealing to \eqref{e.minsep}, $\beta< 4/3$, $\Lambda \geq 2^7$,  and~\eqref{e.am.def}, for every $m\in\N$ we obtain the bound
\begin{align}
\label{e.Mmbound}
M_m 
\leq (2^8 + 2^{18})\sum_{j=0}^m a_j  
\leq  (2^8 + 2^{18}) a_m  \sum_{j\geq 0} \Lambda^{j(\beta-2)} 
\leq  \frac{2^8 + 2^{18}}{1 - \Lambda^{\beta-2}}  a_m 
\leq  \frac{2^8 + 2^{18}}{1 - 2^{-\nicefrac{14}{3}}}  a_m 
\leq  2^{19} a_m .
\end{align}
In view of~\eqref{e.taubounds}, using \eqref{e.Rmbound} and \eqref{e.Mmbound} we find that 
\begin{align*}
\tau^{\prime\prime}_m
\leq 
2^{-24} a_{m-1}^{-1}  \ep_{m-1}^{\delta}
\leq 
2^{-24} a_{m-1}^{-1}  
\leq
2^{-5}  M_{m-1}^{-1}
\,. 
\end{align*}
That is, the hypothesis~\eqref{e.tauprime.indyass} is in fact valid for every $m\in\N$ and is therefore superfluous. 

\smallskip

By induction, we may conclude now that~\eqref{e.indyhyp} holds for every $m\in\N$. Moreover, we have shown that~\eqref{e.flowinv2} and~\eqref{e.phimdiff} are valid for every $m\in\N$, for $n\geq 1$ and respectively $n\geq 2$.
Substituting~\eqref{e.Rmbound},~\eqref{e.Mmbound}, and the bound $2^7\ep_m^{-1} \leq R_m \leq 2^8\ep_m^{-1} $ into the these bounds yields,
for~$m\in\N$ and all $n\in \N$ with $n\geq 2$, that
\begin{align}
\label{e.phimbounds.subbed}
\sup_{t\in\R} \snorm{\phi_{m}(t,\cdot)}_{n,2^8 \ep_{m}^{-1}}
&\leq
2^{5} a_{m} \ep_{m}^{2}   \frac{(n+2)^2}{(n+1)^3} \,, 
\\
\label{e.phimdiff.buff}
\sup_{t\in \R} \snorm{(\phi_m-\phi_{m-1})(t,\cdot) }_{n,2^{7} \ep_m^{-1}}
& 
\leq
10 a_m\ep_m^2 = 10 \ep_m^\beta
\,,
\end{align}
and for every~$s,t\in \R$ with $|t|\leq (2^{25}a_{m})^{-1}$, and all $n\in \N$, that
\begin{equation}
\label{e.flowinv2.subbed}
\snorm{ X^{-1}_{m}(t+s,\cdot,s) - x }_{n, 2^{9}\ep_{m}^{-1}}
\leq
2^{16}  |t|  a_{m} \ep_{m}
\,.
\end{equation}
The bound~\eqref{e.phimdiff.buff} implies~\eqref{e.phi.m.m-1.bounds} for $n \in \NN$ with $n\geq 2$, while the estimate~\eqref{e.flowinv2.subbed} yields~\eqref{e.Xm.regbounds} for every $n\in\N_0$, upon noting that 
\begin{equation*}
\snorm{ \nabla X^{-1}_{m}(t+s,\cdot,s) - \Itwo}_{n,2^{11} \ep_m^{-1}}
\leq 
\frac{ (n+1)^3 2^{9 - 2n} \ep_m^{-1}}{(n+2)^2} \snorm{ X^{-1}_{m}(t+s,\cdot,s) - \cdot}_{n+1,2^{9} \ep_m^{-1}} 
\leq
2^{23} |t| a_m \,.
\end{equation*}

\smallskip

In order to get~\eqref{e.phi.m.m-1.bounds} for~$n=0$, we use the definition of $\phi_m$ in \eqref{e.psi.recursion} and obtain
\begin{align*}
\norm{\phi_m-\phi_{m-1}}_{L^\infty(\R\times\R^2)} \leq a_m \ep_m^2 = \ep_m^\beta
\,.
\end{align*}
For the~$n=1$ bound, we interpolate between the above estimate and~\eqref{e.phimdiff.buff} with $n=2$, to obtain
\begin{equation}
\norm{ \nabla(\phi_m-\phi_{m-1})}_{L^\infty(\R\times\R^2)} 
\leq \bigl( \ep_m^\beta \bigr)^{\nicefrac 12} 
\bigl( 10  a_m \ep_m^2 \tfrac 29 \cdot ( 2^7 \ep_m^{-1})^2 \bigr)^{\nicefrac 12}
\leq   2^8 \ep_m^{\beta-1}
\,,
\label{e.nabphimm1bound}
\end{equation}
and thus
\begin{equation*}
\sup_{t\in \R} \snorm{(\phi_m-\phi_{m-1})(t,\cdot) }_{1,2^{7} \ep_m^{-1}}
\leq 8 \ep_m^{\beta} 
\leq 10 \ep_m^{\beta}  \,.
\end{equation*}
Thus we have also proved~\eqref{e.phi.m.m-1.bounds} for $n\in \{0,1\}$, and hence in view of \eqref{e.phimdiff.buff} for every $n \in \N_0$. 

\smallskip
For future purposes, we note at this stage that upon telescoping the bound~\eqref{e.phi.m.m-1.bounds} for $n\in \{0,1\}$, similarly to \eqref{e.Mmbound} we obtain 
\begin{align}
\sup_{t\in\R} \snorm{\phi_{m}(t,\cdot)}_{n,2^8 \ep_{m}^{-1}}
&\leq \sum_{j=0}^m  \sup_{t\in\R} \snorm{(\phi_{j} - \phi_{j-1})(t,\cdot)}_{n,2^7 \ep_{j}^{-1}} \Biggl( \frac{\ep_j^{-1}}{2 \ep_m^{-1}}\Biggr)^n 
\notag\\
&\leq 10 \cdot 2^{-n} \ep_m^n \sum_{j=0}^m   \ep_j^{\beta-n}
\leq 
10 \cdot 2^{-n} \ep_m^n \sum_{j=0}^m   \Lambda^{-j(\beta-n)}
\leq
\begin{cases}
11, &n= 0, \\
\frac{3 }{\beta-1} \ep_m , &n=1, 
\end{cases}
\label{e.phimbounds.subbed.2}
\end{align}
for $n\in \{0,1\}$. Here we have used that $\beta>1$ and $\Lambda\geq 2^7$. The above estimate with $n=1$ then immediately implies 
\begin{equation}
\norm{\b_{m}}_{L^\infty(\R \times \R^2)}  \leq 2^7 \ep_m^{-1} \sup_t \snorm{\phi_m}_{1,2^8 \ep_m^{-1}} \leq \nicefrac{2^{9}}{(\beta-1)} \,,
\label{e.bm.uniform}
\end{equation}
which shows that the sequence of vector fields $\{ \b_m \}_{m\geq 0}$ is uniformly bounded in space-time, uniformly in $m$. 

\smallskip
In order to conclude the proof, we need to still consider the bounds \eqref{e.C1beta.phi.m.m-1} and \eqref{e.dt.Cbeta.phi.m.m-1}. The first one, namely \eqref{e.C1beta.phi.m.m-1}, follows by interpolating the bounds in \eqref{e.phi.m.m-1.bounds} when $n=1$ and $n=2$. For the H\"older regularity of the time derivative, we  differentiate the expression~\eqref{e.psi.recursion} in time to obtain
\begin{align}
\label{e.psi.recursion.partial.t}
\partial_t 
\bigl( \phi_m - \phi_{m-1} \bigr) 
&
= 
\sum_{k\in\Z} 
\partial_ t \bigl(\hat{\zeta}_{m,l_k} \zeta_{m,k}\bigr) 
\stream_{m,k} \circ \flow_{m-1,l_k}^{-1} 
\notag \\ & \quad 
-
\sum_{k\in\Z} 
\hat{\zeta}_{m,l_k} \zeta_{m,k}
\nabla \stream_{m,k} \circ \flow_{m-1,l_k}^{-1}
\cdot
\b_{m-1} \cdot\nabla \flow_{m-1,l_k}^{-1} 
\,.
\end{align}
Using~\eqref{e.def.streamr},~\eqref{e.zeta.mk},~\eqref{e.zeta.prime.ml.bounds},~\eqref{e.bm.uniform}, and~\eqref{e.nabphimm1bound}, we obtain
\begin{align}
&\norm{\partial_t ( \phi_m - \phi_{m-1} ) }_{L^\infty(\R\times\R^2)}
\notag\\
& 
\qquad \leq
2 \norm{ \partial_t \bigl( \hat{\zeta}_{m,l_k} \zeta_{m,k} \bigr) }_{L^\infty(\R)} 
\sup_{k\in\Z} \norm{ \stream_{m,k} }_{L^\infty(\R^2)}  
\notag\\
&\qquad \qquad
+
2 \norm{\b_{m-1}}_{L^\infty(\R\times\R^2)} \sup_{k\in\Z}  \norm{\nabla \stream_{m,k}}_{L^\infty(\R^2)}
\sup_{t \in \supp \hat{\zeta}_{m,l_k} \zeta_{m,k}} \norm{\nabla X_{m-1,l_k}^{-1}}_{L^\infty(\R^2)}
\notag \\ & 
\qquad \leq
 C a_m \ep_m^2  \tau_{m}^{-1}
+
C a_m\ep_m  
\notag \\ & 
\qquad 
\leq
C 
\bigl(\ep_{m-1}^{q - (2-\beta+4\delta)}
+
1\bigr) 
\ep_m^{\beta-1} 
\leq
C  \ep_m^{\beta-1} 
\,.
\label{e.phim.diff.dt}
\end{align}
The exponent  of $\ep_{m-1}$ in the  last line was computed using~\eqref{e.beta.def.0},~\eqref{e.delta}, and~\eqref{e.supergeo}. Here we have also used that $\beta>1$ and $q>1$. By applying $\nabla^\perp$ to \eqref{e.psi.recursion.partial.t}, similarly to \eqref{e.phim.diff.dt} we deduce
\begin{align*}
&\norm{\partial_t ( \b_m - \b_{m-1} ) }_{L^\infty(\R\times\R^2)}
\notag\\
& 
\qquad \leq
2 \norm{ \partial_t \bigl( \hat{\zeta}_{m,l_k} \zeta_{m,k} \bigr) }_{L^\infty(\R)} 
\sup_{k\in\Z} \norm{\nabla \stream_{m,k} }_{L^\infty(\R^2)}  
\sup_{t \in \supp \hat{\zeta}_{m,l_k}\zeta_{m,k}} \norm{\nabla X_{m-1,l_k}^{-1}}_{L^\infty(\R^2)}
\notag\\
&\qquad \qquad
+
2 \norm{\b_{m-1}}_{L^\infty(\R\times\R^2)} \sup_{k\in\Z}  \norm{\nabla^2 \stream_{m,k}}_{L^\infty(\R^2)}
\sup_{t \in \supp \hat{\zeta}_{m,l_k}\zeta_{m,k}} \norm{\nabla X_{m-1,l_k}^{-1}}_{L^\infty(\R^2)}^2
\notag\\
&\qquad \qquad
+
2 \norm{\b_{m-1}}_{L^\infty(\R\times\R^2)} \sup_{k\in\Z}  \norm{\nabla  \stream_{m,k}}_{L^\infty(\R^2)}
\sup_{t \in \supp \hat{\zeta}_{m,l_k}\zeta_{m,k}} \norm{\nabla^2 X_{m-1,l_k}^{-1}}_{L^\infty(\R^2)}
\notag \\
&\qquad \qquad
+
2 \norm{\nabla \b_{m-1}}_{L^\infty(\R\times\R^2)} \sup_{k\in\Z}  \norm{\nabla  \stream_{m,k}}_{L^\infty(\R^2)}
\sup_{t \in \supp \hat{\zeta}_{m,l_k}\zeta_{m,k}} \norm{\nabla X_{m-1,l_k}^{-1}}_{L^\infty(\R^2)} 
\notag\\
&\qquad \leq C a_m \ep_m \tau_m^{-1} + C a_m + C a_m \ep_m a_{m-1}
\notag\\
&\qquad \leq  C a_m 
= C \ep_m^{\beta-2} \,.
\end{align*}
The upshot of the above estimate is that, after telescoping, we arrive at 
\begin{equation}
\norm{\partial_t  \b_m   }_{L^\infty(\R\times\R^2)} \leq C \sum_{j=0}^m \ep_j^{\beta-2} \leq C \ep_m^{\beta-2} \,.
\label{e.dt.bm}
\end{equation}
Next, we apply one more time derivative to \eqref{e.psi.recursion.partial.t} to obtain
\begin{align}
\label{e.psi.recursion.partial.tt}
&\partial_t^2 
\bigl( \phi_m - \phi_{m-1} \bigr) (t,x)\notag\\
&= 
\sum_{k\in\Z} 
\partial_t^2\bigl(\hat{\zeta}_{m,l_k} \zeta_{m,k}\bigr)
\stream_{m,k} \circ \flow_{m-1,l_k}^{-1} 
\notag \\ & \quad 
-2
\sum_{k\in\Z} 
\partial_t \bigl(\hat{\zeta}_{m,l_k} \zeta_{m,k}\bigr)
\nabla \stream_{m,k} \circ \flow_{m-1,l_k}^{-1}
\cdot
\b_{m-1} \cdot\nabla \flow_{m-1,l_k}^{-1} 
\notag \\ &\quad 
+ 
\sum_{k\in\Z} 
\hat{\zeta}_{m,l_k} \zeta_{m,k}
\nabla^2 \stream_{m,k} \circ \flow_{m-1,l_k}^{-1}
\colon
\bigl(
\b_{m-1} \cdot\nabla \flow_{m-1,l_k}^{-1} 
\bigr)
\otimes
\bigl(
\b_{m-1} \cdot\nabla \flow_{m-1,l_k}^{-1} 
\bigr)
\notag \\ &\quad 
- \sum_{k\in\Z} 
\hat{\zeta}_{m,l_k} \zeta_{m,k}
\nabla \stream_{m,k} \circ \flow_{m-1,l_k}^{-1}
\cdot
\Bigl(\partial_t \b_{m-1} \cdot\nabla \flow_{m-1,l_k}^{-1}
- \b_{m-1} \cdot \nabla \bigl(\b_{m-1} \cdot\nabla \flow_{m-1,l_k}^{-1}\bigr)
\Bigr)
\,.
\end{align}
Similarly to \eqref{e.phim.diff.dt}, and appealing in addition to the estimate \eqref{e.dt.bm}, we obtain
\begin{align}
\norm{\partial_t^2 ( \phi_m - \phi_{m-1} ) }_{L^\infty(\R\times\R^2)}
&\leq C a_m \ep_m^2 \tau_m^{-2} + C a_m \ep_m \tau_m^{-1} + C a_m 
+ C a_m \ep_m \bigl( \ep_{m-1}^{\beta-2} + a_{m-1}\bigr)
\notag\\
&\leq C a_m
= 
C \ep_{m}^{\beta-2} \,.
\label{e.phim.diff.dt2}
\end{align}
The claimed estimate \eqref{e.dt.Cbeta.phi.m.m-1} now follows from \eqref{e.phim.diff.dt} and \eqref{e.phim.diff.dt2} by interpolation, concluding the proof of the Proposition.
\end{proof}

\begin{corollary}
\label{c.phim}
There exists a~$C \in [1, \infty)$ which only depends on $\beta$, such that
for every $m\in\N$,
\begin{equation}
\label{e.phi.m.bounds}
\sup_{t\in \R} \snorm{\phi_{m}(\cdot,t)}_{n, C \ep_m^{-1}}
\leq C \ep_m^{\beta}
 \,, 
\qquad \forall n\in \N\,, n\geq 2
\,,
\end{equation}
and
\begin{equation}
\label{e.C1beta.phi.m}
\sup_{t\in \R} \snorm{\phi_{m}(\cdot,t)}_{n, C \ep_m^{-1}}
\leq C \ep_m^{n}
 \,, 
\qquad \forall n\in \{0,1\}
\,.
\end{equation}
For every $\beta' \in (0, \beta-1)$, the stream function~$\phi$ belongs to~$C^0(\R;C^{1,\beta'}(\R^2)) \cap C^{1,\beta'}(\R;L^\infty(\R^2))$, and in particular, the vector field~$\b$ belongs to~$C^0(\R;C^{0,\beta'}(\R^2)) \cap C^{0,\beta'}(\R;L^\infty(\R^2)) $. 
\end{corollary}

\begin{proof}[Proof of Corollary~\ref{c.phim}]
The bounds for the derivatives of $\phi_m$ of order $n$ with $n\geq 2$, claimed in \eqref{e.phi.m.bounds}, were already established in 
\eqref{e.phimbounds.subbed}. The estimate \eqref{e.C1beta.phi.m} was proven earlier in \eqref{e.phimbounds.subbed.2}. The claimed regularity of $\phi = \lim_{m\to \infty} \phi_m$ follows by telescoping sum $\phi_m = \sum_{j=1}^m  (\phi_j - \phi_{j-1}  )$, appealing to \eqref{e.C1beta.phi.m.m-1} and \eqref{e.dt.Cbeta.phi.m.m-1}, and using the fact that   by~\eqref{e.minsep} we have
\begin{equation*}
\sum_{n=1}^m \ep_m^p \leq 
C|p|^{-1} 
\cdot
\left\{
\begin{aligned}
& 1 & \ \text{if} \ p > 0, \\
& \ep_{m}^p & \ \text{if} \ p < 0. 
\end{aligned}
\right.
\end{equation*}
The regularity of $\b = \nabla^\perp \phi = \lim_{m\to \infty} \b_m$ follows from that of $\phi$ by interpolation.  
\end{proof}

By combining the estimates established in Proposition~\ref{p.SAMS.regularity} with the results of Proposition~\ref{p.ODE.flow}, we obtain the following useful results.
\begin{corollary}
\label{c.flowreg}
For every $s,t \in\R$ with $|t| \leq 2^{-25} a_m^{-1}$, we have that 
\begin{align}
\label{e.Xm.bound.1}
\norm{ \nabla X_{m}(t+s,\cdot,s) - \Itwo }_{L^\infty(\R^2)}
\leq
2^{23} |t| a_{m} 
\leq \tfrac 14
\,.
\end{align}
Moreover, for all $s,t \in\R$ with $|t| \leq 2^{-25} a_m^{-1}$, we have
\begin{align}
\snorm{ \nabla X_m(t+s,X_m^{-1}(t+s,\cdot,s),s) - \Itwo }_{n, 2^{10} \ep_m^{-1}} 
\leq 40 \,,
\qquad \forall  n \in \N_0
\,,
\label{e.Xm.bound.2}
\end{align}
and
\begin{equation}
\snorm{ \nabla X_m(t+s,\cdot,s)}_{n, 2^{14} \ep_m^{-1}} 
\leq 12\,,  
\qquad \forall  n \in \N 
\,.
\label{e.Xm.bound.3}
\end{equation}
In particular, \eqref{e.Xm.bound.1} and \eqref{e.Xm.bound.3} imply that 
\begin{align}
\sup_{|t| \leq 2^{-25} a_m^{-1}}  \norm{\nabla^n X_m(t+s,\cdot,s)}_{L^\infty(\R^2)} 
\leq 2 n! (2^{13} \ep_m^{-1})^{n-1}\,,  
\qquad \forall  n \in \N 
\,.
\label{e.Xm.bound.4}
\end{align}
\end{corollary}

\begin{proof}[Proof of Corollary~\ref{c.flowreg}]
From \eqref{e.phimbounds.subbed}, we deduce that for any $n\in \NN$, 
\begin{align}
\label{e.bm.Dn.new}
\sup_{t\in \R} \; \snorm{ \b_m(t,\cdot) }_{n,2^8 \ep_m^{-1}} 
&\leq \sup_{t\in \R} \frac{2 (n+1)^2}{ n!(2^8 \ep_m^{-1})^{n}}  \norm{\nabla^{n+1} \phi_{m}(t,\cdot)}_{L^\infty(\R^2)} 
\notag\\
&\leq \sup_{t\in \R} \; \snorm{\phi_m(t,\cdot)}_{n+1,2^8 \ep_m^{-1}} \frac{2 (n+1)^2}{ n!(2^8 \ep_m^{-1})^{n}} \frac{(n+1)! (2^8 \ep_m^{-1})^{n+1}}{(n+2)^2}
\notag\\
&\leq 2^{5} a_{m} \ep_{m}^{2}   \frac{(n+2)^2}{(n+1)^3} 
\frac{2 (n+1)^3(2^8 \ep_m^{-1}) }{(n+2)^2}  
\leq 2^{14} a_m \ep_m \,.
\end{align} 
The definition \eqref{e.flow.m.def} suggests that we apply Proposition~\ref{p.ODE.flow} with $\mathbf{f} = \b_m$, $C_{\mathbf{f}} = 2^{14} a_m \ep_m$, and $R_{\mathbf{f}} = 2^8 \ep_m^{-1}$.
The bound \eqref{eq:grad:psi} then directly implies \eqref{e.Xm.bound.1}  since $|t| \leq 2^{-25} a_m^{-1} = (8 C_{\mathbf{f}}R_{\mathbf{f}})^{-1}$.  Similarly, the bound \eqref{e.Xm.bound.2} for $n\geq 1$ follows from \eqref{e.ODE.flow.estimate} with $d=2$, since $R_\f (1+ 8 |t| C_\f R_\f)^2 \leq 4 R_\f = 2^{10} \ep_m^{-1}$.  
The bound \eqref{e.Xm.bound.3} is a direct consequence of \eqref{e.ODE.flow.estimate:new} for $d=2$, since $16R_\f(1+ 16 C_{\mathbf{f}}R_{\mathbf{f}}|t|) \leq 48 R_\f \leq 2^{14} \ep_{m}^{-1}$. Lastly, the estimate \eqref{e.Xm.bound.4} follows from \eqref{e.Xm.bound.1} (for $n=1$) and \eqref{e.Xm.bound.3} (for $n\geq 2$), upon recalling definition \eqref{e.barf}.
\end{proof}

\subsection{Material derivative estimates}
The bound \eqref{e.bm.Dn.new} implies that for all $n \in \N$ we have 
\begin{equation}
\norm{\nabla^n \b_{m}}_{L^\infty(\RR\times \RR^2)} 
\leq 2^{14} a_m \ep_m  \frac{n! (2^8 \ep_m^{-1})^{n}}{(n+1)^2} 
\leq 2^{22} \ep_m^{\beta-2}  (n-1)! (2^8  \ep_m^{-1})^{n-1} 
\,.
\label{e.bm.Dn}
\end{equation}
In order to estimate the time correctors in our two-scale ansatz in Section~\ref{s.multiscale}, it turns out that we also need to have estimates available for 
\begin{equation*}
\norm{\nabla^n \DD_{t,m}^\ell \b_m}_{L^\infty(\RR\times \RR^2)} 
+
\norm{\nabla^{n-1} \DD_{t,m}^\ell \nabla \b_m}_{L^\infty(\RR\times \RR^2)} 
\qquad \forall n,\ell \in \N_0 \mbox{ such that }  1 \leq n + \ell \leq N_*.
\end{equation*} 
Here and throughout the rest of the paper we use the notation $\DD_{t,m}$ for the material derivative along the vector field $\b_m$, which is the scalar differential operator defined by
\begin{equation}
\label{e.D.t.m.def}
\DD_{t,m} = \partial_t + \b_{m} \cdot \nabla 
\,.
\end{equation}
Note that as opposed to \eqref{e.bm.Dn}, in which the index of the space derivatives is allowed to be arbitrarily large ($n\in \N$), in \eqref{e.material.goal} we only are concerned with a total derivative index $n+\ell \leq N_*$ which is finite (in particular, bounded independently of $m$). As such, the implicit constants in these estimates are allowed to depend on $n$ and $\ell$, because this just means that they depend on $N_*$, and so they depend on our choice of $\beta$ (via~\eqref{e.N}). The advantage of this relaxation is that we do not need to keep track of factorial terms (e.g.~$n!, \ell!, (n+\ell)!$), or on powers of constants (e.g.~$C^n, C^\ell$). In particular, as opposed to the previous section, where we had to carefully apply the auxiliary lemmas from Appendix~\ref{prop:multi:Faa}, here we can apply   standard consequences of the Leibniz and chain rules, such as
\begin{align}
\label{e.Moser}
\norm{\nabla^n (f \, g)}_{L^\infty}
&\leq C \norm{\nabla^n f}_{L^\infty} \norm{g}_{L^\infty} + C \norm{f}_{L^\infty} \norm{\nabla^n g}_{L^\infty} 
\,,
\\
\label{e.chain}
\norm{\nabla^n (f\circ g)}_{L^\infty} 
&\leq C \norm{\nabla f}_{L^\infty} \norm{\nabla g}_{C^{n-1}} + C\norm{\nabla f}_{C^{n-1}} \norm{\nabla g}_{L^\infty}^{n}
\,,
\end{align}
for $1 \leq n\leq N_*$, where $C$ only depends on $n$ (hence on $N_*$, hence on $\beta$).  
 
\smallskip
 
The main result of this section is:
\begin{proposition}
\label{p.material.goal}
Assume that $n,\ell \in \NN_0$ are such that $1 \leq n+\ell \leq N_*$. Then, we have that 
\begin{align}
\label{e.material.goal} 
\norm{\nabla^n \DD_{t,m}^\ell \b_m}_{L^\infty(\RR\times \RR^2)} 
&\leq C \ep_m^{\beta-1} \bigl(\ep_m^{\beta-2}\bigr)^\ell \bigl(\ep_m^{-1}\bigr)^{n}
\\
\label{e.monofractal.vomit} 
\norm{\nabla^{n-1} \DD_{t,m}^\ell \nabla \b_m}_{L^\infty(\RR\times \RR^2)} 
&\leq C \ep_m^{\beta-2} \bigl(\ep_m^{\beta-2}\bigr)^\ell \bigl(\ep_m^{-1}\bigr)^{n-1}
\end{align}
where $C$ only depends on $\beta$.
\end{proposition}
\begin{proof}[Proof of Proposition~\ref{p.material.goal}]
We first establish~\eqref{e.material.goal}. 
The bound \eqref{e.material.goal} for $\ell =0$ was already established in \eqref{e.bm.Dn}. We next consider the case $\ell=1$, which is the first interesting case; the proof of this case contains all the main ideas, but without the messy details about commutators.  

\smallskip

We prove \eqref{e.material.goal} for $\ell =1$ by induction on $m$. When $m=0$, then $\b_0=0$ by \eqref{e.phi0def}, so there is nothing to prove. Inductively, let $m\geq 1$ and assume that \eqref{e.material.goal} for $\ell =1$ holds with $m$ replaced by $m' \leq m-1$. Note that the $m$ dependence appears both through the function whose derivatives we study, namely $\b_m$, but also through the differential operator $\DD_{t,m}$ defined in~\eqref{e.D.t.m.def}. This nonlinear dependence on $m$ makes it convenient to introduce the notation $\vv_m$ to denote the ``fast part'' of the vector field $\b_m$, namely
\begin{equation}
\vv_m 
= \b_m - \b_{m-1} 
=  \nabla^\perp \sum_{k\in \ZZ} 
\hat{\zeta}_{m,l_k} \zeta_{m,k} 
\stream_{m,k} \circ X_{m-1,l_k}^{-1}
\,.
\label{eq.vvm.def}
\end{equation}
With this notation~\eqref{e.D.t.m.def} becomes
\begin{equation*}
\DD_{t,m} = \DD_{t,m-1} + \vv_m \cdot \nabla \,,
\end{equation*}
and the bound \eqref{e.phi.m.m-1.bounds} may be recast as
\begin{equation}
\norm{\nabla^n \vv_{m}}_{L^\infty(\RR\times \RR^2)} 
\leq C \ep_m^{\beta-1}  (n-1)! (C \ep_m^{-1})^{n} 
\label{e.vm.Dn}
\end{equation}
for all $n\in \N_0$. The difference with \eqref{e.bm.Dn} is that \eqref{e.vm.Dn} includes the case $n=0$.

\smallskip

Next,  we note that 
\begin{equation}
\DD_{t,m} \b_m - \DD_{t,m-1} \b_{m-1} 
=  
\DD_{t,m-1} \vv_m 
+ (\vv_m \cdot \nabla)\b_{m-1} +(\vv_m \cdot \nabla)\vv_m
\label{e.Dt.vm.0}
\,.
\end{equation}
The reason for the above decomposition lies in the fact that the term $\DD_{t,m-1} \vv_m$ contains an important cancellation, namely $\DD_{t,m-1} X_{m-1,l_k}^{-1} = 0$,  and as such
\begin{align}
\DD_{t,m-1} \vv_m 
&= [\DD_{t,m-1} , \nabla^\perp]  \sum_{k\in \ZZ} \hat{\zeta}_{m,l_k} \zeta_{m,k} \stream_{m,k} \circ X_{m-1,l_k}^{-1} 
+ \nabla^\perp \DD_{t,m-1} \sum_{k\in \ZZ} \hat{\zeta}_{m,l_k} \zeta_{m,k} \stream_{m,k} \circ X_{m-1,k}^{-1}
\notag\\
&= - \nabla^\perp \b_{m-1} \cdot \nabla  \sum_{k\in \ZZ} \hat{\zeta}_{m,l_k} \zeta_{m,k} \stream_{m,k} \circ X_{m-1,l_k}^{-1} 
+ \nabla^\perp \sum_{k\in \ZZ} \partial_t \bigl(\hat{\zeta}_{m,l_k} \zeta_{m,k}\bigr) \stream_{m,k} \circ X_{m-1,l_k}^{-1}
\notag\\
&= - \nabla^\perp \b_{m-1}^\perp \cdot \vv_m
+ \nabla^\perp \sum_{k\in \ZZ} \partial_t \bigl(\hat{\zeta}_{m,l_k} \zeta_{m,k}\bigr) \stream_{m,k} \circ X_{m-1,l_k}^{-1}
\,.
\label{e.Dt.vm.1}
\end{align}
Identity \eqref{e.Dt.vm.1} makes formal the intuition that the ``cost'' of $\DD_{t,m-1}$ acting on $\vv_m$ is equal to the maximum between $\norm{\nabla \b_{m-1}}_{L^\infty_{t,x}}$ and $\norm{ \partial_t  (\hat{\zeta}_{m,l_k} \zeta_{m,k} )}_{L^\infty_t}$. Indeed, from \eqref{e.taubounds}, \eqref{e.zeta.mk}, \eqref{e.zeta.prime.ml.bounds}, \eqref{e.bm.Dn} (with $n=1$), and \eqref{e.vm.Dn} (with $n=0$) we deduce from \eqref{e.Dt.vm.1} that 
\begin{align*}
\norm{\DD_{t,m-1} \vv_m}_{L^\infty(\RR \times \RR^2)} 
\leq C \bigl(  \ep_{m-1}^{\beta-2}   + \tau_m^{-1} \bigr) \ep_m^{\beta-1}
\leq C \tau_m^{-1} \ep_m^{\beta-1}
\,.
\end{align*}
In fact, by appealing to \eqref{e.psimk.snorm}, \eqref{e.Xm.regbounds}, \eqref{e.bm.Dn}, \eqref{e.Moser}, \eqref{e.chain}, and~\eqref{e.vm.Dn}, we deduce from~\eqref{e.Dt.vm.1} that for $0 \leq n \leq N_* - 1$, 
\begin{align}
\norm{\nabla^n \DD_{t,m-1} \vv_m}_{L^\infty(\RR \times \RR^2)} 
\leq C \tau_m^{-1} \ep_m^{\beta-1} \bigl( \ep_{m-1}^{-n} + \ep_m^{-n} + \ep_{m}^{-1} \ep_{m-1}^{-n+1} + \ep_m^{-n} \bigr)
\leq C \tau_m^{-1} \ep_m^{\beta-1} \ep_m^{-n}
\,.
\label{e.material.goal.temp.1}
\end{align}
This handles the estimates for the first term on the right side of \eqref{e.Dt.vm.0}. 
For the second term on the right side of \eqref{e.Dt.vm.0}, by \eqref{e.bm.Dn}, \eqref{e.vm.Dn}, and the Leibniz rule, we deduce that 
\begin{align}
\norm{\nabla^n \bigl( \vv_m \cdot \nabla \b_{m-1}\bigr)  }_{L^\infty(\RR \times \RR^2)} 
&\leq  C \ep_m^{\beta-1} \ep_{m-1}^{\beta-2} \bigl( \ep_m^{-n} + \ep_{m-1}^{-n} \bigr)
\leq  C \ep_m^{\beta-1} \tau_m^{-1} \ep_m^{-n}
\label{e.material.goal.temp.2}
\,.
\end{align}
Also, by \eqref{e.vm.Dn} and the Leibniz rule we obtain
\begin{align}
\norm{\nabla^n \bigl( \vv_m \cdot \nabla \vv_m \bigr)  }_{L^\infty(\RR \times \RR^2)}
&\leq  C \ep_m^{\beta-1}  \ep_m^{\beta-2} \ep_m^{-n}
\,.
\label{e.material.goal.temp.3}
\end{align}
Comparing the bounds in \eqref{e.material.goal.temp.1}--\eqref{e.material.goal.temp.3}, and noting that \eqref{e.taubounds} and~\eqref{e.delta} give
\begin{equation}
\label{e.going.slightly.mad}
\tau_m^{-1} 
\leq C \ep_{m-1}^{\beta-2-4\delta} 
\leq C \ep_{m-1}^{q(\beta-2)} 
\leq C \ep_m^{\beta-2}  
\end{equation}
we deduce from \eqref{e.Dt.vm.0} that for all $0 \leq n \leq N_*-1$, we have that 
\begin{equation*}
\norm{\nabla^n \bigl( \DD_{t,m} \b_m - \DD_{t,m-1} \b_{m-1} \bigr)}_{L^\infty(\RR\times\RR^2)}
\leq C \ep_m^{\beta-1} \ep_m^{\beta-2}  \ep_m^{-n}
\,.
\end{equation*}
Lastly, using that $(\beta-2) + (\beta-1) - n \leq 2\beta-3 \leq \nicefrac 83 -3 = -\nicefrac 13 < 0$ and $\ep_m \leq \ep_{m-1}$, we may use that by induction the estimate~\eqref{e.material.goal} holds at level $m-1$, to deduce  
\begin{equation}
\label{e.material.goal.1}
\norm{\nabla^n  \DD_{t,m} \b_m }_{L^\infty(\RR\times\RR^2)}
\leq 
C\ep_m^{\beta-2}  \ep_m^{\beta-1}   \ep_m^{-n}
+
C \ep_{m-1}^{\beta-2} \ep_{m-1}^{\beta-1}  \ep_{m-1}^{-n}
\leq C \ep_m^{2\beta-3-n}
\,,
\end{equation}
where the constant $C$ is independent of $m$, but may depend on $n \leq N_*-1$. This establishes \eqref{e.material.goal} at level $m$, when $\ell=1$. By induction on $m$, we have proven the bound~\eqref{e.material.goal} for $\ell =1$ and $m\geq 0$. 

\smallskip

The proof of~\eqref{e.material.goal}  for $\ell\geq 2$ proceeds in a similar manner, but it requires a number of commutator estimates, because the operators $\{ \DD_{t,m}, \vv_m \cdot \nabla, \nabla, \partial_t \}$, do not commute. As before, when $m=0$ there is nothing to prove because $\b_0 = 0$. Inductively, let $m\geq 1$ and assume that \eqref{e.material.goal} for $\ell =2$ holds with $m$ replaced by $m' \leq m-1$. By differentiating \eqref{e.Dt.vm.0} with respect to $\DD_{t,m} = \DD_{t,m-1} + \vv_m \cdot \nabla$, we obtain
\begin{align}
\label{e.Dt.vm.2}
&\DD_{t,m}^2 \b_m 
- \DD_{t,m-1}^2 \b_{m-1}
- \DD_{t,m-1}^2 \vv_m
\notag\\
&\qquad 
= (\vv_m \cdot \nabla) \DD_{t,m-1}\b_{m-1} 
+ (\vv_m \cdot \nabla) \DD_{t,m-1} \vv_m 
+ (\DD_{t,m-1} \vv_m \cdot\nabla) \b_m 
\notag\\
&\qquad \qquad 
+ (\vv_m \cdot \nabla) \DD_{t,m} \b_m 
- \bigl((\vv_m \cdot \nabla) \b_{m-1} \cdot \nabla\bigr) \b_m
\,.
\end{align}
All terms on the right side of \eqref{e.Dt.vm.2} contain at most one material derivative, and therefore are already bounded in light of  \eqref{e.vm.Dn}, \eqref{e.material.goal.temp.1}, and \eqref{e.material.goal} with $\ell \in \{0,1\}$. By also appealing to~\eqref{e.Moser} and~\eqref{e.going.slightly.mad}, we obtain for all $n\leq N_* - 2$, 
\begin{align}
\norm{\nabla^n \bigl(\mbox{RHS of  }\eqref{e.Dt.vm.2}\bigr)}_{L^\infty(\RR\times\RR^2)}
&
\leq 
C \ep_m^{\beta-1} \ep_{m-1}^{2(\beta-2)} (\ep_{m}^{-n} + \ep_{m-1}^{-n})
+
C \ep_m^{\beta-1} \tau_m^{-1} \ep_{m}^{\beta-2} \ep_m^{-n}
\notag\\
&\qquad  
+
C \ep_m^{\beta-1} \ep_{m}^{2(\beta-2)}  \ep_{m}^{-n}  
+ 
C \ep_m^{\beta-1}\ep_{m-1}^{\beta-2} \ep_m^{\beta-2}(\ep_{m}^{-n} + \ep_{m-1}^{-n})
\notag\\
&\leq C \ep_m^{\beta-1} \bigl( \ep_m^{\beta-2} \bigr)^2 \ep_m^{-n} \,.
\label{e.material.goal.temp.4}
\end{align}
In order to estimate the contribution from $\DD_{t,m-1}^2 \vv_m$, we apply $\DD_{t,m-1}$ to \eqref{e.Dt.vm.1}, and deduce that 
\begin{align*}
\DD_{t,m-1}^2 \vv_m 
&= - \nabla^\perp \b_{m-1}^\perp \cdot \DD_{t,m-1} \vv_m -  \nabla^\perp \DD_{t,m-1} \b_{m-1}^\perp \cdot  \vv_m + \nabla^\perp \b_{m-1} \cdot \nabla \b_{m-1}^\perp \cdot   \vv_m
\notag\\
& - \nabla^\perp \b_{m-1}^\perp \cdot  \nabla^\perp \sum_{k\in \ZZ} \partial_t \bigl(\hat{\zeta}_{m,l_k} \zeta_{m,k}\bigr) \stream_{m,k} \circ X_{m-1,l_k}^{-1} 
 + \nabla^\perp \sum_{k\in \ZZ} \partial_t^2 \bigl(\hat{\zeta}_{m,l_k} \zeta_{m,k}\bigr) \stream_{m,k} \circ X_{m-1,l_k}^{-1} 
\,.
\end{align*}
Thus, by appealing to~\eqref{e.psimk.snorm}, \eqref{e.zeta.mk}, \eqref{e.zeta.prime.ml.bounds}, \eqref{e.Xm.regbounds},   \eqref{e.bm.Dn}, \eqref{e.Moser}, \eqref{e.material.goal} with $\ell =1$, \eqref{e.vm.Dn}, \eqref{e.material.goal.temp.1}, and~\eqref{e.going.slightly.mad}, analogously to~\eqref{e.material.goal.temp.4} we may deduce that for $n\leq N_*-2$, 
\begin{equation}
\norm{\nabla^n \DD_{t,m-1}^2 \vv_m}_{L^\infty(\RR\times\RR^2)} 
\leq C \ep_m^{\beta-1} \bigl( \ep_m^{\beta-2} \bigr)^2 \ep_m^{-n}  \,.
\label{e.material.goal.temp.5}
\end{equation}
Combining \eqref{e.Dt.vm.2}, \eqref{e.material.goal.temp.4}, \eqref{e.material.goal.temp.5}, 
and the inductive assumption that \eqref{e.material.goal} holds for $\ell =2$ at level $m-1$, 
we deduce that for $n\leq N_*-2$, 
\begin{align*}
\norm{\nabla^n \DD_{t,m}^2 \b_m}_{L^\infty(\RR\times\RR^2)}
&\leq \norm{\nabla^n \DD_{t,m-1}^2 \b_{m-1}}_{L^\infty(\RR\times\RR^2)} 
+ C \ep_m^{\beta-1} \bigl( \ep_m^{\beta-2} \bigr)^2 \ep_m^{-n} 
\notag\\
&\leq C \ep_{m-1}^{\beta-1} \bigl(\ep_{m-2}^{\beta-2}\bigr)^2  \ep_{m-1}^{-n} 
+ C \ep_m^{\beta-1} \bigl( \ep_m^{\beta-2} \bigr)^2 \ep_m^{-n} 
\notag\\
&\leq 
C \ep_m^{\beta-1} \bigl( \ep_m^{\beta-2} \bigr)^2 \ep_m^{-n} \,.
\end{align*}
By induction on $m$, this concludes the proof of \eqref{e.material.goal} for $\ell=2$.

\smallskip

Estimate \eqref{e.material.goal} in the case $3 \leq \ell \leq N_*$  and $0\leq n \leq N_*-\ell$ may be proven in the same way as for $\ell \in \{1,2\}$, save for the bookkeeping, which becomes tedious. Inductively on $\ell$, in analogy to \eqref{e.Dt.vm.0} and  \eqref{e.Dt.vm.2} we may show that the expression $\DD_{t,m}^\ell \b_m - \DD_{t,m-1}^{\ell} \b_{m-1} - \DD_{t,m-1}^\ell \vv_m$ is given by a sum of terms which contain at most $\ell-1$ material derivatives, and are hence bounded by induction. The precise accounting of all terms requires  estimates for high-order commutators with material derivatives (e.g.~$ [\nabla^n, \DD_{t,m-1}^{\ell}  ]$), and for powers of sums of non-commuting operators (e.g.~$(\DD_{t,m-1} + \vv_m \cdot \nabla)^\ell$). Such estimates are given in~\cite[Appendices~A.6 and~A.7]{BMNV}. Using~\cite[Appendices~A.6 and~A.7]{BMNV}, we may show that every additional material derivative $\DD_{t,m}$ landing on $\b_m$ ``costs'' a factor of at most $\ep_m^{\beta-2}$, while every additional space derivatives ``costs'' a factor of at most $\ep_m^{-1}$. We omit these details.

\smallskip

We now turn to the proof of~\eqref{e.monofractal.vomit}. This bound follows from \eqref{e.material.goal} if we are able to estimate the commutator $\big[\DD_{t,m}^{\ell},\nabla\bigr] \b_m$. When $\ell=1$, this commutator   equals $\bigl[ \DD_{t,m}, \nabla\bigr] = -\nabla \b_m \cdot \nabla$, and hence 
\begin{equation*}
\nabla^{n-1} \DD_{t,m} \nabla \b_m =  \nabla^{n} \DD_{t,m} \b_m - \nabla^{n-1} \left( \nabla \b_m \cdot \nabla \b_m\right)
\,.
\end{equation*}
Upon appealing to \eqref{e.bm.Dn}, \eqref{e.Moser}, and \eqref{e.material.goal} we obtain
\begin{equation*}
\norm{\nabla^{n-1} \DD_{t,m} \nabla \b_m}_{L^\infty(\RR\times\RR^2)}  
\leq 
C \ep_m^{\beta-1} \bigl(\ep_m^{\beta-2}\bigr)  \bigl(\ep_m^{-1}\bigr)^{n} 
+
C \bigl(\ep_m^{\beta-2}\bigr)^2 \bigl(\ep_m^{-1}\bigr)^{n-1} 
= 
C \bigl(\ep_m^{\beta-2}\bigr)^2 \bigl(\ep_m^{-1}\bigr)^{n-1} 
\,.
\end{equation*}
This establishes~\eqref{e.monofractal.vomit} when $\ell =1$, for $n \leq N_*-1$.

\smallskip

In order to prove~\eqref{e.monofractal.vomit} for $\ell\geq 2$, we assume by induction that \eqref{e.monofractal.vomit} holds for $\ell' \leq \ell-1$. At this stage, we recall from~\cite[Lemma A.12]{BMNV} that the commutator between high powers of the material derivative operator and a space-gradient is given by 
\begin{equation}
 \big[\DD_{t,m}^{\ell},\nabla\bigr]   
  = \sum_{\ell'=1}^\ell 
 \binom{\ell}{\ell'} 
  \bigl({\rm ad} \DD_{t,m}\bigr)^{\ell'}(\nabla) \DD_{t,m}^{\ell-\ell'}  
 \label{e.vomit.cascade.1}
 \,,
\end{equation}
where $\bigl({\rm ad} \DD_{t,m}\bigr)^0(\nabla) = \nabla$, and  recursively we   define $\bigl({\rm ad} \DD_{t,m}\bigr)^{\ell^\prime}(\nabla) = \bigl[\DD_{t,m}, \bigl({\rm ad} \DD_{t,m}\bigr)^{\ell^\prime-1}(\nabla)\bigr]$.
In turn, using~\cite[Lemma A.13]{BMNV}, we have that 
\begin{equation}
\bigl({\rm ad} \DD_{t,m}\bigr)^\ell(\nabla)  
= \sum_{j=1}^\ell \sum_{\aa  \in \NN_0^j \colon |\aa | = \ell - j} c_{\ell,j,\aa } \prod_{i=1}^j (\DD_{t,m}^{\aa _i} \nabla \b_m) \cdot \nabla   
\,, 
\label{e.vomit.cascade.2}
\end{equation}
where the product $\prod_{i=1}^j$ is the product of matrices, and the coefficients $c_{\ell,j,\aa }$ only depend on $\ell, j, \aa $. 
From \eqref{e.vomit.cascade.1} and \eqref{e.vomit.cascade.2} we thus obtain
\begin{align*}
\norm{\nabla^{n-1} \DD_{t,m}^\ell \nabla \b_m}_{L^\infty_{t,x}} 
&\leq  \norm{\nabla^{n} \DD_{t,m}^\ell  \b_m}_{L^\infty_{t,x}} 
+  \norm{\nabla^{n-1} \bigl[ \DD_{t,m}^\ell  ,\nabla \bigr] \b_m}_{L^\infty_{t,x}} 
\notag\\
&\leq \norm{\nabla^{n} \DD_{t,m}^\ell  \b_m}_{L^\infty_{t,x}} 
+ C \sum_{n'=0}^{n-1} \sum_{\ell'=1}^{\ell}
\norm{\nabla^{n'} \bigl({\rm ad} \DD_{t,m}\bigr)^{\ell'}(\nabla) \nabla^{n-1-n'} \DD_{t,m}^{\ell-\ell'} \b_m}_{L^\infty_{t,x}} 
\,.
\end{align*}
Since $\ell-\ell' \leq \ell-1$, and upon noting that the $\aa $ in \eqref{e.vomit.cascade.2} satisfies $|\aa | = \ell - j  \leq \ell - 1$, we deduce that from the inductive assumption \eqref{e.monofractal.vomit} and the bound \eqref{e.material.goal} that
\begin{align*}
\norm{\nabla^{n-1} \DD_{t,m}^\ell \nabla \b_m}_{L^\infty_{t,x}} 
&\leq C \ep_m^{\beta-1} (\ep_m^{\beta-2})^\ell (\ep_m^{-1})^{n} \notag\\
&\quad 
+C  \sum_{n'=0}^{n-1} \sum_{\ell'=1}^{\ell} \sum_{j=1}^{\ell'} \sum_{\aa  \in \NN_0^j \colon |\aa | = \ell'-j} 
\norm{\nabla^{n'} \prod_{i=1}^j (\DD_{t,m}^{\aa _i} \nabla \b_m)}_{L^\infty_{t,x}}
\norm{\nabla^{n-n'} \DD_{t,m}^{\ell-\ell'} \b_m}_{L^\infty_{t,x}} \notag\\
&\leq C \ep_m^{\beta-1} (\ep_m^{\beta-2})^\ell (\ep_m^{-1})^{n} \notag\\
&\quad 
+ C \sum_{n'=0}^{n-1} \sum_{\ell'=1}^{\ell} \sum_{j=1}^{\ell'} \sum_{\aa  \in \NN_0^j \colon |\aa | = \ell'-j} 
(\ep_m^{\beta-2})^{|\aa|+j} (\ep_m^{-1})^{n'}
\ep_m^{\beta-1} (\ep_m^{\beta-2})^{\ell-\ell'} (\ep_m^{-1})^{n-n'}
\notag\\
&\leq C \ep_m^{\beta-1} (\ep_m^{\beta-2})^\ell (\ep_m^{-1})^{n}
\,.
\end{align*}
A close inspection of the above chain of inequalities reveals that the total number of space, plus the total number of material derivatives, never exceeds $N_*$ as desired.
This establishes the inductive step for \eqref{e.monofractal.vomit},  concluding the proof.
\end{proof}

An immediate consequence of Proposition~\ref{p.material.goal} is an estimate for $\nabla^n \partial_t^\ell$   applied to $\nabla X_{m,l}^{-1} \circ X_{m,l}$. 
\begin{corollary}
\label{c.material.goal}
Assume that $n,\ell \in \NN_0$ are such that $n+\ell \leq N_*$. Then, for all $t,s\in\R$ with $|t| \leq 2^{-25} a_m^{-1}$,  we have that
\begin{align}
\norm{\nabla^n \partial_t^\ell\bigl( \nabla X_{m}^{-1}(t+s,X_{m}(t+s,\cdot,s),s) \bigr)}_{L^\infty(\RR^2)}
&\leq C \ep_m^{-n} \bigl(\ep_m^{\beta-2}\bigr)^\ell
\,,
\label{eq:space:time:X:bounds}
\end{align}
where the constant $C \geq 1$ depends only on $\beta$, through $N_*$.
\end{corollary}

\begin{proof}[Proof of Corollary~\ref{c.material.goal}]
When $\ell=0$ and $n=0$, the bound \eqref{eq:space:time:X:bounds} follows from \eqref{e.Xm.regbounds}, while for $\ell =0$ and $1\leq n \leq N_*$ we additionally appeal to
\eqref{e.Xm.bound.1}, \eqref{e.Xm.bound.4}, and \eqref{e.chain} to deduce
\begin{align*}
\norm{\nabla^n \bigl( \nabla X_{m}^{-1} \circ X_{m} \bigr)}_{L^\infty(\RR^2)}
&\leq 
C \norm{\nabla^2 X_m^{-1}}_{L^\infty} \norm{\nabla X_m}_{C^{n-1}} 
+ 
C \norm{\nabla^2 X_{m}^{-1}}_{C^{n-1}} \norm{\nabla X_m}_{L^\infty}^n
\\
&\leq 
C \ep_m^{-1} (\ep_{m}^{-1})^{n-1} 
+ 
C (\ep_{m}^{-1})^{n} 
= C \ep_m^{-n}
\,.
\end{align*}
As such, it only remains to prove \eqref{eq:space:time:X:bounds} for  $\ell\geq 1$ and $n\leq N_*-\ell$.

\smallskip

Using that $\DD_{t,m} X_m^{-1} = 0$, for $1 \leq \ell \leq N_*$ we have
\begin{align}
\partial_t^\ell \bigl( \nabla X_{m}^{-1} \circ X_{m}\bigr) 
= \left( \DD_{t,m}^\ell \nabla X_{m}^{-1}\right)\circ X_{m}  
=  \left( \bigl[\DD_{t,m}^{\ell},\nabla\bigr]  X_{m}^{-1}\right)  \circ X_{m}   
\label{e.vomit.cascade}
\,.
\end{align}
The formula for the commutator present in \eqref{e.vomit.cascade} was recorded earlier in~\eqref{e.vomit.cascade.1}. By again using that $\DD_{t,m} X_m^{-1} = 0$, we deduce from \eqref{e.vomit.cascade.1} that
\begin{align}
\label{e.vomit.cascade.1a}
 \big[\DD_{t,m}^{\ell},\nabla\bigr]  X_{m}^{-1} 
  = \sum_{\ell'=1}^\ell
  \binom{\ell}{\ell'}  
  \bigl({\rm ad} \DD_{t,m}\bigr)^{\ell'}(\nabla) \DD_{t,m}^{\ell-\ell'} X_m^{-1} 
  = \bigl({\rm ad} \DD_{t,m}\bigr)^\ell(\nabla) X_m^{-1}
  \,,
\end{align}
where the first order differential operator $\bigl({\rm ad} \DD_{t,m}\bigr)^\ell(\nabla)$ is given by~\eqref{e.vomit.cascade.2}. 
With the available bound \eqref{e.monofractal.vomit}, and the identity~\eqref{e.vomit.cascade.2},  we bound the left side of \eqref{e.vomit.cascade.1a} as
\begin{align}
\label{e.salmon.bear.fat}
&\norm{\nabla^n \bigl[ \DD_{t,m}^\ell, \nabla\bigr] X_m^{-1}}_{L^\infty_{t,x}}\notag\\
&\qquad = \norm{\nabla^n \bigl({\rm ad} \DD_{t,m}\bigr)^\ell(\nabla) X_m^{-1}}_{L^\infty_{t,x}} \notag\\
&\qquad \leq C \sum_{n'=0}^{n}  \sum_{j=1}^\ell \sum_{\aa  \in \NN_0^j \colon |\aa | = \ell - j} \sum_{\bb  \in \NN_0^j \colon |\bb | = n'}  \prod_{i=1}^j 
\norm{\nabla^{\bb_i} \DD_{t,m}^{\aa _i} \nabla \b_m}_{L^\infty_{t,x}}  \norm{\nabla^{n-n'+1}  X_m^{-1}}_{L^\infty_{t,x}}
\notag\\
&\qquad \leq C \sum_{n'=0}^{n} \sum_{j=1}^\ell \sum_{\aa  \in \NN_0^j \colon |\aa | = \ell - j}\sum_{\bb  \in \NN_0^j \colon |\bb | = n'}  \prod_{i=1}^j 
(\ep_m^{\beta-2})^{\aa_i+1} (\ep_m^{-1})^{\bb_i}
(\ep_m^{-1})^{n-n'}
\notag\\
&\qquad \leq C  (\ep_m^{\beta-2})^\ell (\ep_m^{-1})^{n}
\,.
\end{align}
In order to conclude the proof of \eqref{eq:space:time:X:bounds}, we combine the above estimate with the identity \eqref{e.vomit.cascade}, and~\eqref{e.chain}, to obtain
\begin{align*}
&\norm{\nabla^n \partial_t^\ell (\nabla X_m^{-1} \circ X_m)}_{L^\infty_{t,x}} \notag\\
&\qquad \leq C \norm{\nabla  \bigl[ \DD_{t,m}^\ell, \nabla\bigr] X_m^{-1}}_{L^\infty_{t,x}} \norm{\nabla X_m}_{L^\infty_t C^{n-1}_x} + C \norm{\nabla  \bigl[ \DD_{t,m}^\ell, \nabla\bigr] X_m^{-1}}_{L^\infty_t C^{n-1}_x} \norm{\nabla X_m}_{L^\infty_{t,x}}^n \notag\\
&\qquad \leq C  (\ep_m^{\beta-2})^\ell   (\ep_m^{-1})^{n }  \,. 
\end{align*}
In the above estimate, the time-support was not written out explicitly, but it was implicitly assumed to be such that $|t|\leq 2^{-25} a_m^{-1}$, so that we could appeal to the bounds \eqref{e.Xm.regbounds} and \eqref{e.Xm.bound.4}. This concludes the proof of the corollary.
\end{proof}

A second consequence of Proposition~\ref{p.material.goal} is the following estimate on~$\nabla^n \DD_{t,m}^\ell \bigl( \nabla X_{m,l} \circ X_{m,l}^{-1}\bigr)$. This  is necessary to estimate the spatial and material derivatives of the matrix~$\s_m$ in Section~\ref{s.Tm.reg}. 

\begin{corollary}
\label{c.material.DX.Xinv} 
Assume that $n,\ell \in \NN_0$ are such that $n+\ell \leq N_*$. Then, for all $t,s\in\R$ with $|t| \leq 2^{-25} a_m^{-1}$,  we have that
\begin{equation}
\norm{\nabla^n \DD_{t,m}^\ell\bigl( \nabla X_{m}(t+s,X_{m}^{-1}(t+s,\cdot,s),s) \bigr)}_{L^\infty(\RR^2)}
\leq C \ep_m^{-n} \bigl(\ep_m^{\beta-2}\bigr)^\ell
\,,
\label{e.spacetime.X.bnds}
\end{equation}
where the constant $C\geq 1$ depends only on $\beta$, through $N_*$.
\end{corollary}
\begin{proof}[Proof of Corollary~\ref{c.material.DX.Xinv}]
When $\ell=0$, the desired bounds were already obtained in \eqref{e.Xm.bound.2}. For $\ell\geq 1$, the proof of \eqref{e.spacetime.X.bnds} starts with the  observation  that  
$
\nabla X_{m} \circ X_{m}^{-1} = ( \nabla X_{m}^{-1} )^{-1}
$
as $2\times 2$ matrices. Since the flows $\b_m$ that define $X_{m}$ are incompressible, we have that  ${\rm det}(\nabla X_{m}^{-1}) = 1$, and so $\nabla X_{m} \circ X_{m}^{-1}$ equals the transpose of the cofactor matrix associated to $\nabla X_{m}^{-1}$. In turn, since we are in two space dimensions, this cofactor matrix equals to $\nabla^\perp (X_{m}^{-1})^\perp$. This leads to the identity
\begin{equation*}
\nabla X_{m} \circ X_{m}^{-1} = ( \nabla X_{m}^{-1} )^{-1}  = \bigl( \nabla^\perp (X_{m}^{-1})^\perp \bigr)^T
\,.
\end{equation*}
The purpose of the above identity is to show that if we have estimates for all the entries of the matrix $\nabla^n \DD_{t,m}^\ell \bigl( \nabla X_{m}^{-1} \bigr)$, then we automatically obtain estimates for 
 the matrix $\nabla^n \DD_{t,m}^\ell \bigl( \nabla X_{m} \circ X_{m}^{-1}\bigr)$.

\smallskip

To conclude, we note that since $\DD_{t,m}^\ell X_m^{-1} = 0$, we have that $\nabla^n \DD_{t,m}^\ell \nabla X_m^{-1} = \nabla^n \bigl[ \DD_{t,m}^\ell, \nabla \bigr] X_m^{-1}$, and precisely this term was previously estimated in \eqref{e.salmon.bear.fat}. In turn, this estimate recovers \eqref{e.spacetime.X.bnds}, as desired.
\end{proof}

We conclude this section by noting that by construction, the vector field $\b$ defined in \eqref{e.def.b} is ``nearly a solution'' of the incompressible Euler equations, as quantified by the following result.

\begin{proposition}
\label{p.STME.Euler} 
The vector field $\b$ constructed in \eqref{e.def.b} solves
\begin{equation}
\partial_t \b + \div ( \b \otimes \b) + \nabla p = \div \Reynolds \,, 
\qquad 
\div \b = 0 \,,
\label{e.almost.Euler}
\end{equation}
for a suitable pressure scalar $p \in C(\RR; C^{0,\beta'}(\RR^2))$ and a traceless  stress tensor $\Reynolds \in C(\RR; C^{0,\beta'}(\RR^2))$, for any $\beta' \in (0,1)$ satisfying
\begin{equation}
\beta' 
< 
2 (\beta - 1) +  \frac{2\delta}{q}  
=
2 (\beta - 1) + 
\frac{(\beta-1)(4-3 \beta)^2}{2 \beta (5 \beta - 4)}
\,.
\label{e.beta'.def}
\end{equation}
Moreover, there exists a constant $C\geq 1$ which only depends on $\beta$ and $\beta'$ as in~\eqref{e.beta'.def}, such that 
\begin{equation}
\norm{p}_{L^\infty(\R; C^{0,\beta'}(\R^2))} + \norm{\Reynolds}_{L^\infty(\R; C^{0,\beta'}(\R^2))} \leq C \Lambda^{-\frac{q}{q-1} ( 2(\beta-1) - \beta')} 
+ C \Lambda^{-\frac{2q}{q-1} (2(\beta-1) + \frac{2\delta}{q} - \beta') } \,.
\label{e.almost.Euler.quant}
\end{equation}
In particular, for $\beta,\beta'$ fixed, the right side of \eqref{e.almost.Euler.quant} can be made arbitrarily small by letting $\Lambda$ be sufficiently large.
\end{proposition}

We note that the parameter $\beta'$ in \eqref{e.beta'.def} is allowed to be strictly larger than $2(\beta-1)$. As such the regularity of the pressure in~\eqref{e.almost.Euler} is strictly better than the regularity of the pressure for a generic $C^0_t C^{0,\beta-1}_x$ weak solution  of the Euler equations, which is $C^0_t C^{0,2(\beta-1)}_x$. Proposition~\ref{p.STME.Euler} follows from a fairly straightforward computation by telescoping~\eqref{e.Dt.vm.0} and~\eqref{e.Dt.vm.1} and the fact that the term~$\vv_m \cdot \nabla \vv_m$ vanishes to leading order due to the shear flows used in the construction.
We do not give the details here, since it is not needed in our analysis, but the interested reader can find the proof commented out in the latex source file (downloadable on the arxiv) below this sentence.

\section{Correctors and renormalized diffusivities}
\label{s.renorm}

In this section, we introduce the sequence of correctors 
and renormalized diffusivities for each scale~$\ep_m$. 

\subsection{The correctors: definitions and estimates}
\label{ss.correctors}
We will introduce a corrector~$\Chi_m^\kappa$ which mediates between scales~$\ep_m$ and~$\ep_{m-1}$. The job of $\Chi_m^\kappa$ is to ``correct'' a solution of the~$\ep_{m-1}$-scale equation
\begin{equation}
\label{e.meq}
\partial_t \theta_{m-1} - \kappa_{m-1} \Delta \theta_{m-1}  + \b_{m-1} \cdot \nabla \theta_{m-1} = 0,
\end{equation}
by adding the wiggles with wavelengths of order~$\ep_m$ we would expect to see in the solution of the~$\ep_m$-scale equation.

\smallskip

As we have seen in the construction of the vector field, the difference between~$\b_m$ and~$\b_{m-1}$ is the inclusion of shear flows oscillating at the length scale~$\ep_m$, in the Lagrangian coordinates corresponding to~$\b_{m-1}$. 
These shear flows alternate between horizontal and vertical shears (with ``quiet'' periods in between) on the time scale~$\tau_m$ which, as we will show, is much longer than the time scale on which in takes for the shear flows to homogenize. These oscillations in space and in time will create oscillations in the solutions~$\theta_m$ which are not present in~$\theta_{m-1}$. We need to introduce correctors which capture, at leading order, these oscillations. Roughly speaking, the correctors which capture the spatial oscillations at scale~$\ep_m$ in Lagrangian coordinates will be denoted by~$\tilde{\Chi}_m$. The time oscillations due to the horizontal and vertical alternation of the shear flows will be corrected by a function denoted by~$\widetilde{H}_m$. 

\smallskip

Since the scale~$\ep_m$ on which the shears oscillate is much smaller than the active scales of the flows~$X_{m-1}$, we should expect the correctors~$\tilde{\Chi}_m$ to be obtained---at least at leading order---from the the composition of the correctors for the (time-independent) simple shear flow with the appropriate Lagrangian flow~$X_{m-1}$. 
As such, we first discuss the derivation of the correctors corresponding to the time-independent shear flows. We denote these by~$\Chi_{m,k}$, and they turn out to be given by a an explicit, well-known and simple formula.

\subsection{Correctors and effective diffusivity for smoothly alternating shear flows}

All of the notation from the previous section is adopted here; in particular we recall that the stream function $\psi_{m,k}$ is defined for each~$k\in\Z$ and~$m\in\N$ in~\eqref{e.def.streamr.0}--\eqref{e.def.streamr} and the time cutoff functions~$\zeta_{m,k}$ and~$\hat{\zeta}_{m,l}$ are defined in~\eqref{e.zetatimecutoff}--\eqref{e.zeta.mk.def} and~\eqref{e.zeta.prime.ml.fitting}--\eqref{e.zeta.prime.ml.bounds}, respectively, and we recall from~\eqref{e.cutoff.overlaps} that the supports of these overlap only when $l = l_k$, with~$l_k$ defined in~\eqref{e.lk.def}.  
We also define the incompressible vector fields~$\mathbf{u}_{m,k}$ by 
\begin{equation}
\label{e.ukm.explicit}
\mathbf{u}_{m,k}(x) := \nabla^\perp \psi_{m,k}(x)
=
\left\{
\begin{aligned}
&2\pi a_m \ep_m \cos\left( \tfrac{2\pi x_1}{\ep_m} \right) \e_2 
& \mbox{if} & \ k\in 4\Z+1, \\
& \!- \!  2\pi a_m \ep_m  \cos\left( \tfrac{2\pi x_2}{\ep_m} \right) \e_1 
& \mbox{if} & \ k\in 4\Z+3, \\
& 0 
& \mbox{if} & \ k\in 2\Z
\,,
\end{aligned}
\right.
\end{equation}
and we set 
\begin{equation}
\label{e.psi.m}
\psi_m(t,x) 
:= 
\sum_{k \in2\Z+1} 
\hat{\zeta}_{m,l_k} (t) \zeta_{m,k}(t) \psi_{m,k}(x) 
\,.
\end{equation}
and
\begin{equation}
\label{e.um.formula}
\mathbf{u}_m(t,x) := \nabla^\perp \psi_{m} (t,x)
=
\sum_{k\in2\Z+1} \hat{\zeta}_{m,l_k} (t)\zeta_{m,k}(t) \mathbf{u}_{m,k}(x) 
\,.
\end{equation}
We let~$\nabla \mathbf{u}_m(t,x)$ denote the~$2$-by-$2$ matrix with entries~$\partial_{x_i} (e_j\cdot \mathbf{u}_m)(t,x)$; it is given by the formula
\begin{equation}
\label{e.um.gradient.formula}
\nabla \mathbf{u}_{m,k}(x)
=
\left\{
\begin{aligned}
& \!- \! 4\pi^2 a_m  \sin\left( \tfrac{2\pi x_1}{\ep_m} \right) \e_1\otimes \e_2 
& \mbox{if} & \ k\in 4\Z+1, \\
& 4\pi^2 a_m \sin\left( \tfrac{2\pi x_2}{\ep_m} \right) \e_2\otimes \e_1 
& \mbox{if} & \ k\in 4\Z+3, \\
& 0 
& \mbox{if} & \ k\in 2\Z
\,.
\end{aligned}
\right.
\end{equation}

We also introduce a special time $t^*_{m,k} := (-\tfrac23 + k)\tau_m$ and, for each~$\kappa>0$ and~$\e \in\R^2$, 
define~$\chi^{\kappa}_{m,k,\e}$ to be the solution of 
\begin{equation}
\label{e.parabcorr.k}
\left\{
\begin{aligned}
& \partial_t \chi^{\kappa}_{m,k,\e} 
- \kappa \Delta \chi^{\kappa}_{m,k,\e}
+\hat{\zeta}_{m,l_k} \zeta_{m,k} \mathbf{u}_{m,k} \cdot \bigl( \e + \nabla \chi^{\kappa}_{m,k,\e} \bigr) = 0 
& \mbox{in} & \ \bigl( -\infty , \infty\bigr) \times  \R^2\,, \\
& \chi^{\kappa}_{m,k,\e} = 0 
& \mbox{in} & \ \bigl( -\infty , t^*_{m,k} \bigr) \times \R^2 \,. \\
\end{aligned}
\right. 
\end{equation}
We observe that~\eqref{e.parabcorr.k} does have a unique solution by first imposing a zero initial condition at time~$t^*_{m,k}$ and then noticing that, thanks to the presence of the cutoff function~$\zeta_{m,k}$, we may extend the solution to earlier times by setting it equal to zero.  
Note that~$\Chi^\kappa_{m,k} \equiv 0$ for $k\in 2\Z$ by~\eqref{e.ukm.explicit}. 

\smallskip

We will use vector notation for these correctors by writing
\begin{equation*}
\Chi_{m,k}^\kappa := 
\begin{pmatrix}
\corr_{m,k,\e_1}^\kappa   \\ 
\corr_{m,k,\e_2}^\kappa 
\end{pmatrix}^\intercal,
\qquad m\in\N.
\end{equation*}
Then~$\nabla \Chi_{m,k}^\kappa$ denotes the~$2\times2$ matrix
\begin{equation}
\label{e.gradcorrmatrix}
\nabla \Chi_{m,k}^\kappa =
\begin{pmatrix}
\partial_{x_1} \corr_{m,k,\e_1}^\kappa 
& \partial_{x_1} \corr_{m,k,\e_2}^\kappa \\ 
\partial_{x_2} \corr_{m,k,\e_1}^\kappa 
& \partial_{x_2} \corr_{m,k,\e_2}^\kappa 
\end{pmatrix}.
\end{equation}
Thanks to the one-dimensional nature of the shear flows~$\mathbf{u}_{m,k}$, we can give a simple explicit formula for~$\Chi^\kappa_{m,k}$.
Indeed, a direct computation yields
\begin{equation}
\label{e.Chimk.formula}
\Chi^{\kappa}_{m,k} (t,x) 
=
\mathbf{u}_{m,k}(x) 
\int_{-\infty}^t 
\hat{\zeta}_{m,l_k}(s)
\zeta_{m,k}(s)
\exp\left( \tfrac{4\pi^2\kappa}{\ep_m^2} (s-t) \right) \,ds
\end{equation}
and thus 
\begin{equation}
\label{e.Chimk.gradient.formula}
\nabla \Chi^{\kappa}_{m,k} (t,x) 
=
\nabla\mathbf{u}_{m,k}(x) 
\int_{-\infty}^t 
\hat{\zeta}_{m,l_k}(s)
\zeta_{m,k}(s)
\exp\left( \tfrac{4\pi^2\kappa}{\ep_m^2} (s-t) \right) \,ds
\,.
\end{equation}
Since $\hat{\zeta}_{m,l_k} \zeta_{m,k} \leq 1$, we have 
\begin{equation*}
\int_{-\infty}^t 
\hat{\zeta}_{m,l_k}(s)
\zeta_{m,k}(s)
\exp\left( \tfrac{4\pi^2\kappa}{\ep_m^2} (s-t) \right) \,ds
\leq
\int_{-\infty}^t 
\exp\left( \tfrac{4\pi^2\kappa}{\ep_m^2} (s-t) \right) \,ds
= 
\frac{\ep_m^2}{4\pi^2\kappa}\,.
\end{equation*}
Therefore, 
\begin{equation}
\label{e.corrbounds.Chi}
\left\| \Chi^\kappa_{m,k} \right\|_{L^\infty(\R\times\TT^2)}
+ \ep_m
\left\| \nabla \Chi^\kappa_{m,k} \right\|_{L^\infty(\R \times\TT^2)}
\leq 
\frac{C a_m \ep_m^3}{\kappa}.
\end{equation}
Since~$\hat{\zeta}_{m,l_k} \leq 1$ and~$\zeta_{m,k}$ vanishes on $[(k+\frac23)\tau_m,\infty)$, we have, for every~$t \geq (k+\frac34)\tau_m$, 
\begin{equation*}
\int_{-\infty}^t 
\hat{\zeta}_{m,l_k}(s)
\zeta_{m,k}(s)
\exp\left( \tfrac{4\pi^2\kappa}{\ep_m^2} (s-t) \right) \,ds
\leq
\int_{-\infty}^{t-\frac1{12}\tau_m} 
\exp\left( \tfrac{4\pi^2\kappa}{\ep_m^2} (s-t) \right) \,ds
= 
\frac{\ep_m^2}{4\pi^2\kappa} \exp\left( - \frac{\pi^2\kappa\tau_m}{3\ep_m^2} \right) 
\,.
\end{equation*}
Therefore, after time $(k+\frac34)\tau_m$ the corrector~$\Chi^\kappa_{m,k}$ becomes exponentially small: we have 
\begin{align}
\label{e.corrbounds.Chi.decay}
\sup_{t\in \bigl( (k+\frac34)\tau_m,\infty\bigr)} \left( 
\left\| \Chi^\kappa_{m,k}(t,\cdot) \right\|_{L^\infty(\TT^2)}
+
\ep_m \left\| \nabla \Chi^\kappa_{m,k}(t,\cdot) \right\|_{L^\infty(\TT^2)}
\right)
&
\leq
\frac{C a_m \ep_m^3}{\kappa}
\exp\left( - \frac{\pi^2\kappa\tau_m}{3 \ep_m^2}  \right)
\notag \\ & 
\leq 
C\ep_m
\exp\left( - \frac{\pi^2\kappa\tau_m}{4 \ep_m^2}  \right)\,,
\end{align}
where we rather crudely used~$a_m\tau_m \leq 1$ in the last line. We will typically encounter the situation in which~$\frac{\kappa \tau_m}{\ep_m^2} \gg 1$. Indeed, it will be a negative power of~$\ep_m$ in practice---see~\eqref{e.exprat.bound} below---and therefore the exponential factor on the right side of~\eqref{e.corrbounds.Chi.decay} is very small.

\smallskip

We next define 
\begin{equation}
\label{e.Chim}
\Chi^\kappa_{m} 
: =
\sum_{k\in \Z}
\xi_{m,k}
\Chi^\kappa_{m,k} 
=
\sum_{k\in2\Z+1} 
\xi_{m,k}
\Chi^\kappa_{m,k},
\end{equation}
Recall that the cutoff function~$\xi_{m,k}$ is defined in~\eqref{e.xi.mk.def} and is locally constant except for times outside the time interval~$[(k-\frac34)\tau_m,(k+\frac34)\tau_m]$, in other words, when the function~$\Chi^\kappa_{m,k}$ is very small by~\eqref{e.corrbounds.Chi.decay}. 

\smallskip

We are able to conclude by~\eqref{e.corrbounds.Chi.decay},~\eqref{e.goodtransition} and superposition that the components of~$\Chi_m^\kappa$ in~\eqref{e.Chim}
are ``almost'' solutions of 
\begin{align}
\label{e.parabcorr}
\left\{
\begin{aligned}
& \partial_t \chi^{\kappa}_{m,\e} 
- \kappa \Delta \chi^{\kappa}_{m,\e}
+\mathbf{u}_{m} \cdot \bigl( \e + \nabla \chi^{\kappa}_{m,\e} \bigr) = 0 
& \mbox{in} & \ \R \times \R^2, \\
& 
\chi^\kappa_{m,\e}(t,x) \quad
\text{is \ $\Z\times\Z^2$--periodic,}
\\
& 
\langle \chi^\kappa_{m,\e}(t,\cdot) \rangle = 0, \quad \forall t\in\R. 
\end{aligned}
\right. 
\end{align}
We recognize~\eqref{e.parabcorr} as the periodic, space-time corrector arising in parabolic homogenization. 
Actually, the equation in the first line of~\eqref{e.parabcorr} is valid only up to an exponentially small error. Our reason for defining~$\Chi_m^\kappa$ slightly differently, not in terms  of~\eqref{e.parabcorr} but rather as the sum~\eqref{e.Chim}, is because the exact formula~\eqref{e.Chimk.formula} is more convenient to work with and the difference between these two is negligible. We still refer to $\Chi_m^\kappa$ as a ``corrector.''

\smallskip

\begin{remark}[A special orthogonality property]
\label{r.special.orthogonality}
An important property inherited from the shear flow structure, which will come to our rescue in the Section~\ref{s.cascade}, is the following pointwise orthogonality property: for every pair of multiindices~$\aa,\bb\in \NN_0^d$,
\begin{equation}
\label{e.SO}
\partial^\aa \mathbf{u}_{m,k} 
\cdot \nabla \partial^\bb \chi^\kappa_{m,k,\e} = 0 \quad \mbox{in} \ \R \times\TT^2. 
\end{equation}
Indeed, if~$k\in 4\Z+1$ (respectively,~$k\in 4\Z+3$), then we see from~\eqref{e.Chimk.formula} that the function~$\partial^\bb \chi^\kappa_{m,k,\e}(t,\cdot)$ depends only on~$x_1$ (resp.,~$x_2$) and therefore its gradient is proportional to~$\e_1$ (resp.,~$\e_2$), while from~\eqref{e.ukm.explicit} we see that $\partial^\aa \mathbf{u}_{m,k}$ is proportional to~$\e_2$ (resp.,~$\e_1$). 
\end{remark}

\subsection{The renormalized diffusivities: recurrence, averaging and estimates}

We introduce the following objects:
\begin{itemize}

\item We denote the spatially-averaged flux of the correctors by
\begin{equation}
\label{e.am.kappa.def}
\J^\kappa_{m} (t) 
:=
\Bigl \langle 
\bigl( \kappa \Itwo + \psi_m(t,\cdot)\sigma\bigr) 
\bigl( \Itwo + \nabla \Chi^\kappa_{m}(t,\cdot) \bigr)
\Bigr \rangle 
\,.
\end{equation}

\item The homogenized matrix is the average of the flux in both space and time, defined by
\begin{equation}
\label{e.Kbarm.def}
\Khom^\kappa_{m} := 
\int_0^1 \J^\kappa_{m}(t)\,dt 
=
\bigl\langle \!\!\bigl\langle
\bigl( \kappa \Itwo + \psi_m\sigma\bigr) 
\bigl( \Itwo + \nabla \Chi^\kappa_{m}\bigr)
\bigr\rangle \!\!\bigr\rangle
\,.
\end{equation}
\end{itemize}
Since~$\Chi_m$ and~$\psi_m$ are invariant under a simultaneous $2\tau_m''\Z$--translation in time and a permutation of the~$\e_1$ and~$\e_2$ axes, it follows that~$\Khom^\kappa_m$ is a scalar matrix. We will therefore abuse notation by allowing~$\Khom^\kappa_m$ to denote both a matrix and the positive scalar constant~$a$ such that~$\Khom^\kappa_m = a\Itwo$, since it will always be clear from the context which is intended. 

\smallskip

Using~\eqref{e.Chimk.formula}, we can find an explicit formula for~$\J^\kappa_m (t)$ and~$\Khom_m^\kappa$, which will be helpful in our computations. Observe first, using the properties of the cutoff functions, the skew-symmetry of~$\sigma$ and the fact that~$\langle \nabla \Chi_{m,k}^\kappa \rangle = 0$, that we may write~\eqref{e.am.kappa.def} as
\begin{equation}
\label{e.Jm.kappa.def.2}
\J^\kappa_{m}(t)
-
\kappa \Itwo 
=
\bigl\langle 
\psi_{m}\sigma \nabla \Chi^{\kappa}_{m} (t,\cdot)
\bigr\rangle
=
\biggl \langle \,
\sum_{k \in2\Z+1}
\hat{\zeta}_{m,l_k}(t)
\zeta_{m,k}(t)
\psi_{m,k}\sigma \nabla \Chi_{m,k} (t,\cdot)
\biggr \rangle
\,.
\end{equation}
Using~\eqref{e.def.streamr},~\eqref{e.Chimk.formula} and $\langle \sin^2 \rangle=\frac12$, we compute, for every $k\in\Z$ and $t\in [(k-\frac12) \tau_m, (k+\frac12) \tau_m]$,
\begin{align}
\label{e.Jm.explicit.mofo}
\lefteqn{
\J^\kappa_{m}(t) - \kappa \Itwo
}  & 
\notag \\ &
= 
2\pi^2 a_m^2 \ep_m^2 
\hat{\zeta}_{m,l_k} (t)
\zeta_{m,k}(t)
\int_{-\infty}^t 
\hat{\zeta}_{m,l_k}(s)
\zeta_{m,k}(s)
\exp\left( \tfrac{4\pi^2\kappa}{\ep_m^2} (s-t) \right) \,ds
\cdot 
\left\{
\begin{aligned}
& \e_2\otimes \e_2
& \mbox{if} & \ k\in 4\Z+1, \\
& \e_1\otimes \e_1
& \mbox{if} & \ k\in 4\Z+3, \\
& 0 
& \mbox{if} & \ k\in 2\Z. \\
\end{aligned}
\right.
\end{align}
In particular, for a universal constant~$C<\infty$,
\begin{equation}
\label{e.akappam.size}
\bigl| \J^\kappa_{m}(t) \bigr|
\leq
\biggl( \kappa +
\frac{Ca_m^2\ep_m^4}{\kappa} \biggr)\,.
\end{equation}
Observe that $\J^\kappa_{m}$ is a~$\tau^{\prime\prime}_m$-periodic function of time. 
We will show next that, up to a very small error, $\J^\kappa_m$ can be written as a sum of products of~$\tau_m$--periodic functions and~$\tau_m^{\prime\prime}$--periodic functions. 
Define 
\begin{equation}
\label{e.Jhat}
\hat{\J}^\kappa_{m}(t)
:=
\kappa \Itwo+
\sum_{n=0}^{N_*-1}
L_{m,n}^\kappa(t) \, \mathbf{j}_{m,n}^\kappa(t)
\end{equation}
where we define, for every~$n\in\{ 0,\ldots,N_*\}$,
\begin{equation}
\label{e.jkmn}
\mathbf{j}_{m,n}^\kappa(t)
:=
\frac{2\pi^2 a_m^2 \ep_m^2}{n!}
\biggl( \frac{\ep_m^2}{4\pi^2\kappa} \biggr)^{\!n}  
\sum_{k\in2\Z+1}
\zeta_{m,k}(t) 
\partial_t^n \zeta_{m,k}(t) 
\bigl( \indc_{ \{ k\in 4\Z+1\}} \e_2\otimes \e_2 + \indc_{ \{ k\in 4\Z+3\}} \e_1\otimes \e_1 \bigr)
\end{equation}
and
\begin{equation}
\label{e.Lmn.def}
L_{m,n}^\kappa (t):=
\sum_{l\in\Z} 
\hat{\zeta}_{m,l} (t)
\int_{-\infty}^t 
\hat{\zeta}_{m,l}(s)
\biggl( \frac{4\pi^2 \kappa (s-t)}{\ep_m^2} \biggr)^{\!n}  
\exp\Bigl( \tfrac{4\pi^2\kappa}{\ep_m^2} (s-t) \Bigr) \,ds
\,.
\end{equation}
Observe that $\mathbf{j}_{m,n}^\kappa$ is indeed $\tau_m$--periodic, and~$L_{m,n}^\kappa$ is~$\tau_m^{\prime\prime}$--periodic. 

\smallskip

The functions~$\mathbf{j}_{m,n}^\kappa$ defined in~\eqref{e.jkmn} satisfy the bounds
\begin{equation}
\label{e.jmn.bound}
\bigl \| \mathbf{j}_{m,n}^\kappa \bigr\|_{L^\infty(\R)}
\leq 
Ca_m^2 \ep_m^2
\biggl( \frac{\ep_m^2}{\kappa \tau_m } \biggr)^{\!\!n}  \,,
\qquad 
\forall n\in\{0,\ldots, N_*\}\,.
\end{equation}
Here the constant~$C$ depends only on~$\beta$ through~$C_{\eqref{e.zetatimecutoff}}$.
Likewise, the functions~$L_{m,n}^\kappa$ satisfy, for every~$n,\ell \in\{0,\ldots,N_*\}$,
\begin{align*}
\| \partial^\ell_t L_{m,n}^\kappa \|_{L^\infty(\R)}
&
\leq 
\sup_{l \in\Z}
\sup_{s\in \R}
\bigl\| \partial_t^\ell 
\bigl( 
\hat{\zeta}_{m,l} 
\hat{\zeta}_{m,l}(\cdot-s )
\bigr) 
\bigr
\|_{L^\infty(\R)} 
\int_{0}^\infty
\biggl( \frac{4\pi^2 \kappa}{\ep_m^2} s \biggr)^{\!n}  
\exp\biggl( - \frac{4\pi^2\kappa}{\ep_m^2} s \biggr) \,ds 
\notag \\ & 
\leq
n! C^n \biggl(\frac{\ep_m^2}{\kappa}
\biggr)
(\tau_m')^{-\ell}
\,,
\end{align*}
for a constant~$C<\infty$ which depends only on~$C_{\eqref{e.zeta.prime.ml.bounds}}$ and thus only on~$\beta$.
Since~$n\leq \N_*$ and~$N_*$ depends only on~$\beta$, we deduce that, for some~$C(\beta)<\infty$,  
\begin{equation}
\label{e.Lmn.bound}
\| \partial^\ell_t L_{m,n}^\kappa \|_{L^\infty(\R)}
\leq
C \biggl(\frac{\ep_m^2}{\kappa}
\biggr)
(\tau_m')^{-\ell}
\,, \qquad \forall \ell \in\{ 0,\ldots, N_*\}\,.
\end{equation}

By Taylor's formula and~\eqref{e.zeta.mk}, for every $s,t\in\R$ with $s\leq t$ and $N \leq N_*$, 
\begin{align*}
\lefteqn{
\biggl| 
\zeta_{m,k} (s) - 
\sum_{n=0}^{N_*-1}
\frac{(s-t)^n}{n!}
\partial_t^n \zeta_{m,k} (t) 
\biggr|
\exp\Bigl( \tfrac{4\pi^2\kappa}{\ep_m^2} (s-t) \Bigr) \,ds
} \qquad & 
\\ &
\leq
\frac{1}{N!}
\bigl\| \partial^{N}_t \zeta_{m,k} \bigr\|_{L^\infty(\R)}
|s-t|^{N} 
\exp\Bigl( \tfrac{4\pi^2\kappa}{\ep_m^2} (s-t) \Bigr) \,ds
\leq
C \biggl(
\frac{\ep_m^2}{\kappa \tau_m}
\biggr)^{\!N}
\exp\Bigl( \tfrac{2\pi^2\kappa}{\ep_m^2} (s-t) \Bigr) \,ds
\,.
\end{align*}
By the previous inequality and the triangle inequality, we obtain
\begin{equation}
\label{e.JJhat}
\bigl| 
\J^\kappa_{m}(t) - \hat{\J}^\kappa_{m}(t)
\bigr|
\leq
\frac{C a_m^2\ep_m^4}{\kappa} \biggl(
\frac{\ep_m^2}{\kappa \tau_m}
\biggr)^{\!\!N_*}
\,.
\end{equation}
Since~$N_*$ is a very large constant, the estimate~\eqref{e.JJhat} says that~$\J^\kappa_{m}$ is indeed well-approximated by the function~$\hat{\J}^\kappa_{m}$, provided that~$\ep_m^2 \ll {\kappa \tau_m}$.

\smallskip

We define a $\tau_m^{\prime\prime}$--periodic function~$\K^\kappa_m$ by averaging out the~$\tau_m$--periodic oscillations from~$\hat{\J}^\kappa_{m}$:
\begin{equation}
\label{e.K}
\K^\kappa_m (t)
:=
\kappa \Itwo+
\sum_{n=0}^{N_*-1}
\bigl\langle \!\!\bigl\langle  \mathbf{j}_{m,n}^\kappa\bigr\rangle \!\!\bigr\rangle
L_{m,n}^\kappa(t) \,.
\end{equation}
Observe that, by~\eqref{e.jmn.bound} and~\eqref{e.Lmn.bound}, if~$\kappa$ satisfies
\begin{equation}
\label{e.condition}
\ep_m^2
\leq 
\frac12
{\kappa \tau_m}
\,,
\end{equation}
then we have that 
\begin{equation}
\label{e.bfK.dervs}
\| \partial_t^\ell \K^\kappa_m \|_{L^\infty(\R)}
\leq 
C
(\tau_m')^{-\ell}
\biggl( \kappa +
\frac{a_m^2\ep_m^4}{\kappa} \biggr)
\,, \qquad \forall \ell \in\N\,.
\end{equation}
By~\eqref{e.JJhat} and the ergodic theorem for periodic functions (Lemma~\ref{l.averages}), we have, under the extra condition~\eqref{e.condition}, 
\begin{align}
\label{e.Kbar.almost.average}
\bigl| 
\bigl\langle \!\!\bigl\langle \K^\kappa_m 
\bigr\rangle \!\!\bigr\rangle
- \Khom^\kappa_m
\bigr| 
\leq
\frac{C a_m^2\ep_m^4}{\kappa} \biggl(
\frac{\ep_m^2}{\kappa \tau_m}
\biggr)^{\!\!N_*}\,,
\end{align}
where we recall from~\eqref{e.Kbarm.def} that~$\Khom^\kappa_m =
\bigl\langle \!\!\bigl\langle
\J^\kappa_m 
\bigr\rangle \!\!\bigr\rangle$. 
Since~$N_*$ is very large, this says that~$\bigl\langle \!\!\bigl\langle \K^\kappa_m \bigr\rangle \!\!\bigr\rangle$ and~$\Khom^\kappa_m$ are very close, provided that~$\ep_m^2 \ll \kappa \tau_m$.

\smallskip

For future reference, we introduce some higher-order time correctors for~$\mathbf{j}_{m,n}^{\kappa}$: for every~$m\in\N$, $n\in \{0,\ldots, N_*\}$, we let~$\{ \mathbf{q}_{m,n,r}^\kappa \}_{r\in\N_0}$ be the sequence of~$\tau_m$--periodic functions on~$\R$ characterized by
\begin{equation}
\label{e.q.mnr.def}
\left\{
\begin{aligned}
&
\mathbf{q}_{m,n,0}^\kappa
:=
\mathbf{j}_{m,n}^\kappa - 
\bigl\langle \!\!\bigl\langle
\mathbf{j}_{m,n}^\kappa
\bigr\rangle \!\!\bigr\rangle
\,,
\\ & 
\bigl\langle \!\!\bigl\langle \mathbf{q}_{m,n,r}^\kappa \bigr\rangle \!\!\bigr\rangle = 0\,, \quad \forall r \in\N\,,
\\ & 
\partial_t \mathbf{q}_{m,n,r+1}^\kappa
= 
- \mathbf{q}_{m,n,r}^\kappa \,, \quad \forall r \in\N\,.
\end{aligned}
\right.
\end{equation}
These satisfy the bounds:
\begin{equation}
\label{e.qmnr.bounds}
\bigl\| \mathbf{q}_{m,n,r}^\kappa \bigr\|_{L^\infty(\R)} 
\leq 
Ca_m^2 \ep_m^2
\biggl( \frac{\ep_m^2}{\kappa \tau_m } \biggr)^{\!n} \frac{(C\tau_m)^r}{r!}
\,,
\qquad 
\forall n\in\{0,\ldots, N_*\}\,, \ r \in\N_0\,.
\end{equation}

For reasons which will become apparent in \eqref{e.ergodic.break.up} below, we need to compute the difference between the averaged flux $\mathbf{J}_m^\kappa(t)$ (which is given explicitly in~\eqref{e.Jm.explicit.mofo}) and the spatially-averaged energy 
\begin{equation}
\label{e.Emkappa.def}
\mathbf{E}_m^\kappa(t)
:=  \sum_{k,k' \in 2 \Z+1} \xi_{m,k}(t) \xi_{m,k'}(t)
\, \Bigl\langle \!
\kappa
(\Itwo + \nabla  \Chi _{m,k}^\kappa)^\intercal
(\Itwo + \nabla   \Chi _{m,k'}^\kappa)
\Bigr\rangle
\,.
\end{equation}
This is the purpose of the next lemma. 

\begin{lemma}
\label{l.flux.to.energy}
Assume that~$\kappa$ satisfies~\eqref{e.condition}. 
Then exists a constant~$C(\beta)<\infty$ such that
\begin{equation}
\label{e.flux.to.energy}
\bigl| 
\J_m^\kappa(t)
-
\mathbf{E}_m^\kappa(t)
\bigr|
\leq
\frac{C a_m^2\ep_m^4}{\kappa} \biggl(
\frac{\ep_m^2}{\kappa \tau_m}
\biggr)
\,.
\end{equation}
\end{lemma}
\begin{proof}
Observe that 
\begin{equation*}
\mathbf{E}_m^\kappa(t)
=  \sum_{k,k' \in 2 \Z+1} \xi_{m,k}(t) \xi_{m,k'}(t)
\kappa\Bigl( 
\Itwo
+
\, \bigl\langle
(\nabla  \Chi _{m,k}^\kappa)^\intercal
\nabla   \Chi _{m,k'}^\kappa
\bigr\rangle
\Bigr)
\,.
\end{equation*}
The supports of~$\xi_{m,k}$ and~$\xi_{m,k'}$ have nonempty intersection only if~$k,k'\in2 \Z+1$ satisfy~$|k-k'| \leq 2$. On the other hand, we see from the formulas~\eqref{e.Chimk.gradient.formula} and~\eqref{e.um.gradient.formula} that 
\begin{equation}
\label{e.alt.orth}
|k-k'| = 2 
\quad \implies \quad
(\nabla  \Chi _{m,k}^\kappa)^\intercal (t,x) 
\nabla   \Chi _{m,k'}^\kappa (t,x')  = 0
\,, \quad
\forall x,x' \in \TT^2\,, \, t \in\R\,.
\end{equation}
Therefore, the only pairs~$k,k'$ contributing to the sum satisfy~$k=k'$. We deduce that
\begin{equation}
\label{e.Emkappa.formula}
\mathbf{E}_m^\kappa(t)
=  \sum_{k \in 2 \Z+1} \xi_{m,k}(t)^2
\kappa\Bigl( 
\Itwo
+
\, \bigl\langle
(\nabla  \Chi _{m,k}^\kappa)^\intercal
\nabla   \Chi _{m,k}^\kappa
\bigr\rangle
\Bigr)
\,.
\end{equation}
We next compute
\begin{equation*}
 \bigl\langle
(\nabla  \Chi _{m,k}^\kappa)^\intercal
\nabla   \Chi _{m,k}^\kappa
\bigr\rangle
=
\bigl\langle (\nabla \mathbf{u}_{m,k} )^t
\nabla \mathbf{u}_{m,k}
\bigr\rangle
\biggl( 
\int_{-\infty}^t 
\hat{\zeta}_{m,l_k}(s)
\zeta_{m,k}(s)
\exp\left( \tfrac{4\pi^2\kappa}{\ep_m^2} (s-t) \right) \,ds
\biggr)^{\!2}
\,.
\end{equation*}
We see from~\eqref{e.um.gradient.formula} that 
\begin{equation}
\label{e.um.gradient.squared.formula}
\bigl( (\nabla \mathbf{u}_{m,k})^t \nabla \mathbf{u}_{m,k} \bigr)
(x)
=
\left\{
\begin{aligned}
& 16\pi^4 a_m^2  \sin^2\left( \tfrac{2\pi x_1}{\ep_m} \right) \e_2\otimes \e_2 
& \mbox{if} & \ k\in 4\Z+1, \\
& 16\pi^4 a_m^2 \sin^2\left( \tfrac{2\pi x_2}{\ep_m} \right) \e_1\otimes \e_1 
& \mbox{if} & \ k\in 4\Z+3, \\
& 0 
& \mbox{if} & \ k\in 2\Z
\,.
\end{aligned}
\right.
\end{equation}
and thus 
\begin{equation}
\label{e.um.gradient.squared.bracket.formula}
\bigl\langle 
(\nabla \mathbf{u}_{m,k})^t \nabla \mathbf{u}_{m,k} 
\bigr\rangle 
=
8\pi^4 a_m^2 \cdot
\left\{
\begin{aligned}
&  \e_2\otimes \e_2 
& \mbox{if} & \ k\in 4\Z+1, \\
& \e_1\otimes \e_1 
& \mbox{if} & \ k\in 4\Z+3, \\
& 0 
& \mbox{if} & \ k\in 2\Z
\,.
\end{aligned}
\right.
\end{equation}
Using~\eqref{e.zeta.mk} and~\eqref{e.zeta.prime.ml.fitting}, we see that  
\begin{equation}
\label{e.integral.in.time}
\biggl|
\hat{\zeta}_{m,l_k}(t)
\zeta_{m,k}(t)
-
\frac{4\pi^2\kappa }{\ep_m^2}
\int_{-\infty}^t 
\hat{\zeta}_{m,l_k}(s)
\zeta_{m,k}(s)
\exp\left( \tfrac{4\pi^2\kappa}{\ep_m^2} (s-t) \right) \,ds
\biggr|
\leq 
C
\left( \frac{C\ep_m^2}{\kappa\tau_m} \right)
\,.
\end{equation}
Combining the above and comparing to the formula for~$\J_m^\kappa$ in~\eqref{e.Jm.explicit.mofo} yields~\eqref{e.flux.to.energy}.
\end{proof}

\begin{lemma}
\label{l.amt.ergodic}
There exists $C(\beta)\in [1,\infty)$ such that, for every~$\kappa>0$ and $m\in\N$,
\begin{equation}
\label{e.enhance.onestep}
\left| \Khom_m^\kappa - \left( \kappa + \frac{9a_m^2\ep_m^4}{80\kappa} \right)
\Itwo 
\right| 
\leq 
\frac{a_m^2\ep_m^4}{\kappa}
\left( \frac{C\ep_m^2}{\kappa\tau_m} + C\ep_{m-1}^{\delta} \right)
\,.
\end{equation}
\end{lemma}
\begin{proof}
Starting from~\eqref{e.Jm.explicit.mofo}, we find that
\begin{equation*}
\Khom^\kappa_{m} - \kappa\Itwo
= 
\frac{\pi^2 a_m^2 \ep_m^2 }{\tau^{\prime\prime}_m}
\sum_{k\in 2\Z+1}
\int_{-\frac12 \tau^{\prime\prime}_m}^{\frac12 \tau^{\prime\prime}_m}
\hat{\zeta}_{m,l_k} (t)
\zeta_{m,k}(t)
\int_{-\infty}^t 
\hat{\zeta}_{m,l_k}(s)
\zeta_{m,k}(s)
\exp\left( \tfrac{4\pi^2\kappa}{\ep_m^2} (s-t) \right) \,ds \,dt \,
\Itwo
\,.
\end{equation*}
It therefore  suffices to show that 
\begin{multline}
\label{e.twentypi}
\biggl|
\frac{9\ep_m^2}{80\pi^2\kappa} 
-
\sum_{k\in 2\Z+1}
\fint_{-\frac12 \tau^{\prime\prime}_m}^{\frac12 \tau^{\prime\prime}_m}
\int_{-\infty}^t 
\hat{\zeta}_{m,l_k} (t)
\hat{\zeta}_{m,l_k}(s)
\zeta_{m,k}(t)
\zeta_{m,k}(s)
\exp\bigl( \tfrac{4\pi^2\kappa}{\ep_m^2} (s-t) \bigr) \,ds \,dt
\biggr|
\\
\leq 
\frac{\ep_m^2}{\kappa}
\left( \frac{\ep_m^2}{\kappa\tau_m} + C\ep_{m-1}^\delta \right)
\,.
\end{multline}
Using~\eqref{e.zetatimecutoff}, we have that
\begin{align*}
\int_{-\infty}^t 
\left| \zeta_{m,k}(t) - \zeta_{m,k}(s) ) \right| 
\exp\left( \tfrac{4\pi^2\kappa}{\ep_m^2} (s-t) \right) \,ds
&
\leq 
\left\| \partial_t \zeta_{m,k} \right\|_{L^\infty(\R)}
\int_{-\infty}^t 
(t-s) 
\exp\left( \tfrac{4\pi^2\kappa}{\ep_m^2} (s-t) \right) \,ds
\notag \\ & 
\leq
C\tau_m^{-1}
\int_{-\infty}^t 
(t-s) 
\exp\left( \tfrac{4\pi^2\kappa}{\ep_m^2} (s-t) \right) \,ds
=
\frac{C\ep_m^4}{\kappa^2 \tau_m}
\,.
\end{align*}
On the other hand, 
\begin{multline*}
\biggl| 
\sum_{k\in 2\Z+1}
\fint_{-\frac12 \tau^{\prime\prime}_m}^{\frac12 \tau^{\prime\prime}_m}
\hat{\zeta}_{m,l_k} (t)
\zeta_{m,k}(t)^2
\int_{-\infty}^t 
\hat{\zeta}_{m,l_k} (s)
\exp\bigl( \tfrac{4\pi^2\kappa}{\ep_m^2} (s-t) \bigr) \,ds \,dt
-
\frac{\ep_m^2}{4\pi^2\kappa} \cdot \frac12
\underbrace{ 
\fint_{-\frac12\tau_m}^{\frac12\tau_m}
\zeta_{m,0}^2 (t)\,dt
}_{=\frac 9{10} \ \text{by~\eqref{e.weirdo}}}
\biggr|
\\
\leq
\frac{C \ep_m^2\tau^\prime_m}{\kappa\tau^{\prime\prime}_m}
\leq 
\frac{C\ep_m^2\ep_{m-1}^\delta}{\kappa}
\,.
\end{multline*}
The triangle inequality and the previous two displays yield~\eqref{e.twentypi}. The proof is now complete.  
\end{proof}

As previously mentioned, we will apply~\eqref{e.corrbounds.Chi.decay} and~\eqref{e.enhance.onestep} when the factor~$\ep_m^2/(\kappa\tau_m)$ on the right side of~\eqref{e.enhance.onestep} is very small, typically a small positive power of~$\ep_m$: see~\eqref{e.exprat.bound} below. Therefore, loosely sense, we have that 
\begin{equation}
\label{e.ahom.shorthand}
\Khom_m^\kappa \approx
\kappa + \frac{9a_m^2\ep_m^4}{80\kappa}\,.
\end{equation}
We now define a sequence~$\{ \kappa_m \}$ of \emph{renormalized diffusivities}, starting from a given ``molecular'' diffusivity~$\kappa$, by the recursion
\begin{equation}
\label{e.kappa.sequence}
\left\{
\begin{aligned}
& \kappa_{m-1} = \Khom_m^{\kappa_m}
\qquad m\in \{1,\ldots,M\} \,, 
\\ & 
\kappa_M = \kappa \,. 
\end{aligned}
\right.
\end{equation}
The idea is that~$\kappa_{m-1}$ represents an effective (or ``eddy'') diffusivity observed at scale $\ep_m$, from the cumulative effects of the diffusion term~$\kappa\Delta$ and all the oscillations in the vector field~$\b$ with wavelengths smaller than~$\ep_{m-1}$. We imagine that we have homogenized all scales below that of~$\ep_{m-1}$ and witnessed an enhancement of diffusivity which results in an effective diffusivity of~$\kappa_{m-1}$. We choose the initial scale~$M$ in such a way that~$\ep_M$ is the critical scale at which the vector field~$\b$ and the diffusion interact in such a way that the recursion~\eqref{e.kappa.sequence} stays under control. 

\smallskip

The next lemma states that, for certain particular values of the molecular diffusivity~$\kappa$, we can control the entire sequence of renormalized diffusivities. 
We denote by~$\mathcal{K}$ the set of \emph{permissible diffusivities}, defined by
\begin{equation}
\label{e.permissible.K}
\mathcal{K}:= 
\bigcup_{m=1}^\infty
\Bigl[ \tfrac 12 \ep_m^{ \frac{2\beta}{\q+1}}, 2 \ep_m^{ \frac{2\beta}{\q+1}} \Bigr]
\,.
\end{equation} 
Recall that~$q$ is introduced in above in~\eqref{e.q.def.0} and is also related to~$\beta$ by the formula~\eqref{e.beta.def.0}. The definitions of the exponents~$\delta$ and~$\gamma$ appearing in the lemma statement below are also given above, in~\eqref{e.delta} and~\eqref{e.gamma}, respectively.

The reason that we can only control sequences~$\{ \kappa_m\}$ satisfying the satisfying recurrence~\eqref{e.kappa.sequence} with initial values~$\kappa$ belonging to this set~$\mathcal{K}$ has to do with the instability of the recurrence formula in~\eqref{e.ahom.shorthand}--\eqref{e.kappa.sequence}. This does not appear to be a technical artifact of our proof, but rather a property of the recursion itself. 

\begin{lemma}[Control of the renormalized diffusivities]
\label{l.recurse}
Suppose~$\kappa \in \mathcal{K}$ and let~$M\in \N$ be such that 
\begin{equation}
\label{e.permitted}
\frac12 \ep_M^{ \frac{2\beta}{\q+1}} \leq \kappa \leq 2 \ep_M^{ \frac{2\beta}{\q+1}}.
\end{equation}
Define a finite sequence~$\kappa_M, \kappa_{M-1},\ldots,\kappa_0$ by the recurrence~\eqref{e.kappa.sequence}, starting from~$\kappa_M:=\kappa$.
Then there exist universal constants $0<c<C<\infty$ such that,
for every $m\in\{ 0,\ldots,M-1\}$,
\begin{equation}
\label{e.kappam.bound}
c a_m \ep_m^{2+\gamma}
\leq \kappa_m 
\leq C a_m \ep_m^{2+\gamma}
\end{equation}
and
\begin{equation}
\label{e.exprat.bound}
c \ep_{m-1}^{2\delta}
\leq
\frac{\ep_m^2}{\kappa_m \tau_m} 
\leq C \ep_{m-1}^{2\delta} 
\,.
\end{equation}
\end{lemma}

\begin{proof}
We proceed by first establishing~\eqref{e.kappam.bound} for the a different sequence~$\{ \kappa^\prime_m \}$, defined by 
\begin{align}
\label{e.kappa.prime.sequence}
\left\{
\begin{aligned}
& \kappa^\prime_{m-1} = \kappa^\prime_m + \frac{9a_m^2\ep_m^4}{80\kappa^\prime_m}\,,
\qquad m\in \{1,\ldots,M\}, 
\\ & 
\kappa^\prime_M = \kappa. 
\end{aligned}
\right.
\end{align}
According to our (imprecise) shorthand~\eqref{e.ahom.shorthand}, we have reasons to expect that~$\kappa^\prime_m$ is close to~$\kappa_m$. Once we have proved~\eqref{e.kappam.bound}, we will argue that the two sequences are indeed close enough that we may obtain essentially the same estimate for~$\kappa_m$. 
Recall that the parameter~$\gamma$ defined in~\eqref{e.gamma} satisfies, in view of~\eqref{e.am.def} and~\eqref{e.q.def}, 
\begin{equation}
\label{e.gamma.range}
a_m\ep_{m}^{2+\gamma} =
\ep_m^{\beta+\gamma}\,.
\end{equation}

\smallskip

\emph{Step 1.} 
We prove that there exist universal constants $0< c \leq C<\infty$ such that 
\begin{equation}
\label{e.kappam.prime.bound}
c a_m \ep_m^{2+\gamma}
\leq 
\kappa^\prime_m 
\leq 
C a_m \ep_m^{2+\gamma}, 
\qquad \forall m \in \{0,\ldots,M-1\}. 
\end{equation}
Denote 
\begin{equation}
\label{e.s.m.def}
s_m:=   \frac{\kappa^\prime_m\sqrt{80/9}}{a_m\ep_m^{2+\gamma}}. 
\end{equation}
We may rewrite the recurrence in~\eqref{e.kappa.prime.sequence} in terms of~$s_m$ as 
\begin{equation}
\label{e.sm.diff.recurrence}
s_{m-1} 
= 
\left( \frac{\ep_m}{\ep_{m-1}} \right)^{\!\beta}
\ep_m^{-\gamma} \ep_{m-1}^{-\gamma}
\cdot
s_m\left( \ep_m^{2\gamma} + \frac{1}{s_m^{2}} \right) \,.
\end{equation}
Notice that the exponent~$\gamma$ has been chosen so that it satisfies~$\q(\beta-\gamma) = \beta+\gamma$. 
Hence by~\eqref{e.supergeo} we have 
\begin{equation}
\label{e.bound.some.stuff}
\biggl|
\left( \frac{\ep_m}{\ep_{m-1}} \right)^{\!\beta} \cdot \ep_m^{-\gamma}\ep_{m-1}^{-\gamma} - 1 \biggr| 
=
\biggl|
\left( \frac{\ep_m}{\ep_{m-1}^\q} \right)^{\beta} - 1 \biggr| 
\leq 
\frac{C\ep_{m-1}}{\beta} \leq C \ep_{m-1}\,. 
\end{equation}
We therefore obtain from~\eqref{e.sm.diff.recurrence} that
\begin{equation}
\label{e.smsimpl}
\biggl| s_{m-1} - s_m\biggl( \ep_m^{2\gamma} + \frac{1}{s_m^{2}} \biggr) \biggr| 
\leq
C \ep_{m-1} s_m\left( \ep_m^{2\gamma} + \frac{1}{s_m^{2}} \right)\,.
\end{equation}
In view of~\eqref{e.gamma.range}, the condition~\eqref{e.permitted} can be written as 
\begin{equation*}
\frac12\ep_M^{-2\gamma} 
\leq s_M 
\leq 2 \ep_M^{-2\gamma}.
\end{equation*}
We deduce from this and~\eqref{e.smsimpl} that
\begin{equation*}
s_{M-1} 
\leq 
2 + 2\ep_{M}^{2\gamma} + C\ep_{M-1}
\leq
4
\end{equation*}
and
\begin{equation*}
s_{M-1} \geq \frac12 - C\ep_{M-1} \geq \frac14. 
\end{equation*}
Similarly, it is easy to check that 
\begin{equation*}
\max \biggl \{ s_{m-1} , \frac1{s_{m-1}} \biggr\}
\leq 
\max \biggl \{ s_{m} , \frac1{s_{m}} \biggr\}
\Bigl( 1 + C\ep_{m-1}^{2\q\theta \wedge 1} \Bigr)
\,.
\end{equation*}
An iteration of the latter inequality therefore yields
\begin{equation}
\label{e.sm.bounded.yes}
\max \biggl \{ s_{m-1} , \frac1{s_{m-1}} \biggr\}
\leq
4 \prod_{j=m}^{M} \Bigl( 1 + C \ep_{j-1}^{2\gamma \q\wedge 1} \Bigr)
\leq 
4 \Bigl( 1 + C \ep_{m-1}^{2\gamma\q\wedge 1}  \Bigr) 
\leq 
C
\,.  
\end{equation}
The proof of~\eqref{e.kappam.prime.bound} is now complete. 

\smallskip

\emph{Step 2.} We show that~\eqref{e.kappam.prime.bound} implies~\eqref{e.kappam.bound}. 
By~\eqref{e.am.def},~\eqref{e.taum.def}
and~\eqref{e.kappam.prime.bound},
we observe that 
\begin{equation*}
\frac{\ep_m^2}{\kappa^\prime_m \tau_m} 
\simeq 
\ep_m^{-\gamma} a_m^{-1} \tau_m^{-1} 
\simeq 
\ep_{m-1}^{(2-\beta)(q-1) - q\gamma -2\delta} 
\,.
\end{equation*}
By the definitions of the exponents in~\eqref{e.delta} and~\eqref{e.gamma},
we have that 
\begin{equation*}
(2-\beta)(q-1) - q\gamma
=
(q-1) \biggl( 2 - \frac{2q+1}{q+1} \beta \biggr)  = 4\delta. 
\end{equation*}
Hence
\begin{equation}
\label{e.exprat.prime.bound}
c \ep_{m-1}^{2\delta}
\leq
\frac{\ep_m^2}{\kappa^\prime_m \tau_m} 
\leq C \ep_{m-1}^{2\delta} 
\,.
\end{equation}
Arguing by induction, suppose that for some~$n\in \{1,\ldots,M-1\}$, we have 
\begin{equation}
\label{e.rat.prime.no.prime}
\frac12 
\leq
\frac{\kappa^\prime_m}{\kappa_m} 
\leq 
2, \quad \forall m \in\{ n,\ldots,M-1\}\,.
\end{equation}
Then, using also~\eqref{e.kappam.prime.bound}, we deduce that, for every~$m \in\{ n,\ldots,M-1\}$, 
\begin{equation*}
\bigl( 1 - C\ep_m^{2\gamma} \bigr) 
\biggl( \frac{\kappa_m}{\kappa^\prime_m} \biggr)^2
\leq
\frac{1+ \frac{9a_m^2\ep_m^4}{80(\kappa_m^\prime)^2} } 
{1+ \frac{9a_m^2\ep_m^4}{80\kappa_m^2} }
\leq 
\bigl( 1 + C\ep_m^{2\gamma} \bigr) 
\biggl( \frac{\kappa_m}{\kappa^\prime_m} \biggr)^2
\,.
\end{equation*}
Next we use~\eqref{e.enhance.onestep},~\eqref{e.exprat.prime.bound} and~\eqref{e.rat.prime.no.prime} to get
\begin{equation}
\label{e.kappa.n.n-1}
\bigl(1 - C \ep_{n-1}^{2\delta} \bigr)
\biggl( 1+ \frac{9a_n^2\ep_n^4}{80\kappa_n^2} \biggr)
\leq
\frac{\kappa_{n-1}}{\kappa_n}
\leq
\bigl(1 + C \ep_{n-1}^{2\delta} \bigr)
\biggl( 1+ \frac{9a_n^2\ep_n^4}{80\kappa_n^2} \biggr)
\end{equation}
Putting these together and using the exact recursion formula for $\kappa_{n-1}^\prime$, we get that 
\begin{equation*}
\max\biggl\{ 
\frac{\kappa_{n-1}}{\kappa^\prime_{n-1}} 
\,,
\frac{\kappa^\prime_{n-1}}{\kappa_{n-1}}
\biggr\} 
\leq
\bigl(1 + C \ep_{n-1}^{2(\delta\wedge \gamma)} \bigr)
\frac{\kappa_{n}}{\kappa^\prime_{n}}
\,.
\end{equation*}
Iterating this and using $\kappa_M = \kappa^\prime_N$, we find that
\begin{equation}
\label{e.ratrat.prime}
\max\biggl\{ 
\frac{\kappa_{n-1}}{\kappa^\prime_{n-1}} 
\,,
\frac{\kappa^\prime_{n-1}}{\kappa_{n-1}}
\biggr\} 
\leq
\prod_{j=n}^{N} \Bigl( 1 + C \ep_{j-1}^{2(\delta\wedge \gamma)} \Bigr)
\leq 
\Bigl( 1 + C \ep_{n-1}^{2(\delta\wedge \gamma)}  \Bigr) 
\,.
\end{equation}
This allows us to remove the condition~\eqref{e.rat.prime.no.prime} and replace it with $C\ep_{n-1}^{2(\delta\wedge \gamma)} \leq 1$; that is, for some $n_0(\data) \in\N$, we have that the inequality~\eqref{e.ratrat.prime} holds for every~$n\geq n_0$.
However, for $m\leq n_0$, we have $c \leq \min\{ \kappa_m, \kappa^\prime_m\} \leq \max\{ \kappa_m, \kappa^\prime_m\} \leq C$, and so we have shown that 
\begin{equation}
\label{e.kappa.kappa.prime}
\max_{m\in\{0,\ldots,M-1\}}
\max\biggl\{ 
\frac{\kappa_{m}}{\kappa^\prime_{m}} 
\,,
\frac{\kappa^\prime_{m}}{\kappa_{m}}
\biggr\} 
\leq 
C
\,.
\end{equation}
In view of~\eqref{e.kappam.prime.bound} and~\eqref{e.exprat.prime.bound}, the proof of the lemma is now complete. 
\end{proof}

In most of the rest of the paper, we assume that the molecular diffusivity constant~$\kappa$ belongs to the set~$\mathcal{K}$ of permissible diffusivities defined in~\eqref{e.permissible.K}, so that the bounds of Lemma~\ref{l.recurse} are valid. Incidentally, the reason we are only able to obtain anomalous diffusion along a subsequence of~$\kappa$'s in Theorem~\ref{t.anomalous.diffusion} is due to the restriction in Lemma~\ref{l.recurse}.

\section{The multiscale ansatz}
\label{s.multiscale}

Now that we have constructed the vector field~$\b$ and defined the renormalized diffusivities, we are ready to begin the proof of anomalous diffusion. 
This will require some delicate asymptotic expansions which will take us the next several sections to develop. 

\smallskip

Throughout, we fix a molecular diffusivity~$\kappa \in \mathcal{K}$, with the set~$\mathcal{K}$ defined in~\eqref{e.permissible.K}. We let~$M$ denote the positive integer satisfying~\eqref{e.permitted}, and we let the finite sequence~$\kappa_M, \kappa_{M-1},\ldots,\kappa_0$ be defined by~\eqref{e.kappa.sequence}. We also select an initial datum, which is~$\Z^2$--periodic function $\theta_0 \in C^\infty(\TT^2)$ with zero mean,  
\begin{align}
\label{e.theta0.meanzero}
\langle \theta_0 \rangle = \int_{\TT^2} \theta_0(x)\,dx = 0,
\end{align}
and which satisfies the quantitative analyticity condition (recall the notation in~\eqref{e.nabla.n})
\begin{align}
\label{e.theta0.anal}
\left\| \nabla^n \theta_0 \right\|_{L^2(\TT^2)}
\leq
\|\theta_0\|_{L^2(\TT^2)} 
 n! R_{\theta_0}^{-n}
\,, \quad \forall n\in\N. 
\end{align}
For each $m\in\{ 0,\ldots,M\}$, we let~$\theta_m$ denote the solution of the initial-value problem 
\begin{align}
\label{e.theta.m}
\left\{
\begin{aligned}
& \partial_t \theta_m - \kappa_m \Delta \theta_m + \b_m \cdot \nabla \theta_m 
= 0 
& \mbox{in} & \ (0,\infty) \times \R^2\,, 
\\ & 
\theta_m = \theta_0 
& \mbox{on} & \ \{ 0 \} \times \R^2\,.
\end{aligned}
\right.
\end{align}
Recall that~$\b_m$ is defined in~\eqref{e.psi.recursion}.
In other words,~$\theta_m$ is the solution of the modified equation in which the stream function~$\phi$ has been replaced by~$\phi_m$, essentially removing the oscillations of~$\phi$ with wavelengths smaller than~$\ep_m$. 
Since~$\phi_m$ is smooth, the equation can be written in terms of the vector field~$\b_m$, as above, but it is often more convenient to write it in terms of the stream function~$\phi_m$ as
\begin{equation}
\label{e.thetam.divform}
\partial_t \theta_m - 
\nabla \cdot \bigl( \kappa_{m} \Itwo + \phi_m \sigma \bigr) \nabla \theta_m 
= 0 
\quad 
\mbox{in} \ (0,\infty) \times \R^2\,.
\end{equation}
It is clear that~$\theta_m \in C^\infty((0,\infty) \times \R^2)$ and, for each time~$t\in(0,\infty)$, the function~$\theta_m(t,\cdot)$ has zero mean and is~$\Z^2$--periodic. 
Note that,~\eqref{e.theta.m} in the case~$m=0$, extends the domain of the given function~$\theta_0$ from $\TT^2$, which we identify with~$\{ 0\} \times\TT^2$, to~$[0,\infty) \times \TT^2$. 

\smallskip

In order to prove Theorem~\ref{t.anomalous.diffusion}, we will propagate lower bounds on the energy dissipation of the solutions of~\eqref{e.theta.m} from~$m-1$ to~$m$. In fact, the key step is show that, for every~$m\in\{ 0,\ldots,M\}$ with~$\ep_{m-1} \leq R_{\theta_0}$, we have
\begin{equation}
\label{e.m.stepdown}
\Biggl| \,
\frac{\kappa_{m} \big\| \nabla {\theta}_m \big\|_{L^2((0,1)\times\TT^2)}^2}
{\kappa_{m-1} \left\| \nabla \theta_{m-1} \right\|_{L^2((0,1)\times\TT^2)}^2}
-1 \,
\Biggr| 
\leq C\ep_{m-1}^\delta
\,.
\end{equation}
This estimate is proved in Proposition~\ref{p.indystepdown}, below. 
From~\eqref{e.m.stepdown}, 
it is a simple matter to obtain the lower bound on the energy dissipation in Theorem~\ref{t.anomalous.diffusion}, as we will see.

\smallskip

The proof of~\eqref{e.m.stepdown} is based on the informal idea that the equation for~$\theta_m$ should homogenize to the equation for~$\theta_{m-1}$. 
To see why we should expect this, 
write the equation for~$\theta_m$ as 
\begin{equation}
\label{e.rewrite.theta.m}
\bigl( \partial_t + \b_{m-1} \cdot \nabla \bigr) \theta_m 
- 
\nabla \cdot \bigl( \kappa_m \Itwo + \tilde{\psi}_m \sigma \bigr) \nabla \theta_m
= 0 
\quad \mbox{in} \ (0,\infty) \times \R^2\,,
\end{equation}
where we define
\begin{align*}
\tilde{\psi}_m (t,x) 
:= 
\phi_{m}(t,x) - \phi_{m-1}(t,x)
=
\sum_{k \in2\Z+1} 
\hat{\zeta}_{m,l_k} (t) \zeta_{m,k}(t) \psi_{m,k}\bigl(X_{m-1,l_k}^{-1}(t,x)\bigr) 
\,.
\end{align*}
We view the vector field~$\b_{m-1}$ in the transport term in the left of~\eqref{e.rewrite.theta.m} as \emph{slow}, as well as the corresponding flows~$X_{m-1,l}$ and inverse flows~$X_{m-1,l}^{-1}$. In contrast, we consider the coefficient matrix~\smash{$\kappa_m\Itwo + \tilde{\psi}_m\sigma$} in the second-order part of the operator to be \emph{fast}. 
Moreover, if we change variables to Lagrangian coordinates with respect to the~``slow'' flows, then the transport operator~$\partial_t + \b_{m-1} \cdot \nabla$ becomes simply~$\partial_t$ and diffusion operator~$\nabla \cdot \bigl( \kappa_m \Itwo + \tilde{\psi}_m \sigma \bigr) \nabla$ becomes~$\nabla \cdot \bigl( \kappa_m \Itwo + {\psi}_m \sigma \bigr) \nabla$, which is a shear flow which switches between the horizontal and vertical directions. 
Given the discussion in the previous section, we expect the fast diffusive operator to homogenize to~$\kappa_{m-1}\Delta$, which leaves us with the equation for~$\theta_{m-1}$ in the original coordinates. 

\smallskip

When we speak here of ``homogenization'' we do not intend for the reader to understand this too literally: no limit is taken, rather the equations will be shown to be close in quantitative sense which is small relative to a power of~$\ep_{m-1}$. 

To make this idea precise, we introduce an multiscale ansatz for~$\theta_{m}$, denoted by~$\tilde\theta_{m}$, which is built from~$\theta_{m-1}$ and the correctors defined in the previous section.
Our strategy is very simple: we will plug~$\tilde{\theta}_{m}$ into the equation for $\theta_m$ and estimate the error. We will show that it is small enough to conclude that~$\tilde{\theta}_{m}$ is close to~$\theta_m$. Since we built~$\tilde{\theta}_{m}$ from~$\theta_{m-1}$, we will be able to relate~$\theta_m$ to~$\theta_{m-1}$ and, in particular, obtain~\eqref{e.m.stepdown}. 

\smallskip

The definition of~$\tilde\theta_{m}$ is motivated by the usual two-scale ansatz used in classical homogenization, in which one attaches the periodic correctors to the (usually smooth) solution of the macroscopic equation, as explained in the discussion above~\eqref{e.ansatz.guess.0}.
However, it is necessarily more complicated than the naive guess~\eqref{e.ansatz.guess.0}, for several reasons. 

\smallskip

First of all, there are actually three different ``fast'' scales in the equation for~$\theta_m$:
\begin{enumerate}
\item[(i)] the smallest spatial scale~$\ep_m$, which is the length scale of the shear flows; 

\item[(ii)] the time scale~$\tau_m$, on which the shear flows switch directions;

\item[(iii)] the time scale~$\tau_m''$, on which the Lagrangian flows~$X_{m-1}$ must refresh. 

\end{enumerate}
We should think of these three fast scales as being well-separated, with the spatial scale~$\ep_m$ being the smallest/fastest. This means that each of these three scales must be separately homogenized! 
Since homogenization estimates require smoothness of the macroscopic data, we must be very careful to maintain sufficient regularity estimates when we homogenize the time scales. 
This is the reason we spend so much effort proving estimates on~$T_{m-1}$ and~$\widetilde{H}_m$ later in this section. 

\smallskip

A second complication is due to the need to compose with the Lagrangian flows. As we have seen informally above, when we homogenize the spatial oscillations (the shear flows), we need to work in Lagrangian coordinates. Rather than actually switching our coordinate system, our definition of~$\tilde{\theta}_m$ will involve compositions with the inverse flows~$X_{m-1}^{-1}$. Unfortunately, the distortion caused by these flows cannot be ignored, and we must therefore introduce corrections in the equation for~$\theta_{m-1}$. This is the reason for the appearance of the matrix~$\mathbf{s}_{m-1}$ defined in~\eqref{e.sm}, and its role in the definition of~$T_{m-1}$. 

\smallskip

We select~$m\in\{ 1,\ldots,M\}$ which is fixed throughout the rest of this section. We also employ the following two notational conventions, which are in force throughout most of the rest of the paper:

\begin{itemize}

\item 
We use the correctors $\Chi_m^\kappa$ and matrices~$\J^\kappa_m$, $\K_m^\kappa$ and $\Khom_m^\kappa$ introduced in the previous section with~$\kappa=\kappa_{m}$ (and never any other choice of the parameter~$\kappa$). In order to lighten the notation, we drop the display of the dependence on~$\kappa_{m}$ from the superscripts, writing for example $\Chi_m$ and $\J_m$ instead of~$\Chi_m^{\kappa_m}$ and $\J_m^{\kappa_m}$. Recall that $\Khom_m^{\kappa_m} = \kappa_{m-1}$ by~\eqref{e.kappa.sequence}.

\item 

We employ the convention that all function compositions are assumed to occur in the spatial variable only. 
In other words, as all function compositions involve the flows~$X_{m-1,l}$ and their inverses~$X_{m-1,l}^{-1}$ (see~\eqref{e.flowabbrev}), instead of writing~$( F(t,\cdot) \circ X_{m-1,l}(t,\cdot))(x)$ we will just write~$F \circ X_{m-1,l}$.
\end{itemize}

In the next subsection, we will introduce the objects~$\mathbf{s}_{m-1}$,~$T_{m-1}$ and~$\widetilde{H}_m$ that are needed in the definition of~$\tilde{\theta}_{m}$, which is then given in Section~\ref{ss.ansatz}. In Sections~\ref{s.Tm.reg} and~\ref{ss.Hm.estimates}, we will prove important regularity estimates on~$T_{m-1}$ and~$\widetilde{H}_m$ which are needed in the following section. 
We will not see why our ansatz is a good one until the analysis in Section~\ref{s.cascade}, which is where we plug it into the equation for~$\theta_{m}$ and compute the error. 
The definitions here are motivated by the computations in Section~\ref{s.cascade}, and so we ask for the reader's patience if they seem a bit mysterious at first glance.

\subsection{Ingredients for the multiscale ansatz}
\label{ss.ingredients}

We proceed by introducing equations with gradually more and more scales, starting from the largest scales, until we arrive at a guess for what~$\theta_m$ should look like.

\smallskip

We first introduce an equation with time oscillations on scale~$\tau_m^{\prime\prime}$, which are due to the reseting of the flows. 
We define the matrix
\begin{equation}
\label{e.sm}
\mathbf{s}_{m-1} 
:=
\K_m
\sum_{l\in\Z}  
\hat{\xi}_{m,l }
\bigl( \nabla X_{m-1,l} \circ X_{m-1,l}^{-1} - \Itwo\bigr)
\,.
\end{equation}
At this point, we want to modify the equation for~$\theta_{m-1}$ by introducing a diffusion coefficient which oscillates on the time scale~$\tau_m''$. 
We will call the resulting solution~$T_{m-1}$. 

\smallskip

The rough idea is to define~$T_{m-1}\in C^\infty([0,1)\times \TT^2)$ to be the solution of the initial-value problem
\begin{align}
\label{e.Tm}
\left\{
\begin{aligned}
& \partial_t T_{m-1} 
- \nabla \cdot 
\bigl( \K_m + \mathbf{s}_{m-1} \bigr) \nabla T_{m-1}
+ \b_{m-1} \cdot \nabla T_{m-1}
= 0 
& \mbox{in} & \ (0,\infty) \times \R^2, 
\\ & 
T_{m-1} = \theta_0 
& \mbox{on} & \ \{ 0 \} \times \R^2
\,.
\end{aligned}
\right.
\end{align}
It is not difficult to see why we should expect this equation to homogenize to the one for~$\theta_{m-1}$: the principal part of the diffusion matrix in~\eqref{e.Tm} is~$\K_m$, which has periodic oscillations in time only, with period~$\tau_m^{\prime\prime}$ and a mean which very close to~$\overline{\K}_m = \kappa_{m-1}$. The matrix~$\mathbf{s}_{m-1}$ is lower-order compared to~$\K_m$, due to~\eqref{e.Xm.bound.1}. The reason for including it in the equation of~$T_{m-1}$ has to do with the need to anticipate some errors arising in the analysis of the smaller (spatial) scales, due to the change to Lagrangian coordinates. 
Note that we may also write the equation~\eqref{e.Tm} as 
\begin{equation}
\label{e.Tm.divform}
\partial_t T_{m-1} 
- \nabla \cdot \bigl( 
\K_m
+ \mathbf{s}_{m-1} 
+ \phi_{m-1} \sigma \bigr) \nabla T_{m-1} 
=0
\quad \mbox{in} \ (0,\infty) \times \R^2
\,.
\end{equation}

\smallskip

We do \emph{not} actually define~$T_{m-1}$ to be the solution of~\eqref{e.Tm}. 
We will instead define it as an approximate solution of~\eqref{e.Tm} through an iteration procedure. The advantage of this is that it allows us to prove better regularity estimates for~$T_{m-1}$. Indeed, the best lower bound on the matrix~$\K_m$ is~$\kappa_m\Itwo$, which is much less than~$\kappa_{m-1}\Itwo$, even though~$\K_m$ is larger than the latter on a proportion of times at least~$1-C\ep_m^{\delta}$. Nevertheless, if this is used in the energy estimates we will get very pessimistic regularity bounds on~$T_{m-1}$ compared to those we have for~$\theta_{m-1}$ in Lemma~\ref{l.theta.m-1.reg.upgrade} below. 
To get better bounds, we work with an approximate solution of~\eqref{e.Tm} which is constructed as follows.

\smallskip

We will choose a large number, which represents the number of iteration steps in our definition of~$T_{m-1}$, and for convenience we may take the large integer $N_*$ defined in~\eqref{e.N}. We initialize the iteration by setting
\begin{equation}
\label{e.T.m-1.0}
T_{m-1}^{(0)}:= \theta_{m-1} \,.
\end{equation}
For every~$1 \leq i \leq N_*$, we recursively define $T_{m-1}^{(i)}$ to be the solution of the initial-value problem
\begin{align}
\label{e.Tm-1.i}
\left\{
\begin{aligned}
& \partial_t T_{m-1}^{(i)} 
- \kappa_{m-1} \Delta T_{m-1}^{(i)}
+ \b_{m-1} \cdot \nabla T_{m-1}^{(i)}
= \nabla \cdot \bigl( \K_m - \kappa_{m-1} \Itwo + \mathbf{s}_{m-1} \bigr) \nabla T_{m-1}^{(i-1)}
& \mbox{in} & \ (0,\infty) \times \R^2, 
\\ & 
T_{m-1}^{(i)} = \theta_0 
& \mbox{on} & \ \{ 0 \} \times \R^2
\,.
\end{aligned}
\right.
\end{align}
It is clear that~$T_{m-1}^{(i)} \in C^\infty([0,\infty) \times \R^2)$. 
Finally, we define 
\begin{equation}
\label{e.T.m-1.Nstar}
T_{m-1}:= T_{m-1}^{( N_* )}
\,.
\end{equation}
By construction,~$T_{m-1}$ satisfies
\begin{align}
\label{e.Tm.true}
\left\{
\begin{aligned}
& \partial_t T_{m-1} 
- \nabla \cdot 
\bigl( \K_m + \mathbf{s}_{m-1} \bigr) \nabla T_{m-1}
+ \b_{m-1} \cdot \nabla T_{m-1}
= \nabla \cdot  \mathbf{e}_{m-1} 
& \mbox{in} & \ (0,\infty) \times \R^2, 
\\ & 
T_{m-1} = \theta_0 
& \mbox{on} & \ \{ 0 \} \times \R^2
\,,
\end{aligned}
\right.
\end{align}
where the error $\mathbf{e}_{m-1}$ is given by 
\begin{equation}
\label{e.E.m-1.def}
\mathbf{e}_{m-1}:= 
\bigl( \K_m - \kappa_{m-1} \Itwo + \mathbf{s}_{m-1} \bigr) \nabla\bigl( T_{m-1}^{( N_* -1)}-  T_{m-1}^{( N_* )}\bigr)
\,.
\end{equation}
Comparing \eqref{e.Tm.true} with \eqref{e.Tm}, we see that the ``true'' $T_{m-1}$ makes an error $\nabla \cdot \mathbf{e}_{m-1}$ in solving the advection-diffusion equation. We will however show that error~$\mathbf{e}_{m-1}$ will be very small, in fact it can be made ``very small'' since $N_*$ is very large (see~\eqref{e.N}).

\smallskip

Below in Section~\ref{s.Tm.reg} we will show that the difference~$T_{m-1}-\theta_{m-1}$ is small (see Lemma~\ref{l.Tm.minus.thetam} for the precise statement). This amounts to homogenizing the temporal oscillations due to the switching of the flows, which have period~$\tau_m''$.

\smallskip

We would next like to write down an equation like~\eqref{e.Tm.divform}, but with~$\J_m$ in place of~$\K_m$. That is, we want to include the temporal oscillations due to switching between horizontal and vertical shear flows. Recall that the function $\J_m$ is essentially the sum of products of $\tau_m$--periodic and $\tau_m^{\prime\prime}$--periodic functions of time only. When the faster time scale is averaged out of~$\J_m$, one obtains~$\K_m$, up to very small errors: see~\eqref{e.Jhat}~\eqref{e.JJhat}, and~\eqref{e.K}, above. 
We could write this equation perhaps as
\begin{align}
\label{e.Tm.tilde}
\left\{
\begin{aligned}
& \partial_t \tilde{T}_{m-1} 
- \nabla \cdot 
\bigl( \J_m + \mathbf{r}_{m-1} \bigr) \nabla \tilde{T}_{m-1}
+ \b_{m-1} \cdot \nabla \tilde{T}_{m-1}
= 0 
& \mbox{in} & \ (0,\infty) \times \R^2, 
\\ & 
\tilde{T}_{m-1} = \theta_0 
& \mbox{on} & \ \{ 0 \} \times \R^2
\,,
\end{aligned}
\right.
\end{align}
where~$\mathbf{r}_{m-1}$ is defined as in~\eqref{e.sm}, with~$\J_m$ in place of~$\K_m$.
We will show the difference~$\tilde{T}_{m-1} - T_{m-1}$ is, up to errors we are able to neglect, given by an expansion which we now introduce. 

\smallskip

We introduce a function~$\widetilde{H}_m$ which is intended to represent, to leading order, the difference between~$\tilde{T}_{m-1} - T_{m-1}$. 
The first idea is to take it to solve the transport-type equation\footnote{See~\eqref{e.Hm.def} for the actual definition of $\widetilde{H}_m$, and~\eqref{e.real.eq.Hm}--\eqref{e.dm} for the equation it solves.}
\begin{equation}
\label{e.Hm}
\left\{
\begin{aligned}
& \bigl( \partial_t + \b_{m-1} \cdot \nabla \bigr) \widetilde{H}_m 
= 
\widetilde{G}_m
& \text{in} & \ (0,\infty) \times\R^2\,,
\\  &
\widetilde{H}_m  = 0 
& \text{on} & \ \{ 0 \} \times\R^2\,,
\end{aligned}
\right.
\end{equation}
where\footnote{
Note that, in the second line of~\eqref{e.Gm}, we used our notational convention (introduced in~\eqref{e.gradcorrmatrix}) of writing vector-valued functions as row vectors and gradients of scalars as column vectors. Hence
\begin{equation}
\label{e.chainrule.Tm1}
\nabla \bigl( T_{m-1} \circ X_{m-1,k} \bigr) \circ X_{m-1,k}^{-1}
=
\bigl(\nabla X_{m-1,k}  \circ X_{m-1,k}^{-1} \bigr) \nabla T_{m-1}
\,.
\end{equation}
}
\begin{align}
\label{e.Gm}
\widetilde{G}_m 
& 
:=
\bigl( \J_m  - 
\K_{m}
\bigr)
: 
\sum_{l\in \Z}  
\hat{\xi}_{m,l} 
\nabla \bigl( 
\nabla \bigl( T_{m-1} \circ X_{m-1,l} \bigr) \circ X_{m-1,l}^{-1}
\bigr) 
\notag \\ & \;
= 
\nabla \cdot 
(\J_m-\K_m)
\sum_{l\in\Z}  
\hat{\xi}_{m,l }
\bigl( 
\nabla X_{m-1,l} \circ X_{m-1,l}^{-1}
\bigr) \nabla T_{m-1}
\,.
\end{align}
We get this by subtracting the equations for~$T_{m-1}$ and~$\widetilde{T}_{m-1}$ and then ignoring the diffusion term~$ \nabla \cdot 
\bigl( \J_m + \mathbf{r}_{m-1} \bigr) (\nabla T_{m-1}-\nabla \widetilde{T}_{m-1})$. 
We will not however define~$\widetilde{H}_m$ to be the solution of~\eqref{e.Hm}, because we are unable to prove sufficient regularity estimates for it. 
Instead we will define~$\widetilde{H}_m$ explicitly in terms of known ingredients which approximate the solution of~\eqref{e.Hm}, but for which we can prove better regularity estimates. 

\smallskip 

In order to define~$\widetilde{H}_m$, we first define a function~$\widetilde{H}_{m,r}$ as follows:
\begin{equation}
\label{e.H.mr.def}
\widetilde{H}_{m,r}(t,x)
:=
\nabla \cdot
\sum_{n=0}^{N_*-1}
\mathbf{A}_{m,n,r}(t,x)
\mathbf{q}_{m,n,r+1}^{\kappa_m} (t)
\end{equation}
where~$\mathbf{q}_{m,n,r} = \mathbf{q}_{m,k,r}^{\kappa_m}$ as defined in~\eqref{e.q.mnr.def} and the tensors $\mathbf{A}_{m,n,r} = (\mathbf{A}_{m,n,r}^{ijk})_{i,j,k=1}^2$ are the $3$-tensors recursively defined as follows:
\begin{equation}
\label{e.A.mnr.def}
\left\{
\begin{aligned}
&\mathbf{A}_{m,n,0}^{ijk}
:= 
- \delta_{ij} L_{m,n}^{\kappa_m}(t) 
\sum_{l \in \Z} \hat{\xi}_{m,l}(t)
\bigl(\partial_k X_{m-1,l}^p \circ X_{m-1,l}^{-1}\bigr) 
\partial_p T_{m-1}
\,, \\  
& \mathbf{A}_{m,n,r+1}^{ijk} 
:= (\partial_t + \b_{m-1}\cdot \nabla) \mathbf{A}_{m,n,r}^{ijk}   +
\nabla_\ell \b_{m-1}^i \mathbf{A}_{m,n,r}^{\ell j k}  
\,.
\end{aligned}
\right.
\end{equation}
In \eqref{e.H.mr.def} we make the following convention regarding the contraction of indices: if $\mathbf{q} = (\mathbf{q}^{jk})_{k,j=1}^2$ is a $2$-tensor, and $\mathbf{A}= (\mathbf{A}^{ijk})_{j,k,i=1}^2$ is a $3$-tensor, then $\nabla \cdot (\mathbf{A} \mathbf{q}) = \partial_i  (\mathbf{A}^{ijk} \mathbf{q}^{jk} )$.

\smallskip

The definition in~\eqref{e.A.mnr.def} above was made in view of the fact that
\begin{equation*}
\bigl[ \bigl(\partial_t + \b_{m-1} \cdot \nabla \bigr) , \nabla \bigr] 
=
\nabla \b_{m-1} \cdot \nabla 
\end{equation*}
so that 
\begin{align*}
\bigl(\partial_t + \b_{m-1} \cdot \nabla \bigr)
\widetilde{H}_{m,r}
&
=
\bigl(\partial_t + \b_{m-1} \cdot \nabla \bigr)
\nabla \cdot
\sum_{n=0}^{N_*-1}
\mathbf{A}_{m,n,r}(t,x)
\mathbf{q}_{m,n,r+1}^{\kappa_m} (t)
\notag \\ & 
=
\nabla \cdot 
\sum_{n=0}^{N_*-1}
\mathbf{A}_{m,n,r+1}
\mathbf{q}_{m,n,r+1}^{\kappa_m} 
-
\nabla \cdot 
\sum_{n=0}^{N_*-1}
\mathbf{A}_{m,n,r}
\mathbf{q}_{m,n,r}^{\kappa_m}
\end{align*}
and therefore, by defining
\begin{equation}
\label{e.Hm.def}
\widetilde{H}_{m}
=
\sum_{r=0}^{\nicefrac{N_*}{2}}
\widetilde{H}_{m,r}
\end{equation}
we obtain, by telescoping the resulting sum, that 
\begin{equation*}
\bigl(\partial_t + \b_{m-1} \cdot \nabla \bigr)
\widetilde{H}_{m}
=
-
\nabla \cdot \!
\sum_{n=0}^{N_*-1}
\mathbf{A}_{m,n,0}
\mathbf{q}_{m,n,0}^{\kappa_m} 
+
\nabla \cdot \!
\sum_{n=0}^{N_*-1}
\mathbf{A}_{m,n,\nicefrac{N_*}{2}}
\mathbf{q}_{m,n,\nicefrac{N_*}{2}}^{\kappa_m} 
\,.
\end{equation*}
We rewrite the first term as
\begin{align*}
\lefteqn{
-
\nabla \cdot \!
\sum_{n=0}^{N_*-1}
\mathbf{A}_{m,n,0}(t,x)
\mathbf{q}_{m,n,0}^{\kappa_m} (t)
}
\qquad & 
\notag \\ & 
=
\partial_{i} \sum_{n=0}^{N_*-1}
\Bigl( \mathbf{j}_{m,n}^\kappa - \bigl\langle \!\!\bigl\langle
\mathbf{j}_{m,n}^\kappa \bigr\rangle \!\!\bigr\rangle
\Bigr)_{jk}
\biggl( 
\delta_{ij} L_{m,n}^{\kappa_m}(t) 
\sum_{l \in \Z} \hat{\xi}_{m,l}(t)
\partial_k (T_{m-1} \circ X_{m-1,l}) \circ X_{m-1,l}^{-1} 
\biggr)
\notag \\ & 
=
\nabla \cdot 
(\hat{\J}_m-\K_m)
\sum_{l\in\Z}  
\hat{\xi}_{m,l }
\bigl( 
\nabla X_{m-1,l} \circ X_{m-1,l}^{-1}
\bigr) \nabla T_{m-1}
\notag \\ & 
=
\widetilde{G}_m 
+
\nabla \cdot 
(\hat{\J}_m-\J_m)
\sum_{l\in\Z}  
\hat{\xi}_{m,l }
\bigl( 
\nabla X_{m-1,l} \circ X_{m-1,l}^{-1}
\bigr) \nabla T_{m-1}
\,.
\end{align*}
We therefore obtain that 
\begin{equation}
\label{e.real.eq.Hm}
\bigl(\partial_t + \b_{m-1} \cdot \nabla \bigr)
\widetilde{H}_{m}
=
\widetilde{G}_m 
+
\nabla \cdot \mathbf{d}_{m} 
\end{equation}
where we define~$\mathbf{d}_{m}$ by
\begin{equation}
\label{e.dm}
\mathbf{d}_{m}
:=
(\hat{\J}_m-\J_m)
\sum_{l\in\Z}  
\hat{\xi}_{m,l }
\bigl( 
\nabla X_{m-1,l} \circ X_{m-1,l}^{-1}
\bigr) \nabla T_{m-1}
+ \!
\sum_{n=0}^{N_*-1}
\mathbf{A}_{m,n,\nicefrac{N_*}{2}}
\mathbf{q}_{m,n,\nicefrac{N_*}{2}}^{\kappa_m} 
\,.
\end{equation}
The~$\mathbf{d}_m$ error will be very small, proportional to~$\ep_{m-1}^{\delta N_*}$ (which will be much less than~$\ep_{m-1}^{1000}$). This is because of the closeness of~$\hat{\J}$ to~$\J$ in~\eqref{e.JJhat}, and the fact that~$\mathbf{A}_{m,n,N_*}$ is similarly small, as we will show in Section~\ref{ss.Hm.estimates}. 

\smallskip

As we will show in Lemma~\ref{p.Hm.bounds} below,~$\widetilde{H}_m$ and its gradient are relatively small. In fact, its gradient is small enough that the error made by plugging it into the diffusion part of the operator is small and can be neglected. This says implicitly that the equation~\eqref{e.Tm.tilde} homogenizes to~\eqref{e.Tm}, which takes care of the temporal oscillations on scale~$\tau_m$.

\subsection{{Definition of the multiscale ansatz~\texorpdfstring{$\widetilde{\theta}_m$}{tilde theta m}}}
\label{ss.ansatz}

Without further ado, we can now present the two-scale ansatz~$\tilde{\theta}_m$, which is defined by 
\begin{align}
\label{e.ansatz}
\tilde{\theta}_m
&
: =
T_{m-1} 
+
\sum_{k,l\in\Z}
\hat{\xi}_{m,l}
\xi_{m,k}
\tilde{\Chi}_{m,k}
\bigl(
\nabla \bigl( T_{m-1} \circ X_{m-1,l} \bigr) \circ X_{m-1,l}^{-1} \bigr)
+
\widetilde{H}_{m}
\\ & \; 
\label{e.ansatz.line2}
=
T_{m-1} 
+
\sum_{k\in 2\Z+1}
\xi_{m,k}
\tilde{\Chi}_{m,k}
\bigl(
\nabla \bigl( T_{m-1} \circ X_{m-1,l_k} \bigr) \circ X_{m-1,l_k}^{-1} \bigr)
+
\widetilde{H}_{m}
\,,
\end{align}
where~$\tilde{\Chi}_{m,k}$ is the ``twisted corrector'' defined by
\begin{equation}
\label{e.Chim.def}
\tilde{\Chi}_{m,k} := \Chi_{m,k} \circ X_{m-1,k}^{-1}  \,.
\end{equation}
Recall that~$l_k$ is defined in~\eqref{e.lk.def} and satisfies~\eqref{e.cutoff.overlaps}. 
The second line~\eqref{e.ansatz.line2} in the display above is valid due to~\eqref{e.hatxi.partition} the fact that~$\Chi_{m,k}$ vanishes if~$k$ is even and
\begin{equation}
\label{e.dancing.partitions}
\hat{\xi}_{m,l} \xi_{m,k} \Chi_{m,k}
=
\xi_{m,k} \Chi_{m,k} \indc_{\{l = l_k\}}
\,.
\end{equation}
In heavier notation, without our simplifying conventions outlined above, we could write~$\tilde{\theta}_m$ as
\begin{align*}
\tilde{\theta}_{m}  (t,x) 
&
= 
T_{m-1} (t,x) 
+ \!\!
\sum_{k\in 2\Z+1} \!\!
\xi_{m,k}(t)
\tilde{\Chi}_{m,k}(t,x)
\nabla \left( T_{m-1}(t,X_{m-1,l_k}(t,\cdot))\right) \bigl(X_{m-1,l_k}^{-1}(t,x)\bigr) 
+
\widetilde{H}_m  (t,x)
\,.
\end{align*}
In order that the tensor contractions are clear, we also mention that the second term of~\eqref{e.ansatz} can be written in coordinates as
\begin{align*}
\sum_{k\in2\Z+1}
\sum_{i=1}^2
(\tilde{\Chi}_{m,k})_i
\xi_{m,k} \,
\partial_{x_i} \left( 
T_{m-1} \circ X_{m-1,l_k}
\right) 
\circ X_{m-1,l_k}^{-1} 
\,.
\end{align*}
Since, as explained above, we should consider $T_{m-1} + \widetilde{H}_m \approx \tilde{T}_{m-1}$, which solves~\eqref{e.Tm.tilde}, the definition of our  ansatz~\eqref{e.ansatz} can be compared to 
\begin{align*}
\tilde{T}_{m-1} 
+ \sum_{k\in 2\Z+1}
\xi_{m,k}
\tilde{\Chi}_{m,k}
\bigl(
\nabla \bigl( \tilde{T}_{m-1} \circ X_{m-1,l_k} \bigr) \circ X_{m-1,l_k}^{-1} \bigr)
\,.
\end{align*}
The latter is similar to a two-scale expansion in classical homogenization, with the role of the macroscopic function being played by~$\tilde{T}_{m-1}$ and the correctors by~$\Chi_{m,k}$. The compositions with the flows implicitly mean that the expansion is with respect to Lagrangian variables. In other words, on each time interval of size $\tau_m^{\prime\prime}$, we have composed with the appropriate flow~$X_{m-1,l_k}$, written the two-scale expansion in these variables, and then composed with the inverse flow. 
Implicit is the assumption that, since the inverse flows~$X_{m-1,l}^{-1}$ are ``slow,'' the correctors should be close to the correctors for the stationary shear flows composed with the inverse flows. 

\smallskip

We will test the validity of our ansatz~\eqref{e.ansatz} by plugging it into the left side of the equation for~$\theta_m$, and estimating the error. This is the focus of Section~\ref{s.cascade}. To prepare for this analysis, we need to obtain good regularity estimates on the ``macroscopic'' ingredients in the expansion: in particular, the function~$T_{m-1}$. Indeed, as we know from classical homogenization theory, the homogenization error depends on the regularity of the macroscopic solution.

\subsection{{Regularity estimates for~\texorpdfstring{$T_{m-1}$}{T (m-1)}}}
\label{s.Tm.reg}

Our first order of business is to show that the equation~\eqref{e.Tm.true} can be considered as a small perturbation of~\eqref{e.theta.m}  with~$m$ replaced by~$m-1$. This intuition is formalized 
by noting that according to~\eqref{e.T.m-1.0}--\eqref{e.T.m-1.Nstar} we have 
\begin{equation}
\label{e.T.minus.theta.m-1}
T_{m-1} - \theta_{m-1} 
= T_{m-1}^{({N_*})} - T_{m-1}^{(0)} 
= \sum_{i=1}^{ {N_*} }   
\underbrace{T_{m-1}^{(i)} - T_{m-1}^{(i-1)}}_{=: V_{m-1}^{(i)}}
\,,
\end{equation}
where the equation for the increment~$V_{m-1}^{(i)}$ can be found by subtracting~\eqref{e.Tm-1.i} with~$i$ and~$i-1$. We obtain, for each~$1\leq i \leq {N_*}$, 
\begin{align}
\label{e.Vm-1.i}
\left\{
\begin{aligned}
& \partial_t V_{m-1}^{(i)} 
- \kappa_{m-1} \Delta V_{m-1}^{(i)}
+ \b_{m-1} \cdot \nabla V_{m-1}^{(i)}
= \nabla \cdot \bigl( \K_m - \kappa_{m-1} \Itwo + \mathbf{s}_{m-1} \bigr) \nabla V_{m-1}^{(i-1)}
& \mbox{in} & \ (0,\infty) \times \R^2, 
\\ & 
V_{m-1}^{(i)} = 0 
& \mbox{on} & \ \{ 0 \} \times \R^2
\,,
\end{aligned}
\right.
\end{align}
where for convenience we denote~$T_{m-1}^{(-1)} := 0$, so that $V_{m-1}^{(0)} = \theta_{m-1}$.

 \smallskip
In light of \eqref{e.T.minus.theta.m-1}--\eqref{e.Vm-1.i} it is apparent that we must obtain good estimates for $\nabla \theta_{m-1}$ and $\nabla V_{m-1}^{(i)}$, which we achieve in Lemmas~\ref{l.theta.m-1.reg.upgrade} and~\ref{l.V.m-1.reg.upgrade} below.

\begin{lemma}[{Estimates on~$\theta_{m-1}$}]
\label{l.theta.m-1.reg.upgrade}
There exists~$1\leq C<\infty$ such that, 
for every~$n\in\N_0$, 
\begin{equation}
\label{e.theta.m-1.reg.upgrade}
\bigl\| \nabla^n \theta_{m-1} \bigr\|_{L^\infty([0,1];L^2(\TT^2))}
+
\kappa_{m-1}^{\nicefrac 12} 
\bigl\| \nabla^{n+1} \theta_{m-1} \bigr\|_{L^2((0,1)\times\TT^2)}
\leq  
\|\theta_0\|_{L^2(\TT^2)}
n! \bigl(C \ep_{m-1}^{-1 - \nicefrac \gamma 2} \vee 2 R_{\theta_0}^{-1}\bigr)^n
\,.
\end{equation}
\end{lemma}
\begin{proof}
We fix a multi-index~$\aa$ of order $|\aa|=n \geq 0$ and apply $\partial^{\aa}$ to both sides of~\eqref{e.thetam.divform} (with $m$ replaced by $m-1$) to obtain:
\begin{align}
\label{e.eqtheta.m-1}
\partial_t \partial^{\aa}\theta_{m-1} 
- \nabla \cdot \bigl( 
\kappa_{m-1} \Itwo 
+ \phi_{m-1} \sigma \bigr) \nabla \partial^{\aa}\theta_{m-1} 
=
\nabla \cdot 
\sum_{\bb<\aa} \binom{\aa}{\bb} 
\partial^{\aa-\bb}   \phi_{m-1} \sigma  \nabla \partial^{\bb} \theta_{m-1}
\end{align}
in $(0,\infty) \times \TT^2$. 
Testing~\eqref{e.eqtheta.m-1} with~$\partial^{\aa}  \theta_{m-1}$ and using that $\sigma$ is skew-symmetric, we get
\begin{align}
\label{e.test.partialaaTm}
&
\sup_{t \in [0,1]}
\bigl\| 
\partial^\aa \theta_{m-1} (t,\cdot) 
\bigr\|_{L^2(\TT^2)}^2
-
\bigl\| 
\partial^\aa \theta_0
\bigr\|_{L^2(\TT^2)}^2
+
2 \kappa_{m-1} \int_0^1 \int_{\TT^2}
\bigl| \nabla \partial^\aa \theta_{m-1} \bigr|^2
\notag \\ & \qquad 
\leq
2 \sum_{\bb<\aa} \binom{\aa}{\bb} 
\left| 
\int_0^1
\int_{\TT^2}
\nabla \partial^\aa \theta_{m-1} \cdot 
\partial^{\aa-\bb}  \phi_{m-1} \sigma  \nabla \partial^{\bb} \theta_{m-1}
\right|
\,.
\end{align}
When $n=0$, the right side of \eqref{e.eqtheta.m-1}, and hence \eqref{e.test.partialaaTm}, vanishes identically. As such, we are only left to consider the case $n\geq 1$.

To upper bound the right side of~\eqref{e.test.partialaaTm}, we split the sum into two parts:  the terms involving~$\phi_{m-1}$ with~$| \aa - \bb | \geq 2$ and the terms involving~$\phi_{m-1}$ with~$| \aa - \bb | =1$. 
For the first group of terms, we use~\eqref{e.phi.m.bounds} and obtain
\begin{align}
\label{e.grouptwo}
\lefteqn{
\sum_{\bb<\aa,|\aa-\bb|\geq 2} \binom{\aa}{\bb} 
\biggl|
\int_0^1
\int_{\TT^2}
\nabla \partial^\aa \theta_{m-1} \cdot 
\partial^{\aa-\bb}\phi_{m-1} \sigma \nabla \partial^{\bb} \theta_{m-1}
\biggr|
} \qquad & 
\notag \\ & 
\leq
\kappa_{m-1}  \bigl\| \nabla \partial^{\aa} \theta_{m-1} \bigr\|_{L^2((0,1)\times\TT^2)}^2
\notag \\ & \qquad 
+
\frac{C\ep_{m-1}^{2\beta}}{\kappa_{m-1}}
\sum_{\bb<\aa,|\aa-\bb|\geq 2} \!
\frac{|\aa|!^2}{|\bb|!^2}
(C \ep_{m-1}^{-1})^{2|\aa-\bb|}
\bigl\| \nabla \partial^{\bb} \theta_{m-1} \bigr\|_{L^2((0,1)\times\TT^2)}^2 \,.
\end{align}
For the second group, in which $|\aa-\bb|=1$, we perform an integration by parts and use the skew-symmetry of~$\sigma$ and~\eqref{e.phi.m.bounds} with~$n=2$ to see that  
\begin{align*}
\biggl|
\int_0^1
\int_{\TT^2}
\nabla \partial^\aa \theta_{m-1} \cdot 
\partial^{\aa-\bb} \phi_{m-1} \sigma  \nabla \partial^{\bb} \theta_{m-1}
\biggr|
&
=
\biggl|
\int_0^1
\int_{\TT^2}
\partial^\aa \theta_{m-1} \cdot 
\nabla(\partial^{\aa-\bb} \phi_{m-1}) \cdot \sigma  \nabla \partial^{\bb} \theta_{m-1}
\biggr|
\\ & 
\leq 
C\ep_{m-1}^{\beta-2}
\bigl\| \partial^{\aa} \theta_{m-1} \bigr\|_{L^2((0,1)\times\TT^2)}
\bigl\| \nabla \partial^{\bb} \theta_{m-1} \bigr\|_{L^2((0,1)\times\TT^2)} 
\,.
\end{align*}
Therefore, if $|\aa-\bb|=1$, then 
\begin{align}
\label{e.groupthree}
\lefteqn{
\sum_{\bb<\aa,|\aa-\bb|= 1} \binom{\aa}{\bb} 
\biggl|
\int_0^1
\int_{\TT^2}
\nabla \partial^\aa \theta_{m-1} \cdot 
\partial^{\aa-\bb}\phi_{m-1} \sigma \nabla \partial^{\bb} \theta_{m-1}
\biggr|
} \qquad  & 
\notag \\ & 
\leq
C|\aa|  \ep_{m-1}^{\beta-2}
\bigl\| \partial^{\aa} \theta_{m-1} \bigr\|_{L^2((0,1)\times\TT^2)}
\max_{\bb<\aa,|\aa-\bb|=1}
\bigl\| \nabla \partial^{\bb} \theta_{m-1} \bigr\|_{L^2((0,1)\times\TT^2)} 
\,.
\end{align}
We next insert the estimates~\eqref{e.grouptwo} and~\eqref{e.groupthree} into the right side of~\eqref{e.test.partialaaTm}, and appeal to \eqref{e.theta0.anal} to bound the initial data term.  After dividing by $|\aa|!^2$, we obtain for $n=|\aa| \geq 1 $ that 
\begin{align}
\label{e.groupthree.follow}
\lefteqn{
\frac{\kappa_{m-1} \bigl\|
\nabla \partial^\aa \theta_{m-1} 
\bigr\|_{L^2((0,1)\times\TT^2)}^2}
{|\aa|!^2}
+
\frac{\sup_{t \in [0,1]}\bigl\| 
\partial^\aa \theta_{m-1} (t,\cdot) 
\bigr\|_{L^2(\TT^2)}^2}{|\aa|!^2}
-
\frac{\|\theta_0\|_{L^2(\TT^2)}^2}{R_{\theta_0}^{2n} } 
} \qquad & 
\notag \\ & 
\leq
\frac{C\ep_{m-1}^{2\beta}}{\kappa_{m-1}^2}
\sum_{\bb<\aa,|\aa-\bb|\geq 2} \!
 (C \ep_{m-1}^{-1})^{2|\aa-\bb|} 
\frac{\kappa_{m-1}\bigl\| \nabla \partial^{\bb} \theta_{m-1} \bigr\|_{L^2((0,1)\times\TT^2)}^2}{|\bb|!^2}
\notag \\ & \qquad 
+ \frac{C \ep_{m-1}^{\beta-2}}{|\aa| \kappa_{m-1}}
\frac{\kappa_{m-1}^{\nicefrac{1}{2}}\bigl\| \partial^{\aa} \theta_{m-1} \bigr\|_{L^2((0,1)\times\TT^2)}}{(|\aa|-1)!}
\max_{\bb<\aa,|\aa-\bb|=1}
\frac{\kappa_{m-1}^{\nicefrac{1}{2}} \bigl\| \nabla \partial^{\bb} \theta_{m-1} \bigr\|_{L^2((0,1)\times\TT^2)}}{|\bb|!}\,.
\end{align}
Fixing a constant~$A\geq 1$, to be selected below (just above~\eqref{e.Dn.recursion}), and defining
\begin{equation}
\label{e.Dn.def}
D_n 
:= 
\frac{1}{\|\theta_0\|_{L^2(\TT^2)}}
\left( \frac{\ep_{m-1}}{A} \right)^n 
\max_{|\aa|=n} \left(
\frac{\kappa_{m-1}^{\nicefrac 12} \bigl\| \nabla \partial^{\aa} \theta_{m-1} \bigr\|_{L^2((0,1)\times\TT^2)}}{|\aa|!}
+ 
\frac{\sup_{t \in [0,1]}\bigl\| 
\partial^\aa \theta_{m-1} (t,\cdot) 
\bigr\|_{L^2(\TT^2)}}{|\aa|!}
\right)
\,,\end{equation}
we then take the maximum of \eqref{e.groupthree.follow} over all multi-indices~$\aa$ with $|\aa|=n$ and rearranging the resulting expression, using also elementary bounds for multinomial coefficients and factorials, to obtain
\begin{align*}
D_n^2
-   
\left( \frac{\ep_{m-1}}{A R_{\theta_0}} \right)^{2n} 
\leq
\frac{C \ep_{m-1}^{2\beta} }{\kappa_{m-1}^2}
\sum_{k=0}^{n-2} 
\biggl( \frac{C}{A} \biggr)^{2(n-k)}  
D_k^2
+
\frac{C\ep_{m-1}^{\beta}}{n \kappa_{m-1} A^2} 
D_{n-1}^2 \,.
\end{align*}
Using the upper and lower bounds for $\kappa_{m-1}$ from~\eqref{e.kappam.bound}, for $n\geq 1$ we obtain from the above estimate that
\begin{align*}
D_n^2
\leq 
\left( \frac{\ep_{m-1}}{A R_{\theta_0}} \right)^{2n} 
+
\frac{C }{ \ep_{m-1}^{2\gamma}}
\sum_{k=0}^{n-2} 
\biggl( \frac{C}{A} \biggr)^{2(n-k)}  
D_k^2
+
\frac{C}{\ep_{m-1}^{\gamma} A^2} 
D_{n-1}^2 \,.
\end{align*}
Note that when $n=0$ only the first term on the right side of the above estimate is present. 
If we choose $A$ by 
\begin{equation*}
A:= \max
\Bigl\{
1,4 C^{\nicefrac 32},(4C)^{\nicefrac 12}, 2 \ep_{m-1}^{1+\nicefrac{\gamma}{2}} R_{\theta_0}^{-1}
\Bigr\}\,,
\end{equation*}
then we obtain
\begin{equation}
\label{e.Dn.recursion}
D_n
\leq
2^{-n} \bigl(\ep_{m-1}^{-\nicefrac \gamma2} \bigr)^n
+
2^{-2}{\bf 1}_{n\geq 2}
\ep_{m-1}^{-\gamma} 
\max\{ D_0,\ldots,D_{n-2} \}
+ 
2^{-2}{\bf 1}_{n\geq 1}
\ep_{m-1}^{-\nicefrac \gamma2} 
D_{n-1}
\,.
\end{equation}
Iterating this inequality, we discover that, for every $n\in\N_0$, 
\begin{equation*}
D_n \leq \bigl(\ep_{m-1}^{-\nicefrac \gamma 2} \bigr)^n 
\,.
\end{equation*}
Recalling \eqref{e.Dn.def} and the notation~\eqref{e.nabla.n}, the above estimate implies 
\begin{equation*}
 \sup_{t \in [0,1]}\bigl\| 
\nabla^n \theta_{m-1} (t,\cdot) 
\bigr\|_{L^2(\TT^2)}
+
\kappa_{m-1}^{\nicefrac 12} \bigl\| \nabla^{n+1} \theta_{m-1} \bigr\|_{L^2((0,1)\times\TT^2)}
\leq  \|\theta_0\|_{L^2(\TT^2)}
n! \bigl(A \ep_{m-1}^{-1 - \nicefrac \gamma 2} \bigr)^n
\,.
\end{equation*}
In view of our choice of $A$, the above estimate gives~\eqref{e.theta.m-1.reg.upgrade}. 
\end{proof}

Next, we aim to obtain similar regularity estimates for~$\{V_{m-1}^{(i)} \}_{i=1}^{{N_*}}$. Since the equation \eqref{e.Vm-1.i}  contains the matrix~$\mathbf{s}_{m-1}$, we first need to obtain suitable estimates for this function. We show~$\mathbf{s}_{m-1}$ is small relative to~$\kappa_{m-1}$, and the scale of its spatial oscillations are large compared to~$\ep_m$. 
Indeed, by~\eqref{e.Xm.bound.1} and~\eqref{e.taum.primeprime.def}, we have that the term defined in \eqref{e.sm} satisfies 
\begin{align}
\label{e.smbound}
\left\| \mathbf{s}_{m-1} \right\|_{L^\infty(\R\times\TT^2)}
\leq
\kappa_{m-1} 
\sup_{l\in \Z} 
\bigl\| \hat{\xi}_{m,l} 
(\nabla X_{m-1,l} - \Itwo )
\bigr\|_{L^\infty(\R \times\TT^2)} 
&
\leq
C\kappa_{m-1} \ep_{m-1}^{2\delta}
\notag \\ & 
= 
C a_{m-1} \ep_{m-1}^{2+\gamma+2\delta}
=
C  \ep_{m-1}^{\beta+\gamma+2\delta}\,.
\end{align}
More generally, we have the following bound on the higher-order spatial derivatives of~$\mathbf{s}_{m-1}$:
\begin{align}
\label{e.smbound.dervs}
\snorm{ \mathbf{s}_{m-1} }_{C\ep_{m-1}^{-1}}
\leq
C \kappa_{m-1}
=
C  \ep_{m-1}^{\beta+\gamma}
\,.
\end{align}
The estimate~\eqref{e.smbound.dervs} is a consequence of~\eqref{e.Xm.bound.2} and Proposition~\ref{prop:compose:analytic}.

\smallskip

Comparing~\eqref{e.smbound.dervs} to~\eqref{e.phi.m.bounds} and~\eqref{e.C1beta.phi.m}, we see that~$\mathbf{s}_{m-1}$ is smaller than~$\phi_m$ by a factor of~$\ep^{\gamma+2\delta}$ while having the same analyticity radius, and is smaller in size than~$\kappa_{m-1} \simeq a_m\ep_m^{2+\gamma}$ by a factor of~$\ep_{m-1}^{2\delta}$. Our next goal is to use this fact to show that the bounds obeyed by~$V_{m-1}^{(i)}$ are better than those satisfied by~$\theta_{m-1}$, by a factor of at least~$\ep_{m-1}^{2\delta}$.

\begin{lemma}[{Estimates on~$V_{m-1}^{(i)}$}]
\label{l.V.m-1.reg.upgrade}
There exists a constant~$C_0<\infty$ such that, 
if~$\ep_{m-1}$ is small enough that
\begin{equation}
\label{l.V.m-1.reg.upgrade.ass}
C_0^3 \ep_{m-1}^{2\delta} \bigl(1 \vee \ep_{m-1}^{2+\gamma}  R_{\theta_0}^{-2}\bigr) \leq 1
\,,
\end{equation}
then, for every~$n,i\in\N_0$, we have the estimate
\begin{multline}
\label{e.Vm-1.reg.upgrade}
\frac{1}{(n+2i)!} 
\bigl(C_0 \ep_{m-1}^{-1 - \nicefrac \gamma 2} \vee C_0 R_{\theta_0}^{-1}\bigr)^{-n}
\left( 
\bigl\| \nabla^n V_{m-1}^{(i)}
\bigr\|_{L^\infty([0,1];L^2(\TT^2))}
+
\kappa_{m-1}^{\nicefrac 12} 
\bigl\| \nabla^{n+1} V_{m-1}^{(i)}
\bigr\|_{L^2((0,1)\times\TT^2)}
\right)
 \\ 
\leq
\mathsf{A}_{m-1,i} \|\theta_0\|_{L^2(\TT^2)}
\,,
\end{multline}
where we have defined the amplitude coefficients appearing in \eqref{e.Vm-1.reg.upgrade} by 
\begin{align}
\label{e.A.m-1.i.def}
\mathsf{A}_{m-1,i}
:= {\bf 1}_{\{i=0\}} 
+ {\bf 1}_{\{i \in \{1,2\}\}} \Bigl(C_0^3 \ep_{m-1}^{2\delta} \bigl(1 \vee \ep_{m-1}^{2+\gamma}  R_{\theta_0}^{-2}\bigr)\Bigr) 
+ {\bf 1}_{\{i\geq 3\}} \Bigl(C_0^3 \ep_{m-1}^{2\delta} \bigl(1 \vee \ep_{m-1}^{2+\gamma}  R_{\theta_0}^{-2}\bigr) \Bigr)^{\frac{i}{2}} 
\,.
\end{align}
\end{lemma}
\begin{proof}
The proof of Lemma~\ref{l.V.m-1.reg.upgrade} is recursive in $i\geq 0$, and closely follows the proof of Lemma~\ref{l.theta.m-1.reg.upgrade}.  Since $V_{m-1}^{(0)} = \theta_{m-1}$, the bound \eqref{l.theta.m-1.reg.upgrade} establishes the inductive step, namely \eqref{e.Vm-1.reg.upgrade} for $i=0$, as long as we ensure $C_0 \geq C_{\eqref{e.theta.m-1.reg.upgrade}}$. 

Next, assume that \eqref{e.Vm-1.reg.upgrade} with $i$ replaced by $i-1$. 
Comparing the $V_{m-1}^{(i)}$ evolution~\eqref{e.Vm-1.i} and the $\theta_{m-1}$ evolution in~\eqref{e.thetam.divform} (with $m$ replaced by $m-1$), we see that the only difference is due to the forcing term $\nabla \cdot \bigl( \K_m - \kappa_{m-1} \Itwo + \mathbf{s}_{m-1} \bigr) \nabla V_{m-1}^{(i-1)}$, and the fact that the initial data for $V_{m-1}^{(i)}$ vanishes identically. As such, since $\K_m$ is only a function of time, \eqref{e.test.partialaaTm} becomes
\begin{align}
\label{e.test.partialaaVm-1}
&
\sup_{t \in [0,1]}
\bigl\| 
\partial^\aa V_{m-1}^{(i)} (t,\cdot) 
\bigr\|_{L^2(\TT^2)}^2
+
2 \kappa_{m-1} \int_0^1 \int_{\TT^2}
\bigl| \nabla \partial^\aa V_{m-1}^{(i)} \bigr|^2
\notag \\ & \qquad 
\leq
2 \sum_{\bb<\aa} \binom{\aa}{\bb} 
\left| 
\int_0^1
\int_{\TT^2}
\nabla \partial^\aa V_{m-1}^{(i)} \cdot 
\partial^{\aa-\bb}  \phi_{m-1} \sigma  \nabla \partial^{\bb} V_{m-1}^{(i)}
\right|
\notag \\ & \qquad \quad 
+ 
2 \sum_{\bb\leq \aa} \binom{\aa}{\bb} 
\left| 
\int_0^1
\int_{\TT^2}
\nabla \partial^\aa V_{m-1}^{(i)} \cdot 
\partial^{\aa-\bb}  \mathbf{s}_{m-1}  \nabla \partial^{\bb} V_{m-1}^{(i-1)}
\right|
\notag \\ & \qquad \quad 
+ 
2 \left| 
\int_0^1
\int_{\TT^2}
\nabla \partial^\aa V_{m-1}^{(i)} \cdot 
\bigl( \K_m - \kappa_{m-1} \Itwo \bigr)  \nabla \partial^{\aa} V_{m-1}^{(i-1)}
\right|
\notag \\ & \qquad
=: \mathsf{Err}_1 +  \mathsf{Err}_2 +  \mathsf{Err}_3
\,.
\end{align}
The first term on the right side of \eqref{e.test.partialaaVm-1}, $\mathsf{Err}_1$, is estimated in exactly the same fashion as \eqref{e.grouptwo} and \eqref{e.groupthree}, resulting in the estimate 
\begin{align}
\label{e.Err.1.bnd}
\mathsf{Err}_1 
&\leq 
\frac{\kappa_{m-1}}{4}  \bigl\| \nabla \partial^{\aa} V_{m-1}^{(i)} \bigr\|_{L^2((0,1)\times\TT^2)}^2
\notag \\ & \qquad 
+
\frac{C\ep_{m-1}^{2\beta}}{\kappa_{m-1}}
\sum_{\bb<\aa,|\aa-\bb|\geq 2} \!
\frac{|\aa|!^2}{|\bb|!^2}
(C \ep_{m-1}^{-1})^{2|\aa-\bb|}
\bigl\| \nabla \partial^{\bb} V_{m-1}^{(i)} \bigr\|_{L^2((0,1)\times\TT^2)}^2 
\notag\\ &\qquad 
+
C  |\aa| \ep_{m-1}^{\beta-2}
\bigl\| \partial^{\aa} V_{m-1}^{(i)}  \bigr\|_{L^2((0,1)\times\TT^2)}
\max_{\bb<\aa,|\aa-\bb|=1}
\bigl\| \nabla \partial^{\bb} V_{m-1}^{(i)}  \bigr\|_{L^2((0,1)\times\TT^2)} 
\,.
\end{align}
In order to estimate the second term on the right side of \eqref{e.test.partialaaVm-1}, $\mathsf{Err}_2$, we appeal to the $\mathbf{s}_{m-1}$ bounds \eqref{e.smbound}--\eqref{e.smbound.dervs} and to the inductive estimate for $V_{m-1}^{(i-1)}$ provided by \eqref{e.Vm-1.reg.upgrade}. We arrive at 
\begin{align}
\label{e.Err.2.bnd}
\mathsf{Err}_2
&\leq 
\frac{\kappa_{m-1}}{4}  \bigl\| \nabla \partial^{\aa} V_{m-1}^{(i)} \bigr\|_{L^2((0,1)\times\TT^2)}^2
\notag\\
&\qquad 
+
C \mathsf{A}_{m-1,i-1}^2
 \|\theta_0\|_{L^2(\TT^2)}^2
 \sum_{\bb<\aa} \!
\frac{|\aa|!^2(|\bb|+2i-2)!^2}{|\bb|!^2}
\bigl(C \ep_{m-1}^{-1}\bigr)^{2|\aa-\bb|}
\bigl(C_0  \ep_{m-1}^{-1 - \nicefrac \gamma 2} \vee C_0  R_{\theta_0}^{-1}\bigr)^{2 |\bb|}
\notag\\
&\qquad 
+
C \ep_{m-1}^{4\delta} \mathsf{A}_{m-1,i-1}^2
 \|\theta_0\|_{L^2(\TT^2)}^2
 (|\aa|+2i-2)!^2 
\bigl(C_0  \ep_{m-1}^{-1 - \nicefrac \gamma 2} \vee C_0  R_{\theta_0}^{-1}\bigr)^{2 |\aa|}
\,.
\end{align}
Bounding the last term on the right side of \eqref{e.test.partialaaVm-1}, $\mathsf{Err}_3$, requires more care.
First, recalling~\eqref{e.kappa.sequence} we note that 
\begin{equation}
\label{e.Km.decompose}
\K_m - \kappa_{m-1} \Itwo 
= \K_m - \Khom_m
= \bigl( \K_m - 
\bigl\langle \!\!\bigl\langle
\K_m 
\bigr\rangle \!\!\bigr\rangle
\bigr) + \bigl(\bigl\langle \!\!\bigl\langle
\K_m 
\bigr\rangle \!\!\bigr\rangle - \Khom_m\bigr)\,.
\end{equation}
By appealing to~\eqref{e.Kbar.almost.average}, \eqref{e.enhance.onestep}, and \eqref{e.exprat.bound}, the term $|\langle \hspace{-2.5pt} \langle
\K_m 
\rangle \hspace{-2.5pt}\rangle - \Khom_m|$ may be made arbitrarily small, which may be combined with \eqref{e.Vm-1.reg.upgrade} at level $i-1$ to deduce the bound
\begin{align}
\label{e.Err.3.bnd}
\mathsf{Err}_3
&
\leq
2 \left| 
\int_0^1
\int_{\TT^2}
\nabla \partial^\aa V_{m-1}^{(i)} \cdot 
\bigl( \K_m - 
\langle \hspace{-2.5pt}\langle
\K_m 
\rangle \hspace{-2.5pt} \rangle
\bigr)  \nabla \partial^{\aa} V_{m-1}^{(i-1)}
\right| 
\notag\\
&\quad 
+ 
\frac{C a_m^2\ep_m^4}{\kappa_{m}} \biggl(
\frac{\ep_m^2}{\kappa_{m} \tau_m}
\biggr)^{\!\!N_*}
\bigl\| \nabla \partial^{\aa} V_{m-1}^{(i)} \bigr\|_{L^2((0,1)\times\TT^2)}
\bigl\| \nabla \partial^{\aa} V_{m-1}^{(i-1)} \bigr\|_{L^2((0,1)\times\TT^2)}
\notag\\
&\leq
2 \left| 
\int_0^1
\int_{\TT^2}
\nabla \partial^\aa V_{m-1}^{(i)} \cdot 
\bigl( \K_m -\langle \hspace{-2.5pt}\langle
\K_m 
\rangle \hspace{-2.5pt} \rangle \bigr)  \nabla \partial^{\aa} V_{m-1}^{(i-1)}
\right| 
\notag\\
&\quad 
+ 
\frac{\kappa_{m-1}}{4}  \bigl\| \nabla \partial^{\aa} V_{m-1}^{(i)} \bigr\|_{L^2((0,1)\times\TT^2)}^2
\notag\\
&\quad 
+ C 
\bigl(C \ep_{m-1}^{2\delta} \bigr)^{\! 2N_*}
\mathsf{A}_{m-1,i-1}^2
\|\theta_0\|_{L^2(\TT^2)}^2
(|\aa|+2i-2)!^2 
\bigl(C_0 \ep_{m-1}^{-1 - \nicefrac \gamma 2} \vee C_0 R_{\theta_0}^{-1}\bigr)^{2|\aa|}
\,.
\end{align}
It thus remains to estimate the first term on the right side of \eqref{e.Err.3.bnd}. For this purpose, we recall from \eqref{e.K} that 
\begin{equation*}
\K_m (t) - \langle \hspace{-2.5pt}\langle
\K_m 
\rangle \hspace{-2.5pt} \rangle
= \sum_{n=0}^{N_*-1}  \langle \hspace{-2.5pt} \langle
\mathbf{j}_{m,n}
\rangle \hspace{-2.5pt} \rangle
\bigl( L_{m,n}(t)  - 
\langle \hspace{-2.5pt}\langle L_{m,n} \rangle \hspace{-2.5pt} \rangle
  \bigr)
\end{equation*}
is a zero-mean symmetric matrix, which is $\tau_m^{\prime\prime}$-periodic in time. As such, we may write 
\begin{equation}
\label{e.Q.m.def}
\K_m (t) - \langle \hspace{-2.5pt}\langle
\K_m 
\rangle \hspace{-2.5pt} \rangle 
= \partial_t {\mathbf{Q}}_m(t) 
= \DD_{t,m-1} {\mathbf{Q}}_m(t) 
\,,
\end{equation}
where
\begin{equation}
\label{e.Q.m.bound}
\| \mathbf{Q}_m \|_{L^\infty([0,1])}
\leq 
C \kappa_{m-1} \tau_m^{\prime\prime}
\,,
\qquad 
\mbox{and}
\qquad
\mathbf{Q}_m(0) =   \mathbf{Q}_m(1) = 0
 \,.
\end{equation}
Using the above two displays, we may integrate by parts (since $\nabla \cdot \b_{m-1}=0$, the $L^2$-adjoint of the operator  $\b_{m-1} \cdot \nabla$ is the operator $-\b_{m-1}\cdot \nabla$) and deduce that 
\begin{align}
\label{e.Err.4.def}
&\int_0^1
\int_{\TT^2}
\nabla \partial^\aa V_{m-1}^{(i)} \cdot 
\bigl( \K_m - \langle \hspace{-2.5pt}\langle
\K_m 
\rangle \hspace{-2.5pt} \rangle \bigr)  \nabla \partial^{\aa} V_{m-1}^{(i-1)} 
\notag\\
&= 
- \int_0^1
\int_{\TT^2}
\nabla \partial^\aa V_{m-1}^{(i)} \cdot 
{\mathbf{Q}}_m \DD_{t,m-1} \nabla \partial^{\aa} V_{m-1}^{(i-1)} 
-
\int_0^1
\int_{\TT^2}
\DD_{t,m-1} \nabla \partial^\aa V_{m-1}^{(i)} \cdot 
{\mathbf{Q}}_m \nabla \partial^{\aa} V_{m-1}^{(i-1)} 
\notag\\
&= 
\int_0^1
\int_{\TT^2}
\nabla \partial^\aa V_{m-1}^{(i)} \cdot 
{\mathbf{Q}}_m \nabla \b_{m-1} \cdot \nabla \partial^{\aa} V_{m-1}^{(i-1)} 
+
\int_0^1
\int_{\TT^2}
\nabla \b_{m-1}\cdot \nabla \partial^\aa V_{m-1}^{(i)} \cdot 
{\mathbf{Q}}_m \nabla \partial^{\aa} V_{m-1}^{(i-1)} 
\notag\\
&\qquad  
- \int_0^1
\int_{\TT^2}
\nabla \partial^\aa V_{m-1}^{(i)} \cdot 
{\mathbf{Q}}_m \nabla \DD_{t,m-1}  \partial^{\aa} V_{m-1}^{(i-1)} 
-
\int_0^1
\int_{\TT^2}
\nabla \DD_{t,m-1}  \partial^\aa V_{m-1}^{(i)} \cdot 
{\mathbf{Q}}_m \nabla \partial^{\aa} V_{m-1}^{(i-1)} 
\notag\\
&=: \mathsf{Err}_{4,1} + \mathsf{Err}_{4,2} + \mathsf{Err}_{4,3} + \mathsf{Err}_{4,4}
\,.
\end{align}
Using the $\nabla \b_{m-1}$ estimate in \eqref{e.bm.Dn}, the ${\mathbf{Q}_m}$ bound in \eqref{e.Q.m.bound}, the estimate $\tau_m^{\prime\prime} \ep_{m-1}^{\beta-2}  \leq C \ep_{m-1}^{2\delta}$, we obtain $|\nabla \b_{m-1}| \, |\mathbf{Q}_m| \leq C \kappa_{m-1} \ep_{m-1}^{2\delta}$, and so, by also using the inductive estimate for  $\nabla \partial^\aa V_{m-1}^{(i-1)}$ in \eqref{l.V.m-1.reg.upgrade}, we obtain
\begin{align}
\label{e.Err.4a.bnd}
\bigl| \mathsf{Err}_{4,1} \bigr| + \bigl| \mathsf{Err}_{4,2} \bigr|
&\leq 
\frac{\kappa_{m-1}}{24}  \bigl\| \nabla \partial^{\aa} V_{m-1}^{(i)} \bigr\|_{L^2((0,1)\times\TT^2)}^2
\notag\\
&\qquad 
+ C \ep_{m-1}^{4\delta}
\mathsf{A}_{m-1,i-1}^2
\|\theta_0\|_{L^2(\TT^2)}^2
(|\aa|+2i-2)!^2 
\bigl(C_0 \ep_{m-1}^{-1 - \nicefrac \gamma 2} \vee C_0 R_{\theta_0}^{-1}\bigr)^{2|\aa|}
\,.
\end{align}
Next, we bound the $\mathsf{Err}_{4,3}$ and $\mathsf{Err}_{4,4}$ terms appearing on the right side of \eqref{e.Err.4.def}. For this purpose, we note that \eqref{e.Vm-1.i} gives 
\begin{align}
\label{e.Vm-1.i.material}
\nabla \DD_{t,m-1} \partial^{\aa} V_{m-1}^{(i)}
&= 
\kappa_{m-1} \Delta \nabla \partial^{\aa}  V_{m-1}^{(i)}
- \nabla \sum_{\bb < \aa} \binom{\aa}{\bb} 
\partial^{\aa-\bb} \b_{m-1} \cdot \nabla \partial^{\bb} V_{m-1}^{(i)}
\notag\\
&\qquad 
+
{\bf 1}_{i\geq 1} 
\nabla  \nabla \cdot \bigl( \K_m - \kappa_{m-1} \Itwo \bigr) \nabla \partial^\aa V_{m-1}^{(i-1)}
\notag\\
&\qquad 
+
{\bf 1}_{i\geq 1} 
\nabla \nabla \cdot \sum_{\bb \leq \aa} \binom{\aa}{\bb}
\partial^{\aa-\bb} \mathbf{s}_{m-1} \nabla \partial^{\bb} V_{m-1}^{(i-1)} 
\,.
\end{align}
Using \eqref{e.Vm-1.i.material} we first bound the more difficult term, $\mathsf{Err}_4$. The additional complication arises from the fact that the first line on the right side of \eqref{e.Vm-1.i.material} contains terms with $2$, and respectively $1$, additional derivatives on top of $\nabla \partial^{\aa}$, and this apparently prevents us from closing our estimates; this is however not an issue, as these derivatives may be integrated by parts onto the $T_{m-1}^{(i-1)}$ term, for which we have already estimated all the space derivatives (including those of order $n+3$). To be precise, \eqref{e.Vm-1.i.material} allows us to rewrite 
\begin{align}
\label{e.Err.44.rewrite}
\mathsf{Err}_{4,4}
&=
- \kappa_{m-1}
\int_0^1
\int_{\TT^2}
\nabla  \partial^\aa V_{m-1}^{(i)} \cdot 
{\mathbf{Q}}_m \Delta \nabla \partial^{\aa} V_{m-1}^{(i-1)} 
\notag\\
&\qquad 
- \sum_{\bb < \aa} \binom{\aa}{\bb} 
\int_0^1
\int_{\TT^2}
\partial^{\aa-\bb} \b_{m-1} \cdot \nabla \partial^{\bb} V_{m-1}^{(i)} \cdot 
{\mathbf{Q}}_m \nabla \nabla \partial^{\aa} V_{m-1}^{(i-1)} 
\notag\\
&\qquad 
+ {\bf 1}_{i\geq 1} 
\int_0^1
\int_{\TT^2}
 \Bigl( \nabla \cdot \bigl( \K_m - \kappa_{m-1} \Itwo \bigr) \nabla \partial^\aa V_{m-1}^{(i-1)}\Bigr)
\Bigl(\nabla \cdot {\mathbf{Q}}_m  \nabla \partial^{\aa} V_{m-1}^{(i-1)} \Bigr)
\notag\\
&\qquad 
- {\bf 1}_{i\geq 1} 
\sum_{\bb \leq \aa} \binom{\aa}{\bb}
\int_0^1
\int_{\TT^2}
\nabla \nabla \cdot \bigl( \partial^{\aa-\bb} \mathbf{s}_{m-1} \nabla \partial^{\bb} V_{m-1}^{(i-1)} \bigr)
\cdot  {\mathbf{Q}}_m \nabla \partial^{\aa} V_{m-1}^{(i-1)} 
\,.
\end{align}
Note that the ``gain'' we expect for $i=1$ (see \eqref{e.A.m-1.i.def}) is larger than for $i\geq 2$, and because of that, special care must be devoted to the third term on the right side of \eqref{e.Err.44.rewrite}, when $i=1$. Using \eqref{e.Km.decompose} and \eqref{e.Q.m.def}, and recalling that $V_{m-1}^{(0)} = \theta_{m-1}$, we rewrite
\begin{align}
\label{e.Err.44.rewrite.again}
& 
{\bf 1}_{i=1} 
\int_0^1
\int_{\TT^2}
 \Bigl( \nabla \cdot \bigl( \K_m - \kappa_{m-1} \Itwo \bigr) \nabla \partial^\aa V_{m-1}^{(i-1)}\Bigr)
\Bigl( \nabla \cdot {\mathbf{Q}}_m  \nabla \partial^{\aa} V_{m-1}^{(i-1)} \Bigr)
\notag\\&\qquad
=  
\int_0^1
\int_{\TT^2}
\Bigl( \nabla \cdot \bigl(\langle \hspace{-2.5pt}\langle
\K_m 
\rangle \hspace{-2.5pt} \rangle - \Khom_m\bigr) \nabla \partial^\aa \theta_{m-1}\Bigr)
\Bigl( \nabla \cdot {\mathbf{Q}}_m  \nabla \partial^{\aa} \theta_{m-1}\Bigr)
\notag\\&\qquad\qquad
+ 
\int_0^1
\int_{\TT^2}
\Bigl( \DD_{t,m-1} {\mathbf{Q}}_m \colon \nabla^2 \partial^\aa \theta_{m-1} \Bigr)
\Bigl({\mathbf{Q}}_m  \colon \nabla^2  \partial^{\aa} \theta_{m-1} \Bigr)
\notag\\&\qquad
= 
\int_0^1
\int_{\TT^2}
\Bigl( \nabla \cdot \bigl(\langle \hspace{-2.5pt}\langle
\K_m 
\rangle \hspace{-2.5pt} \rangle - \Khom_m\bigr) \nabla \partial^\aa \theta_{m-1}\Bigr)
\Bigl( \nabla \cdot {\mathbf{Q}}_m  \nabla \partial^{\aa} \theta_{m-1}\Bigr)
\notag\\&\qquad\qquad
-
\int_0^1
\int_{\TT^2}
\Bigl(  {\mathbf{Q}}_m \colon \DD_{t,m-1} \nabla^2 \partial^\aa \theta_{m-1} \Bigr)
\Bigl({\mathbf{Q}}_m  \colon \nabla^2  \partial^{\aa} \theta_{m-1} \Bigr)
\,.
\end{align}
The last term in the above expression may then be rewritten upon noting that $\DD_{t,m-1} \theta_{m-1} = \kappa_{m-1} \Delta \theta_{m-1}$, and therefore
\begin{align*}
\lefteqn{ 
\DD_{t,m-1} 
\partial^{\aa+e_r+e_p}
\theta_{m-1}
} \qquad & 
\notag \\ & 
= 
\kappa_{m-1} \Delta \partial^{\aa+e_r+e_p} \theta_{m-1}
+ \!\!
\sum_{\bb < \aa + e_r + e_p} \binom{\aa + e_r + e_p}{\bb}
\partial^{\aa+e_r+e_p - \bb} \b_{m-1} \cdot \nabla \partial^{\bb}\theta_{m-1}
\,.
\end{align*}
Combining the above identity with \eqref{e.Err.44.rewrite.again} and \eqref{e.Err.44.rewrite}, and appealing to the $\b_{m-1}$ estimate in \eqref{e.bm.Dn}, the ${\mathbf{Q}_m}$ bound in \eqref{e.Q.m.bound}, the $\mathbf{s}_{m-1}$ estimate in \eqref{e.smbound.dervs}, the inductive bound \eqref{l.V.m-1.reg.upgrade} at level $i-1$, and to~\eqref{e.Kbar.almost.average} and \eqref{e.exprat.bound} for $i=1$,
we deduce
\begin{align}
\label{e.Err.4b.bnd}
\bigl| \mathsf{Err}_{4,4} \bigr|
&\leq 
\frac{\kappa_{m-1}}{24}  \bigl\| \nabla \partial^{\aa} V_{m-1}^{(i)} \bigr\|_{L^2((0,1)\times\TT^2)}^2
\notag\\
&\quad 
+ C (C  \kappa_{m-1} \tau_{m}^{\prime\prime})^2
\mathsf{A}_{m-1,i-1}^2
\|\theta_0\|_{L^2(\TT^2)}^2
(|\aa|+2i)!^2 
\bigl(C_0  \ep_{m-1}^{-1 - \nicefrac \gamma 2} \vee C_0  R_{\theta_0}^{-1}\bigr)^{2|\aa|+4}
\notag\\
&\quad 
+ C \sum_{\bb < \aa} \frac{|\aa|!  (|\aa|+2i-1)!}{|\bb|! |\aa-\bb|}  
(C \ep_{m-1}^{-1})^{|\aa-\bb|}
\kappa_{m-1}^{\nicefrac 12}
\bigl\| \nabla \partial^{\bb} V_{m-1}^{(i)} \bigr\|_{L^2((0,1)\times\TT^2)}
\notag\\
&\qquad \qquad \times 
\ep_{m-1}^{\beta-1} \tau_{m}^{\prime \prime} 
\|\theta_0\|_{L^2(\TT^2)}  
\mathsf{A}_{m-1,i-1}
\bigl(C_0 \ep_{m-1}^{-1 - \nicefrac \gamma 2} \vee C_0 R_{\theta_0}^{-1}\bigr)^{|\aa|+1}
\notag\\
&\quad 
+ 
C  {\bf 1}_{i\geq 2} 
\kappa_{m-1} \tau_{m}^{\prime\prime} 
\mathsf{A}_{m-1,i-1}^2
\|\theta_0\|_{L^2(\TT^2)}^2
(|\aa|+2i-1)!^2 \bigl(C_0 \ep_{m-1}^{-1 - \nicefrac \gamma 2} \vee C_0 R_{\theta_0}^{-1}\bigr)^{2|\aa|+2}
\notag\\
&\quad 
+ 
{\bf 1}_{i = 1} 
C  \kappa_{m-1} \tau_{m}^{\prime\prime} 
\bigl(C \ep_{m-1}^{2\delta}\bigr)^{N_*}
\|\theta_0\|_{L^2(\TT^2)}^2
(|\aa|+1)!^2 \bigl(C_0 \ep_{m-1}^{-1 - \nicefrac \gamma 2} \vee C_0 R_{\theta_0}^{-1}\bigr)^{2|\aa|+2}
\notag\\
&\quad 
+ 
{\bf 1}_{i = 1} 
\bigl(C  \kappa_{m-1} \tau_{m}^{\prime\prime} \bigr)^2
\|\theta_0\|_{L^2(\TT^2)}^2
(|\aa|+1)! (|\aa|+3)! \bigl(C_0 \ep_{m-1}^{-1 - \nicefrac \gamma 2} \vee C_0 R_{\theta_0}^{-1}\bigr)^{2|\aa|+4}
\notag\\
&\quad 
+ 
{\bf 1}_{i = 1} 
\bigl(C  \kappa_{m-1} \tau_{m}^{\prime\prime} \bigr)^2
\kappa_{m-1}^{-1}
\|\theta_0\|_{L^2(\TT^2)}^2 
(|\aa|+1)!  \bigl(C_0 \ep_{m-1}^{-1 - \nicefrac \gamma 2} \vee C_0 R_{\theta_0}^{-1}\bigr)^{|\aa|+1}
\notag\\
&\qquad \qquad \times
\sum_{|\bb| \leq |\aa|+1}
\frac{(|\aa|+2)!}{|\aa-\bb|+2}
\bigl(C \ep_{m-1}^{\beta-2}\bigr)
\bigl(C \ep_{m-1}^{-1}\bigr)^{|\aa-\bb|+1}
\bigl(C_0 \ep_{m-1}^{-1-\nicefrac{\gamma}{2}} \vee C_0 R_{\theta_0}^{-1}\bigr)^{|\bb|}
\notag\\
&\quad 
+ 
C {\bf 1}_{i\geq 1} 
\sum_{\bb < \aa} \frac{|\aa|! (|\aa|+2i)!   (|\bb|+2i-2)!}{|\bb|! (|\aa-\bb|+1)^2}
(C \ep_{m-1}^{-1})^{|\aa-\bb|}
\notag\\
&\qquad \qquad \times 
\bigl(\kappa_{m-1} \tau_{m}^{\prime\prime}  \bigr)
\|\theta_0\|_{L^2(\TT^2)}^2
\mathsf{A}_{m-1,i-1}^2
\bigl(C_0 \ep_{m-1}^{-1 - \nicefrac \gamma 2} \vee C_0 R_{\theta_0}^{-1}\bigr)^{|\aa|+|\bb|+2}
\notag\\
&\quad 
+ 
C {\bf 1}_{i\geq 1} 
 (|\aa|+2i)!   (|\aa|+2i-2)!
\notag\\
&\qquad \qquad \times 
\bigl( \kappa_{m-1} \tau_{m}^{\prime\prime} \ep_{m-1}^{2\delta}  \bigr)
\|\theta_0\|_{L^2(\TT^2)}^2
\mathsf{A}_{m-1,i-1}^2
\bigl(C_0 \ep_{m-1}^{-1 - \nicefrac \gamma 2} \vee C_0 R_{\theta_0}^{-1}\bigr)^{2|\aa|+2}
\,.
\end{align}
Returning to \eqref{e.Err.4.def}, we are left to consider the term $\mathsf{Err}_{4,3}$. The difference between this term and $\mathsf{Err}_{4,4}$ is that $\DD_{t,m-1}$ acts on $V_{m-1}^{(i-1)}$ instead of $V_{m-1}^{(i)}$, and as such we need to appeal to the identity \eqref{e.Vm-1.i.material} with $i$ replaced by $i-1$. We do not however need to integrate by parts terms with a derivative count larger than $|\aa|+1$ because they occur only on $V_{m-1}^{(i-1)}$ and $V_{m-1}^{(i-2)}$. As such, similarly to \eqref{e.Err.44.rewrite} we may rewrite 
\begin{align*}
\mathsf{Err}_{4,3}
&=
- \kappa_{m-1}
\int_0^1
\int_{\TT^2}
\nabla  \partial^\aa V_{m-1}^{(i)} \cdot 
{\mathbf{Q}}_m \Delta \nabla \partial^{\aa} V_{m-1}^{(i-1)} 
\notag\\
&\qquad 
+ \sum_{\bb < \aa} \binom{\aa}{\bb} 
\int_0^1
\int_{\TT^2}
\nabla \partial^{\aa} V_{m-1}^{(i)}   \cdot 
{\mathbf{Q}}_m \nabla\bigl(\partial^{\aa-\bb} \b_{m-1} \cdot \nabla \partial^{\bb} V_{m-1}^{(i-1)}  \bigr)
\notag\\
&\qquad 
-  {\bf 1}_{i\geq 2} 
\int_0^1
\int_{\TT^2}
\nabla \partial^{\aa} V_{m-1}^{(i)}
\cdot {\mathbf{Q}}_m   
\nabla \nabla \cdot \bigl( \K_m - \kappa_{m-1} \Itwo \bigr) \nabla \partial^\aa V_{m-1}^{(i-2)} 
\notag\\
&\qquad 
- {\bf 1}_{i\geq 2} 
\sum_{\bb \leq \aa} \binom{\aa}{\bb}
\int_0^1
\int_{\TT^2}
\nabla \partial^{\aa} V_{m-1}^{(i)} 
\cdot  {\mathbf{Q}}_m 
\nabla \nabla \cdot \bigl( \partial^{\aa-\bb} \mathbf{s}_{m-1} \nabla \partial^{\bb} V_{m-1}^{(i-2)} \bigr)
\,,
\end{align*}
and similarly to \eqref{e.Err.4b.bnd} we may bound
\begin{align}
\label{e.Err.4c.bnd}
\bigl| \mathsf{Err}_{4,3} \bigr|
&\leq 
\frac{\kappa_{m-1}}{24}  \bigl\| \nabla \partial^{\aa} V_{m-1}^{(i)} \bigr\|_{L^2((0,1)\times\TT^2)}^2
\notag\\
&\quad 
+ C (C  \kappa_{m-1} \tau_{m}^{\prime\prime})^2
\mathsf{A}_{m-1,i-1}^2 \|\theta_0\|_{L^2(\TT^2)}^2
(|\aa|+2i)!^2 \bigl(C_0  \ep_{m-1}^{-1 - \nicefrac \gamma 2} \vee C_0 R_{\theta_0}^{-1}\bigr)^{2|\aa|+4}
\notag\\
&\quad 
+ C \sum_{\bb < \aa} \frac{|\aa|!^2 (|\bb|+2i-1)!^2}{|\bb|!^2}  
(C \ep_{m-1}^{-1})^{2|\aa-\bb|}
\notag\\
&\qquad \qquad \times 
\bigl(\ep_{m-1}^{\beta-1} \tau_{m}^{\prime \prime}\bigr)^2 
\|\theta_0\|_{L^2(\TT^2)}^2
\mathsf{A}_{m-1,i-1}^2
\bigl(C_0 \ep_{m-1}^{-1 - \nicefrac \gamma 2} \vee C_0 R_{\theta_0}^{-1}\bigr)^{2|\bb|+2}
\notag\\
&\quad 
+
C  {\bf 1}_{i\geq 2} 
(\kappa_{m-1} \tau_{m}^{\prime\prime})^2
\mathsf{A}_{m-1,i-2}^2
\|\theta_0\|_{L^2(\TT^2)}^2
(|\aa|+2i-2)!^2 \bigl(C_0  \ep_{m-1}^{-1 - \nicefrac \gamma 2} \vee C_0  R_{\theta_0}^{-1}\bigr)^{2|\aa|+4}
\notag\\
&\quad 
+ 
C {\bf 1}_{i\geq 2} 
\sum_{\bb < \aa} \frac{|\aa|!^2 (|\bb|+2i-2)!^2}{|\bb|!^2}
(C \ep_{m-1}^{-1})^{2|\aa-\bb|}
 \notag\\
 &\qquad \qquad \times 
 \bigl(\kappa_{m-1} \tau_{m}^{\prime\prime}\bigr)^2
 \|\theta_0\|_{L^2(\TT^2)}^2
\mathsf{A}_{m-1,i-2}^2
 \bigl(C_0 \ep_{m-1}^{-1 - \nicefrac \gamma 2} \vee C_0  R_{\theta_0}^{-1}\bigr)^{2|\bb|+4}
 \notag\\
&\quad 
+ C {\bf 1}_{i\geq 2} 
 (|\aa|+2i-2)!^2  
 \notag\\
 &\qquad \qquad \times 
\bigl(\kappa_{m-1} \tau_{m}^{\prime\prime} \ep_{m-1}^{2\delta}\bigr)^2
 \|\theta_0\|_{L^2(\TT^2)}^2
\mathsf{A}_{m-1,i-2}^2
 \bigl(C_0 \ep_{m-1}^{-1 - \nicefrac \gamma 2} \vee C_0  R_{\theta_0}^{-1}\bigr)^{2|\aa|+4}
\,.
\end{align}
Next, in analogy to \eqref{e.Dn.def} we define 
\begin{align}
\label{e.Dn.i.def}
D_n^{(i)}
&:= 
\frac{\bigl(C_0  \ep_{m-1}^{-1 - \nicefrac \gamma 2} \vee C_0 R_{\theta_0}^{-1}\bigr)^{-n}}{\mathsf{A}_{m-1,i} \|\theta_0\|_{L^2(\TT^2)}}
\notag\\
&\qquad \qquad \times
\max_{|\aa|=n} \left(
\frac{\kappa_{m-1}^{\nicefrac 12} \bigl\| \nabla \partial^{\aa} V_{m-1}^{(i)} \bigr\|_{L^2((0,1)\times\TT^2)}}{(|\aa|+2i)!}
+ 
\frac{\sup_{t \in [0,1]}\bigl\| 
\partial^\aa V_{m-1}^{(i)} (t,\cdot) 
\bigr\|_{L^2(\TT^2)}}{(|\aa|+2i)!}
\right)
\,,
\end{align}
so that proving \eqref{e.Vm-1.reg.upgrade} amounts to showing that $D_n^{(i)} \leq 1$.
To achieve this bound, we combine~\eqref{e.test.partialaaVm-1}, \eqref{e.Err.1.bnd}, \eqref{e.Err.2.bnd}, \eqref{e.Err.3.bnd}, \eqref{e.Err.4a.bnd}, \eqref{e.Err.4a.bnd}, \eqref{e.Err.4c.bnd}, and absorb the appropriate term on the left side of the inequality. 
By also using the parameter inequalities \eqref{e.taubounds}, \eqref{e.kappam.bound}, the fact that $\gamma \geq 4\delta$ (a consequence of \eqref{e.delta}, \eqref{e.gamma}, and $q>1$), and upon denoting\footnote{The inequality $\mathcal{G} \leq \nicefrac 12$ follows from $\ep_{m-1}\leq 1$ and $C_0 \geq 2 C$.}
\begin{align*}
\mathcal{F} 
&:= 
\ep_{m-1}^{1+\nicefrac{\gamma}{2}}  \bigl(C_0  \ep_{m-1}^{-1-\nicefrac{\gamma}{2}} \vee C_0  R_{\theta_0}^{-1}\bigr) 
= C_0 
\bigl(1 \vee \ep_{m-1}^{1+\nicefrac{\gamma}{2}}  R_{\theta_0}^{-1}\bigr) 
\geq C_0 \geq 1
\,,
\\
\mathcal{G} 
&:= 
(C \ep_{m-1}^{-1}) 
\bigl(C_0  \ep_{m-1}^{-1 - \nicefrac \gamma 2} \vee C_0 R_{\theta_0}^{-1}\bigr)^{-1}
=
\bigl(C C_0^{-1}\bigr)
 \bigl(\ep_{m-1}^{\nicefrac \gamma 2} \wedge \ep_{m-1}^{-1} R_{\theta_0}\bigr) 
 \leq \nicefrac 12
 \,,
\end{align*}
for all $n\geq 0$ and $i\geq 1$ we arrive at
\begin{align}
\bigl(D_{n}^{(i)}\bigr)^2  
&\leq
\frac{C}{\mathcal{F}^2}
\bigl(D_{n-1}^{(i)}\bigr)^2  
+
\frac{C}{\mathcal{F}^4}
\sum_{k=0}^{n-2}
2^{-2(n-2-k)}
\bigl(D_{k}^{(i)}\bigr)^2  
+ 
\frac{C \ep_{m-1}^{2\delta}  \mathsf{A}_{m-1,i-1}}{\mathsf{A}_{m-1,i}} 
\sum_{k=0}^{n-1} 
2^{-(n-1-k)}
D_k^{(i)}
\notag \\ & \qquad 
+
\frac{C \mathsf{A}_{m-1,i-1}^2}{\mathsf{A}_{m-1,i}^2}
\Bigl( 
\bigl(C \ep_{m-1}^{2\delta} \bigr)^{\! 2N_*} 
+  \ep_{m-1}^{4\delta}   \mathcal{F}^4 
\Bigr)
+ 
{\bf 1}_{i = 1} 
\left(
\frac{C \ep_{m-1}^{2\delta} \bigl(C \ep_{m-1}^{2\delta}\bigr)^{N_*}}{ \mathsf{A}_{m-1,1}^2}
\mathcal{F}^{2}
+ 
\frac{ C  \ep_{m-1}^{4\delta}}{ \mathsf{A}_{m-1,1}^2}
\mathcal{F}^{4}
\right)
\notag\\ &\qquad 
+ {\bf 1}_{i\geq 2} 
\left(
\frac{C \ep_{m-1}^{4\delta} \mathsf{A}_{m-1,i-2}^2}{ \mathsf{A}_{m-1,i}^2}
\mathcal{F}^{4}
+ 
\frac{C \ep_{m-1}^{2\delta} \mathsf{A}_{m-1,i-1}^2}{\mathsf{A}_{m-1,i}^2}
\mathcal{F}^{2}
\right)
\,.
\label{e.big.ass.display}
\end{align}
We note that upon taking $N_*$ sufficiently large as in~\eqref{e.N}, and using \eqref{l.V.m-1.reg.upgrade.ass}, we may ensure that 
\begin{equation}
\bigl(C \ep_{m-1}^{2\delta} \bigr)^{\! N_*}  \leq \ep_{m-1}^{2\delta},
 \label{e.big.ass.display.app}
\end{equation}
so that in \eqref{e.big.ass.display} we may bound $\bigl(C \ep_{m-1}^{2\delta} \bigr)^{\! 2N_*} \leq \ep_{m-1}^{4\delta}   \mathcal{F}^4 $, and $\ep_{m-1}^{2\delta} \bigl(C \ep_{m-1}^{2\delta}\bigr)^{N_*} \mathcal{F}^{2} 
\leq \ep_{m-1}^{4\delta}  \mathcal{F}^{4}$.

In order to initiate the induction in $n\geq 0$, we first consider estimate \eqref{e.big.ass.display} for $n=0$. Recalling the definition~\eqref{e.A.m-1.i.def}, and the bound $\ep_{m-1}^{2\delta} \mathcal{F}^2 \leq C_0^{-1}$ (which is equivalent to \eqref{l.V.m-1.reg.upgrade.ass}), \eqref{e.big.ass.display} becomes
\begin{align}
\label{e.D0i}
\bigl(D_{0}^{(i)}\bigr)^2  
&\leq  
{\bf 1}_{i = 1} 
\frac{C \ep_{m-1}^{4\delta} \mathcal{F}^{4}}{\mathsf{A}_{m-1,1}^2}
+ {\bf 1}_{i\geq 2} 
\biggl(
\frac{C \mathsf{A}_{m-1,i-1}^2}{\mathsf{A}_{m-1,i}^2}
\Bigl( 
\ep_{m-1}^{4\delta} \mathcal{F}^4
+ \ep_{m-1}^{2\delta} \mathcal{F}^2
\Bigr)
+ 
\frac{C \mathsf{A}_{m-1,i-2}^2}{ \mathsf{A}_{m-1,i}^2}
\ep_{m-1}^{4\delta}  \mathcal{F}^{4}
\biggr)
\notag\\
&\leq  
{\bf 1}_{i \in\{ 1,2\}} 
\frac{C}{C_0^2}
+ {\bf 1}_{i = 2} 
C\ep_{m-1}^{2\delta} \mathcal{F}^2 \bigl(1+ \ep_{m-1}^{2\delta} \mathcal{F}^2\bigr)
+ {\bf 1}_{i\geq 3} 
\frac{C \bigl(1+\ep_{m-1}^{2\delta} \mathcal{F}^2 \bigr)}{C_0}
\leq  
\frac{2 C}{C_0} 
\,.
\end{align}
As long as $C_0$ is taken to be sufficiently large, this established the bound 
$D_{0}^{(i)} \leq 1$. We now inductively assume $D_k^{(i)} \leq 1$ for all $k \in \{0,1,\ldots,n-1\}$, and aim to establish that $D_n^{(i)} \leq 1$; in turn this would conclude the proof of \eqref{e.Vm-1.reg.upgrade}. To do so, we return to \eqref{e.big.ass.display}, use \eqref{e.big.ass.display.app}, the definition~\eqref{e.A.m-1.i.def},  and the inductive bound $D_k^{(i)} \leq 1$ for $k\leq n-1$, to conclude 
\begin{align*}
\bigl(D_{n}^{(i)}\bigr)^2  
&\leq
\frac{C}{\mathcal{F}^2}
+
\frac{2C}{\mathcal{F}^4} 
+ 
\frac{2C \ep_{m-1}^{2\delta}  \mathsf{A}_{m-1,i-1}}{\mathsf{A}_{m-1,i}} 
\notag \\ & \qquad 
+
{\bf 1}_{i = 1} 
\frac{ C  \ep_{m-1}^{4\delta}}{ \mathsf{A}_{m-1,1}^2}
\mathcal{F}^{4}
+ {\bf 1}_{i\geq 2} 
\biggl(
\frac{C \mathsf{A}_{m-1,i-1}^2}{\mathsf{A}_{m-1,i}^2}
\bigl( \ep_{m-1}^{4\delta}  \mathcal{F}^{4} + \ep_{m-1}^{2\delta}  \mathcal{F}^{2}\bigr)
+
\frac{C \mathsf{A}_{m-1,i-2}^2}{ \mathsf{A}_{m-1,i}^2}
\ep_{m-1}^{4\delta}  \mathcal{F}^{4}
\biggr)
\,.
\end{align*}
The second line of the above estimate precisely matches the upper bound in \eqref{e.D0i}, which was shown to be $\leq 2 C C_0^{-1}$ under the standing assumptions. Using that $\mathcal{F} \geq C_0 \geq 1$, and recalling the definition~\eqref{e.Vm-1.reg.upgrade}, we may bound also the first line of the above estimate, and finally deduce
\begin{align*}
\bigl(D_{n}^{(i)}\bigr)^2  
&\leq 
\frac{5C}{C_0^2}
+ 
{\bf 1}_{i=1}
\frac{2C}{C_0^3} 
+ 
{\bf 1}_{i=2}
2C \ep_{m-1}^{2\delta}
+
{\bf 1}_{i\geq 3}
\frac{2C \ep_{m-1}^{\delta}}{C_0^{\nicefrac 32}} 
\leq 
\frac{5C}{C_0^2}
+ 
{\bf 1}_{i \geq 1}
\frac{4C}{C_0^3} 
\,.
\end{align*}
Upon choosing $C_0$ sufficiently large with respect to $C$, we establish the bound necessary for the inductive step $D_{n}^{(i)}\leq 1$, and thus conclude the proof of the Lemma.
\end{proof}

Direct consequences of the bounds in Lemma~\ref{l.V.m-1.reg.upgrade} and of the definition \eqref{e.T.minus.theta.m-1} are the following regularity estimates for $T_{m-1}$. 
\begin{lemma}[{Estimates on~$T_{m-1}$}]
\label{l.Tm.reg.upgrade}
Let $C :=4 C_0^3$, where $C_0\geq 1$ is the  universal constant from Lemma~\ref{l.V.m-1.reg.upgrade}. If $\ep_{m-1}$ is small enough to ensure  
\begin{equation}
\label{e.ep.m-1.small}
\ep_{m-1}^{1+ \nicefrac{\gamma}{2}} \leq R_{\theta_0} 
\qquad \mbox{and} \qquad 
\ep_{m-1}^{2\delta} \leq C^{-1}
\,,
\end{equation}
then, for every $n\in\N_0$ we have
\begin{align}
\label{e.Tm.reg.upgrade}
&
\bigl\|  \nabla^n T_{m-1} \bigr\|_{L^\infty([0,1];L^2(\TT^2))}
+
\kappa_{m-1}^{\nicefrac 12} 
\bigl\| \nabla^{n+1} T_{m-1} \bigr\|_{L^2((0,1)\times\TT^2)}
\leq 
C_{{N_*}} \|\theta_0\|_{L^2(\TT^2)}
n! \bigl(C \ep_{m-1}^{-1 - \nicefrac \gamma 2} \bigr)^n
\,,
\end{align}
where we have defined $C_{{N_*}} := 2^{2{N_*}} (2{N_*})!$.
\end{lemma}
\begin{proof}
Assumption \eqref{e.ep.m-1.small} gives that 
\begin{equation}
\label{e.T.m-1.temp.b}
C_0^3 \ep_{m-1}^{2\delta} ( 1 \vee \ep_{m-1}^{2+\gamma} R_{\theta_0}^{-2}) 
= C_0^3 \ep_{m-1}^{2\delta} \leq C_0^3   C^{-1} = \nicefrac 14
\,.
\end{equation} 
Thus assumption~\eqref{l.V.m-1.reg.upgrade.ass} holds, and we are allowed to apply Lemma~\ref{l.V.m-1.reg.upgrade}. 
For compactness of notation, denote the left side of \eqref{e.Tm.reg.upgrade} as 
\begin{equation*}
\mathsf{F}_n := 
\max_{|\aa|=n} \left( 
\bigl\|  \partial^\aa T_{m-1} 
\bigr\|_{L^\infty([0,1];L^2(\TT^2))}
+
\kappa_{m-1}^{\nicefrac 12} 
\bigl\| \nabla \partial^{\aa} T_{m-1} 
\bigr\|_{L^2((0,1)\times\TT^2)}
\right)
\,.
\end{equation*}
From \eqref{e.T.minus.theta.m-1} and the identification $\theta_{m-1} = V_{m-1}^{(0)}$, we have that 
\begin{equation*}
T_{m-1} =  \sum_{i=0}^{{N_*}} V_{m-1}^{(i)}
\,.
\end{equation*}
Hence, the bound \eqref{e.Vm-1.reg.upgrade}, the definition \eqref{e.A.m-1.i.def}, and assumption~\eqref{l.V.m-1.reg.upgrade.ass} imply that
\begin{align}
\label{e.T.m-1.temp.a}
\frac{\mathsf{F}_n}{n! \bigl(C_0 \ep_{m-1}^{-1 - \nicefrac \gamma 2} \vee C_0 R_{\theta_0}^{-1}\bigr)^{n}} 
&\leq
\|\theta_0\|_{L^2(\TT^2)}
\sum_{i=0}^{{N_*}}
\mathsf{A}_{m-1,i}  \frac{(n+2i)!}{n!}
\notag\\
&\leq 
(2{N_*})! \|\theta_0\|_{L^2(\TT^2)}
\sum_{i=0}^{{N_*}}
   \binom{n+2i}{n} \frac{(2i)!}{(2{N_*})!}
\notag\\
&\leq
2^{n+2{N_*}} (2{N_*})!  \|\theta_0\|_{L^2(\TT^2)}  
\,.
\end{align}
With assumption~\eqref{e.ep.m-1.small} and the definition $C=4 C_0^3$, the proof is completed. 
\end{proof}

In the above proof we have merely used that for all $i\geq 0$ the amplitude coefficients $\mathsf{A}_{m-1,i}$ appearing in~\eqref{e.A.m-1.i.def} satisfy the bound $\mathsf{A}_{m-1,i} \leq 1$. Now, we use the precise structure of these coefficients to deduce two further consequences.

\begin{lemma}[{$T_{m-1}$ and $\theta_{m-1}$ are close and $T_{m-1}$ almost solves \eqref{e.Tm.divform}}]
\label{l.Tm.minus.thetam}
Under the assumptions of Lemma~\ref{l.Tm.reg.upgrade}, we have that
\begin{multline}
\label{e.Tm.thetam}
\left\| 
T_{m-1} -  \theta_{m-1}
\right\|_{L^\infty((0,1); L^2( \TT^2))} 
+
\kappa_{m-1}^{\nicefrac12}
\left\| 
\nabla T_{m-1} - \nabla \theta_{m-1}
\right\|_{L^2((0,1)\times \TT^2)} 
\\
\leq
(2{N_*})! \bigl(C_{\eqref{e.Tm.reg.upgrade}} \ep_{m-1}^{2\delta}\bigr) \|\theta_0\|_{L^2(\TT^2)} 
\,,
\end{multline}
and the error term appearing on the right side of \eqref{e.Tm.true} satisfies
\begin{align}
\label{e.Em-1.thetam}
\left\|\nabla^{n} \mathbf{e}_{m-1} \right\|_{L^2((0,1)\times \TT^2)}
\leq
C C_{{N_*}} \kappa_{m-1}^{\nicefrac 12}
n! 
\bigl(C_{\eqref{e.Tm.reg.upgrade}} \ep_{m-1}^{-1-\nicefrac{\gamma}{2}}\bigr)^{n}
\bigl( C_{\eqref{e.Tm.reg.upgrade}} \ep_{m-1}^{2\delta} \bigr)^{\nicefrac{{N_*}}{2}} 
\|\theta_0\|_{L^2(\TT^2)}
\,,
\end{align}
for $n \in \NN_0$  and a universal constant $C\geq 1$.
\end{lemma}

\begin{proof}
In order to prove \eqref{e.Tm.thetam}, we recall from \eqref{e.T.minus.theta.m-1} that $T_{m-1}-\theta_{m-1} =\sum_{i=1}^{{N_*}} V_{m-1}^{(i)}$. Therefore, similarly to \eqref{e.T.m-1.temp.a} with $n=0$, we deduce from the bound \eqref{e.Vm-1.reg.upgrade}, the definition \eqref{e.A.m-1.i.def}, and assumption~\eqref{e.ep.m-1.small} (which implies~\eqref{l.V.m-1.reg.upgrade.ass} and also~\eqref{e.T.m-1.temp.b}), that
\begin{align*}
&
\left\| 
T_{m-1} -  \theta_{m-1}
\right\|_{L^\infty((0,1); L^2( \TT^2))} 
+
\kappa_{m-1}^{\nicefrac12}
\left\| 
\nabla T_{m-1} - \nabla \theta_{m-1}
\right\|_{L^2((0,1)\times \TT^2)} 
\notag\\
&\qquad
\leq
\|\theta_0\|_{L^2(\TT^2)}
\sum_{i=1}^{{N_*}}
(2i)! \mathsf{A}_{m-1,i}  
\notag\\
&\qquad
\leq 
C_0^3 \ep_{m-1}^{2\delta} \|\theta_0\|_{L^2(\TT^2)}   
\sum_{i=1}^{{N_*}} (2i)! \bigl(C_0^3 \ep_{m-1}^{2\delta}\bigr)^{\frac{(i-2)_+}{2}}
\notag\\
&\qquad
\leq 
C_{\eqref{e.Tm.reg.upgrade}} \ep_{m-1}^{2\delta} \|\theta_0\|_{L^2(\TT^2)}   
\sum_{i=1}^{{N_*}} (2i)! 2^{-(i-2)_+ - 2}
\notag\\
&\qquad
\leq
(2{N_*})! C_{\eqref{e.Tm.reg.upgrade}}  \ep_{m-1}^{2\delta} \|\theta_0\|_{L^2(\TT^2)}   
\,. 
\end{align*}
This gives \eqref{e.Tm.thetam}.

In view of the definition of $\mathbf{e}_{m-1}$ in \eqref{e.E.m-1.def}, proving \eqref{e.Em-1.thetam} for $n\in \{0,1,2\}$ amounts to combining the $V_{m-1}^{({N_*})}$ estimate from \eqref{e.Vm-1.reg.upgrade}, the $\mathbf{s}_{m-1}$ bound in \eqref{e.smbound.dervs}, and the estimate  $|  \K_m - \kappa_{m-1} | \leq C \kappa_{m-1}$ (which follows from~\eqref{e.bfK.dervs}):
\begin{align*}
&
\left\|\partial^{\aa} \mathbf{e}_{m-1} \right\|_{L^2((0,1)\times \TT^2)}
\notag\\ 
&\quad 
\leq 
C \kappa_{m-1} 
\bigl\| \partial^\aa \nabla V_{m-1}^{({N_*})} \bigr\|_{L^2((0,1)\times \TT^2)}
+
C \kappa_{m-1} 
\sum_{\bb \leq \aa} \frac{|\aa|!}{|\bb|!} 
\bigl(C \ep_{m-1}^{-1} \bigr)^{|\aa-\bb|}
\bigl\| \partial^\bb \nabla V_{m-1}^{({N_*})} \bigr\|_{L^2((0,1)\times \TT^2)}
\notag\\
&\quad 
\leq 
C \kappa_{m-1}^{\nicefrac 12} 
(2 {N_*}+|\aa|)! 
\mathsf{A}_{m-1,{N_*}} 
\|\theta_0\|_{L^2(\TT^2)}
\bigl(C_0 \ep_{m-1}^{-1-\nicefrac{\gamma}{2}}\bigr)^{|\aa|}
\notag\\
&\quad \leq 
C \kappa_{m-1}^{\nicefrac 12} 
(2 {N_*}+|\aa|)! 
\bigl(C_{\eqref{e.Tm.reg.upgrade}} \ep_{m-1}^{2\delta} \bigr)^{\nicefrac{{N_*}}{2}}
\|\theta_0\|_{L^2(\TT^2)}
\bigl(C_0 \ep_{m-1}^{-1-\nicefrac{\gamma}{2}}\bigr)^{|\aa|}
\,.
\end{align*}
Since $(2 {N_*}+|\aa|)! \leq 2^{|\aa|+2{N_*}} (2{N_*})! |\aa|! = C_{{N_*}} 2^{|\aa|} |\aa|!$ and $2 C_0 \leq C_{\eqref{e.Tm.reg.upgrade}}$, this concludes the proof of \eqref{e.Em-1.thetam} and thus of the lemma. 
\end{proof}

The estimates on all space derivatives of $\nabla T_{m-1}$ obtained in Lemma~\ref{l.Tm.reg.upgrade} imply, when combined with the available bounds on $\b_{m-1}$ a control on mixed space-and-material derivatives of $\nabla T_{m-1}$.

\begin{lemma}
\label{l.T.m-1.material.mixed}
Under the assumptions of Lemma~\ref{l.Tm.reg.upgrade}, for all $n,\ell \in \NN_0$ with $\ell \geq 1$ and $n+ 2 \ell \leq  N_*$, we have 
\begin{align} 
\norm{\nabla^n \DD_{t,m-1}^\ell \nabla T_{m-1}}_{L^2([0,1] \times \TT^2)}
&\leq 
C \ep_{m-1}^{3\delta}  \kappa_{m-1}^{-\nicefrac12} 
\|\theta_0\|_{L^2(\TT^2)}  \ep_{m-1}^{-(1 + \nicefrac \gamma 2 ) n}   \bigl(\tau_m^\prime \bigr)^{-\ell} 
\label{e.barf.cascade}
\end{align}
for a sufficiently large constant $C = C(N_*) \geq 1$.
\end{lemma}
 
\begin{proof} 
In \eqref{e.barf.cascade} we only consider $\ell \geq 1$  because for $\ell = 0$ a different bound is already available in \eqref{e.Tm.reg.upgrade}. 
In order to prove \eqref{e.barf.cascade}, we recall from~\eqref{e.T.m-1.0}--\eqref{e.T.m-1.Nstar} that $T_{m-1} = T_{m-1}^{{N_*}}$, where 
the functions $\{ \nabla T_{m-1}^{(i)} \}_{i=0}^{{N_*}}$ solve
\begin{align}
\label{eq:Tm-1:i:evo}
\DD_{t,m-1} \nabla T_{m-1}^{(i)} 
&= 
\kappa_{m-1} \Delta \nabla T_{m-1}^{(i)} 
-
\nabla \b_{m-1} \cdot \nabla T_{m-1}^{(i)}
\notag\\
&\qquad 
+
{\bf 1}_{\{i\geq 1\}}
\nabla \nabla \cdot \bigl( \K_m - \kappa_{m-1} \Itwo + \mathbf{s}_{m-1} \bigr) \nabla T_{m-1}^{(i-1)}
\,.
\end{align}
We claim that for all $0\leq i \leq {N_*}$, $n \in \NN_0, \ell \in \NN$ with $n+2\ell \leq N_*$ we have the bound
\begin{align}
\kappa_{m-1}^{ \nicefrac12}  
\max_{|\aa|=n}  
\norm{\partial^\aa \DD_{t,m-1}^\ell \nabla T_{m-1}^{(i)}}_{L^2([0,1] \times \TT^2)}
&\leq 
C  \ep_{m-1}^{3\delta}  
\|\theta_0\|_{L^2(\TT^2)} \bigl(C \ep_{m-1}^{-1 - \nicefrac \gamma 2} \bigr)^{n}   \bigl(\tau_m^\prime \bigr)^{-\ell} 
\,.
\label{e.barf.cascade.0}
\end{align}
Specializing \eqref{e.barf.cascade.0} to the case $i= {N_*}$ gives \eqref{e.barf.cascade}, upon noting that the factor of $C^n\leq C^{N_*}$ may be absorbed in the constant $C_{\eqref{e.barf.cascade}}$.

We first prove \eqref{e.barf.cascade.0} for $\ell=1$, as this contains the main idea. The generalization to $\ell \geq 2$ is a matter of accounting, and the upper bound obtained is allowed to have a large amplitude, by a factor of $\ep_{m-1}^{-3\delta}$.
Since $T_{m-1}^{(i)} = \sum_{i'=0}^{i} V_{m-1}^{(i')}$, by using that $A_{m-1,i'} \leq 1$ for all $i'\geq 0$, we deduce similarly to \eqref{e.Tm.reg.upgrade} that for all $0 \leq i \leq {N_*}$ and $n\in \mathbb{N}_0$, we have
\begin{align}
\label{e.Tm.i.reg.upgrade}
&
\kappa_{m-1}^{\nicefrac 12}
\max_{|\aa|=n}  
 \bigl\|  \partial^{\aa} \nabla T_{m-1}^{(i)} \bigr\|_{L^2((0,1)\times\TT^2)}
\leq 
C \|\theta_0\|_{L^2(\TT^2)}
n! \bigl(C \ep_{m-1}^{-1 - \nicefrac \gamma 2} \bigr)^n
\,.
\end{align}
From \eqref{e.bm.Dn}, \eqref{e.bfK.dervs}, \eqref{e.enhance.onestep},  \eqref{e.smbound.dervs}, \eqref{e.ep.m-1.small}, \eqref{e.Tm.i.reg.upgrade},  and the Leibniz rule, it follows that for all $0\leq i \leq {N_*}$, we have
\begin{align*}
&
\|\theta_0\|_{L^2(\TT^2)}^{-1} 
\kappa_{m-1}^{\nicefrac 12} 
\max_{|\aa|=n}   
\norm{\partial^\aa \DD_{t,m-1} \nabla T_{m-1}^{(i)}}_{L^2([0,1] \times \TT^2)}
\notag\\ &\qquad 
\leq C \kappa_{m-1} 
\bigl(C \ep_{m-1}^{-1 - \nicefrac \gamma 2} \bigr)^{n+2}
(n+2)!  
+
\bigl(C \ep_{m-1}^{\beta-2}\bigr) n!
\sum_{k=0}^{n}
 \bigl(C \ep_{m-1}^{-1}\bigr)^{n-k}
 \bigl(C \ep_{m-1}^{-1 - \nicefrac \gamma 2} \bigr)^{k}
\notag\\ &\qquad \qquad
+  
C \kappa_{m-1}  
\bigl(C  \ep_{m-1}^{-1 - \nicefrac \gamma 2} \bigr)^2 (n+2)!
\sum_{k=0}^{n+2}
\bigl(C \ep_{m-1}^{-1}\bigr)^{n-k+2}
\bigl(C  \ep_{m-1}^{-1 - \nicefrac \gamma 2} \bigr)^{k-2}
\notag\\ &\qquad 
\leq C  (n+2)!  \ep_{m-1}^{\beta-2}  \bigl(C \ep_{m-1}^{-1 - \nicefrac \gamma 2} \bigr)^{n}
\,.
\end{align*}
Since $\ep_{m-1}^{\beta- 2} \leq  \ep_{m-1}^{3\delta} (\tau_m^\prime)^{-1}$ (see~\eqref{e.taum.def}--\eqref{e.taum.prime.def}), the above estimate gives the proof of \eqref{e.barf.cascade.0} for $\ell = 1$.

From~\eqref{eq:Tm-1:i:evo} it is clear that proving~\eqref{e.barf.cascade.0} for $\ell \geq 2$, requires a bound for the space-and-material derivatives of $\mathbf{s}_{m-1}$ (which we recall was defined in \eqref{e.sm}). For this purpose, for all $n,\ell \in \NN_0$ with $n+\ell \leq N_*$ we claim that
\begin{align}
\norm{\nabla^n \DD_{t,m-1}^\ell  \mathbf{s}_{m-1}}_{L^\infty(\R\times \R^2)}
\leq 
C \kappa_{m-1}  \ep_{m-1}^{-n} (\tau_m^\prime)^{-\ell}
\,.
\label{e.sm.material}
\end{align}
As usual in such terms, we do not keep track of factorials because the constant $C$ in \eqref{e.sm.material} depends (only) on $N_*$. 
When $\ell=0$, the bound \eqref{e.sm.material} follows from \eqref{e.smbound.dervs}. For $\ell\geq 1$, we use the Leibniz rule, \eqref{e.taum.prime.def}, \eqref{e.taum.primeprime.def}, \eqref{e.taubounds}, \eqref{e.zeta.prime.ml.fitting}, \eqref{e.zeta.prime.ml.bounds},  \eqref{e.Xm.bound.1}, \eqref{e.spacetime.X.bnds},  and \eqref{e.bfK.dervs},
\begin{align*}
\norm{\nabla^n \DD_{t,m-1}^\ell  \mathbf{s}_{m-1}}_{L^\infty(\R\times \R^2)}
 &\leq C  \sum_{\ell'=0}^\ell \sum_{\ell''=0}^{\ell'} \sum_{l\in\Z}
 \norm{\partial_t^{\ell-\ell' } \hat{\xi}_{m,l}}_{L^\infty(\R)} 
 \norm{\partial_t^{\ell'-\ell''} \K_m}_{L^\infty(\R)}
 \notag\\
 &\qquad \qquad \times
 \norm{\nabla^n \DD_{t,m-1}^{\ell''} \bigl( \nabla X_{m-1,l} \circ X_{m-1,l}^{-1} - \Itwo\bigr)}_{L^\infty(\supp \hat{\xi}_{m,l} \times \R^2)}
 \notag\\
 &\leq  C \kappa_{m-1} \sum_{\ell''=0}^\ell 
 (\tau_m^\prime)^{\ell''-\ell} 
 \ep_{m-1}^{-n} (\ep_{m-1}^{\beta-2})^{\ell''}
 \notag\\
  &\leq C \kappa_{m-1}  \ep_{m-1}^{-n} (\tau_m^\prime)^{-\ell}
  \,.
\end{align*}
In the last inequality we have used that $\tau_m^\prime \ep_{m-1}^{\beta-2} \leq C \ep_{m-1}^{3\delta} \leq 1$. This concludes the proof of~\eqref{e.sm.material}.

With~\eqref{e.sm.material} in hand, we return to proving~\eqref{e.barf.cascade.0} for $\ell \geq 2$.  In view of \eqref{eq:Tm-1:i:evo}, in order to estimate higher order material derivatives of $\nabla T_{m-1}^{(i)}$, we need to understand the commutator between $\DD_{t,m}^{\ell-1}$ and $\nabla^2$. In this direction, from~\cite[Lemma A.12]{BMNV} we recall that 
\begin{align}
 \big[\DD_{t,m-1}^{\ell-1},\nabla^2\bigr]  f
  = \sum_{\ell'=1}^{\ell-1}
  \sum_{\ell'' = 0}^{\ell'}
  c_{\ell, \ell',\ell''}
  \bigl({\rm ad} \DD_{t,m-1}\bigr)^{\ell''}(\nabla)   \bigl({\rm ad} \DD_{t,m-1}\bigr)^{\ell'-\ell''}(\nabla) \DD_{t,m-1}^{\ell-1-\ell'} f
 \label{e.barf.cascade.1}
\end{align}
where $  c_{\ell, \ell',\ell''}>0$ are explicitly computable combinatorial coefficients, and we recall from \eqref{e.vomit.cascade.2} that $\bigl({\rm ad} \DD_{t,m-1}\bigr)^{r}(\nabla)$ is a first order differential operator for any $r\geq 0$. With \eqref{e.barf.cascade.1},  we return to \eqref{eq:Tm-1:i:evo} and obtain (ignoring the precise contraction of tensors) that 
\begin{align}
\DD_{t,m-1}^\ell \nabla T_{m-1}^{(i)}
&= 
\kappa_{m-1}
\sum_{\ell'=0}^{\ell-1}    
\sum_{\ell'' = 0}^{\ell'}
c_{\ell, \ell',\ell''}
\bigl({\rm ad} \DD_{t,m-1}\bigr)^{\ell''}(\nabla)   
\bigl({\rm ad} \DD_{t,m-1}\bigr)^{\ell'-\ell''}(\nabla) 
\DD_{t,m-1}^{\ell-1-\ell'} \nabla T_{m-1}^{(i)}
\notag\\
&\qquad 
- 
\sum_{\ell'=0}^{\ell-1} 
\binom{\ell-1}{\ell'}
\DD_{t,m-1}^{\ell'}  \nabla \b_{m-1} \cdot  
\DD_{t,m-1}^{\ell-1-\ell'} \nabla T_{m-1}^{(i)}
\notag\\
&\qquad 
+ {\bf 1}_{\{i\geq 1\}} \sum_{\ell'=0}^{\ell-1}    
\sum_{\ell'' = 0}^{\ell'}
c_{\ell, \ell',\ell''}
\bigl({\rm ad} \DD_{t,m-1}\bigr)^{\ell''}(\nabla)   
\bigl({\rm ad} \DD_{t,m-1}\bigr)^{\ell'-\ell''}(\nabla) 
\notag\\
&\qquad \qquad \qquad \qquad \qquad \times
\DD_{t,m-1}^{\ell-1-\ell'} 
\bigl( \K_m - \kappa_{m-1} \Itwo + \mathbf{s}_{m-1} \bigr) \nabla T_{m-1}^{(i-1)}
 \,.
\label{e.barf.cascade.3}
\end{align}
We note that the number of $\DD_{t,m-1}$ material derivatives acting on $\nabla T_{m-1}^{(i)}$ and $\nabla T_{m-1}^{(i-1)}$ on the right side of \eqref{e.barf.cascade.3} is at most $\ell-1$, whereas on the left side of \eqref{e.barf.cascade.3} we have $\ell$-many $\DD_{t,m-1}$ material derivatives acting on acting on $\nabla T_{m-1}^{(i)}$. As such the bound \eqref{e.barf.cascade.0} is established inductively in $\ell \geq 1$, with the base step $\ell=1$ being already proven.

In order to bound the terms on the right side of \eqref{e.barf.cascade.3}, we note that \eqref{e.monofractal.vomit} implies that
\begin{align}
&\norm{\nabla^n \bigl({\rm ad} \DD_{t,m-1}\bigr)^{r}(\nabla) f}_{L^p(\RR\times \RR^2)}
\notag\\
&\qquad \leq C \bigl( \ep_{m-1}^{\beta-2}\bigr)^r \norm{\nabla^{n+1} f}_{L^p(\RR\times \RR^2)}
+ C \bigl( \ep_{m-1}^{-1}\bigr)^n \bigl( \ep_{m-1}^{\beta-2}\bigr)^r \norm{\nabla f}_{L^p(\RR\times \RR^2)}
\label{e.barf.cascade.2}
\end{align}
holds for all $n+r \leq N_*$ and $f\in W^{n+1,p}$. By combining \eqref{e.barf.cascade.2} and \eqref{e.barf.cascade.3}, for $ n+2\ell \leq N_*$ with $\ell \geq 2$, we obtain that 
\begin{align*}
&\norm{\nabla^n \DD_{t,m-1}^\ell \nabla T_{m-1}^{(i)}}_{L^2([0,1]\times \R^2)}
\notag\\
&\leq  C \kappa_{m-1} 
\sum_{\ell'=0}^{\ell-1}  
\ep_{m-1}^{(\beta-2)\ell'} 
\Bigl(
\norm{\nabla^{n+2} \DD_{t,m-1}^{\ell-1-\ell'}  \nabla T_{m-1}^{(i)}}_{L^2((0,1)\times\TT^2)}
+ 
\ep_{m-1}^{-(n+1)} 
\norm{\nabla  \DD_{t,m-1}^{\ell-1-\ell'}  \nabla T_{m-1}^{(i)}}_{L^2((0,1)\times\TT^2)}
\Bigr)
\notag\\
&\quad 
+ C \sum_{n'=0}^{n} \sum_{\ell'=0}^{\ell-1} 
\norm{\nabla^{n'} \DD_{t,m-1}^{\ell'}  \nabla \b_{m-1} }_{L^\infty((0,1)\times\TT^2)}
\norm{\nabla^{n-n'} \DD_{t,m-1}^{\ell-1-\ell'} \nabla T_{m-1}^{(i)}}_{L^2((0,1)\times\TT^2)}
\notag\\
&\quad 
+ C {\bf 1}_{\{i\geq 1\}}
\sum_{\ell'=0}^{\ell-1} \sum_{\ell''=0}^{\ell-1-\ell'}
\ep_{m-1}^{(\beta-2)\ell'} 
\notag\\
&\qquad \qquad \qquad
\times \biggl(
\norm{\nabla^{n+2} \Bigl( \DD_{t,m-1}^{\ell-1-\ell' - \ell''}  \bigl( \K_m - \kappa_{m-1} \Itwo + \mathbf{s}_{m-1} \bigr) \DD_{t,m-1}^{\ell''}  \nabla T_{m-1}^{(i-1)}\Bigr)}_{L^2((0,1)\times\TT^2)}
\notag\\
&\qquad \qquad \qquad \qquad  
+ 
\ep_{m-1}^{-(n+1)} 
\norm{\nabla \Bigl(  \DD_{t,m-1}^{\ell-1-\ell'-\ell''}  \bigl( \K_m - \kappa_{m-1} \Itwo + \mathbf{s}_{m-1} \bigr) \DD_{t,m-1}^{\ell''}  \nabla T_{m-1}^{(i-1)}\Bigr)}_{L^2((0,1)\times\TT^2)}
\biggr)
\,.
\end{align*}
Recalling the $\nabla \b_{m-1}$ bound in~\eqref{e.monofractal.vomit}, the $\mathbf{s}_{m-1}$ estimate~\eqref{e.sm.material},  the $\K_m$ bound in~\eqref{e.bfK.dervs},   the $\nabla T_{m-1}^{(i)}$ bound with no material derivatives~\eqref{e.Tm.i.reg.upgrade}, the inductive bound \eqref{e.barf.cascade.0} for $\DD_{t,m-1}^{\ell'}\nabla T_{m-1}^{(i)}$ with $1\leq \ell' \leq \ell-1$, and the parameter inequality $\ep_{m-1}^{\beta - 2}\tau_{m}^{\prime} = \ep_{m-1}^{3\delta} \ll 1$, we obtain from the above estimate that 
\begin{align}
&\|\theta_0\|_{L^2(\TT^2)}^{-1}  \kappa_{m-1}^{\nicefrac 12}  
\norm{\nabla^n \DD_{t,m-1}^\ell \nabla T_{m-1}^{(i)}}_{L^2([0,1]\times \R^2)}
\notag\\
&\qquad \leq  
C \kappa_{m-1}    
\sum_{\ell'=0}^{\ell-1}  
\ep_{m-1}^{(\beta-2)\ell'} 
\Bigl( 
(\ep_{m-1}^{-1-\nicefrac{\gamma}{2}})^{n+2} (\tau_m^\prime)^{-(\ell-1-\ell')}
+ 
\ep_{m-1}^{-(n+1)} 
(\ep_{m-1}^{-1-\nicefrac{\gamma}{2}})  (\tau_m^\prime)^{-(\ell-1-\ell')}
\Bigr)
\notag\\
&\qquad \qquad 
+ C  \ep_{m-1}^{\beta-2}
\sum_{n'=0}^{n} \sum_{\ell'=0}^{\ell-1} 
(\ep_{m-1}^{\beta-2})^{\ell'} 
(\ep_{m-1}^{-1})^{n'}
(\ep_{m-1}^{-1-\nicefrac{\gamma}{2}})^{n-n'}
(\tau_m^\prime)^{-(\ell-1-\ell')}
\notag\\
&\qquad \qquad 
+ C {\bf 1}_{\{i\geq 1\}} \kappa_{m-1} 
\sum_{\ell'=0}^{\ell-1} \sum_{\ell''=1}^{\ell-1-\ell'} \sum_{n'=0}^{n+2}
\ep_{m-1}^{(\beta-2)\ell'} 
(\ep_{m-1}^{-1})^{n+2-n'} (\tau_m^\prime)^{-(\ell-1-\ell' - \ell'')}
(\ep_{m-1}^{-1-\nicefrac{\gamma}{2}})^{n'} (\tau_m^\prime)^{-\ell''}
\notag\\
&\qquad \leq  
C \Bigl( \kappa_{m-1} \ep_{m-1}^{-2-\gamma} + \ep_{m-1}^{\beta-2}  \Bigr)
(\ep_{m-1}^{-1-\nicefrac{\gamma}{2}})^{n}
\sum_{\ell'=0}^{\ell-1}  
\ep_{m-1}^{(\beta-2)\ell'}  (\tau_m^\prime)^{-(\ell-1-\ell')}
\notag\\
&\qquad \leq 
C  \ep_{m-1}^{\beta-2} 
(\ep_{m-1}^{-1-\nicefrac{\gamma}{2}})^{n}(\tau_m^\prime)^{-(\ell-1)}
\notag\\
&\qquad = 
C  \ep_{m-1}^{3\delta} 
(\ep_{m-1}^{-1-\nicefrac{\gamma}{2}})^{n}(\tau_m^\prime)^{-\ell}
\label{e.barf.cascade.4}
\end{align}
for all  $n+2\ell \leq N_*$. Note that the bound on the term  $\nabla^{n+2} \DD_{t,m-1}^{\ell-1-\ell'} \mathbf{s}_{m-1}$, cf.~\eqref{e.sm.material}, requires that $n+2 + \ell-1-\ell' \leq N_*$; this condition holds because $n+2+\ell-1-\ell' \leq n+1+\ell \leq n+2\ell$ for $\ell\geq 1$. 
By induction on $\ell$, this concludes the proof of \eqref{e.barf.cascade.0}, and thus of the Lemma.
\end{proof}

\subsection{{Estimates for~\texorpdfstring{$\widetilde{H}_m$}{tilde H m}}}
\label{ss.Hm.estimates}

Before estimating the function $\widetilde{H}_m$ defined in \eqref{e.Hm.def}, we need to obtain estimates for space-and-material derivatives of the tensors $\mathbf{A}_{m,n,r}$ defined in \eqref{e.A.mnr.def}. 
\begin{proposition}
\label{p.Amnr.bounds}
Under the assumptions of Lemma~\ref{l.Tm.reg.upgrade}, there exists a constant $C\geq 1$ such that for all $0 \leq r \leq \nicefrac{N_*}{2}$, $m\in \N$, and $0\leq n \leq N_*-1$ we have 
\begin{equation}
\label{e.Amnr.bounds}
\bigl\|
\nabla^k \DD_{t,m-1}^\ell \mathbf{A}_{m,n,r}
\bigr\|_{L^2([0,1]\times \TT^2)}
\leq   C  \|\theta_0\|_{L^2(\TT^2)}
\bigl(\ep_m^2 \kappa_m^{-1}\bigr)
\kappa_{m-1}^{-\nicefrac 12}
\ep_{m-1}^{-(1+\nicefrac\gamma2) k } \bigl(\tau_m^\prime \bigr)^{-\ell-r}
\end{equation}
for all $k + 2\ell \leq N_* - 2r$.
\end{proposition}
\begin{proof}
We appeal to the space-and-material bounds for $\nabla \b_{m-1}$ available from \eqref{e.material.goal}, the space-and-material estimates for $\nabla X_{m-1,\ell} \circ X_{m-1,\ell}^{-1} $ in \eqref{e.spacetime.X.bnds}, the space-and-material bounds for $\nabla T_{m-1}$ in \eqref{e.Tm.reg.upgrade} and \eqref{e.barf.cascade},   the time derivative bounds for $L_{m,n}^{\kappa_m}$ in \eqref{e.Lmn.bound}, and the time derivative bounds for $\hat{\xi}_{m,l}$ in \eqref{e.zeta.prime.ml.bounds}. Using these bounds, the product rule, and the definition of $\mathbf{A}_{m,n,0}$ in \eqref{e.A.mnr.def}, we deduce
\begin{equation}
\bigl\|
\nabla^k \DD_{t,m-1}^\ell \mathbf{A}_{m,n,0}
\bigr\|_{L^2([0,1]\times \TT^2)}
\leq C  \|\theta_0\|_{L^2(\TT^2)}
\bigl(\ep_m^2 \kappa_m^{-1}\bigr)
\kappa_{m-1}^{-\nicefrac 12}
\ep_{m-1}^{-(1+\nicefrac\gamma2) k } \bigl(\tau_m^\prime \bigr)^{-\ell}
\end{equation}
for all $k + 2\ell \leq N_*$. Here we have used implicitly the bounds $\ep_{m-1}^{\beta-2} \leq C (\tau_m^\prime)^{-1}$ and $\kappa_m \ep_m^{-2} \leq C (\tau_m^\prime)^{-1}$. 
Inductively in $r$, it is then direct to establish
\begin{equation}
\label{e.Amnr.bounds.induction}
\bigl\|
\nabla^k \DD_{t,m-1}^\ell \mathbf{A}_{m,n,r}
\bigr\|_{L^2([0,1]\times \TT^2)}
\leq C  \|\theta_0\|_{L^2(\TT^2)}
\bigl(\ep_m^2 \kappa_m^{-1}\bigr)
\kappa_{m-1}^{-\nicefrac 12}
\ep_{m-1}^{-(1+\nicefrac\gamma2) k } \bigl(\tau_m^\prime \bigr)^{-\ell-r}
\end{equation}
but only for $k$ and $\ell$ that satisfy $k + 2(\ell+r) \leq N_*$. To see this, note that the recursion relation in \eqref{e.A.mnr.def} gives $\mathbf{A}_{m,n,r+1} = \DD_{t,m-1} \mathbf{A}_{m,n,r} + \nabla \b_{m-1} \mathbf{A}_{m,n,r}$, with suitable contraction. If only the first term in this relation would be present, then \eqref{e.Amnr.bounds.induction} would simply follow by induction. The second term in this relation requires that we use the Leibniz rule to decompose $\nabla^k \DD_{t,m-1}^\ell (\nabla \b_{m-1} \mathbf{A}_{m,n,r} ) = \sum_{k'=0}^k \sum_{\ell'=0}^{\ell} 
\binom{k}{k'} \binom{\ell}{\ell'} \nabla^{k-k'} \DD_{t,m-1}^{\ell-\ell'} \nabla \b_{m-1} \, \nabla^{k'} \DD_{t,m-1}^{\ell'} \mathbf{A}_{m,n,r}$. The desired bound at level $r+1$ is then a consequence of \eqref{e.material.goal} and \eqref{e.Amnr.bounds.induction} at level $r$.
\end{proof}

\begin{proposition}
\label{p.Hm.bounds}
Under the assumptions of Lemma~\ref{l.Tm.reg.upgrade},  there exists a constant $C >0$, which only depends only on~$N_*$, such that 
\begin{align}
\label{e.Hm.Linfty}
\norm{\widetilde{H}_m(t,\cdot)}_{L^{\infty}_tL^2_x([0,1]\times \TT^2)} 
& \leq C \ep_{m-1}^{\delta} \|\theta_0\|_{L^2(\TT^2)} 
\,.
\end{align}
and 
\begin{align}
\label{e.Hm.gradient.L2}
\norm{\nabla \widetilde{H}_m(t,\cdot)}_{L^{2}_tL^2_x([0,1]\times \TT^2)} 
& \leq  
C  \ep_{m-1}^{4\delta}
\kappa_{m-1}^{-\nicefrac 12}
\|\theta_0\|_{L^2(\TT^2)} 
\,.
\end{align}
\end{proposition}

\begin{proof}
Recall that
\begin{equation}
\label{e.Hm.twiddle.again}
\widetilde{H}_m(t,x)
=
\nabla \cdot
\sum_{r=0}^{\nicefrac{N_*}{2}}
\sum_{n=0}^{N_*-1}
\mathbf{A}_{m,n,r}(t,x)
\mathbf{q}_{m,n,r+1} (t)
\,.
\end{equation}

\emph{Step 1.}
The uniform-in-time estimate~\eqref{e.Hm.Linfty}. 
First, we observed that for all $l \in \Z$, and $t \in [(l-\nicefrac 12) \tau_m^{\prime\prime},(l+\nicefrac 12) \tau_m^{\prime\prime}]$, since the flow $X_{m-1,l}(t,\cdot)$ is volume preserving, we have that $\| \widetilde{H}_{m}(t,\cdot) \|_{L^2_x} = \| \widetilde{H}_{m} \circ X_{m-1,l}(t,\cdot) \|_{L^2_x}$. Second, we note that by the construction of $\hat{\zeta}_{m,l}$, this function vanishes identically in a $2\tau_m^\prime$-neighborhood of $(l\pm \nicefrac 12) \tau_m^{\prime\prime}$ (see~\eqref{e.zeta.prime.ml.fitting}), and thus by the definition of $L_{m,n}^{\kappa_m}$  in \eqref{e.Lmn.def}, we have   
\begin{equation}
\bigl(\partial_t^\ell L_{m,n}^{\kappa_m} \bigr)( (l \pm \nicefrac 12) \tau_m^{\prime\prime}) =  0
\label{e.In.ell.no.bdry.terms}
\end{equation}
for all $\ell \in \N$. In turn,~\eqref{e.In.ell.no.bdry.terms} and the recursive definition of the $\mathbf{A}_{m,n,r}$ tensors in \eqref{e.A.mnr.def} gives that $\mathbf{A}_{m,n,r} ( (l \pm \nicefrac 12) \tau_m^{\prime\prime},\cdot) =  0$ for all $0\leq r \leq \nicefrac{N_*}{2}$, and thus
\begin{equation*}
\widetilde{H}_m( (l \pm \nicefrac 12) \tau_m^{\prime\prime},\cdot) =  0
\,,
\end{equation*}
for all $l \in \Z$. Combining these two observations with the fundamental theorem of calculus in time, we deduce that for all $t \in [(l-\nicefrac 12) \tau_m^{\prime\prime},(l+\nicefrac 12) \tau_m^{\prime\prime}]$
\begin{align}
\| \widetilde{H}_{m}(t,\cdot) \|_{L^2_x}^2
&= \int_{(l-\nicefrac 12)\tau_m^{\prime\prime}}^t \frac{d}{dt'} \| \widetilde{H}_{m} (t',X_{m-1,l}(t',\cdot) \|_{L^2_x}^2 d t'
\notag\\
&= 2 \int_{(l-\nicefrac 12)\tau_m^{\prime\prime}}^t \int_{\TT^2}  \widetilde{H}_{m} (t',X_{m-1,l}(t',x) (\DD_{t,m-1} \widetilde{H}_{m})(t',X_{m-1,l}(t',x)  dx dt'
\notag\\
&\leq 
2 (\tau_m^{\prime\prime})^{\nicefrac 12} 
\norm{\DD_{t,m-1} \widetilde{H}_{m}}_{L^2([0,1]\times \TT^2)}
\sup_{t\in [(l-\nicefrac 12)\tau_m^{\prime\prime},(l+\nicefrac 12)\tau_m^{\prime\prime}]} \| \widetilde{H}_{m}(t,\cdot) \|_{L^2_x} 
\,.
\label{e.tilde.Hm.temp.1}
\end{align}
Next, using~\eqref{e.Hm.twiddle.again},~\eqref{e.qmnr.bounds} and~\eqref{e.Amnr.bounds}, we derive
\begin{align*}
\lefteqn{
\bigl\| 
\DD_{t,m-1} \widetilde{H}_m
\bigr\|_{L^{2}([0,1]\times \TT^2)} 
} \qquad & 
\notag \\ & 
\leq 
\sum_{r=0}^{\nicefrac{N_*}{2}}
\sum_{n=0}^{N_*-1}
\biggl(
\bigl\|\nabla \DD_{t,m-1}  \mathbf{A}_{m,n,r}\bigr\|_{L^{2}([0,1]\times \TT^2)} 
\bigl\| \mathbf{q}_{m,n,r+1} \bigr\|_{L^{\infty}([0,1])}
\notag\\
&\qquad \qquad \qquad 
+
\bigl\|  \nabla \b_{m-1}\bigr\|_{L^{\infty}([0,1]\times \TT^2)} 
\bigl\|  \nabla \mathbf{A}_{m,n,r}\bigr\|_{L^{2}([0,1]\times \TT^2)} 
\bigl\| \mathbf{q}_{m,n,r+1} \bigr\|_{L^{\infty}([0,1])}
\notag\\
&\qquad \qquad \qquad 
+
\bigl\| \nabla  \mathbf{A}_{m,n,r}\bigr\|_{L^{2}([0,1]\times \TT^2)} 
\bigl\| \partial_t \mathbf{q}_{m,n,r+1} \bigr\|_{L^{\infty}([0,1])}
\biggr)
\notag \\ & 
\leq 
C \|\theta_0\|_{L^2(\TT^2)} \!
\sum_{r=0}^{\nicefrac{N_*}{2}}
\sum_{n=0}^{N_*-1}
\biggl(
\biggl\{ 
\bigl(\ep_m^2 \kappa_m^{-1}\bigr)
\kappa_{m-1}^{-\nicefrac 12}
\ep_{m-1}^{-(1+\nicefrac\gamma2)} \bigl(\tau_m^\prime \bigr)^{-r-1}
\biggr\}
\biggl\{ 
\frac{a_m^2 \ep_m^2}{n!}
\biggl( \frac{C\ep_m^2}{\kappa_m \tau_m } \biggr)^n \frac{(C\tau_m)^{r+1}}{(r+1)!}
\biggr\} 
\notag\\
&\qquad \qquad \qquad\qquad\qquad \quad +
\biggl\{ 
\bigl(\ep_m^2 \kappa_m^{-1}\bigr)
\kappa_{m-1}^{-\nicefrac 12}
\ep_{m-1}^{-(1+\nicefrac\gamma2)} \bigl(\tau_m^\prime \bigr)^{-r}
\biggr\}
\biggl\{ 
\frac{a_m^2 \ep_m^2}{n!}
\biggl( \frac{C\ep_m^2}{\kappa_m \tau_m } \biggr)^n \frac{(C\tau_m)^{r}}{r!}
\biggr\} \biggr)
\,.
\end{align*}
We next use that
\begin{equation}
\label{e.tilde.Hm.temp.2}
\sum_{r=0}^\infty
\biggl( \frac{C\tau_m}{\tau_m'} \biggr)^{\!\!r} 
\leq
\sum_{r=0}^\infty
\bigl( C \ep_{m-1}^{\delta} \bigr)^{r} 
\leq 2\,,
\end{equation}
and 
\begin{equation}
\label{e.tilde.Hm.temp.3}
\sum_{n=0}^\infty
\biggl( \frac{C\ep_m^2}{\kappa_m \tau_m } \biggr)^{\!\!n} 
\leq
\sum_{n=0}^\infty
\bigl( C \ep_{m-1}^{2\delta} \bigr)^{n}
\leq 2\,.
\end{equation}
Combining the three displays above, we arrive at 
\begin{equation}
\bigl\| 
\DD_{t,m-1} \widetilde{H}_m
\bigr\|_{L^{2}([0,1]\times \TT^2)} 
\leq 
C \|\theta_0\|_{L^2(\TT^2)} 
\bigl( a_m^2 \ep_m^4 \kappa_m^{-1} \bigr)
\kappa_{m-1}^{-\nicefrac 12}
\ep_{m-1}^{-(1+\nicefrac\gamma2)}  
\,.
\label{e.tilde.Hm.temp.4}
\end{equation}
Returning to \eqref{e.tilde.Hm.temp.1}, we take the supremum in time over $t \in [(l-\nicefrac 12) \tau_m^{\prime\prime},(l+\nicefrac 12) \tau_m^{\prime\prime}]$, absorb the suitable term in the left side, and then taking a supremum over $l \in \Z$, we deduce
\begin{align*}
\bigl\|
\widetilde{H}_m
\bigr\|_{L^\infty([0,1];L^2(\TT^2))}
&\leq C \|\theta_0\|_{L^2(\TT^2)}
\bigl( a_m^2 \ep_m^4 \kappa_m^{-1} \bigr)
(\tau_m^{\prime\prime})^{\nicefrac 12} 
\kappa_{m-1}^{-\nicefrac 12}
\ep_{m-1}^{-(1+\nicefrac\gamma2)} 
\\
&\leq C \|\theta_0\|_{L^2(\TT^2)}
\kappa_{m-1}
(a_{m-1}^{-1} \ep_{m-1}^{2\delta})^{\nicefrac 12} 
\kappa_{m-1}^{-\nicefrac 12}
\ep_{m-1}^{-(1+\nicefrac\gamma2)}  
\\
&\leq C \|\theta_0\|_{L^2(\TT^2)}
\ep_{m-1}^{\delta}  
\,.
\end{align*}
This concludes the proof of \eqref{e.Hm.Linfty}.

\smallskip

\emph{Step 2.} The gradient estimate~\eqref{e.Hm.gradient.L2}. 
Using~\eqref{e.Hm.twiddle.again},~\eqref{e.qmnr.bounds} and~\eqref{e.Amnr.bounds}, we obtain
\begin{align}
\label{e.nabla.Hm.plug}
\lefteqn{
\bigl\| 
\nabla \widetilde{H}_m
\bigr\|_{L^{2}([0,1]\times \TT^2)} 
}  \ \ \ & 
\notag \\ & 
\leq 
\sum_{r=0}^{\nicefrac{N_*}{2}}
\sum_{n=0}^{N_*-1}
\bigl\| \nabla^2 \mathbf{A}_{m,n,r}\bigr\|_{L^{2}([0,1]\times \TT^2)} 
\bigl\| \mathbf{q}_{m,n,r+1} \bigr\|_{L^{\infty}([0,1])}
\notag \\ & 
\leq 
C \|\theta_0\|_{L^2(\TT^2)} \!
\sum_{r=0}^{\nicefrac{N_*}{2}}
\sum_{n=0}^{N_*-1}
\biggl\{ 
\bigl(\ep_m^2 \kappa_m^{-1}\bigr)
\kappa_{m-1}^{-\nicefrac 12}
\ep_{m-1}^{-2(1+\nicefrac\gamma2)} \bigl(\tau_m^\prime \bigr)^{-r}
\biggr\}
\biggl\{ 
\frac{a_m^2 \ep_m^2}{n!}
\biggl( \frac{C\ep_m^2}{\kappa_m \tau_m } \biggr)^n \frac{(C\tau_m)^{r+1}}{(r+1)!}
\biggr\} 
\,.
\end{align}
Inserting the bounds \eqref{e.tilde.Hm.temp.2} and \eqref{e.tilde.Hm.temp.3} into~\eqref{e.nabla.Hm.plug}, we obtain that
\begin{align*}
\bigl\| 
\nabla \widetilde{H}_m
\bigr\|_{L^{2}([0,1]\times \TT^2)} 
&
\leq
C \|\theta_0\|_{L^2(\TT^2)}
\Bigl\{ 
\bigl(\ep_m^2 \kappa_m^{-1}\bigr)
\kappa_{m-1}^{-\nicefrac 12}
\ep_{m-1}^{-2(1+\nicefrac\gamma2)}
\Bigr\}
\Big\{
a_m^2 \ep_m^2 \tau_m
\Bigr\}
\notag \\ & 
=
C \|\theta_0\|_{L^2(\TT^2)}
\frac{a_m^2 \ep_m^4 }{\kappa_m}
\tau_m
\ep_{m-1}^{-(2+\gamma)}
\kappa_{m-1}^{-\nicefrac 12}
\notag \\ & 
=
C \|\theta_0\|_{L^2(\TT^2)}
a_{m-1} 
\tau_m 
\kappa_{m-1}^{-\nicefrac 12}
=
C \|\theta_0\|_{L^2(\TT^2)}
\ep_{m-1}^{4\delta}
\kappa_{m-1}^{-\nicefrac 12}
\,,
\end{align*}
which proves~\eqref{e.Hm.gradient.L2}.
\end{proof}

\begin{proposition}
\label{p.dm.bounds}
Under the assumptions of Lemma~\ref{l.Tm.reg.upgrade},  there exists a constant $C >0$, which only depends only on~$N_*$, such that 
\begin{equation}
\label{e.dm.bounds}
\bigl\|
\mathbf{d}_m
\bigr\|_{L^2([0,1]\times \TT^2)}
\leq  
C \kappa_{m-1}^{\nicefrac 12}
\bigl(C \ep_{m-1}^{\delta} \bigr)^{\nicefrac{N_*}{2}} \|\theta_0\|_{L^2(\TT^2)}
\,.
\end{equation}
\end{proposition}
\begin{proof}
Recalling the definition of $\mathbf{d}_m$ in \eqref{e.dm} we have 
\begin{align*}
\bigl\|
\mathbf{d}_{m}
\bigr\|_{L^2([0,1]\times \TT^2)}
&\leq
2 \bigl\| 
\hat{\J}_m-\J_m 
\bigr \|_{L^\infty([0,1])}
\sup_{l\in\Z}  
\bigl\| \hat{\xi}_{m,l }
\bigl( 
\nabla X_{m-1,l} \circ X_{m-1,l}^{-1}
\bigr)
\bigr\|_{L^\infty([0,1] \times \TT^2)}
\bigl\|
\nabla T_{m-1}
\bigr\|_{L^2([0,1]\times \TT^2)}
\notag\\
&\qquad 
+ 
\sum_{n=0}^{N_*-1}
\bigl\|
\mathbf{A}_{m,n,\nicefrac{N_*}{2}}
\bigr\|_{L^2([0,1]\times \TT^2)}
\bigl\|
\mathbf{q}_{m,n,\nicefrac{N_*}{2}}^{\kappa_m}  
\bigr\|_{L^\infty([0,1]\times \TT^2)}
\,.
\end{align*}
Appealing to the closeness of~$\hat{\J}$ to~$\J$ in~\eqref{e.JJhat}, the $\mathbf{A}$ bound in~\eqref{e.Amnr.bounds}, the $\mathbf{q}$ bound in~\eqref{e.qmnr.bounds}, the $\nabla T_{m-1}$ estimate in~\eqref{e.Tm.reg.upgrade}, the summability in $n$ from~\eqref{e.tilde.Hm.temp.3}, and the flow bound~\eqref{e.spacetime.X.bnds}, we deduce
\begin{align*}
\bigl\|
\mathbf{d}_{m}
\bigr\|_{L^2([0,1]\times \TT^2)}
&\leq
\frac{C a_m^2\ep_m^4}{\kappa_m} \biggl(
\frac{\ep_m^2}{\kappa_m \tau_m}
\biggr)^{\!\!N_*}
\kappa_{m-1}^{-\nicefrac 12} \|\theta_0\|_{L^2(\TT^2)}
\notag\\
&\qquad 
+ C \|\theta_0\|_{L^2(\TT^2)}
\sum_{n=0}^{N_*-1}
\bigl(\ep_m^2 \kappa_m^{-1}\bigr)
\kappa_{m-1}^{-\nicefrac 12}
\bigl(\tau_m^\prime \bigr)^{-\nicefrac{N_*}{2}} 
\frac{a_m^2 \ep_m^2}{n!}
\biggl( \frac{C\ep_m^2}{\kappa_m \tau_m } \biggr)^n   \tau_m^{\nicefrac{N_*}{2}} 
\notag\\
&\leq
C \|\theta_0\|_{L^2(\TT^2)} 
\kappa_{m-1}^{\nicefrac 12}
\bigl(
C \ep_{m-1}^{2\delta}
\bigr)^{\!N_*}
+ 
C \|\theta_0\|_{L^2(\TT^2)}
\kappa_{m-1}^{\nicefrac 12}
\bigl(C \ep_{m-1}^{\delta} \bigr)^{\nicefrac{N_*}{2}} 
\,.
\end{align*}
Since the second of the above two terms is larger, this gives \eqref{e.dm.bounds}. 
\end{proof}

\section{Homogenization cascade up the inertial-convection subrange}
\label{s.cascade}

In this section, we give the proof of Theorem~\ref{t.anomalous.diffusion}. 
We begin by plugging the ansatz into the advection-diffusion operator and computing the error. This is the purpose of the next subsection.
In Section~\ref{ss.estimate.error} we estimate the error term, which is then used in Section~\ref{ss.one.renormalization.step} to complete the main induction step, summarized in Proposition~\ref{p.indystepdown}. The proof of the theorem appears finally in Section~\ref{ss.proof}. 

\subsection{Computing the error in the multiscale ansatz}

In this subsection, we compute an explicit expression for the error obtained when we insert the two-scale ansatz~$\tilde{\theta}_m$ it into the left side of~\eqref{e.Tm}. That is, we compute~$(\partial_t - \kappa_m \Delta + \b_m \cdot \nabla ) \tilde{\theta}_m$. The main result is the Big Display on Page~\pageref{e.monster}.

\smallskip

Throughout, we use the abbreviated notations 
\begin{align*}
\left\{
\begin{aligned}
& \tilde{\psi}_{m,k}(t,x) :=
\bigl(\psi_{m,k} \circ X_{m-1,k}^{-1}\bigr) (t,x) , 
\\ 
& \tilde{\psi}_m(t,x) : = \textstyle{\sum_{k\in\Z} }
\hat{\zeta}_{m,l_k}
\zeta_{m,k}(t) \tilde{\psi}_{m,k}(t,x).
\end{aligned}
\right.
\end{align*}
As in Section~\ref{s.multiscale}, we use the notational convention that function compositions are with respect to the space variables only. 
Observe that the recurrence in~\eqref{e.psi.recursion} may be written as
\begin{align}
\label{e.phimincr}
\phi_m(t,x) - \phi_{m-1}(t,x)
=  
\tilde{\psi}_m(t,x),
\qquad \forall m\in\N \cap[1,\infty).
\end{align}
We proceed by splitting the operator:
\begin{align}
\label{e.operator.splitting}
(\partial_t - \kappa_m \Delta + \b_m \cdot \nabla ) \tilde{\theta}_m
& 
=
\bigl(\partial_t + \b_{m-1} \cdot \nabla \bigr) \tilde{\theta}_m
+
\bigl( - \kappa_m \Delta + (\b_m - \b_{m-1} ) \cdot \nabla \bigr)\tilde{\theta}_m
\notag \\ & 
=
\underbrace{
\bigl(\partial_t + \b_{m-1} \cdot \nabla \bigr) \tilde{\theta}_m
}_{\text{the transport term}}
-
\underbrace{
\nabla \cdot 
\bigl( \kappa_m \Itwo +  \tilde{\psi}_m \sigma \bigr) \nabla \tilde{\theta}_m
}_{\text{the diffusion term}}
\,.
\end{align}
We will compute the transport term and the diffusion term separately. 

\subsubsection{Computation of the transport term}

It should come as no surprise that we will use Lagrangian coordinates to compute the transport term. 
We make use of the following two identities:
\begin{equation}
\label{e.tbm1.one}
\bigl( \partial_t + \b_{m-1} \cdot\nabla \bigr)
\tilde{\Chi}_{m,k}
=
\partial_t\Chi_{m,k} \circ X_{m-1,l_k}^{-1}
\,,
\end{equation}
and
\begin{align}
\label{e.tbm1.two}
\lefteqn{
\bigl( \partial_t + \b_{m-1} \cdot\nabla \bigr)
\bigl( 
\nabla \left( T_{m-1} \circ X_{m-1,l} \right) \circ X_{m-1,l}^{-1} \bigr)
} \qquad & 
\notag \\ & 
=
\bigl( \nabla X_{m-1,l}\circ X_{m-1,l}^{-1} \bigr)
\nabla 
\nabla \cdot \bigl(
\bigl(  \K_m  + \mathbf{s}_{m-1} \bigr) \nabla T_{m-1} +\mathbf{e}_{m-1}\bigr)
\,.
\end{align}
To prove these, recall that if $Z(t,x)$ is a flow for $\b_{m-1}$, that is, a solution of the ODE
\begin{equation*}
\partial_t Z = \b_{m-1}(t,Z)
\,,
\end{equation*}
then the inverse flow $Z^{-1}$ satisfies the transport equation (cf.~\eqref{e.PDE.backwards})
\begin{equation*}
\bigl( \partial_t + \b_{m-1} \cdot \nabla \bigr) Z^{-1} = 0
\,.
\end{equation*}
Moreover, if~$Z^{-1}$ is smooth, then any function of~$Z^{-1}$ also satisfies the same transport equation. 
Applying this to~$Z=X_{m-1,k}$, while keeping in mind the convention that function compositions in our formulas are with respect to the spatial variable only, we find that, for any function~$F(t,x)$ of both~$t$ and~$x$, we have
\begin{equation}
\label{e.transport.calculus.1}
\bigl( \partial_t + \b_{m-1} \cdot \nabla \bigr) \bigl( F\circ Z^{-1} \bigr) 
=
\partial_t F \circ Z^{-1}
\,,
\end{equation}
and
\begin{equation}
\label{e.transport.calculus.2}
\bigl( \bigl(\partial_t + \b_{m-1}\cdot\nabla \bigr) F \bigr) \circ Z
=
\partial_t (F\circ Z)
\,.
\end{equation}
The first claimed identity~\eqref{e.tbm1.one} is then immediate from~\eqref{e.transport.calculus.1}.
To obtain~\eqref{e.tbm1.two}, we use~\eqref{e.transport.calculus.1},~\eqref{e.transport.calculus.2} and the equation~\eqref{e.Tm.true} for $T_{m-1}$ as follows:
\begin{align*}
\lefteqn{
\bigl( \partial_t + \b_{m-1} \cdot\nabla \bigr)
\bigl( 
\nabla (  T_{m-1} \circ X_{m-1,l} ) \circ X_{m-1,l}^{-1} \bigr)
} \qquad & 
\\ &
= \bigl( \partial_t \nabla (  T_{m-1} \circ X_{m-1,l} ) \bigr) 
\circ X_{m-1,l}^{-1} 
\\ & 
=
\bigl( \nabla \partial_t (  T_{m-1} \circ X_{m-1,l} ) \bigr) \circ X_{m-1,l}^{-1} 
\\ & 
=
\bigl( 
\nabla \bigl(  \bigl( \partial_t  T_{m-1} + \b_{m-1} \cdot \nabla  T_{m-1}\bigr) \circ X_{m-1,l} \bigr) 
\bigr)
\circ X_{m-1,l}^{-1}
\\ & 
= 
\bigl( \nabla X_{m-1,l}\circ X_{m-1,l}^{-1} \bigr)
\nabla
\bigl[ 
\nabla \cdot 
\bigl( \bigl( \K_m + \mathbf{s}_{m-1} \bigr) \nabla T_{m-1} + \mathbf{e}_{m-1} \bigr)
\bigr]
\,.
\end{align*}
We are ready to apply $\bigl( \partial_t + \b_{m-1} \cdot \nabla \bigr)$ to both sides of~\eqref{e.ansatz.line2}. Using the product rule, the equation for~$\widetilde{H}_m$ in~\eqref{e.real.eq.Hm}, the definition of~$\tilde{G}_m$ in~\eqref{e.Gm}, and the above identities~\eqref{e.tbm1.one} and~\eqref{e.tbm1.two}, we obtain
\begin{align}
\label{e.timecomp.0}
\bigl( \partial_t + \b_{m-1} \cdot \nabla \bigr) \tilde{\theta}_m
&
=
\bigl( \partial_t + \b_{m-1} \cdot \nabla \bigr) T_{m-1} 
+
\sum_{k\in 2\Z+1}\!\!
(\partial_t\xi_{m,k} \tilde{\Chi}_{m,k})\cdot
\nabla \bigl(  T_{m-1} \circ X_{m-1,l_k} \bigr) \circ X_{m-1,l_k}^{-1} 
\notag \\ & \quad 
+
\sum_{k\in 2\Z+1}\!\!
\xi_{m,k} 
\bigl( \partial_t \Chi_{m,k} \circ X_{m-1,l_k}^{-1} \bigr) \cdot
\nabla \bigl(  T_{m-1} \circ X_{m-1,l_k} \bigr) \circ X_{m-1,l_k}^{-1} 
\notag \\ & \quad 
+
\sum_{k\in 2\Z+1}\!\!
\xi_{m,k} 
\tilde{\Chi}_{m,k}\cdot
\bigl( \nabla X_{m-1,l_k}\circ X_{m-1,l_k}^{-1} \bigr) 
\nabla \nabla \cdot 
\bigl( \bigl( \K_m + \mathbf{s}_{m-1} \bigr) \nabla T_{m-1} + \mathbf{e}_{m-1} \bigr)
\notag \\ & \quad 
+
\nabla \cdot 
\sum_{l\in \Z}
\hat{\xi}_{m,l}
\bigl( \J_m  - \K_m \bigr) 
\nabla
\bigl( T_{m-1} \circ X_{m-1,l} \bigr) \circ X_{m-1,l}^{-1}
+\nabla \cdot \mathbf{d}_m
\,.
\end{align}
In view of the definition of~$\mathbf{s}_{m-1}$ in~\eqref{e.sm}, we can 
write the equation for~$T_{m-1}$ as 
\begin{equation*}
\bigl( \partial_t + \b_{m-1} \cdot \nabla \bigr) T_{m-1} 
=
\nabla \cdot 
\biggl( 
\sum_{l\in \Z}
\hat{\xi}_{m,l}
\K_m 
\nabla
\bigl( T_{m-1} \circ X_{m-1,l} \bigr) \circ X_{m-1,l}^{-1}
+ \mathbf{e}_{m-1}
\biggr)
\,.
\end{equation*}
Using this, we can cancel the first term on the right of~\eqref{e.timecomp.0} with part of the first term on the last line (the expression involving~$\mathbf{K}_m$). Note that we are using here the fact that~$\mathbf{K}_m$ is a scalar matrix, and it therefore commutes with~$\bigl( \nabla X_{m-1,l_k}\circ X_{m-1,l_k}^{-1} \bigr)$. 
We therefore obtain
\begin{align}
\label{e.timecomp}
\bigl( \partial_t + \b_{m-1} \cdot \nabla \bigr) \tilde{\theta}_m
&
=
\sum_{k\in 2\Z+1}\!\!
(\partial_t\xi_{m,k} \tilde{\Chi}_{m,k})\cdot
\nabla \bigl(  T_{m-1} \circ X_{m-1,l_k} \bigr) \circ X_{m-1,l_k}^{-1} 
\notag \\ & \quad 
+
\sum_{k\in 2\Z+1}\!\!
\xi_{m,k} 
\bigl( \partial_t \Chi_{m,k} \circ X_{m-1,l_k}^{-1} \bigr) \cdot
\nabla \bigl(  T_{m-1} \circ X_{m-1,l_k} \bigr) \circ X_{m-1,l_k}^{-1} 
\notag \\ & \quad 
+
\sum_{k\in 2\Z+1}\!\!
\xi_{m,k} \tilde{\Chi}_{m,k}\cdot
\bigl( \nabla X_{m-1,l_k}\circ X_{m-1,l_k}^{-1} \bigr) 
\nabla \nabla \cdot 
\bigl( \bigl( \K_m + \mathbf{s}_{m-1} \bigr) \nabla T_{m-1}+ \mathbf{e}_{m-1}\bigr)
\notag \\ & \quad 
+
\J_m : 
\sum_{l\in \Z}
\hat{\xi}_{m,l} 
\nabla \bigl( \nabla
\bigl( T_{m-1}\circ X_{m-1,l} \bigr) \circ X_{m-1,l}^{-1} \bigr)
+\nabla \cdot \bigl(\mathbf{d}_m + \mathbf{e}_{m-1} \bigr)
\, .
\end{align}
Below we will insert the identity~\eqref{e.timecomp} for the transport term back into the right side of~\eqref{e.operator.splitting}. The first, third and fifth terms on the right side of~\eqref{e.timecomp} are ``acceptable errors,'' that is, we will eventually show that they are negligible for our purposes. The second and fourth terms will cancel some expressions arising in our computation of the diffusion term, which we pursue next. 

\subsubsection{Computation of the diffusion term}

We write the diffusive term in divergence form as
\begin{equation}
\label{e.diffusionterm.2}
\bigl( - \kappa_{m} \Delta + (\b_m - \b_{m-1}) \cdot \nabla\bigr) \tilde{\theta}_m
=
-\nabla \cdot \bigl( \bigl( \kappa_{m} \Itwo + \tilde{\psi}_m \sigma \bigr) \nabla \tilde{\theta}_m \bigr).
\end{equation}
Returning to the formula~\eqref{e.ansatz.line2} to compute the gradient of~$\tilde{\theta}_m$, we find
\begin{align}
\label{e.grad.tildetheta}
\nabla \tilde{\theta}_m
&
=
\sum_{k\in2\Z+1} \!\!
\xi_{m,k}
\bigl( \Itwo + \nabla \tilde{\Chi}_{m,k} \bigr)
\sum_{l\in\Z}
\hat{\xi}_{m,l} 
\nabla \bigl( T_{m-1} \circ X_{m-1,l} \bigr) \circ X^{-1}_{m-1,l}
\notag \\ & \qquad 
+
\sum_{l\in\Z} 
\hat{\xi}_{m,l}
\bigl( 
\Itwo - \nabla X_{m-1,l} \circ X_{m-1,l}^{-1} \bigr) \nabla T_{m-1}
\notag \\ & \qquad
+ \!\!
\sum_{k\in 2\Z+1} \!\!
\xi_{m,k} \tilde{\Chi}_{m,k}
\nabla \bigl( 
\nabla \bigl(  T_{m-1} \circ X_{m-1,l_k} \bigr) \circ X_{m-1,l_k}^{-1} 
\bigr)
+\nabla \widetilde{H}_m
\end{align}
Inserting this into the operator on the right side of~\eqref{e.diffusionterm.2}, we obtain
\begin{align}
\label{e.initial.plugin}
\lefteqn{ 
- \nabla \cdot 
\bigl( 
\bigl( \kappa_m \Itwo + \tilde{\psi}_m \sigma \bigr) \nabla \tilde{\theta}_m
\bigr) 
} \qquad & 
\notag \\ & 
=
- \nabla \cdot  \!\!
\sum_{k\in 2\Z+1} \!\!
\xi_{m,k}
\bigl( \kappa_m \Itwo + \tilde{\psi}_m \sigma \bigr) 
\bigl( \Itwo + \nabla \tilde{\Chi}_{m,k} \bigr)
\sum_{l\in\Z} 
\hat{\xi}_{m,l} 
\nabla \bigl(  T_{m-1} \circ X_{m-1,l} \bigr) \circ X_{m-1,l}^{-1} 
\notag \\ & \qquad 
- \nabla \cdot 
\sum_{l\in\Z} 
\hat{\xi}_{m,l}
\bigl( \kappa_m \Itwo + \tilde{\psi}_m \sigma \bigr) 
\bigl( 
\Itwo - \nabla X_{m-1,l} \circ X_{m-1,l}^{-1} \bigr) \nabla T_{m-1}
\notag \\ & \qquad
- \nabla \cdot \!\!
\sum_{k\in 2\Z+1} \!\!
\xi_{m,k} 
\bigl( \kappa_m \Itwo + \tilde{\psi}_m \sigma \bigr) 
\tilde{\Chi}_{m,k}
\nabla 
\bigl( \nabla \bigl(  T_{m-1} \circ X_{m-1,l_k} \bigr) \circ X_{m-1,l_k}^{-1}\bigr) 
\notag \\ & \qquad
-\nabla \cdot 
\bigl( \kappa_m \Itwo + \tilde{\psi}_m \sigma \bigr) 
\nabla \widetilde{H}_m
\,.
\end{align}
The second, third and fourth terms on the right side of~\eqref{e.initial.plugin} are acceptable errors: we will show in the next subsection that they are negligible for our purposes. 

\smallskip

Let's look more closely at what is inside the divergence in the first term on the right side of~\eqref{e.initial.plugin}.
Using
the properties~\eqref{e.zeta.mk} and~\eqref{e.cutoff.xi} of the cutoff functions and the fact that both~$\tilde\Chi_{m,k}^\kappa$ and~$\tilde{\psi}_{m,k}$ vanish when $k$ is even, we see that 
\begin{align*}
\xi_{m,k} \zeta_{m,k'} \tilde{\psi}_{m,k'}
=
\zeta_{m,k} \tilde{\psi}_{m,k} \indc_{\{ k = k' \}}
\quad 
\mbox{and}
\quad
\xi_{m,k} \zeta_{m,k'} \nabla \tilde{\Chi}_{m,k'}
=
\zeta_{m,k} \nabla \tilde{\Chi}_{m,k} \indc_{\{ k = k' \}}
\,.
\end{align*}
We therefore obtain
\begin{align}
\label{e.fourlinesofcutoff}
\lefteqn{
\sum_{k\in 2\Z+1} \!\!
\xi_{m,k}
\bigl( \kappa_m \Itwo + \tilde{\psi}_m \sigma \bigr) 
\bigl( \Itwo + \nabla \tilde{\Chi}_{m,k} \bigr)
\sum_{l\in\Z} \hat{\xi}_{m,l}
\nabla \bigl(  T_{m-1} \circ X_{m-1,l} \bigr) \circ X_{m-1,l}^{-1} 
} \quad & 
\notag \\ & 
=
\sum_{k\in 2\Z+1} \!\!
\xi_{m,k}
\Bigl( \kappa_m \Itwo + \sum_{k' \in\Z} \hat{\zeta}_{m,l_{k'}}
 \zeta_{m,k'}\tilde{\psi}_{m,k'} \sigma\Bigr)
\bigl( \Itwo + \nabla \tilde{\Chi}_{m,k} \bigr)
\sum_{l\in\Z} \hat{\xi}_{m,l}
\nabla \bigl(  T_{m-1} \circ X_{m-1,l} \bigr) \circ X_{m-1,l}^{-1} 
\notag \\ & 
=
\sum_{k\in 2\Z+1} \!\!
\xi_{m,k}
\bigl( \kappa_m \Itwo +  \hat{\zeta}_{m,l_k}
 \zeta_{m,k}\tilde{\psi}_{m,k} \sigma\bigr)
\bigl( \Itwo + \nabla \tilde{\Chi}_{m,k} \bigr)
\sum_{l\in\Z} \hat{\xi}_{m,l}
\nabla \bigl(  T_{m-1} \circ X_{m-1,l} \bigr) \circ X_{m-1,l}^{-1} 
\,.
\end{align}
Before we compute the divergence of this expression, we need to use a special property of the shear flow structure, which is that 
\begin{align}
\label{e.ortho.tilde}
\nabla \cdot \bigl( \tilde{\psi}_{m,k} \sigma \nabla \tilde{\Chi}_{m,k} \bigr)
&
=
\nabla \tilde{\psi}_{m,k}\cdot \sigma  \nabla \tilde{\Chi}_{m,k}
=0
\,.
\end{align}
Indeed, from~\eqref{e.def.streamr.0}--\eqref{e.def.streamr},~\eqref{e.ukm.explicit} and~\eqref{e.Chimk.formula}, we see that ${\psi}_{m,k}$ and~$\Chi_{m,k}$ depend only on one coordinate~$x_i$ for some $i\in\{1,2\}$, which is the same for both functions and depends only on~$k$. Using coordinates with the summation convention, we therefore compute
\begin{align*}
\nabla \tilde{\psi}_{m,k} \cdot\sigma  \nabla \tilde{\Chi}_{m,k}
&
=
\partial_{x_i} \tilde{\psi}_{m,k} \sigma_{ij} \partial_{x_j}  \tilde{\Chi}_{m,k}
\\ & 
=
(\partial_{x_l} \psi_{m,k} \circ X^{-1}_{m-1})
\partial_{x_i} (X^{-1}_{m-1})_l  
\sigma_{ij} 
(\partial_{x_{l'}} \Chi_{m,k} \circ X^{-1}_{m-1} )
\partial_{x_j} (X^{-1}_{m-1})_{l'}  
\\ & 
=
\underbrace{
\sigma_{ij} 
\partial_{x_i} (X^{-1}_{m-1})_l  
\partial_{x_j} (X^{-1}_{m-1})_l  
}_{=0}
(\partial_{x_l} \psi \circ X^{-1}_{m-1})
(\partial_{x_l} \Chi_{m,k} \circ X^{-1}_{m-1} )
= 0
\,.
\end{align*}
(See also Remark~\ref{r.special.orthogonality}.)
Using~\eqref{e.ortho.tilde}, we may write
\begin{align*}
\lefteqn{
\nabla \cdot \bigl( 
\bigl( \kappa_m \Itwo + \hat{\zeta}_{m,l_k}
\zeta_{m,k}\tilde{\psi}_{m,k} \sigma\bigr)
\bigl( \Itwo +  \nabla \tilde{\Chi}_{m,k} \bigr)
\bigr) 
} \qquad & 
\\ &
=
\nabla \cdot \bigl( 
\hat{\zeta}_{m,l_k}
\zeta_{m,k}\tilde{\psi}_{m,k} \sigma
+ \kappa_m \nabla \tilde{\Chi}_{m,k} \bigr)
\\ & 
=
\nabla \cdot \Bigl( 
\hat{\zeta}_{m,l_k}
\zeta_{m,k}\tilde{\psi}_{m,k} \sigma
+ \kappa_m \nabla {\Chi}_{m,k} \circ X_{m-1,l_k}^{-1} 
+
\kappa_m 
\bigl( \nabla X_{m-1,l_k}^{-1}- \Itwo \bigr)
\nabla {\Chi}_{m,k} \circ X_{m-1,l_k}^{-1}
\Bigr)
\,.
\end{align*}
To proceed, we next use the fact that, for any smooth vector field~$\g$ and smooth, measure-preserving map $M:\R^2\to\R^2$, we have 
\begin{equation}
\label{e.pushforward}
\nabla \cdot ( (\nabla M^{-1}\g)\circ M) = (\nabla \cdot \g)\circ M
\,.
\end{equation}
In view of~\eqref{e.pushforward}, applied with~$\g = \hat{\zeta}_{m,l_k}
\zeta_{m,k}{\psi}_{m,k} \sigma
+ \kappa_m \nabla {\Chi}_{m,k}$ and~$M= X_{m-1,l_k}^{-1}$,
we find that 
\begin{align}
\label{e.push.it.forward}
\lefteqn{
\nabla \cdot \bigl( 
\hat{\zeta}_{m,l_k}
\zeta_{m,k}\tilde{\psi}_{m,k} \sigma
+ \kappa_m \nabla {\Chi}_{m,k} \circ X_{m-1,l_k}^{-1} \bigr)
} \quad & 
\notag \\ & 
=
\bigl( 
\nabla \cdot \bigl( \hat{\zeta}_{m,l_k}
\zeta_{m,k}{\psi}_{m,k} \sigma
+ \kappa_m \nabla {\Chi}_{m,k} \bigr) 
\bigr) \circ X_{m-1,l_k}^{-1}
\notag \\ & \qquad
+
\nabla \cdot
\bigl( 
\bigl( \Itwo-
\nabla X_{m-1,l_k} \circ X_{m-1,l_k}^{-1} 
\bigr)
\bigl( 
\hat{\zeta}_{m,l_k} \zeta_{m,k}{\psi}_{m,k} \sigma
+ \kappa_m \nabla {\Chi}_{m,k} 
\bigr) \circ X_{m-1,l_k}^{-1} 
\bigr)
\,.
\end{align}
Now comes the crucial point at which the corrector equation~\eqref{e.parabcorr} for~$\Chi_{m,k} = {\Chi}_{m,k}^{\kappa_m}$  is used. 
By~\eqref{e.parabcorr} and the fact that, similar to~\eqref{e.ortho.tilde}, 
\begin{align*}
\nabla \cdot (  \psi_{m,k} \sigma \nabla\Chi_{m,k}^{\kappa_m}) = \nabla \psi_{m,k} \cdot \nabla^\perp \Chi_{m,k}^{\kappa_m} = 0
\,, 
\end{align*}
we can rewrite the first term of the right side of~\eqref{e.push.it.forward} as
\begin{align*}
\bigl( 
\nabla \cdot \bigl( \hat{\zeta}_{m,l_k} \zeta_{m,k}{\psi}_{m,k} \sigma
+ \kappa_m \nabla {\Chi}_{m,k} \bigr) 
\bigr) \circ X_{m-1,l_k}^{-1}
=
\partial_t  {\Chi}_{m,k} \circ X_{m-1,l_k}^{-1}
\,.
\end{align*}
Combining the previous displays, we obtain 
\begin{align}
\label{e.divtwist}
\lefteqn{
\nabla \cdot \bigl( 
\bigl( \kappa_m \Itwo + \hat{\zeta}_{m,l_k} \zeta_{m,k}\tilde{\psi}_{m,k} \sigma\bigr)
\bigl( \Itwo +  \nabla \tilde{\Chi}_{m,k} \bigr)
\bigr) 
} \qquad  & 
\notag \\ &
=
\partial_t  {\Chi}_{m,k} \circ X_{m-1,l_k}^{-1}
+\nabla \cdot \bigl( 
\kappa_m 
\bigl( \nabla X_{m-1,l_k}^{-1}- \Itwo \bigr)
\nabla {\Chi}_{m,k} \circ X_{m-1,l_k}^{-1}
\bigr)
\notag \\ & \qquad
+\nabla \cdot
\bigl( 
\bigl( \Itwo-
\nabla X_{m-1,l_k} \circ X_{m-1,l_k}^{-1} 
\bigr)
\bigl( 
\hat{\zeta}_{m,l_k} \zeta_{m,k}{\psi}_{m,k} \sigma
+ \kappa_m \nabla {\Chi}_{m,k} 
\bigr) \circ X_{m-1,l_k}^{-1} 
\bigr)
\,.
\end{align}
Finally, using~\eqref{e.divtwist}, we can compute the divergence of the last line of~\eqref{e.fourlinesofcutoff}, which is also equal to negative of the first line of~\eqref{e.initial.plugin}:
\begin{align}
\label{e.diffterm.firstline}
\lefteqn{
\nabla\cdot \!\!
\sum_{k\in 2\Z+1} \!\!
\xi_{m,k}
\bigl( \kappa_m \Itwo + \hat{\zeta}_{m,l_k}
 \zeta_{m,k}\tilde{\psi}_{m,k} \sigma\bigr)
\bigl( \Itwo + \nabla \tilde{\Chi}_{m,k} \bigr)
\nabla \bigl(  T_{m-1} \circ X_{m-1,k} \bigr) \circ X_{m-1,k}^{-1} 
} \quad & 
\notag \\ & 
=
\sum_{k\in 2\Z+1}
\xi_{m,k}
\partial_t \Chi_{m,k} \circ X_{m-1,k}^{-1}
\cdot 
\nabla \bigl(  T_{m-1} \circ X_{m-1,l_k} \bigr) \circ X_{m-1,l_k}^{-1} 
\notag \\ & \quad
+
\sum_{k\in 2\Z+1}
\xi_{m,k}
\bigl( \nabla \cdot \bigl( 
\kappa_m 
\bigl( \nabla X_{m-1,k}^{-1}- \Itwo \bigr)
\nabla {\Chi}_{m,k} \circ X_{m-1,k}^{-1}
\bigr)
\bigr)
\cdot 
\nabla \bigl(  T_{m-1} \circ X_{m-1,l_k} \bigr) \circ X_{m-1,l_k}^{-1} 
\notag \\ & \quad
+
\sum_{k\in 2\Z+1}
\xi_{m,k}
\nabla \cdot
\bigl( 
\bigl( \Itwo-
\nabla X_{m-1,l_k} \circ X_{m-1,l_k}^{-1} 
\bigr)
\bigl( 
\hat{\zeta}_{m,l_k} \zeta_{m,k}{\psi}_{m,k} \sigma
+ \kappa_m \nabla {\Chi}_{m,k} 
\bigr) \circ X_{m-1,l_k}^{-1} 
\bigr)
\notag \\ & \quad \qquad \qquad 
\cdot 
\nabla \bigl(  T_{m-1} \circ X_{m-1,l_k} \bigr) \circ X_{m-1,l_k}^{-1} 
\notag \\ & \quad
+
\sum_{k\in 2\Z+1}
\xi_{m,k}
\bigl( \kappa_m \Itwo + \hat{\zeta}_{m,l_k} \zeta_{m,k}\tilde{\psi}_{m,k} \sigma\bigr)
\bigl( \Itwo + \nabla \tilde{\Chi}_{m,k} \bigr)
\!:\!
\sum_{l\in\Z} \hat{\xi}_{m,l} 
\nabla \bigl( \nabla \bigl(  T_{m-1} \circ X_{m-1,l} \bigr) \circ X_{m-1,l}^{-1} \bigr)
\,.
\end{align}
This expression will be substituted for the first line of~\eqref{e.initial.plugin}, in view of~\eqref{e.fourlinesofcutoff}, which will then be substituted for the second term of~\eqref{e.operator.splitting}. The second and third terms on the right side of~\eqref{e.diffterm.firstline} are acceptable errors which are estimated in the next subsection. The first term on the right side of~\eqref{e.diffterm.firstline} will cancel the second term on the on the right side of~\eqref{e.timecomp} when we combine~\eqref{e.timecomp} and~\eqref{e.initial.plugin}. The last term will be combined with the term involving~$\J_m$ on the last line of~\eqref{e.timecomp}, which centers its mean and renders the resulting expression an acceptable error.

\subsubsection{Formula for the error of the ansatz}

We now combine~\eqref{e.timecomp} and~\eqref{e.initial.plugin},~\eqref{e.fourlinesofcutoff} and~\eqref{e.diffterm.firstline} to obtain an explicit expression for $(\partial_t - \kappa_m \Delta + \b_m \cdot \nabla ) \tilde{\theta}_m$. 
As we merge~\eqref{e.timecomp} and~\eqref{e.initial.plugin}, 
we recall that the second line of~\eqref{e.timecomp} cancels the first line of~\eqref{e.diffterm.firstline}, and that the last line of~\eqref{e.timecomp} can be nicely combined with the last line of~\eqref{e.diffterm.firstline}. 

\smallskip

The result is the following equation satisfied by the ansatz~$\tilde{\theta}_m$:
\label{e.monster}
\begin{align}
& \bigl(
\partial_t - \kappa_m \Delta + \b_m \cdot \nabla 
\bigr) \tilde{\theta}_m
\notag \\ & \ 
\label{e.monster.cutoff1}
=
\sum_{k\in 2\Z+1}\!\!
\bigl(
\partial_t\xi_{m,k} 
\tilde{\Chi}_{m,k} \bigr)
\cdot 
\nabla \bigl(  T_{m-1} \circ X_{m-1,l_k} \bigr) \circ X_{m-1,l_k}^{-1} 
\\ & \quad 
\label{e.monster.twistie1}
+
\sum_{k\in 2\Z+1}\!\!
\xi_{m,k} 
\tilde{\Chi}_{m,k}
\cdot 
\bigl( \nabla X_{m-1,l_k}\circ X_{m-1,l_k}^{-1} \bigr) 
\nabla \bigl(  \nabla \cdot 
\bigl( \K_m + \mathbf{s}_{m-1} \bigr) \nabla T_{m-1} 
\bigr)
\\ & \quad 
\label{e.monster.twistie3}
- \nabla \cdot
\sum_{l\in \Z}
\hat{\xi}_{m,l}
\bigl( \kappa_m \Itwo + \tilde{\psi}_m \sigma \bigr) 
\bigl( 
\nabla  T_{m-1} - \nabla \bigl(  T_{m-1} \circ X_{m-1,l} \bigr) \circ X_{m-1,l}^{-1} 
\bigr)
\\ & \quad
\label{e.monster.twistie4}
-
\sum_{k\in 2\Z+1}\!\!
\xi_{m,k}
\nabla \cdot \Bigl( 
\kappa_m 
\bigl( \nabla X_{m-1,l_k}^{-1}- \Itwo \bigr)
\nabla {\Chi}_{m,k} \circ X_{m-1,l_k}^{-1}
\Bigr)
\nabla \bigl(  T_{m-1} \circ X_{m-1,l_k} \bigr) \circ X_{m-1,l_k}^{-1} 
\\ & \quad 
\label{e.monster.twistie5}
-
\sum_{k\in 2\Z+1}\!\!
\xi_{m,k}
\nabla \cdot
\Bigl( 
\bigl( \Itwo-
\nabla X_{m-1,l_k} \circ X_{m-1,l_k}^{-1} 
\bigr)
\bigl( 
\hat{\zeta}_{m,l_k}
\zeta_{m,k}\tilde{\psi}_{m,k} \sigma
+ \kappa_m \nabla {\Chi}_{m,k} 
\circ X_{m-1,l_k}^{-1} \bigr)
\Bigr)
\notag \\ & \quad \qquad\qquad \times
\nabla \bigl(  T_{m-1} \circ X_{m-1,l_k} \bigr) \circ X_{m-1,l_k}^{-1} 
\\ & \quad
\label{e.monster.normie1}
- \nabla \cdot \!\!
\sum_{k\in 2\Z+1}\!\!
\xi_{m,k} 
\bigl( \kappa_m \Itwo + \tilde{\psi}_m \sigma \bigr) 
\tilde{\Chi}_{m,k}
\nabla 
\bigl( \nabla \bigl(  T_{m-1} \circ X_{m-1,l_k} \bigr) \circ X_{m-1,l_k}^{-1}\bigr) 
\\ & \quad
\label{e.monster.normie2}
-\nabla \cdot 
\bigl( \kappa_m \Itwo + \tilde{\psi}_m \sigma \bigr) 
\nabla \widetilde{H}_{m} 
\\ & 
\label{e.monster.tiny}
\quad
+ \nabla \cdot
\bigl( \mathbf{d}_m + \mathbf{e}_{m-1} \bigr)
+
\sum_{k\in 2\Z+1}\!\!
\xi_{m,k} 
\tilde{\Chi}_{m,k}
\bigl( \nabla X_{m-1,l_k}\circ X_{m-1,l_k}^{-1} \bigr) 
\cdot 
\nabla \bigl(  \nabla \cdot  \mathbf{e}_{m-1}\bigr)
\\ & \quad
\label{e.monster.normie3}
- \!\!\!
\sum_{k\in 2\Z+1} \!\!\!
\xi_{m,k}
\Bigl( \J_m {-} \bigl( \kappa_m \Itwo {+} \hat{\zeta}_{m,l_k}
 \zeta_{m,k}\tilde{\psi}_{m,k} \sigma\bigr)
\bigl( \Itwo + \nabla \tilde{\Chi}_{m,k} \bigr) 
\Bigr)
\!:\!
\sum_{l\in\Z} \hat{\xi}_{m,l} 
\nabla \bigl( \nabla \bigl(  T_{m-1} {\circ} X_{m-1,l} \bigr) {\circ} X_{m-1,l}^{-1} \bigr)
\,.
\end{align}
The term~\eqref{e.monster.cutoff1} is due to the time cutoff~$\xi_{m,k}$ and is extremely small by~\eqref{e.corrbounds.Chi.decay}, as we will see. 
The terms in~\eqref{e.monster.twistie1}--\eqref{e.monster.twistie5} are errors caused by the ``twisting'' introduced into our ansatz by the composing~$T_{m-1}$ with the flows~$X_{m-1,l}$ and inverse flows~$X_{m-1,l}^{-1}$; these will be controlled using the bounds on the flows and inverse flows we proved in Section~\ref{s.construction}: see~\eqref{e.Xm.regbounds} and~\eqref{e.Xm.bound.1}. 
The terms~\eqref{e.monster.normie1},~\eqref{e.monster.normie2} and~\eqref{e.monster.normie3} are ``routine homogenization errors'' which are expected: these are small because the correctors are small, as we find in error estimates in classical homogenization theory. 
Finally, the very tiny error terms in~\eqref{e.monster.tiny} is reflection of the way we constructed~$T_{m-1}$ and~$\widetilde{H}_m$, which actually solve the equations~\eqref{e.Tm.true} and~\eqref{e.real.eq.Hm}, respectively, which are slight perturbations of the equations~\eqref{e.Tm} and~\eqref{e.Hm} we initially wanted them to solve. 

\smallskip

In the next subsection, we will show that the~$L^2_t\dot{H}^{-1}_x$ norm of each of the numerous error terms on the right side of the expression above is at most~$O(\kappa_m^{\nicefrac12}\ep_m^\delta)$. This will subsequently permit us to deduce that the~$L^\infty_tL^2_x$ difference of the ansatz~$\tilde{\theta}_m$ and the true solution~$\theta_m$ of~\eqref{e.theta.m} is small, and thus to compare the decay in time of the~$L^2_x$ norm of~$\theta_m$ with that of~$\theta_{m-1}$, up to a suitable error.

\subsection{Estimates of the nine error terms}
\label{ss.estimate.error}

In this subsection, we prove the existence of a constant~$C(\beta)<\infty$ such that, if the parameter~$\Lambda$ satisfies
\begin{equation}
\label{e.Lambda.restriction}
\Lambda\geq C
\end{equation}
then, for every~$m\in\N$ satisfying  
\begin{equation}
\label{e.m.restriction}
m \geq m_{\theta_0}:=
\inf \Bigl \{ n\in \N \,:\, n\geq 2\,, \ 
\ep_{n-1}^{1+ \nicefrac{\gamma}{2}} \leq R_{\theta_0} \Bigr \}\,,
\end{equation}
we have the estimate
\begin{align}
\label{e.bigbound}
 \kappa_m^{-\nicefrac12} \bigl\| 
\bigl(\partial_t - \kappa_m \Delta + \b_m \cdot \nabla \bigr) \tilde{\theta}_m
\bigr\|_{L^2((0,1);\dot{H}^{-1}(\TT^2))}
\leq 
C\ep_{m-1}^{\delta} 
\left\| \theta_{0} \right\|_{L^2( \TT^2)}
\,.
\end{align}
The proof of~\eqref{e.bigbound} amounts to showing that each of the error terms in~\eqref{e.monster.cutoff1}--\eqref{e.monster.normie3} can be estimated by the right side of~\eqref{e.bigbound}.

\begin{remark}
\label{r.Lambda.m.conds}
The role of the restrictions~$\Lambda\geq C$ and~$m\leq m_{\theta_0}$ is to ensure that the conditions in~\eqref{e.ep.m-1.small} are valid and, therefore, all of the estimates on~$T_{m-1}$ and~$\widetilde{H}_m$ proved in Section~\ref{s.multiscale} are in force. 
\end{remark}

\subsubsection{Application of ergodic lemma for controlling nondivergence form terms}

For each of the four nondivergence terms in~\eqref{e.monster.twistie1},~\eqref{e.monster.twistie4},~\eqref{e.monster.twistie5} and~\eqref{e.monster.normie3}, we are faced with estimating the~$L^2_t \dot{H}^{-1}_x$ norm of a function of the form~$f (g\circ Z^{-1})$, which is the product of a ``fast'' function~$g\circ Z^{-1}$, the composition of an~$\ep_m$--periodic, mean-zero function~$g$ with a smooth inverse flow~$Z^{-1}$ which has analyticity radius $\ep_{m-1}$, and a ``slow'' function~$f$ which is~$1$--periodic, smooth and has analyticity radius of order~$\ep_{m-1}$. 
We expect that a weak norm of~$f (g\circ Z^{-1})$ would inherit smallness from the relatively fast oscillations of~$g$, which modulate the slower signal~$f$, up to an error which depends on the scale separation between~$\ep_m$ and~$\ep_{m-1}$. Thanks to the periodicity of~$g$, this error turns out to be exponentially small in the ratio~$\ep_{m-1} / \ep_m$, as we show in Appendix~\ref{a.ergodic}. 
Roughly, what we have is that 
\begin{align}
\label{e.roughly.exponential}
\lefteqn{
\bigl\| f (g\circ Z^{-1}) \bigr\|_{L^2((0,1);\dot{H}^{-1}(\TT^2))}
} \qquad & 
\notag \\ & 
\leq
C \ep_{m} 
\bigl\| f \bigr\|_{L^2((0,1)\times\TT^2)} 
\bigl\| g \bigr\|_{L^2((0,1)\times\TT^2)}
+
C \bigl\| g \bigr\|_{L^2((0,1)\times\TT^2)}
\exp\biggl( 
- \frac{\ep_{m-1}}{C\ep_m}
\bigg)
\,.
\end{align}
The exact statement can be found in Remark~\ref{r.Hminusone.flow}. 

\smallskip

The first term in~\eqref{e.roughly.exponential} represents the scaling of the~$\dot{H}^{-1}(\TT^2)$ norm compared to the~$L^2(\TT^2)$ norm for a mean-zero,~$\ep_{m}$--periodic function. For instance, in the case~$f\equiv1$ and~$Z=\Id$ and~$g$ is constant in time, we  have~$\| g \|_{\dot{H}^{-1}(\TT^2)} \simeq \ep_{m} \bigl\| g \bigr\|_{L^2((0,1)\times\TT^2)}$. So essentially what~\eqref{e.roughly.exponential} says in fact is that the modulation by~$f$ and~$Z$ does not alter this estimate, up to an error which is exponentially small in the scale separation. The constants~$C$ in the second term on the right of~\eqref{e.roughly.exponential} depend on the appropriate analyticity norms of~$f$ and~$Z$.

\smallskip

Here we check the applicability of this estimate to the terms~\eqref{e.monster.twistie1},~\eqref{e.monster.twistie4},~\eqref{e.monster.twistie5} and~\eqref{e.monster.normie3}. In all cases, the role of the measure-preserving mapping~$Z$ is played by~$X_{k-1,m}$ and it needs only to be estimated for times in the support of the cutoff function~$\xi_{m,k}$. 
The required analyticity bounds are then a consequence of Corollary~\ref{c.flowreg}, in particular the estimate~\eqref{e.Xm.bound.4}, which implies, in view of~\eqref{e.taum.def} and the definition of~$\xi_{m,k}$, 
\begin{equation}
\label{e.flow.for.ergodic}
\sup_{t \in \supp \xi_{m,k} } \max_{|\aa|=n} \norm{ \partial^\aa X_{m-1,k} (t,\cdot)}_{L^\infty(\R^2)} 
\leq C n! (C \ep_{m-1}^{-1})^{n-1}\,,  \qquad \forall n\in\N\,.
\end{equation}
The choices of~$f$ and~$g$ we need to make are different in each estimate, but in every situation we show that our choice of~$f$ satisfies, for some constant~$C<\infty$, 
\begin{equation}
\label{e.checking.f.ergodic}
\sup_{t \in \supp \xi_{m,k} } \max_{|\aa|=n} \norm{ \partial^\aa f (t,\cdot)}_{L^\infty(\R^2)} 
\leq C n! (C \ep_{m-1}^{-1})^{n+1}\,,  \qquad \forall n\in\N\,.
\end{equation}
We then may apply Remark~\ref{r.Hminusone.flow} with $R=r=C_X=C\ep_{m-1}^{-1}$ and $N=\ep_{m}^{-1}$. The result~\eqref{e.Hminusone.ergodic} then yields
\begin{align}
\label{e.what.ergodic.gives}
\lefteqn{
\bigl\| f (g\circ Z^{-1}) \bigr\|_{L^2((0,1);\dot{H}^{-1}(\TT^2))}
} \qquad & 
\notag \\ & 
\leq
C \ep_{m} 
\bigl\| f \bigr\|_{L^2((0,1)\times\TT^2)} 
\bigl\| g \bigr\|_{L^2((0,1)\times\TT^2)}
+
C \bigl\| g \bigr\|_{L^2((0,1)\times\TT^2)} \ep_{m-1}^{-1} 
\exp\biggl( 
- \frac{\ep_{m-1}}{C\ep_m}
\bigg)
\notag \\ & 
\leq 
C \ep_{m} 
\bigl\| f \bigr\|_{L^2((0,1)\times\TT^2)} 
\bigl\| g \bigr\|_{L^2((0,1)\times\TT^2)}
+
C \ep_{m-1}^{500} \bigl\| g \bigr\|_{L^2((0,1)\times\TT^2)} 
\,.
\end{align}
In the last line we used that, by~\eqref{e.supergeo}, 
\begin{align*}
\ep_{m-1}^{-1} 
\exp\biggl( - \frac{\ep_{m-1}}{C\ep_m} \bigg)
\leq 
\ep_{m-1}^{-1} 
\exp\bigl( - c \ep_{m-1}^{-(q-1)} \bigr)
\leq 
C \ep_{m-1}^{500}\,.
\end{align*}

\smallskip

For~\eqref{e.monster.twistie1}, we use the choices
\begin{equation}
\label{e.fg.choice.1}
f = 
\xi_{m,k}
\bigl( \nabla X_{m-1,l_k}\circ X_{m-1,l_k}^{-1} \bigr) 
\nabla \nabla \cdot 
\bigl( \bigl( \K_m + \mathbf{s}_{m-1} \bigr) \nabla T_{m-1}  \bigr)
\qquad \mbox{and} \qquad
g= {\Chi}_{m,k}
\,.
\end{equation}
The desired bound for~$f$ in~\eqref{e.checking.f.ergodic} is a consequence of~\eqref{e.Xm.bound.2}, the bounds for~$\mathbf{s}_{m-1}$ in~\eqref{e.smbound.dervs}, the estimate for $\kappa_{m-1}$ in~\eqref{e.kappam.bound}, Lemma~\ref{l.Tm.reg.upgrade} and the product estimate of Lemma~\ref{l.product}.  

\smallskip

The term~\eqref{e.monster.twistie4} is not in divergence form, nor is it entirely nondivergence form. Therefore, we split it into the sum of two terms, one in nondivergence form and the other in divergence form: 
\begin{align}
\label{e.monster.twistie4.split}
\text{\eqref{e.monster.twistie4}}
&
=
\!\! \sum_{k\in 2\Z+1}\!\!\xi_{m,k}
\kappa_m 
\bigl( \nabla X_{m-1,l_k}^{-1}- \Itwo \bigr)
\nabla {\Chi}_{m,k} \circ X_{m-1,l_k}^{-1}
\bigr)
\nabla \bigl( \nabla \bigl(  T_{m-1} {\circ} X_{m-1,l_k} \bigr) {\circ} X_{m-1,l_k}^{-1} \bigr)
\notag \\ & \quad 
-
\nabla \cdot \! \!\!
\sum_{k\in 2\Z+1}\!\!\xi_{m,k}
\kappa_m 
\bigl( \nabla X_{m-1,l_k}^{-1} {-}\, \Itwo \bigr)
\bigl( \nabla {\Chi}_{m,l_k} {\circ} X_{m-1,l_k}^{-1} \bigr) 
\nabla \bigl(  T_{m-1} {\circ} X_{m-1,l_k} \bigr) {\circ} X_{m-1,l_k}^{-1} 
\,.
\end{align}
Only the first term in nondivergence form requires the use of the ergodic lemma. We apply it with the choices
\begin{equation}
\label{e.fg.choice.2}
f = 
\xi_{m,k}
\kappa_m 
\bigl( \nabla X_{m-1,k}^{-1}- \Itwo \bigr) 
\nabla \bigl( 
\nabla \bigl(  T_{m-1} \circ X_{m-1,k} \bigr) \circ X_{m-1,k}^{-1} \bigr) 
\qquad  \mbox{and} \qquad
g  = \nabla {\Chi}_{m,k}\,.
\end{equation}  
Recall that 
\begin{equation}
\label{e.expandTmXXinv}
\nabla \bigl(  T_{m-1} \circ X_{m-1,k} \bigr) \circ X_{m-1,k}^{-1} 
=
\bigl( \nabla X_{m-1,k} \circ X_{m-1,k}^{-1} \bigr) \nabla T_{m-1}
\,.
\end{equation}
Therefore the desired bound for~$f$ in~\eqref{e.checking.f.ergodic} is a consequence of~\eqref{e.Xm.regbounds},~\eqref{e.Xm.bound.2} and the product estimate of Lemma~\ref{l.product}.  

\smallskip

We next consider the term~\eqref{e.monster.twistie5} which, as for~\eqref{e.monster.twistie4}, must be split since it is not completely in nondivergence form. In order to write our expressions more compactly, we denote
\begin{equation}
\label{e.Ymk}
Y_{m-1,l}:= \Itwo - \nabla X_{m-1,l} \circ X_{m-1,l}^{-1} 
\end{equation}
and we write~\eqref{e.monster.twistie5} as 
\begin{align}
\label{e.monster.twistie5.split}
\text{\eqref{e.monster.twistie5}}
&
=
\sum_{k\in 2\Z+1}\!\!
\xi_{m,k}
Y_{m-1,l_k}
\bigl( 
\zeta_{m,k}\tilde{\psi}_{m,k} \sigma
+ \kappa_m \nabla {\Chi}_{m,k} 
{\circ} X_{m-1,l_k}^{-1} \bigr)
\nabla \bigl(
\nabla \bigl(  T_{m-1} {\circ} X_{m-1,l_k} \bigr) {\circ} X_{m-1,l_k}^{-1} 
\bigr)
\notag \\ & \qquad
-
\nabla \cdot \!\! \!\!
\sum_{k\in 2\Z+1}\!\!
\xi_{m,k}
Y_{m-1,k}
\bigl( 
\zeta_{m,k}\tilde{\psi}_{m,k} \sigma
+ \kappa_m \nabla {\Chi}_{m,k} 
{\circ} X_{m-1,l_k}^{-1} \bigr)
\nabla \bigl(  T_{m-1} {\circ} X_{m-1,l_k} \bigr) {\circ} X_{m-1,l_k}^{-1} 
\,.
\end{align}
Only the first term on the right of~\eqref{e.monster.twistie5.split} will require the use of the ergodic lemma, and we will apply it with the choice
\begin{equation}
\label{e.fg.choice.3}
\left\{ 
\begin{aligned}
& f = 
\xi_{m,k}
\bigl( \Itwo - \nabla X_{m-1,l_k} \circ X_{m-1,l_k}^{-1} \bigr) \nabla \bigl(
\nabla \bigl(  T_{m-1} \circ X_{m-1,l_k} \bigr) \circ X_{m-1,l_k}^{-1} 
\bigr), \quad \mbox{and} 
\\ & 
g  = \zeta_{m,k}{\psi}_{m,k} \sigma
+ \kappa_m \nabla {\Chi}_{m,k} 
\,.
\end{aligned}
\right.
\end{equation}  
In view of~\eqref{e.expandTmXXinv}, the desired analyticity estimate~\eqref{e.checking.f.ergodic} for~$f$ is a consequence of estimate~\eqref{e.Xm.bound.2} of Corollary~\ref{c.flowreg}, Lemma~\ref{l.Tm.reg.upgrade} and the product estimate of Lemma~\ref{l.product}.

\smallskip

Finally, for~\eqref{e.monster.normie3}, we apply the ergodic lemma (with the time variable frozen) with
\begin{equation}
\label{e.fg.choice.4} 
\left\{
\begin{aligned}
& f = 
\xi_{m,k}
\hat{\xi}_{m,l}
\nabla \bigl( 
\nabla \bigl(  T_{m-1} \circ X_{m-1,l} \bigr) \circ X_{m-1,l}^{-1} \bigr) \,, 
\\ &  
g  = 
\J_m - 
\bigl( \kappa_m \Itwo + \hat{\zeta}_{m,l_k} \zeta_{m,k}{\psi}_{m,k} \sigma\bigr)
\bigl( \Itwo + \nabla {\Chi}_{m,k} \bigr) \,.
\end{aligned}
\right.
\end{equation} 
In view of~\eqref{e.expandTmXXinv}, the desired analyticity estimate~\eqref{e.checking.f.ergodic} for~$f$ is a consequence of estimate~\eqref{e.Xm.bound.2} of Corollary~\ref{c.flowreg}, Lemma~\ref{l.Tm.reg.upgrade} and the product estimate of Lemma~\ref{l.product}.

\subsubsection{A reference list of basic estimates}

The following identities and estimates, proved above and collected here for the convenience of the reader, will be used repeatedly in the estimates of the terms on the right side of the Big Display on Page~\pageref{e.monster}:
\begin{gather}
\label{e.emq}
\ep_m \simeq \ep_{m-1}^q\,,
\\
\label{e.amem}
a_m = \ep_{m}^{\beta-2}\,,
\\
\label{e.tam}
\tau_m \simeq a_{m-1}^{-1} \ep_{m-1}^{2\delta}\,,
\\
\label{e.kappam}
\kappa_m 
\simeq 
a_m\ep_m^{2+\gamma} 
= \ep_m^{\beta+\gamma} 
= \ep_m^{\frac{2q}{q+1}\beta}\,,
\\
\label{e.expqunt}
\frac{\ep_m^2}{\kappa_m \tau_m} 
\lesssim \ep_{m-1}^{2\delta}\,,
\\
\label{e.corrm}
\left\| \Chi^\kappa_{m,k} \right\|_{L^\infty(\R\times\TT^2)}
+ \ep_m
\left\| \nabla \Chi^\kappa_{m,k} \right\|_{L^\infty(\R \times\TT^2)}
\lesssim
\frac{a_m \ep_m^3}{\kappa_m}
\simeq
\ep_m^{1-\gamma}\,,
\\
\label{e.psikmb}
\big\| \tilde{\psi}_m \big\|_{L^\infty(\R\times\TT^2)}
\simeq  
\| {\psi}_{m,k} \|_{L^\infty(\R\times\TT^2)}
\simeq 
a_m\ep_m^2 
= 
\ep_m^{\beta}\,,
\\
\label{e.hcorrm}
\norm{\nabla \widetilde{H}_m(t,\cdot)}_{L^{2}([0,1]\times \TT^2)} 
\lesssim
\ep_{m-1}^{4\delta}
\kappa_{m-1}^{-\nicefrac 12}
\|\theta_0\|_{L^2(\TT^2)} \,,
\\
\label{e.nabm2}
\left\| \nabla^{n+1} T_{m-1} \right\|_{L^2((0,1)\times \TT^d)}
\lesssim
n! \bigl(C\ep_{m-1}^{-(1+\nicefrac\gamma2)} \bigr)^{n} 
\kappa_{m-1}^{-\nicefrac12}
\left\| \theta_{0} \right\|_{L^2(\TT^2)}\,, \quad \forall n\in\N\,.
\end{gather}
In each of the above inequalities, the symbols~$\lesssim$ and~$\simeq$ are to be interpreted as asserting inequalties and two-sided inequalities, respectively, with implicit prefactor constants which depend only on the parameter~$\beta$. 
For~\eqref{e.emq},~\eqref{e.amem} and~\eqref{e.tam}, see~\eqref{e.supergeo},~\eqref{e.am.def} and~\eqref{e.taum.def}, respectively;
for~\eqref{e.kappam} 
see~\eqref{e.kappam.bound}; 
for~\eqref{e.expqunt}, see~\eqref{e.exprat.bound};
for~\eqref{e.corrm}, see~\eqref{e.corrbounds.Chi}; 
for~\eqref{e.hcorrm}, see~\eqref{e.Hm.gradient.L2};
for~\eqref{e.psikmb}, see~\eqref{e.def.streamr}; finally, 
for~\eqref{e.nabm2}, see Lemma~\ref{l.Tm.reg.upgrade}.  

\smallskip

We present the estimates for the terms on the right side of the Big Display on Page~\pageref{e.monster}, in consecutive order. 

\subsubsection*{The estimate of~\eqref{e.monster.cutoff1}} 
The smallness of this term is a reflection of the fact that the spatial oscillations in the shear flow homogenize on time scales of order~$\ep_m^2/\kappa_m$, which is much less than~$\tau_m$ by a factor of~$\ep_{m-1}^{2\delta}$. 
While it is in nondivergence form, we can bound it brutally using an exponential factor so we do not need to use the ergodic lemma (Remark~\ref{r.Hminusone.flow}) in the appendix; an $L^2$ estimate is more than sufficient. The claim is that 
\begin{align}
\label{e.monster.est.3}
&
\biggl\| 
\sum_{k\in 2\Z+1}
( \partial_t\xi_{m,k} \tilde{\Chi}_{m,k} )\cdot 
\nabla \bigl(  T_{m-1}\circ X_{m-1,l_k} \bigr) \circ X_{m-1,l_k}^{-1} 
\biggr\|_{L^2((0,1)\times \TT^2)}
\leq
C\kappa_m^{\nicefrac12} \ep_{m-1}^{500}
\left\| \theta_{0} \right\|_{L^2( \TT^2)}
\,.
\end{align}
Observe that~\eqref{e.corrbounds.Chi.decay},~\eqref{e.cutoff.xi} and~\eqref{e.expqunt} imply that~$\tilde{\Chi}_{m,k}$ and~$\h_m$ is very small on the support of~$\partial_t \xi_{m,k}$. 
We can therefore estimate the left side of~\eqref{e.monster.est.3} by
\begin{align*}
\sup_{k\in2\N} 
\left\| \partial_t \xi_{m,k} \tilde{\Chi}_{m,k}
\right\|_{L^\infty(\R\times\TT^d)}
\left\| \nabla T_{m-1} \right\|_{L^2((0,1)\times\TT^d))}
&\leq
C 
\exp\left( -c\ep_{m-1}^{-2\delta} \right)
\left\| \nabla \theta_{m-1} \right\|_{L^2((0,1)\times\TT^d))}
\\ &
\leq 
C\kappa_{m}^{\nicefrac12}
\underbrace{
\left(
\kappa_{m}^{-\nicefrac12} \kappa_{m-1}^{-\nicefrac12} 
\exp\left( -c\ep_{m-1}^{-2\delta} \right)
\right)
}_{\,\leq\, C\ep_{m-1}^{500}}
\left\| \theta_{0} \right\|_{L^2( \TT^2)}
\,.
\end{align*}
This is~\eqref{e.monster.est.3}.
\qed

\subsubsection*{The estimates of~\eqref{e.monster.twistie1}} 
We claim that
\begin{multline}
\label{e.monster.est.5}
\biggl\| 
\sum_{k\in 2\Z+1}
\xi_{m,k} 
\tilde{\Chi}_{m,k}
\cdot
\bigl( \nabla X_{m-1,l_k}\circ X_{m-1,l_k}^{-1} \bigr) 
\nabla \nabla \cdot 
\bigl( \bigl( \kappa_{m-1} \Itwo + \mathbf{s}_{m-1} \bigr) \nabla T_{m-1}\bigr)
\biggr\|_{L^2((0,1);\dot{H}^{-1}(\TT^2))}
\\
\leq 
C \kappa_{m}^{\nicefrac12} 
C\ep_{m-1}^{8\delta}
\left\| \theta_{0} \right\|_{L^2( \TT^2)}
\,.
\end{multline}
We will use the ergodic lemma of Remark~\ref{r.Hminusone.flow}, as explained above. We apply~\eqref{e.what.ergodic.gives} with~$f$ and~$g$ chosen as in~\eqref{e.fg.choice.1} to find that the left side of~\eqref{e.monster.est.5} is bounded from above by the sum of
\begin{align*}
C \ep_{m-1}^{500}
\bigl\| g \bigr\|_{L^2((0,1)\times\TT^2)}
=
C \ep_{m-1}^{500} \bigl\| {\Chi}_{m,k} \bigr\|_{L^2((0,1)\times\TT^2)} 
\leq 
C \ep_{m-1}^{500}  \ep_m^{1-\gamma}
\leq
C \ep_{m-1}^{500},
\end{align*} 
and
\begin{align*}
\lefteqn{
C \ep_{m} 
\bigl\| f \bigr\|_{L^2((0,1)\times\TT^2)} 
\bigl\| g \bigr\|_{L^2((0,1)\times\TT^2)}
} \qquad & 
\notag \\ & 
\leq
C \ep_m
\left\| \Chi_{m,k} \right\|_{L^\infty} 
\sup_{k\in\N} \left\| \xi_{m,k} \nabla X_{m-1,l_k} \right\|_{L^\infty(\R\times\TT^2)} 
\left\|\nabla \nabla \cdot 
\bigl( \bigl(\K_m + \mathbf{s}_{m-1} \bigr) \nabla T_{m-1}\bigr)
\right\|_{L^2((0,1)\times \TT^2)}
\notag \\ & 
\leq 
C \ep_m
\cdot
\ep_{m}^{1-\gamma}
\cdot C \cdot 
\kappa_{m-1}
\ep_{m-1}^{-2-\gamma} 
\kappa_{m-1}^{-\nicefrac12} 
\left\| \theta_{0} \right\|_{L^2((0,1)\times\TT^2)}
\notag \\ & 
=
C \kappa_{m}^{\nicefrac12} 
\underbrace{
\biggl( \frac{\kappa_{m-1}}{\kappa_m} \biggr)^{\!\nicefrac12} 
\biggl( \frac{\ep_m}{\ep_{m-1}} \biggr)^{\!2}
\ep_m^{-\gamma}\ep_{m-1}^{-\gamma}
}_{\,\leq \, C\ep_{m-1}^{8\delta }}
\left\| \theta_{0} \right\|_{L^2( \TT^2)}
\leq 
C \kappa_{m}^{\nicefrac12} 
C\ep_{m-1}^{8\delta}
\left\| \theta_{0} \right\|_{L^2( \TT^2)}
\,.
\end{align*}
In the previous display, we used the size and regularity estimates for~$\s_{m-1}$ in~\eqref{e.smbound} and~\eqref{e.smbound.dervs} as well as~\eqref{e.emq},~\eqref{e.kappam},~\eqref{e.corrm},~\eqref{e.Xm.bound.1} and~\eqref{e.nabm2} and the fact that
\begin{equation*}
\underbrace{\left( \frac{\kappa_{m-1}}{\kappa_m} \right)^{\!\nicefrac12} }_{\,\leq\, C \ep_{m-1}^{-\gamma\beta}\!}\!
\underbrace{
\biggl( \frac{\ep_m}{\ep_{m-1}} \biggr)^{\!2}}_{\,\leq\, C\ep_{m-1}^{2(q-1)}}
\!\!\underbrace{ \phantom{\bigg(}\!\!
\ep_m^{-\gamma}\ep_{m-1}^{-\gamma}
}_{\,\leq\, C \ep_{m-1}^{-(q-1)\beta}}
\leq
C \ep_{m-1}^{8\delta} 
\end{equation*}
since,  by~\eqref{e.delta} and~\eqref{e.gamma},
\begin{equation*}
2(q-1) - \gamma\beta - (q-1) \beta 
=
(q-1) \left( 2 - \frac{\beta}{q+1}\left( 2q + 1 \right) \right) 
=
2(q-1) \left( 1 - \frac{2q+1}{2q+2} \beta \right) 
=8\delta
\,.
\end{equation*}
This completes the proof of~\eqref{e.monster.est.5}. 
\qed

\subsubsection*{The estimate of~\eqref{e.monster.twistie3}} 
We will show that 
\begin{multline}
\label{e.monster.est.7}
\biggl\| 
\sum_{k\in 2Z+1}
\xi_{m,k}
\bigl( \kappa_m \Itwo + \tilde{\psi}_m \sigma \bigr) 
\bigl( 
\nabla  T_{m-1}- \nabla \bigl(  T_{m-1}\circ X_{m-1,l_k} \bigr) \circ X_{m-1,l_k}^{-1} 
\bigr)
\biggr\|_{L^2((0,1)\times \TT^2)}
\\
\leq 
C \kappa_m^{\nicefrac12} \ep_{m-1}^{2\delta} 
\left\| \theta_{0} \right\|_{L^2( \TT^2)}
\,.
\end{multline}
Using the formula 
\begin{align*}
\nabla  T_{m-1}- \nabla \bigl(  T_{m-1}\circ X_{m-1,l_k} \bigr) \circ X_{m-1,l_k}^{-1} 
=
\bigl( \Itwo - \nabla X_{m-1,l_k}\circ X_{m-1,l_k}^{-1} \bigr) \nabla T_{m-1}
\end{align*}
and the estimates~\eqref{e.psikmb},~\eqref{e.Xm.bound.2} and~\eqref{e.Tm.divform} with $n=0$, we can bound the left side of~\eqref{e.monster.est.7} from above by
\begin{multline*}
\big\| \kappa_m \Itwo + \tilde{\psi}_m\sigma \big\|_{L^\infty(\R\times\TT^2)}
\sup_{k\in\N}
\big\| \xi_{m,k}\bigl(\Itwo - \nabla X_{m-1,l_k}\bigr) \big\|_{L^\infty(\R\times\TT^2)}
\left\| \nabla T_{m-1} \right\|_{L^2((0,1)\times\TT^2)}
\\
\leq
\underbrace{C \bigl( \kappa_m + a_m\ep_m^2\bigr)}_{\, \leq \, C \kappa_m^{\nicefrac12}\kappa_{m-1}^{\nicefrac12} 
} 
\ep_{m-1}^{2\delta} 
\left\| \nabla T_{m-1} \right\|_{L^2((0,1)\times\TT^2)}
\leq
C\kappa_m^{\nicefrac12} \ep_{m-1}^{2\delta} 
\left\| \theta_{0} \right\|_{L^2( \TT^2)}
\,.
\end{multline*}
This yields~\eqref{e.monster.est.7}. 
\qed

\subsubsection*{The estimate of~\eqref{e.monster.twistie4}} 
We use the identity~\eqref{e.monster.twistie4.split} to split~\eqref{e.monster.twistie4} into a divergence-form part and a nondivergence form part. 
The claimed estimates for these are as follows: 
\begin{multline}
\label{e.monster.est.8a}
\biggl\|
\sum_{k\in 2\Z+1}
\xi_{m,k} 
\kappa_m 
\bigl( \nabla X_{m-1,l_k}^{-1}- \Itwo \bigr)
\nabla {\Chi}_{m,k} \circ X_{m-1,l_k}^{-1}
\nabla \bigl(  T_{m-1}\circ X_{m-1,l_k} \bigr) \circ X_{m-1,l_k}^{-1} 
\biggr\|_{L^2((0,1)\times \TT^2)} 
\\
\leq
C\kappa_m^{\nicefrac12} \ep_{m-1}^{2\delta} 
\left\| \theta_{0} \right\|_{L^2( \TT^2)}
\end{multline}
and
\begin{multline}
\label{e.monster.est.8b}
\biggl\|
\sum_{k\in 2\Z+1}
\xi_{m,k} 
\kappa_m 
\bigl( \nabla X_{m-1,l_k}^{-1}- \Itwo \bigr)
\nabla {\Chi}_{m,k} \circ X_{m-1,l_k}^{-1}
\nabla \bigl( \nabla \bigl(  T_{m-1}\circ X_{m-1,l_k} \bigr) \circ X_{m-1,l_k}^{-1} \bigr)
\biggr\|_{L^2((0,1);\dot{H}^{-1}(\TT^2))} 
\\
\leq
C\kappa_m^{\nicefrac12} \ep_{m-1}^{2\delta} 
\left\| \theta_{0} \right\|_{L^2( \TT^2)}
\,.
\end{multline}
For the $L^2$ estimate~\eqref{e.monster.est.8a}, we use~\eqref{e.Xm.regbounds},~\eqref{e.corrm} and~\eqref{e.nabm2} to bound the left side from above by
\begin{align}
\label{e.L2.monster8a}
& 
\sup_{k\in\N}
\left\| \xi_{m,k} \bigl( \nabla X_{m-1,l_k}^{-1} - \Itwo \bigr) \right\|_{L^\infty(\R\times\TT^2)}
\kappa_m
\left\| \nabla \Chi_{m,k} \right\|_{L^\infty(\R\times\TT^2)}
\bigl\| \nabla \bigl(  T_{m-1}\circ X_{m-1,l_k} \bigr) \circ X_{m-1,l_k}^{-1}  \bigr\|_{L^2((0,1)\times\TT^2)}
\notag \\ & \quad
\leq 
C\ep_{m-1}^{2\delta} 
\kappa_m \ep_m^{-\gamma} 
\left\| \nabla T_{m-1} \right\|_{L^2((0,1)\times\TT^2)}
\notag \\ & \quad
= 
C\kappa_{m}^{\nicefrac12} \ep_{m-1}^{2\delta}
\underbrace{
\left(
\frac{\kappa_m}{\kappa_{m-1}}
\right)^{\nicefrac12} 
\ep_m^{-\gamma} 
}_{\,\leq\, C \, \text{by\,\eqref{e.kappam}\,\&\,\eqref{e.gamma}} }
\left\| \theta_{0} \right\|_{L^2( \TT^2)}
\leq
C\kappa_m^{\nicefrac12} \ep_{m-1}^{2\delta} 
\left\| \theta_{0} \right\|_{L^2( \TT^2)}
\,.
\end{align}
This completes the proof of~\eqref{e.monster.est.8a}. 

\smallskip

For~\eqref{e.monster.est.8b}, we use the ergodic lemma of Remark~\ref{r.Hminusone.flow}, as explained above. We apply~\eqref{e.what.ergodic.gives} with~$f$ and~$g$ chosen as in~\eqref{e.fg.choice.2} to find that the left side of~\eqref{e.monster.est.8b} is bounded from above by the sum of
\begin{equation*}
C \ep_{m-1}^{500} \bigl\| \nabla {\Chi}_{m,k} \bigr\|_{L^2((0,1)\times\TT^2)} 
\leq 
C \ep_{m-1}^{500}  \ep_m^{-\gamma}
\leq
C \ep_{m-1}^{499},
\end{equation*} 
and
\begin{align*}
&
C \ep_{m} 
\bigl\| \xi_{m,k}
\kappa_m 
\bigl( \nabla X_{m-1,l_k}^{-1}- \Itwo \bigr) 
\nabla \bigl( 
\nabla \bigl(  T_{m-1} \circ X_{m-1,l_k} \bigr) \circ X_{m-1,l_k}^{-1} \bigr)  \bigr\|_{L^2((0,1)\times\TT^2)} 
\bigl\| \nabla {\Chi}_{m,k} \bigr\|_{L^2((0,1)\times\TT^2)}
\\ & \qquad
\leq
C\ep_m\kappa_m
\sup_{k\in\N}
\bigl\| \xi_{m,k} \bigl( \nabla X_{m-1,l_k}^{-1} - \Itwo \bigr) \bigr\|_{L^\infty(\R\times\TT^2)}
\bigl\| \nabla \Chi_{m,k} \bigr\|_{L^\infty(\R\times\TT^2)}
\notag \\ & \qquad\qquad  
\times
\bigl\| \nabla \bigl( \nabla \bigl(  T_{m-1}\circ X_{m-1,l_k} \bigr) \circ X_{m-1,l_k}^{-1} \bigr) \bigr\|_{L^2((0,1)\times\TT^2)}\,.
\end{align*}
The latter term is almost the same as in~\eqref{e.L2.monster8a}, the only differences being that we have one extra derivative on the $T_m$ expression, which costs exactly $\ep_{m-1}^{-(1+\nicefrac \gamma2)}$, and we also have an extra factor of~$\ep_m = \ep_{m-1}^{q}$. This more than compensates, since, in view of~\eqref{e.gamma},
\begin{equation}
\label{e.kill.half.gamma}
q - 1 > (q-1) \frac{\beta}{2} > \gamma > \frac 12 \gamma \,.
\end{equation}
This proof of~\eqref{e.monster.est.8b} is therefore complete. 
\qed

\subsubsection*{The estimate of~\eqref{e.monster.twistie5}}  
We will prove that 
\begin{equation}
\label{e.monster.est.9}
\bigl\| \mbox{\eqref{e.monster.twistie5}} \bigr\|_{L^2((0,1);\dot{H}^{-1}(\TT^2))}
\leq
C\kappa_m^{\nicefrac12} \ep_{m-1}^{2\delta} 
\left\| \theta_{0} \right\|_{L^2( \TT^2)}
\,.
\end{equation}
We use the identity~\eqref{e.monster.twistie5.split} to split~\eqref{e.monster.twistie5} into divergence and nondivergence form terms. 
The estimate for the first term on the right side of~\eqref{e.monster.twistie5.split} is obtained with the help of the ergodic lemma in Appendix~\ref{a.ergodic}, as explained below~\eqref{e.monster.twistie5.split}. We apply~\eqref{e.what.ergodic.gives} with the choice of~$f$ and~$g$ given by~\eqref{e.fg.choice.3}. 
Up to the exponentially small error, the estimate is reduced therefore to 
\begin{align*}
&
C \ep_{m} 
\bigl\| f \bigr\|_{L^2((0,1)\times\TT^2)} 
\bigl\| g \bigr\|_{L^2((0,1)\times\TT^2)}
\\ & \qquad 
\leq
C\ep_m
\sup_{k\in\N}
\left\| \xi_{m,k} \bigl( \nabla X_{m-1,l_k}  - \Itwo \bigr) \right\|_{L^\infty(\R\times\TT^2)}
\bigl\|  \nabla \bigl(
\nabla \bigl(  T_{m-1} \circ X_{m-1,l_k} \bigr) \circ X_{m-1,l_k}^{-1} 
\bigr) \bigr\|_{L^2((0,1)\times\TT^2)}
\\ & \qquad \qquad 
\times
\left\| \psi_{m}+ \kappa_m |\nabla \Chi_{m,k}| \right\|_{L^\infty(\R\times\TT^2)}
\notag \\ & \qquad
\leq
C \kappa_m \ep_m^{1-\gamma}  \ep_{m-1}^{2\delta} \ep_{m-1}^{-(1+\nicefrac \gamma2)} 
\kappa_{m-1}^{-\nicefrac12}
\left\| \theta_{0} \right\|_{L^2( \TT^2)}
\leq 
C\kappa_m^{\nicefrac12} \ep_{m-1}^{q-1-\nicefrac \gamma2+2\delta} 
\left\| \theta_{0} \right\|_{L^2( \TT^2)}
\,.
\end{align*}
As above in~\eqref{e.kill.half.gamma}, we have that~$q-1-\nicefrac \gamma2>0$, therefore the right side of the previous display is bounded from above by the right side of~\eqref{e.monster.est.9}.

\smallskip

The second term on the right side of~\eqref{e.monster.twistie5.split} is in divergence form, so we just need to estimate the $L^2$ norm of what is under the divergence. The claim is that 
\begin{multline}
\label{e.monster.est.9a}
\biggl\|
\sum_{k\in 2\Z+1}\!\!
\xi_{m,k}
Y_{m-1,k}
\bigl( 
\hat{\zeta}_{m,l_k} \zeta_{m,k}\tilde{\psi}_{m,k} \sigma
+ \kappa_m \nabla {\Chi}_{m,k} 
\circ X_{m-1,l_k}^{-1} \bigr)
\nabla \bigl(  T_{m-1} \circ X_{m-1,l_k} \bigr) \circ X_{m-1,l_k}^{-1} 
\biggr\|_{L^2((0,1)\times \TT^2)} 
\\
\leq
C\kappa_m^{\nicefrac12} \ep_{m-1}^{2\delta} 
\left\| \theta_{0} \right\|_{L^2( \TT^2)}
\,.
\end{multline}
The proof of~\eqref{e.monster.est.9a} is almost the same as~\eqref{e.monster.est.8a}. Compared to the latter, we use the bound 
\begin{align*}
\bigl\| \tilde{\psi}_{m,k} \bigr\|_{L^\infty}
+ \kappa_m \big\|  \nabla {\Chi}_{m,k} \bigr\|_{L^\infty} 
\leq
C \kappa_m \ep_m^{-\gamma}
\end{align*}
instead of just the bound for~$\kappa_m \|  \nabla {\Chi}_{m,k} \|_{L^\infty}$, and we substitute the bound 
\begin{align*}
\left\| \xi_{m,k} Y_{m-1,k} \right\|_{L^\infty(\R\times\TT^2)}
\leq C\ep_{m-1}^{2\delta},
\end{align*}
in place of the bound for~$\| \xi_{m,k} \bigl( \nabla X_{m-1,l_k}^{-1} - \Itwo \bigr) \|_{L^\infty(\R\times\TT^2)}$, which is essentially the same. 
Recall that~$Y_{m-1,k}$ is defined in~\eqref{e.Ymk} and the above bound is a consequence of~\eqref{e.Xm.bound.1} and~\eqref{e.tam}. 
The proof of~\eqref{e.monster.est.9a} and hence of~\eqref{e.monster.est.9} is now complete. 
\qed

\subsubsection*{The estimates of~\eqref{e.monster.normie1}}
The claimed estimate is
\begin{multline}
\label{e.monster.est.10}
\biggl\|
\sum_{k\in 2\Z+1}
\xi_{m,k} 
\bigl( \kappa_m \Itwo + \tilde{\psi}_m \sigma \bigr) 
\tilde{\Chi}_{m,k}
\nabla 
\bigl( \nabla \bigl(  T_{m-1}\circ X_{m-1,l_k} \bigr) \circ X_{m-1,l_k}^{-1}\bigr) 
\biggr\|_{L^2((0,1)\times\TT^2)}
\\
\leq 
C\kappa_{m}^{\nicefrac12} 
\ep_{m-1}^{4\delta} 
\left\| \theta_{0} \right\|_{L^2( \TT^2)}
\,.
\end{multline}
Using~\eqref{e.corrm},~\eqref{e.psikmb} and~\eqref{e.nabm2}, we bound the left hand side of~\eqref{e.monster.est.10} by
\begin{align*}
& \left\| \kappa_m\Itwo+\psi_m\sigma \right\|_{L^\infty} 
\left\| \Chi_{m,k} \right\|_{L^\infty} 
\sup_{k\in\N}
\left\| \xi_{m,k} \nabla 
\bigl( \nabla \bigl(  T_{m-1}\circ X_{m-1,l_k} \bigr) \circ X_{m-1,l_k}^{-1}\bigr) 
\right\|_{L^2((0,1)\times\TT^2)}
\notag \\ & \qquad \qquad 
\leq
C \kappa_m \left(1 + \frac{a_m\ep_m^2}{\kappa_m} \right) 
\ep_m^{1-\gamma}
\ep_{m-1}^{-(1+\nicefrac \gamma2)}
\kappa_{m-1}^{-\nicefrac12}  
\left\| \theta_{0} \right\|_{L^2( \TT^2)}
\notag \\ & \qquad \qquad 
=
C\kappa_{m}^{\nicefrac12} 
\underbrace{
\kappa_{m}^{\nicefrac12} 
\left(1 + \frac{a_m\ep_m^2}{\kappa_m} \right)
}_{\,\leq\, C\kappa_{m-1}^{\nicefrac12} }
\underbrace{
\left( \frac{\ep_m}{\ep_{m-1}} \right) 
\ep_{m}^{-\gamma} \ep_{m-1}^{-\nicefrac\gamma2}
}_{\,\leq\, C \ep_{m-1}^{4\delta}}
\kappa_{m-1}^{-\nicefrac12}  
\left\| \theta_{0} \right\|_{L^2( \TT^2)}
\leq
C\kappa_{m}^{\nicefrac12} 
\ep_{m-1}^{4\delta} 
\left\| \theta_{0} \right\|_{L^2( \TT^2)}
\,.
\end{align*}
In the above display, we used that 
\begin{equation}
\label{e.put.pipe.smoke.it}
\left( \frac{\ep_m}{\ep_{m-1}} \right) 
\ep_{m}^{-\gamma} \ep_{m-1}^{-\nicefrac\gamma2}
\leq
C\ep_{m-1}^{q-1-q\gamma-\nicefrac \gamma2}
=
C\ep_{m-1}^{4\delta}
\end{equation}
since, by~\eqref{e.delta} and~\eqref{e.gamma},
\begin{equation*}
q-1-q\gamma - \frac12 \gamma 
=
(q-1) \left( 1 - \frac{\beta}{q+1}\left( q + \frac12 \right) \right) 
=
(q-1) \left( 1 - \frac{2q+1}{2q+2} \beta \right) 
=4\delta
\,.
\end{equation*}
The proof of~\eqref{e.monster.est.10} is complete. 
\qed

\subsubsection*{The estimate of~\eqref{e.monster.normie2}}
The claimed estimate is
\begin{equation}
\label{e.monster.est.11}
\bigl\| 
\bigl( \kappa_m \Itwo + \tilde{\psi}_m \sigma \bigr) \nabla \widetilde{H}_{m} 
\bigr\|_{L^2((0,1)\times\TT^2)}
\leq
C\kappa_{m}^{\nicefrac12} 
\ep_{m-1}^{\delta} 
\left\| \theta_{0} \right\|_{L^2( \TT^2)}
\,.
\end{equation}
This follows easily from~\eqref{e.hcorrm},~\eqref{e.psikmb} and~\eqref{e.nabm2}. Indeed, we have that 
\begin{align*}
\bigl\| 
\bigl( \kappa_m \Itwo {+} \tilde{\psi}_m \sigma \bigr) \nabla \widetilde{H}_{m} 
\bigr\|_{L^2((0,1)\times\TT^2)}
& \leq
\left\| \kappa_m\Itwo{+}\psi_m\sigma \right\|_{L^\infty} 
\bigl\| \nabla \widetilde{H}_{m} \bigr\|_{L^2((0,1)\times \TT^2)} 
\notag \\ & 
\leq 
\underbrace{
C \kappa_m \left(1 {+} \frac{a_m\ep_m^2}{\kappa_m} \right) 
}_{\,\leq\, C \kappa_m^{\nicefrac12}\kappa_{m-1}^{\nicefrac12}}
\ep_{m-1}^{4\delta} 
\kappa_{m-1}^{-\nicefrac12}
\left\| \theta_{0} \right\|_{L^2( \TT^2)}
\leq
C\kappa_{m}^{\nicefrac12} 
\ep_{m-1}^{4\delta} 
\left\| \theta_{0} \right\|_{L^2( \TT^2)}
\,.
\end{align*}
This completes the proof of~\eqref{e.monster.est.11}. 
\qed

\subsubsection*{The estimate of~\eqref{e.monster.tiny}}
The claimed estimate is
\begin{equation}
\label{e.monster.est.tiny}
\bigl\| \mbox{\eqref{e.monster.tiny}} \bigr\|_{L^2((0,1);\dot{H}^{-1}(\TT^2))}
\leq
C\kappa_m^{\nicefrac12} \ep_{m-1}^{2\delta} 
\left\| \theta_{0} \right\|_{L^2( \TT^2)}
\,.
\end{equation}
We first show the following bound for~$\mathbf{d}_m$:
\begin{equation}
\label{e.monster.est.11.a.d}
\bigl\|
\mathbf{d}_m
\bigr\|_{L^2([0,1]\times \TT^2)}
\leq  
\kappa_m^{\nicefrac12} \ep_{m-1}^{2\delta} 
\|\theta_0\|_{L^2(\TT^2)}
\,.
\end{equation}
This bound follows from \eqref{e.dm.bounds}, upon taking $N_*$ to be sufficiently large to ensure that 
\begin{equation*}
\kappa_{m-1}^{\nicefrac 12}
\bigl(C \ep_{m-1}^{\delta} \bigr)^{\nicefrac{N_*}{2}}  
\leq 
\kappa_{m}^{\nicefrac 12} \ep_{m-1}^{2\delta}
\,.
\end{equation*}
In turn, this estimate follows from 
$C \ep_{m-1}^{\nicefrac{\delta}{2}}\leq 1$
and the definition of $N_*$ in~\eqref{e.N}, which implies 
$N_*  \geq 8 + 4(q-1)(\beta+\gamma) \delta^{-1} = 8 +  \frac{128 q^2}{q-1}$. The last equality follows from \eqref{e.beta.def.0}, \eqref{e.delta}, and \eqref{e.gamma}.

\smallskip

We next prove that
\begin{equation}
\label{e.monster.est.11.a.e}
\bigl\|
\mathbf{e}_{m-1}
\bigr\|_{L^2([0,1]\times \TT^2)}
\leq  
C\kappa_m^{\nicefrac12} \ep_{m-1}^{2\delta} 
\|\theta_0\|_{L^2(\TT^2)}
\,.
\end{equation}
The inequality~\eqref{e.Em-1.thetam} with~$n=0$ says that 
\begin{equation}
\label{e.Em-1.thetam.n=0}
\left\| \mathbf{e}_{m-1} \right\|_{L^2((0,1)\times \TT^2)}
\leq
C C_{{N_*}} \kappa_{m-1}^{\nicefrac 12}
\bigl( C \ep_{m-1}^{2\delta} \bigr)^{\nicefrac{{N_*}}{2}} 
\|\theta_0\|_{L^2(\TT^2)}
\,.
\end{equation}
If~${N_*}$ is sufficiently large to ensure that 
\begin{equation*}
\kappa_{m-1}^{\nicefrac 12}
\bigl(C \ep_{m-1}^{2\delta} \bigr)^{\nicefrac{{N_*}}{2}}  
\leq 
\kappa_{m}^{\nicefrac 12} \ep_{m-1}^{2\delta}\,,
\end{equation*}
then \eqref{e.Em-1.thetam.n=0} directly implies \eqref{e.monster.est.11.a.e}.
Similar to the previous paragraph, the above estimate follows from the definition of~$N_*$ in~\eqref{e.N}, which implies that~${N_*}\geq 8+128q^2(q-1)$. 
Note that since~${N_*}$ is chosen as in~\eqref{e.N}, the constant~$C_{{N_*}}$ also depends only on~$\beta$, and this gives us~\eqref{e.monster.est.11.a.e}. 

\smallskip

The claimed estimate for the third and final term in~\eqref{e.monster.tiny} is
\begin{align}
\label{e.tiny.est.thirdterm}
\biggl\| \sum_{k\in 2\Z+1}\!\!
\xi_{m,k} 
\tilde{\Chi}_{m,k}
\bigl( \nabla X_{m-1,l_k}\circ X_{m-1,l_k}^{-1} \bigr) 
\cdot 
\nabla \bigl(  \nabla \cdot  \mathbf{e}_{m-1}\bigr)
\biggr\|_{L^2((0,1)\times \TT^2)}
\leq
C\ep_{m-1}^{2\delta} \kappa_m^{\nicefrac12} 
\|\theta_0\|_{L^2(\TT^2)}
\,.
\end{align}
We proceed similarly as above. Using~\eqref{e.Xm.bound.1},~\eqref{e.corrm} and~\eqref{e.Em-1.thetam} with~$n=2$, We have that 
\begin{align*}
\lefteqn{
\biggl\| \sum_{k\in 2\Z+1}\!\!
\xi_{m,k} 
\tilde{\Chi}_{m,k}
\bigl( \nabla X_{m-1,l_k}\circ X_{m-1,l_k}^{-1} \bigr) 
\cdot 
\nabla \bigl(  \nabla \cdot  \mathbf{e}_{m-1}\bigr)
\biggr\|_{L^2((0,1)\times \TT^2)}
} 
\qquad &
\notag \\ & 
\leq 
\sup_{k\in\N}
\bigl\| \tilde{\Chi}_{m,k} \bigr\|_{L^\infty(\R\times\TT^2)}
\bigl\| \xi_{m,k} \bigl(\nabla X_{m-1,l_k}  \circ X_{m-1,l_k}^{-1}\bigr) 
\bigr\|_{L^\infty(\R\times\TT^2)}
\bigl\|
\nabla^2 \mathbf{e}_{m-1}
\bigr\|_{L^2((0,1)\times \TT^2)}
\notag \\ & 
\leq 
C\ep_m^{1-\gamma} \cdot 
C C_{{N_*}} \kappa_{m-1}^{\nicefrac 12}
\bigl(C \ep_{m-1}^{-2-\gamma}\bigr)
\bigl( C \ep_{m-1}^{\delta} \bigr)^{{N_*}} 
\|\theta_0\|_{L^2(\TT^2)}
\notag \\ & 
\leq 
CC_{{N_*}} 
\biggl(
\ep_{m-1}^{-1-(q-1)(\beta-1) -2q\gamma}
\bigl( C \ep_{m-1}^{\delta} \bigr)^{{N_*}}  
\biggr)
\kappa_m^{\nicefrac12} 
\|\theta_0\|_{L^2(\TT^2)}
\,.
\end{align*}
Arguing as above, we need to ensure that~${N_*}$ was chosen so large that the term in parentheses is bounded by a constant times~$\ep_{m-1}^{2\delta}$. It suffices to take 
\begin{equation*}
N_* \geq 2 + \delta^{-1} \bigl( 1 + (q-1)(\beta-1) -2 q \gamma\bigr)\,,
\end{equation*}
which is clearly satisfied by the $N_*$ defined in~\eqref{e.N}. This completes the proof of~\eqref{e.tiny.est.thirdterm} and thus of~\eqref{e.monster.est.tiny}. 
\qed

\subsubsection*{The estimates of~\eqref{e.monster.normie3}}
We will show that 
\begin{align}
\label{e.monster.est.12}
&
\biggl\|
\sum_{k\in 2\Z+1} \!\!\!
\xi_{m,k}
\Bigl( \J_m -  
\bigl( \kappa_m \Itwo {+} \hat{\zeta}_{m,l_k} \zeta_{m,k}\tilde{\psi}_{m,k} \sigma\bigr)
\bigl( \Itwo {+} \nabla \tilde{\Chi}_{m,k} \bigr) \Bigr)
\nabla \bigl( \nabla \bigl(  T_{m-1} {\circ} X_{m-1,l_k} \bigr) {\circ} X_{m-1,l_k}^{-1} \bigr)
\biggr\|_{L^2_t \dot{H}^{-1}_x}
\notag \\ & \qquad \qquad
\leq
C\kappa_{m}^{\nicefrac12} \ep_{m-1}^{4\delta} 
\left\|  \theta_{0} \right\|_{L^2(\TT^2)} 
\,.
\end{align}
This term is in nondivergence form, so we need to apply the ergodic lemma in Appendix~\ref{a.ergodic}, as explained above. We apply~\eqref{e.what.ergodic.gives} with the choice of~$f$ and~$g$ given by~\eqref{e.fg.choice.4}. 
Discarding the exponentially small error term, this reduces the estimate~\eqref{e.monster.est.12} to an estimate for 
\begin{align*}
&
C \ep_{m} 
\bigl\| f \bigr\|_{L^2((0,1)\times\TT^2)} 
\bigl\| g \bigr\|_{L^2((0,1)\times\TT^2)}
\\ & \qquad 
\leq
C\ep_m
\sup_{k\in\N}
\bigl\|  \xi_{m,k}
\nabla \bigl( 
\nabla \bigl(  T_{m-1} \circ X_{m-1,l_k} \bigr) \circ X_{m-1,l_k}^{-1} \bigr) 
\bigr\|_{L^2((0,1)\times\TT^2)}
\\ & \qquad \qquad \times
\left\|  \J_m - 
\bigl( \kappa_m \Itwo + \hat{\zeta}_{m,l_k}\zeta_{m,k}{\psi}_{m,k} \sigma\bigr)
\bigl( \Itwo + \nabla {\Chi}_{m,k} \bigr) \right\|_{L^\infty(\R\times\TT^2)}
\notag \\ & \qquad
\leq
C \ep_m \cdot C\ep_{m-1}^{-(1+\nicefrac \gamma2)}
\kappa_{m-1}^{-\nicefrac12}
\left\| \theta_{0} \right\|_{L^2(\TT^2)}
\cdot 
C \kappa_m \ep_m^{-2\gamma}
\notag \\ & \qquad 
\leq
C\kappa_{m}^{\nicefrac12} 
\underbrace{ \!\!\!\!\phantom{\biggl(} \ep_{m}^{1-\gamma} \ep_{m-1}^{-(1+\nicefrac \gamma2)} }_{\,\leq \, C\ep_{m-1}^{4\delta}}   \underbrace{ \ep_m^{-\gamma} \biggl( \frac{\kappa_m}{\kappa_{m-1}} \biggr)^{\!\!\nicefrac12} }_{\,\leq\,  C}
\left\| \theta_{0} \right\|_{L^2(\TT^2)}
\leq
C\kappa_{m}^{\nicefrac12} \ep_{m-1}^{4\delta} 
\left\| \theta_{0} \right\|_{L^2( \TT^2)}
\,.
\end{align*}
Here we used~\eqref{e.kappam},~\eqref{e.corrm},~\eqref{e.psikmb} and~\eqref{e.put.pipe.smoke.it} again. 
This completes the proof of~\eqref{e.monster.est.12}.
\qed

\medskip

We have now estimated every one of the terms on the right side of the Big Display on Page~\pageref{e.monster} and shown that they are each bounded by the right side of~\eqref{e.bigbound}. The proof of~\eqref{e.bigbound} is therefore complete.

\subsection{Energy cascade down the scales} 
\label{ss.one.renormalization.step}

In this section, we complete the proof of Theorem~\ref{t.anomalous.diffusion}. The main step remaining is to use the estimates from the previous section to obtain lower bounds on the energy dissipation of~$\theta_m$ in terms of that of~$\theta_{m-1}$, 
thereby formalizing the expectation that energy is pushed by advection into smaller scales, down the inertial-convection subrange, until one finally reaches a scale small enough that molecular diffusivity dominates.

\smallskip

Recall that we have fixed a small parameter~$\kappa \in \mathcal{K}$, with the set~$\mathcal{K}$ defined in~\eqref{e.permissible.K}, and chosen~$M \in \N$ satisfying~\eqref{e.permitted}; the finite sequence~$\kappa_M, \kappa_{M-1},\ldots,\kappa_0$ is defined by~\eqref{e.kappa.sequence}.

\begin{proposition}[Main induction step]
\label{p.indystepdown}
There exist constants~$C(\beta)<\infty$, such that, if the minimal scale separation parameter~$\Lambda$ satisfies $\Lambda \geq C$, then the following statement is valid.
For every~$R_{\theta_0}>0$ and~$\theta_0\in C^\infty(\TT^2)$ with~$\langle \theta_0 \rangle = 0$ which satisfies the quantitative analyticity condition 
\begin{align}
\label{e.theta0.anal.encore}
\max_{|\aa|=n}
\left\| \partial^\aa \theta_0 \right\|_{L^2(\TT^2)}
\leq
\|\theta_0\|_{L^2(\TT^2)} 
\frac{n!}{R_{\theta_0}^{n}}
\,, \quad \forall n\in\N \,,
\end{align}
if we define 
\begin{equation}
\label{e.mtheta0.def}
m_{\theta_0}:=
\min \Bigl \{ m\in \N \,:\, m\geq 2\,, \ 
\ep_{m-1}^{1+ \nicefrac{\gamma}{2}} \leq R_{\theta_0} \Bigr \}\,,
\end{equation}
then, for every $m\in\{m_{\theta_0},\ldots,M\}$, 
we have the estimates
\begin{equation}
\label{e.homogenization.m}
\big\| \theta_m - \theta_{m-1}\big\|_{L^\infty((0,1);L^2(\TT^2))}+
\kappa_m^{\nicefrac12} 
\big\| \nabla \theta_m - \nabla \tilde{\theta}_{m} \big\|_{L^2((0,1)\times \TT^2)}
\leq 
C \ep_{m-1}^\delta 
\left\| \theta_{0} \right\|_{L^2( \TT^2)}
\end{equation}
and
\begin{equation}
\label{e.tildethetam.energy.diss}
\Biggl| \,
\frac{\kappa_{m} \big\| \nabla {\theta}_m \big\|_{L^2((0,1)\times\TT^2)}^2}
{\kappa_{m-1} \left\| \nabla \theta_{m-1} \right\|_{L^2((0,1)\times\TT^2)}^2}
-1 \,
\Biggr| 
\leq C\ep_{m-1}^\delta.
\end{equation}
\end{proposition}

The proof of Proposition~\ref{p.indystepdown} is based on~\eqref{e.bigbound}, which implies the bound on~$\nabla \theta_{m} - \nabla \tilde{\theta}_m$ in~\eqref{e.homogenization.m}. We then use this estimate to obtain the rest of the estimates in the proposition by doing the computations for $\tilde{\theta}_{m}$ and then switching back to~$\theta_m$ with the triangle inequality. 

\smallskip

We will assume that the constant~$C$ is large enough that the conditions~\eqref{e.Lambda.restriction} and~\eqref{e.m.restriction} are valid. Therefore the estimate~\eqref{e.bigbound} proved in the previous subsection is also valid. 

\subsubsection*{The proof of the estimate of $\theta_m - \tilde{\theta}_{m}$ in $L^2_tH^1_x$}
We prove the estimate
\begin{equation}
\label{e.twoscale.m}
\big\| \theta_m - \tilde{\theta}_m\big\|_{L^\infty((0,1);L^2(\TT^2))}
+
\kappa_m^{\nicefrac12} 
\big\| \nabla \theta_m - \nabla \tilde{\theta}_{m} \big\|_{L^2((0,1)\times \TT^2)}
\leq
C \ep_{m-1}^\delta 
\left\| \theta_{0} \right\|_{L^2( \TT^2)}
\,.
\end{equation}
Since $\theta_m = \tilde{\theta}_m$ at~$t=0$, the energy estimate and~\eqref{e.bigbound} yield
\begin{align*}
&
\big\| \theta_m - \tilde{\theta}_m\big\|_{L^\infty((0,1);L^2(\TT^2))}^2
+
\kappa_{m}
\big\| \nabla \theta_m - \nabla \tilde{\theta}_m\big\|_{L^2((0,1)\times\TT^2)}^2
& 
\notag \\ & \qquad 
\leq
\frac{C}{\kappa_{m}}
\big\| (\partial_t - \kappa_m\Delta + \b_m\cdot\nabla) \tilde{\theta}_m \big\|_{L^2((0,1); \dot{H}^{-1}(\TT^2))}^2
\leq 
C \ep_{m-1}^{2 \delta} 
\left\| \theta_{0} \right\|_{L^2( \TT^2)}
\,.
\end{align*}
This is~\eqref{e.twoscale.m}. 
\qed

\subsubsection*{Estimate of the energy dissipation of~${\theta}_{m}$ in terms of~$\theta_{m-1}$}
We next prove that~\eqref{e.tildethetam.energy.diss} holds. 
In view of~\eqref{e.homogenization.m}, it is essentially equivalent to prove~\eqref{e.tildethetam.energy.diss} with~$\tilde{\theta}_m$ in place of~$\theta_m$. We therefore return to the identity~\eqref{e.grad.tildetheta} and rearrange it in the form
\begin{align}
\label{e.grad.tildetheta.again}
\nabla \tilde{\theta}_m 
&
=
\sum_{k\in 2\Z+1}
\xi_{m,k}
\bigl( \Itwo +  \nabla  \Chi_{m,k}  \bigr) \circ X_{m-1,l_k}^{-1}
\nabla T_{m-1}
\notag \\ & \qquad 
+
\sum_{k\in 2\Z+1}
\xi_{m,k}
\bigl(  \nabla X_{m-1,l_k}^{-1} - \Itwo  \bigr)(\nabla  \Chi_{m,k}) \circ X_{m-1,l_k}^{-1}
\nabla T_{m-1}
\notag \\ & \qquad 
+
\sum_{k\in 2\Z+1}
\xi_{m,k}
\nabla \tilde{\Chi}_{m,k}
\bigl( \nabla X_{m-1,l_k} \circ X_{m-1,l_k}^{-1}   - \Itwo \bigr) 
\nabla T_{m-1}
\notag \\ & \qquad
+
\sum_{k\in 2\Z+1}
\xi_{m,k} \tilde{\Chi}_{m,k}
\nabla \bigl( 
\nabla \bigl(  T_{m-1} \circ X_{m-1,l_k} \bigr) \circ X_{m-1,l_k}^{-1} 
\bigr)
+ 
\nabla \widetilde{H}_m
\,.
\end{align}
The first term on the right side of~\eqref{e.grad.tildetheta.again} makes the leading order contribution. 
We proceed by estimating the second through fifth terms on the right side of~\eqref{e.grad.tildetheta.again}, showing that they are suitably small. 
For the second of these terms we use~\eqref{e.Xm.bound.1} and~\eqref{e.tam} in place of~\eqref{e.Tm.thetam}:
\begin{align*}
& 
\kappa_m^{\nicefrac12} 
\left\| 
\sum_{k\in 2\Z+1}
\xi_{m,k}
\bigl(  \nabla X_{m-1,l_k}^{-1} - \Itwo  \bigr)
(\nabla  \Chi_{m,k}) \circ X_{m-1,l_k}^{-1}
\nabla T_{m-1}
\right\|_{L^2((0,1)\times\TT^2)} 
\notag \\ & \qquad 
\leq
C\kappa_m^{\nicefrac12} 
\sup_{k\in\N}
\left\| \nabla X_{m-1,l_k}^{-1} - \Itwo \right\|_{L^\infty(\supp \xi_{m,k} \times \TT^2)}
\left\| \xi_{m,k}\nabla  \Chi_{m,k} \right\|_{L^\infty(\R\times\TT^2)}
\left\| \nabla T_{m-1} \right\|_{L^2((0,1)\times \TT^2)}
\notag \\ & \qquad 
\leq 
C\ep_{m-1}^{2\delta} \underbrace{
\left( \frac{\kappa_m}{\kappa_{m-1}} \right)^{\nicefrac12} 
\ep_m^{-\gamma}
}_{\,\leq\, C \text{\,by~\eqref{e.kappam}\,\&\,\eqref{e.corrm}}}
\!\!
\left\| \theta_{0} \right\|_{L^2( \TT^2)} 
\leq 
C \ep_{m-1}^{2\delta} 
\left\| \theta_{0} \right\|_{L^2( \TT^2)}
\,.
\end{align*}
For the third of these terms we use~\eqref{e.Xm.bound.1} and~\eqref{e.tam} in place of~\eqref{e.Tm.thetam}:
\begin{align*}
& 
\kappa_m^{\nicefrac12} 
\left\| 
\sum_{k\in 2\Z+1}
\xi_{m,k}
\nabla \tilde{\Chi}_{m,k}
\bigl( \nabla X_{m-1,l_k} \circ X_{m-1,l_k}^{-1} - \Itwo \bigr) 
\nabla T_{m-1}
\right\|_{L^2((0,1)\times\TT^2)} 
\notag \\ & \qquad 
\leq
C\kappa_m^{\nicefrac12} 
\sup_{k\in\N}
\left\| \xi_{m,k}\nabla \tilde{\Chi}_{m,k} \right\|_{L^\infty(\R\times\TT^2)}
\left\| \nabla X_{m-1,l_k} - \Itwo \right\|_{L^\infty(\supp \xi_{m,k} \times \TT^2)}
\left\| \nabla T_{m-1} \right\|_{L^2((0,1)\times \TT^2)}
\notag \\ & \qquad 
\leq 
C\underbrace{
\left( \frac{\kappa_m}{\kappa_{m-1}} \right)^{\nicefrac12} 
\ep_m^{-\gamma}
}_{\,\leq\, C \text{\,by~\eqref{e.kappam}\,\&\,\eqref{e.corrm}}}
\!\!
\ep_{m-1}^{2\delta} 
\left\| \theta_{0} \right\|_{L^2( \TT^2)}
\leq 
C \ep_{m-1}^{2\delta} 
\left\| \theta_{0} \right\|_{L^2( \TT^2)}
\,.
\end{align*}
For the first term on the last line of~\eqref{e.grad.tildetheta.again}, from \eqref{e.Xm.bound.4}, \eqref{e.corrm}, and \eqref{e.nabm2}, we have 
\begin{align*}
& 
\kappa_m^{\nicefrac12} 
\biggl\| 
\sum_{k\in 2\Z+1}
\xi_{m,k} \tilde{\Chi}_{m,k}
\nabla \bigl( 
\nabla \bigl(  T_{m-1} \circ X_{m-1,l_k} \bigr) \circ X_{m-1,l_k}^{-1} 
\bigr)
\biggr\|_{L^2((0,1)\times\TT^2)} 
\notag \\ & \qquad 
\leq
C\kappa_m^{\nicefrac12} 
\sup_{k\in\N}
\left\| \xi_{m,k} \tilde{\Chi}_{m,k} \right\|_{L^\infty(\supp \xi_{m,k} \times \TT^2)}
\left\| \nabla^2 T_{m-1} \right\|_{L^2((0,1)\times\TT^2)}
\notag\\ &\qquad \qquad 
+C\kappa_m^{\nicefrac12} 
\sup_{k\in\N}
\left\| \xi_{m,k} \tilde{\Chi}_{m,k} \nabla^2 X_{m-1,l_k}  \right\|_{L^\infty(\supp \xi_{m,k} \times \TT^2)}  \left\| \nabla T_{m-1} \right\|_{L^2((0,1)\times\TT^2)}
\notag \\ & \qquad 
\leq
C\underbrace{
\left( \frac{\kappa_m}{\kappa_{m-1}} \right)^{\nicefrac12} 
\ep_m^{-\gamma}
}_{\,\leq\, C \text{\,by~\eqref{e.kappam}}}
\underbrace{
\left( \frac{\ep_m}{\ep_{m-1}} \right)
\ep_{m-1}^{-\nicefrac \gamma2}
}_{\,\leq\, C\ep_{m-1}^{2(q-1)/3}}
\left\| \theta_{0} \right\|_{L^2(\TT^2)}
\leq 
C \ep_{m-1}^{\frac23(q-1)} 
\left\| \theta_{0} \right\|_{L^2(\TT^2)}
\,.
\end{align*}
Finally, in view of the estimate for $\nabla \widetilde{H}_m$ in~\eqref{e.Hm.gradient.L2}, we have
\begin{equation}
\kappa_m^{\nicefrac 12} \norm{\nabla \widetilde{H}_m(t,\cdot)}_{L^{2}([0,1]\times \TT^2)} 
\leq  
C\ep_{m-1}^{\delta+\gamma}   
\left\| \theta_{0} \right\|_{L^2(\TT^2)}
\,.
\end{equation}
By~\eqref{e.grad.tildetheta.again}, the triangle inequality and the previous five displays, we therefore  obtain that 
\begin{equation}
\label{e.leadingord}
\kappa_m^{\nicefrac12} 
\biggl\| 
\nabla \tilde{\theta}_m 
-
\sum_{k\in 2\Z+1}
\xi_{m,k}
\bigl( \Itwo + \nabla  \Chi_{m,k} \bigr)\circ X_{m-1,l_k}^{-1}
\nabla T_{m-1}
\biggr\|_{L^2((0,1)\times\TT^2)} 
\leq 
C \ep_{m-1}^{2\delta} 
\left\| \theta_{0} \right\|_{L^2(\TT^2)}
\,.
\end{equation}
In view of~\eqref{e.Tm.thetam},~\eqref{e.homogenization.m} and~\eqref{e.leadingord}, in order to prove~\eqref{e.tildethetam.energy.diss} we have left to show
\begin{multline}
\label{e.left.to.show}
\Biggl| 
\kappa_m 
\biggl\| 
\sum_{k\in 2\Z+1}
\xi_{m,k}
\bigl( \Itwo + \nabla \Chi_{m,k} \bigr)\circ X_{m-1,l_k}^{-1}
\nabla T_{m-1}
\biggr\|_{L^2((0,1)\times\TT^2)}^2
-
\kappa_{m-1} 
\left\| 
\nabla T_{m-1}
\right\|_{L^2((0,1)\times\TT^2)}^2
\Biggr| 
\\
\leq
C \ep_{m-1}^{2\delta} 
\left\| \theta_{0} \right\|_{L^2(\TT^2)}^2
\,.
\end{multline}
The proof of~\eqref{e.left.to.show} is based on an application of an ergodic lemma in the appendix (Lemma~\ref{l.flow.averages}). 
Before we apply the lemma, we 
define the matrix
\begin{equation*}
\mathbf{F}(t,x):= 
\sum_{k\in 2\Z+1}
\xi_{m,k}
\bigl( \Itwo + \nabla  \Chi_{m,k} \bigr)\circ X_{m-1,l_k}^{-1}
\,.
\end{equation*}
We can therefore break up the left side of~\eqref{e.left.to.show} as follows:
\begin{align}
\label{e.ergodic.break.up}
\lefteqn{
\Biggl| 
\kappa_m 
\biggl\| 
\sum_{k\in 2\Z+1}
\xi_{m,k}
\bigl( \Itwo + \nabla  \Chi_{m,k} \bigr) \circ X_{m-1,l_k}^{-1}
\nabla T_{m-1}
\biggr\|_{L^2((0,1)\times\TT^2)}^2
-
\kappa_{m-1} 
\left\| 
\nabla T_{m-1}
\right\|_{L^2((0,1)\times\TT^2)}^2\,
\Biggr| 
} \qquad & 
\notag \\ & 
\leq
\kappa_m
\biggl| 
\int_0^1 
\bigl\| 
\mathbf{F} (t,\cdot)
\nabla T_{m-1}(t,\cdot)
\bigr\|_{L^2(\TT^2)}^2
\,dt
-
\int_0^1 \!
\int_{\TT^2} 
\nabla T_{m-1}(t,x)\cdot 
\bigl\langle
(\mathbf{F}^t \mathbf{F})(t,\cdot)
\bigr\rangle
\nabla T_{m-1}(t,x)
\,dx
\,dt\,
\biggr|
\notag \\ & \qquad 
+
\int_0^1 
\Bigl| 
\kappa_m \bigl\langle
(\mathbf{F}^t \mathbf{F})(t,\cdot)
\bigr\rangle
-
\J_m(t)
\Bigr|
\bigl\|
\nabla T_{m-1}(t,\cdot)
\bigr\|_{L^2(\TT^2)}^2
\,dt
\notag \\ & \qquad 
+
\biggl| 
\int_0^1 
\Bigl\langle
\nabla T_{m-1}(t,\cdot)\cdot 
\J_m(t) 
\nabla T_{m-1}(t,\cdot)
\Bigr\rangle
\,dt
-
\int_{0}^1 
\Bigl\langle
\nabla T_{m-1}(t,\cdot)\cdot 
\langle \hspace{-2.5pt} \langle
 \J_m \rangle \hspace{-2.5pt} \rangle
\nabla T_{m-1}(t,\cdot)
\Bigr\rangle
\,
\biggr|
\notag \\ & \qquad 
+
\biggl|
\int_{0}^1 
\Bigl\langle
\nabla T_{m-1}(t,\cdot)\cdot 
\langle \hspace{-2.5pt} \langle
 \J_m \rangle \hspace{-2.5pt} \rangle
\nabla T_{m-1}(t,\cdot)
\Bigr\rangle
-
\kappa_{m-1} 
\left\| 
\nabla T_{m-1}
\right\|_{L^2((0,1)\times\TT^2)}^2 \,
\biggr|
\end{align}
The first and third terms on the right side of~\eqref{e.ergodic.break.up} are estimated using an ergodic lemma (Lemma~\ref{l.flow.averages} in the appendix); for the first term, we apply the ergodic lemma in space only, and for the third term we apply the ergodic lemma in time only. The second term on the right side of~\eqref{e.ergodic.break.up} is a consequence of Lemma~\ref{l.flux.to.energy}. The fourth term is actually equal to zero, because~$\kappa_m \Itwo$ is equal to~$\langle \J_m \rangle$ by definition: see~\eqref{e.Kbarm.def} and~\eqref{e.kappa.sequence}.

\smallskip

We next bound the second term on the right side of~\eqref{e.ergodic.break.up}.
We observe that
\begin{align*}
\mathbf{F}^t \mathbf{F} 
=
\Itwo +
\!\!\sum_{k\in 2\Z+1}\!\!
\xi_{m,k} 
\bigl( 
\bigl(  \nabla  \Chi_{m,k} \bigr)^t
{+}
\nabla  \Chi_{m,k} \bigr)
\circ X_{m-1,l_k}^{-1}
+
\!\!\sum_{k\in 2\Z+1}\!\!
\xi_{m,k}^2 
\bigl( 
\bigl(  \nabla  \Chi_{m,k} \bigr)^t
\nabla  \Chi_{m,k} 
\bigr)
\circ X_{m-1,l_k}^{-1}
\,.
\end{align*}
Indeed, this follows from the fact that the supports of~$\xi_{m,k}$ and~$\xi_{m,k'}$ have nonempty intersection only if~$k,k'\in2 \Z+1$ satisfy~$|k-k'| \leq 2$ and, by~\eqref{e.alt.orth}, we have  
\begin{equation}
|k-k'| = 2 
\quad \implies \quad
\bigl(\nabla  \Chi _{m,k}
\circ X_{m-1,l_k}^{-1}\bigr)^\intercal 
\bigl(
\nabla   \Chi _{m,k'}
\circ
X_{m-1,l_{k'}}^{-1}\bigr)
= 0
\,.
\end{equation}
Taking the average in space and multiplying by~$\kappa_m$ yields, in view of~\eqref{e.Emkappa.formula},
\begin{align}
\label{e.bracket.FF}
\kappa_m
\bigl\langle \mathbf{F}^t \mathbf{F}  \bigr\rangle (t)
&
=
\kappa_m\Itwo +
\kappa_m\sum_{k\in 2\Z+1}
\xi_{m,k}^2 (t) 
\, \Bigl\langle 
\bigl( 
\bigl(  \nabla  \Chi_{m,k} \bigr)^t
\nabla  \Chi_{m,k} 
\bigr)
\circ X_{m-1,l_k}^{-1}
\Bigr\rangle (t)
\notag \\ &
=
\kappa_m\Itwo +
\kappa_m\sum_{k\in 2\Z+1}
\xi_{m,k}^2 (t) 
\, \bigl\langle 
\bigl(  \nabla  \Chi_{m,k} \bigr)^t
\nabla  \Chi_{m,k} 
\bigr\rangle (t)
=
\mathbf{E}_m(t)
\,.
\end{align}
Therefore, thanks to Lemma~\ref{l.flux.to.energy}, we obtain
\begin{align}
\label{e.ergodic.break.up.second}
\lefteqn{
\int_0^1 
\bigl| 
\kappa_m \bigl\langle
(\mathbf{F}^t \mathbf{F})(t,\cdot)
\bigr\rangle
-
\J_m(t)
\bigr|
\bigl\|
\nabla T_{m-1}(t,\cdot)
\bigr\|_{L^2(\TT^2)}^2
\,dt
} \quad &
\notag \\ &
=
\int_0^1 
\bigl| 
\mathbf{E}_m (t,\cdot)
-
\J_m(t)
\bigr|
\bigl\|
\nabla T_{m-1}(t,\cdot)
\bigr\|_{L^2(\TT^2)}^2
\,dt
\notag \\ &
\leq 
\kappa_{m-1} 
\biggl(
\frac{\ep_m^2}{\kappa_m \tau_m}
\biggr)
\bigl\|
\nabla T_{m-1}
\bigr\|_{L^2((0,1) \times \TT^2)}^2
\leq 
\ep_{m-1}^{2\delta}
\left\| \theta_{0} \right\|_{L^2(\TT^2)}^2
\,,
\end{align}
where we used~\eqref{e.expqunt} and~\eqref{e.nabm2} to get the last inequality. 

\smallskip

We next estimate the first term on the right side of~\eqref{e.ergodic.break.up}. The claimed estimate is
\begin{multline}
\label{e.ergodic.space.kill}
\kappa_m
\biggl| 
\int_0^1 
\bigl\| 
\mathbf{F} (t,\cdot)
\nabla T_{m-1}(t,\cdot)
\bigr\|_{L^2(\TT^2)}^2
\,dt
-
\int_0^1 \!
\int_{\TT^2} 
\nabla T_{m-1}(t,x)\cdot 
\bigl\langle
(\mathbf{F}^t \mathbf{F})(t,\cdot)
\bigr\rangle
\nabla T_{m-1}(t,x)
\,dx
\,dt\,
\biggr|
\\
\leq 
\ep_{m-1}^{500}
\left\| \theta_{0} \right\|_{L^2(\TT^2)}^2
\,.
\end{multline}
For this we apply an ergodic lemma from Appendix~\ref{a.ergodic}, namely Lemma~\ref{l.flow.averages} with the time variable frozen. 
We therefore fix~$t\in (0,1)$,~$k\in 2\Z+1$ and~$i,j\in\{1,2\}$ and we use the following choices for the~$X$,~$f$ and~$g$ appearing in that lemma:
\begin{equation*}
\left\{
\begin{aligned}
& 
X = X_{m-1,l_k}
\,,
\\ & 
f = \partial_{x_i} T_{m-1}(t,\cdot) \partial_{x_j} T_{m-1}(t,\cdot)
\,,
\\ & 
g = 
\Bigl( 
\Itwo +
\xi_{m,k} 
\bigl( 
\bigl(  \nabla  \Chi_{m,k} \bigr)^t
{+}
\nabla  \Chi_{m,k} \bigr)
+
\xi_{m,k}^2 
\bigl( 
\bigl(  \nabla  \Chi_{m,k} \bigr)^t
\nabla  \Chi_{m,k} 
\bigr)
\Bigr)_{ij}
\,.
\end{aligned}
\right.
\end{equation*}
and with the following choices of the constants appearing the hypotheses of that lemma:
\begin{equation*}
\left\{
\begin{aligned}
& C_X = C\ep_{m-1} \,, 
\\ & 
R = C \ep_{m-1}^{-1} \,,
\\ & 
C_f = C_{\theta_0} R_{\theta_0}^{-1}  \kappa_{m-1}^{-\nicefrac12}
\\ & 
r = R_{\theta_0} ^{-1} \ep_{m-1}^{1+\nicefrac\gamma2}
\\ & 
N = \ep_{m}^{-1}\,.
\end{aligned}
\right.
\end{equation*}
Let's check the hypotheses of Lemma~\ref{l.flow.averages} in this situation. We first note that~\eqref{e.ass.X.anal} is valid since~\eqref{e.Xm.regbounds}. 
The hypothesis~\eqref{e.ass.f.anal.flows} is valid due to~\eqref{e.Tm.reg.upgrade}. The periodicity assumption for~$g$ is clear from the construction of the correctors~$\Chi_{m,k}$, as these are~$\ep_m$--periodic. Finally, the condition~\eqref{e.N.ergodic.constraint} is valid since
\begin{align*}
\frac{N r}{R (r +  d C_X)} 
& =
\frac{\ep_{m-1}^{2+\nicefrac\gamma2}}{C\ep_m R_{\theta_0}(R_{\theta_0} ^{-1} \ep_{m-1}^{1+\nicefrac\gamma2} +\ep_{m-1}) } 
\geq 
\frac{\ep_{m-1}^{1+\nicefrac\gamma2}}{C R_{\theta_0} \ep_m} 
=
\frac{\ep_{m-1}^{ -(q-1) ( 1 - \frac{\beta}{2(q+1)}  )}}{C R_{\theta_0}} 
\geq 
\frac{\ep_{m-1}^{ - 2(q-1)/3 }}{C R_{\theta_0}} 
\,.
\end{align*}
Clearly the expression on the right side is at least~$1$, provided that~$\Lambda$ is chosen sufficiently large. 
Note that, after summing over~$k\in 2\Z+1$, the function~$g\circ X^{-1}$ is equal to the~$ij$th component of~$\mathbf{F}^t \mathbf{F}(t,\cdot)$. 
Therefore, an application of the lemma, namely~\eqref{e.new.expdecay}, gives us the bound
\begin{multline*}
\kappa_m
\biggl| 
\bigl\| 
\mathbf{F} (t,\cdot)
\nabla T_{m-1}(t,\cdot)
\bigr\|_{L^2(\TT^2)}^2
-
\int_{\TT^2} 
\nabla T_{m-1}(t,x)\cdot 
\bigl\langle
(\mathbf{F}^t \mathbf{F})(t,\cdot)
\bigr\rangle
\nabla T_{m-1}(t,x)
\,dx
\biggr|
\\
\leq 
C \kappa_{m-1}^{\nicefrac12}
\exp \biggl( 
-
\frac{\ep_{m-1}^{ - 2(q-1)/3 }}{C R_{\theta_0}} 
\biggr) 
\left\| \theta_{0} \right\|_{L^2(\TT^2)}^2
\leq \ep_{m-1}^{500}\left\| \theta_{0} \right\|_{L^2(\TT^2)}^2
\,.
\end{multline*}
Integrating over~$t$ yields~\eqref{e.ergodic.space.kill}. 

\smallskip
Returning to \eqref{e.ergodic.break.up}, we consider the third term on the right side. Our claim is that 
\begin{multline}
\label{e.ergodic.break.up.last}
\biggl| 
\int_0^1 
\Bigl\langle
\nabla T_{m-1}(t,\cdot)\cdot 
\J_m(t) 
\nabla T_{m-1}(t,\cdot)
\Bigr\rangle
\,dt
-
\int_{0}^1 
\Bigl\langle
\nabla T_{m-1}(t,\cdot)\cdot 
\langle \hspace{-2.5pt} \langle
 \J_m \rangle \hspace{-2.5pt} \rangle
\nabla T_{m-1}(t,\cdot)
\Bigr\rangle
\,
\biggr|
\\ 
\leq
C\ep_{m-1}^{2\delta}
\left\| \theta_{0} \right\|_{L^2( \TT^2)}^2
\,.
\end{multline}
To see this, we first recall that the bound for $\J_m - \hat{\J}_m$ in  \eqref{e.JJhat}, together with \eqref{e.exprat.bound}, implies that the left side of \eqref{e.ergodic.break.up.last} is bounded from above as 
\begin{align}
\label{e.TJT.break}
&
\biggl| 
\int_0^1 
\Bigl\langle
\nabla T_{m-1}(t,\cdot)\cdot 
\J_m(t) 
\nabla T_{m-1}(t,\cdot)
\Bigr\rangle
\,dt
-
\int_{0}^1 
\Bigl\langle
\nabla T_{m-1}(t,\cdot)\cdot 
\langle \hspace{-2.5pt} \langle
 \J_m \rangle \hspace{-2.5pt} \rangle
\nabla T_{m-1}(t,\cdot)
\Bigr\rangle
\,
\biggr|
\notag\\
& 
\leq \frac{C a_m^2\ep_m^4}{\kappa_m} \biggl(
\frac{\ep_m^2}{\kappa_m \tau_m}
\biggr)^{\!\!N_*}
\bigl\|
\nabla T_{m-1}
\bigr\|_{L^2([0,1]\times \TT^2)}^2
+ 
\biggl| 
\int_0^1 \!
\bigl( 
\hat{\J}_m(t) 
-
\langle \hspace{-2.5pt} \langle
\hat{\J}_m \rangle \hspace{-2.5pt} \rangle
\bigr)
\colon
\Bigl\langle \!
\nabla T_{m-1}(t,\cdot) 
\otimes  
\nabla T_{m-1}(t,\cdot)\!
\Bigr\rangle
dt
\biggr|
\notag\\ &  
\leq 
C \bigl( C \ep_{m-1}^{2\delta} \bigr)^{\!N_*}
\kappa_{m-1}
\bigl\|
\nabla T_{m-1}
\bigr\|_{L^2([0,1]\times \TT^2)}^2
+ 
\biggl| 
\int_0^1 \!
\bigl( 
\hat{\J}_m(t) 
-
\langle \hspace{-2.5pt} \langle
\hat{\J}_m \rangle \hspace{-2.5pt} \rangle
\bigr)
\colon
\Bigl\langle \!
\nabla T_{m-1}(t,\cdot) 
\otimes  
\nabla T_{m-1}(t,\cdot) \!
\Bigr\rangle
\,dt
\,
\biggr|
\notag\\ &  
\leq 
C \bigl( C \ep_{m-1}^{2\delta} \bigr)^{\!N_*}
\left\| \theta_{0} \right\|_{L^2(\TT^2)}^2
+ 
\biggl| 
\int_0^1 
\bigl( 
\hat{\J}_m(t) 
-
\langle \hspace{-2.5pt} \langle
\hat{\J}_m \rangle \hspace{-2.5pt} \rangle
\bigr)
\colon
\Bigl\langle
\nabla T_{m-1}(t,\cdot) 
\otimes  
\nabla T_{m-1}(t,\cdot)
\Bigr\rangle
\,dt
\,
\biggr|
\,.
\end{align}
Next, we recall from \eqref{e.Jhat} that $\hat{\J}_m(t) - \langle \hspace{-2.5pt} \langle
\hat{\J}_m \rangle \hspace{-2.5pt} \rangle$ is a zero--mean $\tau_{m}^{\prime\prime}$--periodic function of time (recall~\eqref{e.ratios.taum}), and so we may write \begin{equation*}
\hat{\J}_m(t) 
-
\langle \hspace{-2.5pt} \langle
\hat{\J}_m \rangle \hspace{-2.5pt} \rangle
=
\sum_{n=0}^{N_* -1}
\bigl(
L_{m,n}(t) \, \mathbf{j}_{m,n}(t)
-
\langle \hspace{-2.5pt} \langle
L_{m,n} \, \mathbf{j}_{m,n}  \rangle \hspace{-2.5pt} \rangle
\bigr)
= 
\partial_t \hat{\mathbf{Q}}_m(t)
\end{equation*}
where
\begin{equation*}
\| \hat{\mathbf{Q}}_m \|_{L^\infty([0,1])}
\leq 
C \kappa_{m-1} \tau_m^{\prime\prime}
\,,
\qquad 
\mbox{and}
\qquad
\hat{\mathbf{Q}}_m(0) =  \hat{\mathbf{Q}}_m(1) = 0
 \,.
\end{equation*}
In light of the above two displays, we integrate by parts in time, use the identity 
\begin{align*}
&\int_0^1 \hat{\mathbf{Q}}_m(t) \partial_t \, \Bigl\langle \!
\nabla T_{m-1}(t,\cdot) 
\otimes  
\nabla T_{m-1}(t,\cdot) \!
\Bigr\rangle
\,dt
\notag\\
&= 
\int_0^1 \hat{\mathbf{Q}}_m(t)
\int_{\TT^2} \Bigl( \DD_{t,m-1} \nabla T_{m-1} \otimes \nabla T_{m-1}  
+
 \nabla T_{m-1}  \otimes \DD_{t,m-1}   \nabla T_{m-1}\Bigr) (t,x)  dx dt
\,,
\end{align*}
which follows since $\nabla \cdot \b_{m-1}=0$,
and appeal to the estimates~\eqref{e.Tm.reg.upgrade} and \eqref{e.barf.cascade} to deduce
\begin{align*}
&
\biggl| 
\int_0^1 
\bigl( 
\hat{\J}_m(t) 
-
\langle \hspace{-2.5pt} \langle
\hat{\J}_m \rangle \hspace{-2.5pt} \rangle
\bigr)
\colon
\Bigl\langle
\nabla T_{m-1}(t,\cdot) 
\otimes  
\nabla T_{m-1}(t,\cdot)
\Bigr\rangle
\,dt
\,
\biggr|
\notag\\
&\qquad 
\leq  
C \| \hat{\mathbf{Q}}_m \|_{L^\infty([0,1])}
\|\nabla T_{m-1} \|_{L^2([0,1]\times \TT^2)}
\|\DD_{t,m-1}\nabla T_{m-1} \|_{L^2([0,1]\times \TT^2)}
\notag\\
&\qquad 
\leq  
C \kappa_{m-1} \tau_m^{\prime\prime}
\cdot 
\kappa_{m-1}^{-\nicefrac 12}
\| \theta_0\|_{L^2(\TT^2)}
\cdot
C \ep_{m-1}^{3\delta}
\bigl(\tau_m^{\prime}\bigr)^{-1}
\kappa_{m-1}^{-\nicefrac 12}
\| \theta_0\|_{L^2(\TT^2)}
\leq 
C \ep_{m-1}^{2\delta}
\left\| \theta_{0} \right\|_{L^2(\TT^2)}^2
\,.
\end{align*}
Combining this display with~\eqref{e.TJT.break} and using that
\begin{equation*}
C \bigl( C \ep_{m-1}^{2\delta} \bigr)^{\!N_*}
\kappa_{m-1}
\bigl\|
\nabla T_{m-1}
\bigr\|_{L^2([0,1]\times \TT^2)}^2 \leq\ep_{m-1}^{500} \left\| \theta_{0} \right\|_{L^2( \TT^2)}^2
\end{equation*}
for $N_*$ sufficiently large, 
we obtain~\eqref{e.ergodic.break.up.last}.

\smallskip

Finally, we collect the bounds we have obtained for the terms on the right side of \eqref{e.ergodic.break.up}, namely \eqref{e.ergodic.break.up.second}, \eqref{e.ergodic.space.kill}, \eqref{e.ergodic.break.up.last}, and obtain that 
\begin{multline*}
\biggl| 
\kappa_m 
\biggl\| 
\sum_{k\in 2\Z+1}
\xi_{m,k}
\bigl( \Itwo + \nabla  \Chi_{m,k} \bigr) \circ X_{m-1,l_k}^{-1}
\nabla T_{m-1}
\biggr\|_{L^2((0,1)\times\TT^2)}^2
-
\kappa_{m-1} 
\left\| 
\nabla T_{m-1}
\right\|_{L^2((0,1)\times\TT^2)}^2\,
\biggr| 
\\
\leq
C\ep_{m-1}^{2\delta}
\left\| \theta_{0} \right\|_{L^2( \TT^2)}^2
\,.
\end{multline*}
This concludes the proof of \eqref{e.left.to.show}.
\qed

\subsubsection*{The estimate of $\tilde{\theta}_{m} - T_{m-1}$ in $L^\infty_tL^2_x$}
We next prove the estimate
\begin{align}
\label{e.tildethetam.to.Tm}
\big\| \tilde{\theta}_{m} - T_{m-1} \big\|_{L^\infty((0,1);L^2(\TT^2))}
\leq
C\ep_{m-1}^{\delta}
\left\| \theta_{0} \right\|_{L^2(\TT^2)}
\,.
\end{align}
We proceed by bounding the last two terms on the right side of~\eqref{e.ansatz}.
First, 
using~\eqref{e.cutoff.xi},~\eqref{e.corrm}, the fact that $X_{m-1,l_k}^{-1}$ is volume preserving, and that $|\nabla X_{m-1,l_k}| \leq 2$ on $\supp(\xi_{m,k})$, we obtain
\begin{align}
\label{e.tildethetam.to.Tm.A.temptemp}
\lefteqn{
\left\| 
\sum_{k\in 2\Z+1}
\xi_{m,k}
\tilde{\Chi}_{m,k}
\bigl(
\nabla \bigl( T_{m-1} \circ X_{m-1,l_k} \bigr) \circ X_{m-1,l_k}^{-1} \bigr)
\right\|_{L^\infty((0,1);L^2(\TT^2))}
} \qquad & 
\notag \\ & 
\leq 
\sup_{k\in\N}
\left\| \tilde{\Chi}_{m,k} \right\|_{L^\infty(\R\times\TT^2)}
\left\| 
\sum_{k\in 2\Z+1}
\xi_{m,k}
\bigl(
\nabla \bigl( T_{m-1} \circ X_{m-1,l_k} \bigr) \circ X_{m-1,l_k}^{-1} \bigr)
\right\|_{L^\infty((0,1);L^2(\TT^2))}
\notag \\ & 
\leq 
C\ep_m^{1-\gamma} 
\sup_{k\in\N} \left\| 
\xi_{m,k}
( \nabla T_{m-1}) \circ X_{m-1,l_k}  
\right\|_{L^\infty((0,1);L^2(\TT^2))}
\,.
\end{align}
The right side of the above display requires that we bound $\|(\nabla T_{m-1})\circ X_{m-1,l_k})(t,\cdot)\|_{L^2(\TT^2)}$ in $L^\infty(0,1)$, instead of $L^2(0,1)$ (a bound that would have been available from~\eqref{e.nabm2}). Instead, we appeal to the fundamental theorem of calculus in time which in light of $\xi_{m,k}((k-\frac 54)\tau_m) = 0$, gives
\begin{align*}
&\sup_{t\in\RR}  \left\| 
\xi_{m,k}(t)
(\nabla  T_{m-1}) \circ X_{m-1,l_k}(t,\cdot)
\right\|_{L^2(\TT^2)}
\notag\\
&\leq 
\int_{(k-\frac 54)\tau_m}^{(k+\frac 54)\tau_m}
\Bigl( \|\partial_t \xi_{m,k}\|_{L^\infty(\RR)} \left\| 
\nabla   T_{m-1}  (t,\cdot)
\right\|_{L^2(\TT^2)}  
+
\|\xi_{m,k}\|_{L^\infty(\RR)} \left\| 
\DD_{t,m-1} \nabla  T_{m-1}(t,\cdot)
\right\|_{L^2(\TT^2)}  
\Bigr) dt
\,.
\end{align*}
By further appealing to~\eqref{e.Tm.reg.upgrade} with~$n=0$ and to~\eqref{e.barf.cascade} with $n=0$ and~$\ell=1$, and using also~\eqref{e.xi.mk.def}, we deduce from the previous display that\begin{align}
\lefteqn{ 
\sup_{t\in\RR}  \left\| 
\xi_{m,k}(t)
(\nabla  T_{m-1}) \circ X_{m-1,l_k}(t,\cdot)
\right\|_{L^2(\TT^2)}
} \qquad &
\notag \\ &
\leq 
C \tau_m^{-\nicefrac12}
\left\| 
\nabla   T_{m-1}
\right\|_{L^2(0,1;L^2(\TT^2))}
+ C \tau_m^{\nicefrac 12}
\left\| 
\DD_{t,m-1}\nabla   T_{m-1}  
\right\|_{L^2(0,1;L^2(\TT^2))}
\notag\\
&\leq C
\Bigl(
\tau_m^{-\nicefrac12}
+
\tau_m^{\nicefrac 12}
(\tau_m^\prime)^{-1}
\Bigr)
\kappa_{m-1}^{-\nicefrac 12}
\left\| \theta_{0} \right\|_{L^2(\TT^2)}
\leq C
\tau_m^{-\nicefrac 12}
\kappa_{m-1}^{-\nicefrac 12}
\left\| \theta_{0} \right\|_{L^2( \TT^2)}
\,.
\label{e.tildethetam.to.Tm.A.temp}
\end{align}
By combining~\eqref{e.tildethetam.to.Tm.A.temptemp} and~\eqref{e.tildethetam.to.Tm.A.temp}, we deduce
\begin{multline}
\label{e.tildethetam.to.Tm.A}
\biggl\| 
\sum_{k\in 2\Z+1}
\xi_{m,k}
\tilde{\Chi}_{m,k}
\bigl(
\nabla \bigl( T_{m-1} \circ X_{m-1,l_k} \bigr) \circ X_{m-1,l_k}^{-1} \bigr)
\biggr\|_{L^\infty((0,1);L^2(\TT^2))}
\\
\leq C 
\underbrace{
\ep_m^{1-\gamma} \tau_m^{-\nicefrac 12}\kappa_{m-1}^{-\nicefrac 12}
}_{\leq C \ep_{m-1}^\delta}
\leq C  \ep_{m-1}^\delta
\left\| \theta_{0} \right\|_{L^2( \TT^2)}
\,.
\end{multline}
In the last inequality we used that 
\begin{equation*}
\ep_m^{1-\gamma} \tau_{m}^{-\nicefrac12} \kappa_{m-1}^{-\nicefrac 12}
=
\ep_{m-1}^{q (1-\gamma) - 1 + \nicefrac{\beta}{2} - 2\delta} \ep_{m-1}^{- \frac{q\beta}{q+1}}
=
\ep_{m-1}^{q (1-\gamma) - 1 + \nicefrac{\beta}{2} - \frac{q\beta}{q+1} - 3\delta}
\ep_{m-1}^{\delta}
\end{equation*}
and, since~$\beta \in (1,\nicefrac43)$,
\begin{align*}
q (1-\gamma) - 1 + \frac{\beta}{2}- \frac{q\beta}{q+1} - 3\delta
&=
(q-1) 
+
\beta \biggl( \frac 12 - \frac{q^2}{q+1}\biggr)
-
\frac{3(q-1)^2}{4(q+1)(4q-1)}
= \frac{(4-3\beta)^2}{80 \beta - 64}
\geq 0\,.
\end{align*}
Finally, by the estimate for $H_m$ in~\eqref{e.Hm.Linfty}, we have
\begin{equation}
\norm{\widetilde{H}_m(t,\cdot)}_{L^{\infty}_tL^2_x([0,1]\times \TT^2)} 
\leq  
C\ep_{m-1}^{\delta}
\left\| \theta_{0} \right\|_{L^2(\TT^2)}
\,.
\end{equation}
The previous two displays,~\eqref{e.ansatz} and the triangle inequality yield~\eqref{e.tildethetam.to.Tm}. 
\qed

\smallskip

The proof of Proposition~\ref{p.indystepdown} is now complete, as the estimate for the first term on the left side of~\eqref{e.homogenization.m} follows from~\eqref{e.Tm.thetam},~\eqref{e.twoscale.m} and~\eqref{e.tildethetam.to.Tm}.

\begin{remark}[Uniform continuity in time of the scalar variance]
\label{r.LeBron}
By a simple interpolation argument, the estimate for the first term in~\eqref{e.homogenization.m} can be improved to obtain some positive regularity in the time variable. Indeed, we have that there exists an exponent~$\mu(\beta)>0$ and a constant~$C(\beta) <\infty$ such that 
\begin{equation}
\label{e.continuous.indeed}
\big\| \theta_m - \theta_{m-1}\big\|_{C^{0,\mu}([0,1];L^2(\TT^2))}
\leq 
C \ep_{m-1}^{\nicefrac \delta2} 
\left\| \theta_{0} \right\|_{H^1( \TT^2)}
\,.
\end{equation}
To prove~\eqref{e.continuous.indeed}, we first interpolate~$C^{0,\mu}$ between~$L^\infty$ and~$C^{0,1}$ and apply~\eqref{e.homogenization.m} to get
\begin{align}
\label{e.Carm.interp}
\big\| \theta_m - \theta_{m-1}\big\|_{C^{0,\mu}([t_0,1];L^2(\TT^2))}
&
\leq
\big\| \theta_m - \theta_{m-1}\big\|_{L^\infty([0,1];L^2(\TT^2))}^{1-\mu}
\big\| \partial_t ( \theta_m - \theta_{m-1} ) \big\|_{L^\infty([t_0,1];L^2(\TT^2))}^{\mu}
\notag \\ &
\leq 
C \ep_{m-1}^{(1-\mu) \delta} 
\left\| \theta_{0} \right\|_{L^2( \TT^2)}^{1-\mu}
\big\| \partial_t ( \theta_m - \theta_{m-1} ) \big\|_{L^\infty([t_0,1];L^2(\TT^2))}^{\mu}\,.
\end{align}
The second factor may be estimated very crudely using standard parabolic Schauder estimates. These imply that, any equation of the form~\eqref{e.passive.scalar} with~$\b(t,\cdot)$ divergence-free and satisfying~\eqref{e.b.reg}, 
the solution~$\theta^\kappa$ of~\eqref{e.passive.scalar} satisfies 
\begin{equation}
\label{e.Schauder}
\| \theta^\kappa \|_{C^{1+\nicefrac\alpha2}_t C^{2+\alpha}_x((t_0,1)\times\TT^d)}
\leq 
C (t_0^2 \wedge \kappa)^{-\frac{d+2}{2\alpha}}
\left\| \theta_{0} \right\|_{L^2( \TT^2)}
\,.
\end{equation}
To get this, cover the torus~$\TT^d$ with balls of radius~$(t_0^2 \wedge \kappa)^{-\nicefrac1\alpha}$ and apply the interior Schauder estimates (after rescaling) in the corresponding parabolic cylinders. Giving up volume factors then leads to~\eqref{e.Schauder}. (The power of~$\kappa$ which appears is not important, we just need that it is some negative power of~$\kappa$.) Applying this to~$\theta_m$ and~$\theta_{m-1}$ separately and then inserting the bound into~\eqref{e.Carm.interp}, we obtain 
\begin{equation}
\label{e.continuous.indeed.t0}
\big\| \theta_m - \theta_{m-1}\big\|_{C^{0,\mu}([t_0,1];L^2(\TT^2))}
\leq
C \ep_{m-1}^{(1-\mu) \delta} 
(t_0^2 \wedge \ep_{m}^{\beta+\gamma})^{-\frac{\mu(d+2)}{2\alpha}} 
\left\| \theta_{0} \right\|_{L^2( \TT^2)}
\,.
\end{equation}
For short times, we use that for any equation of the form~\eqref{e.passive.scalar} with~$\kappa>0$ and~$\b(t,\cdot)$ divergence-free and satisfying~\eqref{e.b.reg}, 
the solution~$\theta^\kappa$ of~\eqref{e.passive.scalar} satisfies 
\begin{equation}
\label{e.Duhamel.shorttime}
\big\llbracket \theta^\kappa \big\rrbracket_{C^{0,\nicefrac12}([0,\kappa^{2/\alpha}];L^2(\TT^d))}
\leq 
C \kappa^{-1-\frac{d+2}{2\alpha}}
\left\| \theta_{0} \right\|_{H^1( \TT^d)}
\,.
\end{equation}
To get this, cover the torus~$\TT^d$ with balls of radius~$\kappa^{\nicefrac1\alpha}$ and 
We therefore get, using an interpolation of~$C^{0,\mu}$ between~$L^\infty$ and~$C^{0,\nicefrac12}$ and~\eqref{e.homogenization.m},
\begin{align*}
\lefteqn{
\big\| \theta_m - \theta_{m-1}\big\|_{C^{0,\mu}([0,\ep_m^{2(\beta+\gamma)/(\beta-1)}];L^2(\TT^2))}
} \qquad & 
\\ & 
\leq
\big\| \theta_m - \theta_{m-1}\big\|_{L^\infty([0,\ep_m^{2(\beta+\gamma)/(\beta-1)}];L^2(\TT^2))}^{1-2\mu}
\big\llbracket
\theta_m - \theta_{m-1}  \big\rrbracket_{C^{0,\nicefrac12}([0,\ep_m^{2(\beta+\gamma)/(\beta-1)}];L^2(\TT^2))}^{2\mu}
\notag \\ &
\leq 
C \ep_{m-1}^{(1-2\mu) \delta - 2q\mu ( 1+\frac{4}{2(\beta-1)}) } 
\left\| \theta_{0} \right\|_{H^1( \TT^2)}
\,.
\end{align*}
Combining this with the above, we get 
\begin{align*}
\lefteqn{
\big\| \theta_m - \theta_{m-1}\big\|_{C^{0,\mu}([0,1];L^2(\TT^2))}
} 
\quad &
\\ & 
\leq 
C
\Bigl(
\ep_{m-1}^{(1-\mu) \delta} 
(\ep_m^{4(\beta+\gamma)/(\beta-1)} \wedge \ep_{m}^{\beta+\gamma})^{-\frac{4\mu}{2\alpha}} 
+
\ep_{m-1}^{(1-2\mu) \delta - 2q\mu ( 1+\frac{4}{2(\beta-1)}) } 
\Bigr)  
\left\| \theta_{0} \right\|_{H^1( \TT^2)}
\,.
\end{align*}
For~$\mu$ small enough, depending only on~$\delta$, we obtain~\eqref{e.continuous.indeed}. 
\end{remark}

\subsection{The proof of Theorem~\ref{t.anomalous.diffusion}}
\label{ss.proof}

We turn to the proof of Theorem~\ref{t.anomalous.diffusion}.

\smallskip

We first prove the theorem in the case that the initial datum $\theta_0$ satisfies the analyticity condition~\eqref{e.theta0.anal}. It will be convenient to define $m_* \in \NN$ by
\begin{equation}
\label{e.m*.def}
m_*:=
m_{\theta_0} - 1 =
\min \Bigl \{ m\in \N \,:\, m\geq 2\,, \ 
\ep_{m-1}^{1+ \nicefrac{\gamma}{2}} \leq R_{\theta_0} \Bigr \} - 1 \,.
\end{equation}
We start with the observation that, due  to~\eqref{e.kappam.bound} and the definition of~$m_{*}$ above, there are constants~$c_*>0$ and~$C_*<\infty$ which depend only on~$\beta$, such that
\begin{equation}
\label{e.kappa*.def}
c_*
R_{\theta_0}^{\frac{2q(\beta+\gamma)}{2+\gamma}}
\leq 
c_*\ep_{m_{*}-1}^{q(\beta+\gamma)} 
\leq 
c_* \ep_{m_{*}}^{\beta+\gamma}  \leq \kappa_{m_{*}} \leq C_*\ep_{m_{*}}^{\beta+\gamma}
\leq 
C_*
R_{\theta_0}^{\frac{2(\beta+\gamma)}{2+\gamma}}
\,.
\end{equation}
Recall that~$\theta_{m_{*}}$ satisfies the advection-diffusion equation with diffusivity~$\kappa_{m_{*}}$ and vector field~$\b_{m_*}$:
\begin{equation*}
\partial_t \theta_{m_{*}}  -\kappa_{m_{*}} \Delta \theta_{m_{*}}
+\b_{m_{*}} \cdot \nabla \theta_{m_{*}}
 = 0 \quad \mbox{in} \ (0,1)\times \TT^2\,.
\end{equation*}
The standard~$L^2$ energy estimate and the Poincar\'e inequality are applicable. Since~$\langle \theta_0 \rangle =0$, we find that
\begin{equation*}
\frac{d}{dt} \norm{\theta_{m_*}(t,\cdot)}_{L^2(\TT^2)}^2 = - 2 \kappa_{m_{*}} \norm{\nabla \theta_{m_*}(t,\cdot)}_{L^2(\TT^2)}^2
\leq - 8\kappa_{m_{*}} \pi^2 \norm{\theta_{m_*}(t,\cdot)}_{L^2(\TT^2)}^2 \,.
\end{equation*}
Therefore, we obtain
\begin{align*}
\int_0^1 \kappa_{m_{*}} \norm{\nabla \theta_{m_*}(s,\cdot)}_{L^2(\TT^2)}^2 ds
&= \frac12 \Bigl( \norm{\theta_0}_{L^2(\TT^2)}^2 - \norm{\theta_{m_*}(1,\cdot)}_{L^2(\TT^2)}^2 \Bigr)
\notag\\
&\geq \norm{\theta_0}_{L^2(\TT^2)}^2 \frac{1 - e^{-8\kappa_{m_{*}} \pi^2 }}{2}  
\geq 
 2 \pi^2\kappa_{m_{*}} \norm{\theta_0}_{L^2(\TT^2)}^2 
\,.
\end{align*}
The last inequality is valid only if~$8\kappa_{m_{*}} \pi^2\leq 1$, but this can be assumed to be valid in view of~\eqref{e.m*.def} and \eqref{e.kappa*.def}, by taking~$\Lambda$ larger if necessary. Therefore, we obtain
\begin{equation}
\label{e.theta.m*.anomaly}
2 \pi^2c_*
R_{\theta_0}^{\frac{2q(\beta+\gamma)}{2+\gamma}}
\norm{\theta_0}_{L^2(\TT^2)}^2 
\leq 
\kappa_{m_{*}} \int_0^1  \norm{\nabla \theta_{m_*}(s,\cdot)}_{L^2(\TT^2)}^2 ds
\,.
\end{equation}

By~\eqref{e.tildethetam.energy.diss} of Proposition~\ref{p.indystepdown}, for every $m\in\{m_{*}+1,\ldots,M\}$, it holds that
\begin{align}
\label{e.crunch}
\Bigl ( 1 - C\ep_{m-1}^\delta\Bigr )
\kappa_{m-1} \left\| \nabla \theta_{m-1} \right\|_{L^2((0,1)\times\TT^2)}^2
&
\leq
\kappa_m \left\| \nabla \theta_{m} \right\|_{L^2((0,1)\times\TT^2)}^2
\notag \\ &
\leq
\Bigl ( 1 + C\ep_{m-1}^\delta\Bigr )
\kappa_{m-1} \left\| \nabla \theta_{m-1} \right\|_{L^2((0,1)\times\TT^2)}^2
\,.
\end{align}
We take the parameter~$\Lambda$ to be large enough that, with~$C$ as in the previous display, we have that 
\begin{equation}
\label{e.Lambda.choice}
C \Lambda^{-\delta} 
\leq \frac{1}{100}\,.
\end{equation}
This ensures that
\begin{equation*}
1 - C\ep_{m-1}^\delta \geq 1 - C\ep_{1}^\delta \geq \frac{99}{100}\,.
\end{equation*}
By induction, in view of~\eqref{e.minsep}, it follows from~\eqref{e.crunch} that 
\begin{align*}
\min_{m\in\{ m_*, \ldots,M\}}
\kappa_m \left\| \nabla \theta_{m} \right\|_{L^2((0,1)\times\TT^2)}^2
\geq
\biggl (
\prod_{m=m_*+1}^M
\Bigl ( 1 - C\ep_{m-1}^\delta \Bigr ) 
\biggr )
\kappa_{m_*} \left\| \nabla \theta_{m_*} \right\|_{L^2((0,1)\times \TT^2)}^2
\,,
\end{align*}
where $M = M(\kappa)$ is is an integer satisfying~\eqref{e.permitted}.
Observe that, using the elementary inequality 
$-2x \leq \log (1-x)$ valid for all $x\in (0,\nicefrac12]$, 
we get
\begin{align*}
\log 
\prod_{m=m_*+1}^M
\Bigl ( 1 - C\ep_{m-1}^\delta\Bigr )
=
\sum_{m=m_*+1}^M
\log 
\Bigl ( 1 - C\ep_{m-1}^\delta\Bigr )
\geq 
-2 C 
\sum_{m=m_*+1}^M
\ep_{m-1}^\delta
\geq 
-C 
\ep_{m_*}^\delta
\,.
\end{align*}
Note that we used~\eqref{e.Lambda.choice} to get the last inequality in the above display. 
We therefore obtain 
\begin{equation*}
\prod_{m=m_*+1}^M
\Bigl ( 1 - C\ep_{m-1}^\delta\Bigr )
\geq 
\exp \bigl( 
-C 
\ep_{m_*}^\delta
\bigr)
\geq 
\frac {3}{4}
\,,
\end{equation*}
and hence
\begin{equation*}
\min_{m\in\{ m_*, \ldots,M\}}
\kappa_m \left\| \nabla \theta_{m} \right\|_{L^2((0,1)\times\TT^2)}^2
\geq
\frac34
\kappa_{m_*} \left\| \nabla \theta_{m_*} \right\|_{L^2((0,1)\times \TT^2)}^2
\,.
\end{equation*}
By a very similar argument, using the upper bound of~\eqref{e.crunch} rather than the lower bound, we also obtain that
\begin{equation*}
\max_{m\in\{ m_*, \ldots,M\}}
\kappa_m \left\| \nabla \theta_{m} \right\|_{L^2((0,1)\times\TT^2)}^2
\leq
\frac43
\kappa_{m_*} \left\| \nabla \theta_{m_*} \right\|_{L^2((0,1)\times \TT^2)}^2\,.
\end{equation*}
In particular, since $\kappa = \kappa_M$, we have 
\begin{align}
\label{e.thetaN.anom}
\frac34
\kappa_{m_*} \left\| \nabla \theta_{m_*} \right\|_{L^2((0,1)\times \TT^2)}^2
\leq 
\kappa \left\| \nabla \theta_{M} \right\|_{L^2((0,1)\times\TT^2)}^2
\leq 
\frac43
\kappa_{m_*} \left\| \nabla \theta_{m_*} \right\|_{L^2((0,1)\times \TT^2)}^2
\,. 
\end{align}

\smallskip
We can write an equation for the difference~$\theta-\theta_M$ as
\begin{equation*}
\partial_t (\theta-\theta_M)
-\kappa \Delta (\theta-\theta_M)
+\b\cdot\nabla (\theta-\theta_M)
=
\nabla \cdot \bigl ( (\phi - \phi_M) \nabla\theta_M\bigr)
\quad \text{in} \ (0,1) \times \TT^d\,.
\end{equation*}
Recall that, by~\eqref{e.phi.m.m-1.bounds},
\begin{equation*}
\| \phi - \phi_M\|_{L^{\infty}(\R\times\TT^d)}
\leq C \ep_{M+1}^{\beta}
=
C \ep_{M}^{\beta q}
\,.
\end{equation*}
We may therefore compare $\theta_M$ to $\theta$ using the above displays,~\eqref{e.permitted},~\eqref{e.thetaN.anom} and an energy estimate: 
\begin{align}
&
\left\| \theta - \theta_M \right\|_{L^\infty((0,1);L^2 (\TT^2))}^2
+
\kappa \left\| \nabla \theta - \nabla \theta_M \right\|_{L^2((0,1)\times \TT^2)}^2
\notag \\ & \quad
\leq 
\frac{C}{\kappa} 
\left\| \phi - \phi_M \right\|_{L^\infty((0,1)\times\TT^2)}^2
\left\| \nabla \theta_M \right\|_{L^2((0,1)\times\TT^2)}^2
\notag \\ & \quad
\leq 
\frac{C}{\kappa_M^2}
\left\| \phi - \phi_M \right\|_{L^\infty((0,1)\times\TT^2)}^2
\kappa_{m_*} \left\| \nabla \theta_{m_*} \right\|_{L^2((0,1)\times \TT^2)}^2
\notag \\ & \quad
\leq
C \ep_{M}^{-\frac{4\beta}{q+1}} \ep_{M}^{2\beta q}
\cdot
\Bigl (
\kappa_{m_*} \left\| \nabla \theta_{m_*} \right\|_{L^2((0,1)\times \TT^2)}^2\Bigr )
\notag \\ & \quad
=
C \ep_M^{2 \beta ( q - \frac{2}{q+1}) }
\Bigl (
\kappa_{m_*} \left\| \nabla \theta_{m_*} \right\|_{L^2((0,1)\times \TT^2)}^2\Bigr )
\label{e.theta.thetaM.string}
\\ & \quad
\leq C\Lambda^{-2M\beta ( q - \frac{2}{q+1})}\Bigl (
\kappa_{m_*} \left\| \nabla \theta_{m_*} \right\|_{L^2((0,1)\times \TT^2)}^2\Bigr )
\notag 
\,.
\end{align}
Since~$q>1$, the exponent of $\Lambda$ in the last term on the right side is negative. Enlarging~$\Lambda$, if necessary, and using that $M\geq 1$, we obtain
\begin{equation*}
\left\| \theta - \theta_M \right\|_{L^\infty((0,1);L^2 (\TT^2))}^2
+
\kappa \left\| \nabla \theta - \nabla \theta_M \right\|_{L^2((0,1)\times \TT^2)}^2
\leq 
\frac14\Bigl (
\kappa_{m_*} \left\| \nabla \theta_{m_*} \right\|_{L^2((0,1)\times \TT^2)}^2\Bigr )
\,.
\end{equation*}
Combining the previous display with~\eqref{e.thetaN.anom} and~\eqref{e.theta.m*.anomaly} therefore yields
\begin{align}
\label{e.theta.yes.anom}
\kappa^{\nicefrac12} \left\| \nabla \theta \right\|_{L^2((0,1)\times\TT^2)}
&
\geq 
\kappa^{\nicefrac12} \bigl ( \left\| \nabla \theta_M \right\|_{L^2((0,1)\times\TT^2)}
-
\left\| \nabla \theta - \nabla \theta_M  \right\|_{L^2((0,1)\times\TT^2)}\bigr )
\notag \\ &
\geq
\frac1{\sqrt{2}}
\kappa_{m_*}^{\nicefrac12} \left\| \nabla \theta_{m_*} \right\|_{L^2((0,1)\times \TT^2)}
\notag \\ &
\geq
\pi \sqrt{c_*}
R_{\theta_0}^{\frac{q(\beta+\gamma)}{2+\gamma}}
\norm{\theta_0}_{L^2(\TT^2)} 
\,.
\end{align}
This completes the proof of Theorem~\ref{t.anomalous.diffusion} in the case that the analyticity condition~\eqref{e.theta0.anal} is satisfied.

\smallskip

We now turn to the argument for general~$\theta_0 \in H^1(\TT^2)$ of zero mean. We introduce the length scale~$L_{\theta_0}$ implicitly appearing in~\eqref{e.varrho.def}: 
\begin{equation}
\label{e.L.theta0}
L_{\theta_0} := 
\frac{\| \theta_0  \|_{L^2(\TT^2)}}{\| \theta_0  \|_{H^{1}(\TT^2)} }
\end{equation}
and we mollify~$\theta_0$ (viewed as a periodic function on~$\R^2$) with the standard heat kernel~$\Phi$ (with diffusion coefficient of unit size) at time~$\alpha^2 L_{\theta_0}^2$:
\begin{equation}
\tilde{\theta}_0:= \theta_0 \ast \Phi(\alpha^2 L_{\theta_0}^2,\cdot)\,.
\label{e.analytic.approximation}
\end{equation}
Here~$\alpha>0$ is a small parameter we will choose below. It is clear that~$\tilde{\theta}_0$ is a periodic and mean-zero function. It is furthermore analytic, since~$\Phi(\alpha^2 L_{\theta_0}^2,\cdot)$ is, and satisfies, for a universal constant~$C<\infty$, 
\begin{equation*}
\max_{|\aa|=n}
\bigl\| \partial^\aa \tilde{\theta}_0 \bigr\|_{L^2(\TT^2)}
\leq
\| \theta_0 \|_{L^2(\TT^2)}
n!
\biggl( \frac{C}{\alpha L_{\theta_0}}\biggr)^{\!n}
\,, \quad \forall n\in\N. 
\end{equation*}
That is, the analyticity condition~\eqref{e.theta0.anal} is valid for~$R_{\tilde{\theta}_0} = c\alpha L_{\theta_0}$. 
If we let~$\tilde{\theta}$ be the solution of~\eqref{e.passive.scalar} with~$\tilde{\theta}_0$ in place of~$\theta_0$, then~$\theta-\tilde{\theta}$ is also a solution of the same equation with initial data~$\theta_0 - \tilde{\theta}_0$, and therefore the incompressibility of $\b$ implies as in~\eqref{e.energy.balance} that 
\begin{equation}
\label{e.energy.balance.2}
\bigl\| (\theta - \tilde{\theta})(t,\cdot) \bigr\|_{L^2(\TT^2)}^2 
+
2\kappa
\bigl\| \nabla (\theta - \tilde{\theta})\bigr\|_{L^2((0,t)\times \TT^2)}^2
= 
\bigl\| \theta_0 - \tilde{\theta}_0\bigr\|_{L^2(\TT^2)}^2 
\,, \quad \forall t>0\,.
\end{equation}
In view of~\eqref{e.theta.yes.anom} (applied to $\tilde{\theta}$ replacing $\theta$ and $\tilde{\theta}_0$ replacing $\theta_0$) and  the triangle inequality, it suffices to show that we can choose~$\alpha$ in such a way that
\begin{equation}
\label{e.wts}
\bigl\| \theta_0 - \tilde{\theta}_0\bigr\|_{L^2(\TT^2)} \leq \frac12 \pi \sqrt{c_*}
R_{\tilde{\theta}_0}^{\frac{q(\beta+\gamma)}{2+\gamma}}
\bigl\| \theta_0\bigr\|_{L^2(\TT^2)}
\,.
\end{equation}
We next compute
\begin{equation*}
\bigl\| \tilde{\theta}_0 \bigr\|_{L^2(\TT^2)}^2
=
\bigl\| \theta_0 \bigr\|_{L^2(\TT^2)}^2
-
2\int_0^{\alpha^2 L_{\theta_0}^2} 
\int_{\TT^2}
\bigl| (\theta_0 \ast \nabla \Phi( s,\cdot) )(x)\bigr|^2\,dx\,dt
\,.
\end{equation*}
Using that 
\begin{equation*}
\int_0^{\alpha^2 L_{\theta_0}^2} 
\int_{\TT^2}
\bigl| (\theta_0 \ast \nabla \Phi( s,\cdot) )(x)\bigr|^2\,dx\,dt
\leq
\alpha^2 L_{\theta_0}^2
\bigl\| \nabla \theta_0 \bigr\|_{L^2(\TT^2)}^2
\leq 
C \alpha^2 \| \theta_0 \|_{L^2(\TT^2)}^2
\,,
\end{equation*}
we therefore obtain
\begin{equation}
\label{e.them.mollifiers}
(1- C \alpha)\bigl\| \theta_0 \bigr\|_{L^2(\TT^2)}
\leq
\bigl\| \tilde{\theta}_0 \bigr\|_{L^2(\TT^2)}
\leq
\bigl\| \theta_0 \bigr\|_{L^2(\TT^2)}
\,.
\end{equation}
Since~$R_{\tilde{\theta}_0} = c\alpha L_{\theta_0}$, the inequality~\eqref{e.wts} will be valid provided that~$\alpha$ satisfies  
\begin{equation*}
\alpha \leq c L_{\theta_0}^{\frac{q(\beta+\gamma)}{2+\gamma-q(\beta+\gamma)}}
\,.
\end{equation*}
With this choice of~$\alpha$, we observe that \eqref{e.theta.yes.anom},~\eqref{e.energy.balance.2}, and~\eqref{e.wts} imply 
\begin{equation*}
\kappa^{\nicefrac12} \left\| \nabla \theta \right\|_{L^2((0,1)\times\TT^2)}
\geq
\frac12 \pi \sqrt{c_*}
R_{\theta_0}^{\frac{q(\beta+\gamma)}{2+\gamma}}
\norm{\theta_0}_{L^2(\TT^2)} 
= 
c
L_{\theta_0}^{\frac{q(\beta+\gamma)}{2+\gamma-q(\beta+\gamma)}}
\norm{\theta_0}_{L^2(\TT^2)} 
\,.
\end{equation*}
This completes the proof of the theorem.

\smallskip

As a final remark, we note that the exponent of~$L_{\theta_0}$, namely~$\frac{q(\beta+\gamma)}{2+\gamma-q(\beta+\gamma)}$, can be taken arbitrarily close to~$\frac{\beta}{2-\beta}$, by taking~$q$ closer to~$1$ than in~\eqref{e.q.def.0}, at the cost of all constants depending additionally on~$q$. With~$\alpha = \beta-1$, this matches what was promised in~\eqref{e.varrho.def}.
\qed

\begin{remark}[Uniform continuity in time of the scalar variance]
\label{r.LeBron.2}
By using~\eqref{e.continuous.indeed} and repeating the argument of Remark~\ref{r.LeBron} to get a similar bound for the difference~$\theta - \theta_M$, using in  particular~\eqref{e.Schauder},~\eqref{e.Duhamel.shorttime} and~\eqref{e.theta.thetaM.string}, we obtain 
\begin{align}
\label{e.continuous.indeed.indeed}
\lefteqn{
\big\| \theta \big\|_{C^{0,\mu}([0,1];L^2(\TT^2))}
} \quad & \notag \\
&
\leq 
\big\| \theta_{m_*} \big\|_{C^{0,\mu}([0,1];L^2(\TT^2))}
+
\big\| \theta - \theta_M \big\|_{C^{0,\mu}([0,1];L^2(\TT^2))}
+
\sum_{m=m_*}^M
\big\| \theta_{m} - \theta_{m-1} \big\|_{C^{0,\mu}([0,1];L^2(\TT^2))}
\notag \\ & 
\leq 
C_{\theta_0}
\,,
\end{align}
where~$C_{\theta_0}$ depends only on~$\left\| \theta_{0} \right\|_{H^1( \TT^2)}$ and~$\beta$.
\end{remark}

\subsection{Lack of selection in the vanishing diffusivity limit}
\label{ss.proof.selection}
 
The goal of this section is to prove the following statement concerning the lack of selection principle for vanishing diffusivity limits of the advection-diffusion equation to the transport equation. Our result applies to a large class of $\dot{H}^2(\TT^2)$ initial conditions, but not to all of them.
\begin{proposition}
\label{p.no.selection.principle}
Fix $\beta \in [\nicefrac{68}{67},\nicefrac{4}{3})$.
There exists a constant $C_* = C_*(\beta)\geq 1$ such that the following holds.
For every parameters $A \in (0,1]$ and $B>1$, assume that $\Lambda$ is taken sufficiently large with respect to $\beta,A,B$ to ensure that (see condition~\eqref{e.take.Lambda.large.A.B} below): 
\begin{equation*}
C_* \Lambda^{-\delta} \leq \min\left\{ A , B^{- \frac{2}{2+\gamma}(2 -\beta +\frac{2\beta}{q+1}) } \right\}
\, .
\end{equation*}
Fix the sequence $\{\ep_m\}_{m\geq 0}$ according to \eqref{e.epm.choice}.
Choose an initial datum $\theta_0\in \dot{H}^2(\TT^2)$. Define the length scale $L_{\theta_0} := \frac{\|\theta_0\|_{L^2(\TT^2)}}{\|\nabla \theta_{0}  \|_{L^2( \TT^2)}}$ and  let $m_* \geq 1$ be the unique integer such that 
$
\ep_{m_*} \leq C_* L_{\theta_0}^{ \frac{2}{2+\gamma - 2 q \delta}} < \ep_{m_*-1}
$.
Assuming that $\theta_0$ satisfies   (see conditions~\eqref{e.theta.0.condition.1} and~\eqref{e.theta.0.condition.2} below): 
\begin{equation*}
\frac{\|\nabla\theta_0\|_{L^2(\TT^2)}^4}{\|\theta_{0}  \|_{L^2( \TT^2)}^2 \|\Delta \theta_0\|_{L^2(\TT^2)}^2} 
 \geq A\,,
 \qquad \mbox{and}\qquad
 C_* L_{\theta_0}^{ \frac{2}{2+\gamma - 2 q \delta}} \leq B \ep_{m_*}
 \,,
\end{equation*}
there exist two sequences of diffusivities $\{ \kappa^{(1)}_m \}_{m\in\N}$ and~$\{ \kappa^{(2)}_m \}_{m\in\N}$, both converging to $0$ as $m\to \infty$, such that the corresponding solutions $\{\theta^{\kappa_m^{(1)}}\}_m$ and $\{ \theta^{\kappa_m^{(2)}}\}_m$ of the advection-diffusion equation with initial data $\theta_0$ and drift $\b$, converge as $m\to \infty$  in $C^{0,\mu}((0,1);L^2(\TT^2))$ (for some $\mu>0$) to two {\em distinct} weak solutions of the associated transport equation. 
\end{proposition} 
 
In order to prove Proposition~\ref{p.no.selection.principle}, we are going to find two sequences~$\{ \kappa^{(1)}_m \}_{m\in\N}$ and~$\{ \kappa^{(2)}_m \}_{m\in\N}$ such that:
\begin{itemize}
\item $\kappa^{(1)}_m$ and $\kappa^{(2)}_m$ belong to the~$m$th interval in~\eqref{e.permissible.K}, namely~$[ \tfrac 12 \ep_m^{ \frac{2\beta}{\q+1}}, 2 \ep_m^{ \frac{2\beta}{\q+1}}]$; 

\item The choice is made such that the ratio of~$\kappa^{(1)}_m$ to $\kappa^{(2)}_m$ is not close to one in the sense that it lies in~$[\frac12,2] \,\setminus\, [ \frac 34, \frac  43]$.

\item Each of these infinite sequences satisfies the required recurrence formula for the effective diffusivities, namely~\eqref{e.kappa.sequence.again} below. 

\end{itemize} 

The existence of this pair of sequences is a consequence of Lemma~\ref{l.kappa.m.flippity.floppity}, below. 
We then turn our attention to the two corresponding solutions of the advection-diffusion equation~\eqref{e.theta.m} with~$m=m_*$, with~$m_*$ defined as above in~\eqref{e.m*.def}, which we denote by~$\theta_{m_*}^{(1)}$ and~$\theta_{m_*}^{(2)}$. We show in Lemma~\ref{l.solutions.no.same} that these two solutions have to be different, at least for short times~$t$, as they experience a different amount of diffusion (since~$\kappa_{m_*}^{(1)} \neq \kappa_{m_*}^{(2)}$). Finally, we appeal to the bound~\eqref{e.homogenization.m} proved above to obtain a positive lower bound on the size of~$\| \theta_{m}^{(1)}- \theta_{m}^{(2)} \|_{L^2}$ for all~$m\geq m_*$. The estimate~\eqref{e.homogenization.m} also gives us the convergence in~$L^\infty_tL^2_x$ of~$\theta_{m}^{(i)}$ to a solution of the transport equation, which is evidently different for~$i=1$ and~$i=2$.

\smallskip

We begin with the following lemma, which revisits the analysis of the recurrence~\eqref{e.kappa.sequence} given in Lemma~\ref{l.recurse} and extracts some extra information. 

\begin{lemma}
\label{l.kappa.m.flippity.floppity}

There exists~$C(\beta)<\infty$ and~$\rho(\beta)>0$ such that, if~$\Lambda \geq C$, then, for each~$\tau \in [\frac12,2]$, there exists an infinite sequence~$\{ \kappa_m\}_{m\in\N}$ satisfying
\begin{equation}
\label{e.kappa.sequence.again}
\kappa_{m-1} = \Khom_m^{\kappa_m}\,, \quad \forall m\geq 1\,, 
\end{equation}
such that 
\begin{equation}
\label{e.kappa.flip.flop.infinite}
\biggl| 
\frac{ \kappa_m  \cdot\tfrac43\sqrt{5} } {\ep_m^{\frac{2\beta}{q+1}} }
-
\tau^{(-1)^{m}}
\biggr| 
\leq C \ep_{m}^{\rho}
\leq
\frac{1}{100}\,, \quad \forall m\in\N\,.
\end{equation}
\end{lemma}
\begin{proof}
We first suppose that~$M\in\N$ and~$\{ \kappa_0, \kappa_1,\ldots,\kappa_M \} \subseteq (0,\infty)$ is a sequence of positive numbers satisfying, for every~$m\in \{1,\ldots,M\}$ the recurrence relation~\eqref{e.kappa.sequence.again}. We also assume that 
\begin{equation}
\label{e.permitted.again}
\tau := 
\frac{ \kappa_M \cdot \tfrac43\sqrt{5}} {\ep_M^{\frac{2\beta}{q+1}} }
\in \Bigl[ \frac12 ,2 \Bigr]\,.
\end{equation}
Using the recurrence relation~\eqref{e.smsimpl} for the effective diffusivities, expressed in terms of the quantity~$s_m$ defined in the proof of Lemma~\ref{l.recurse} by
\begin{equation*}
s_m:= \frac{\kappa^\prime_m\cdot\tfrac43\sqrt{5}}{a_m\ep_m^{2+\gamma}}
\,,
\end{equation*}
and~$\kappa^\prime_m$ is defined in~\eqref{e.kappa.prime.sequence}.
Recall that~$a_m$ and the exponents~$\beta$,~$\gamma$ and~$q$ are defined in such a way that~$a_m\ep_{m}^{2+\gamma} =
\ep_m^{\beta+\gamma} = \ep_n^{\frac{2\beta}{q+1}}$. 
The recurrence relation for~$s_m$ appeared above in~\eqref{e.sm.diff.recurrence}, and in view of the identity in~\eqref{e.bound.some.stuff} it can be written as 
\begin{equation}
\label{e.sm.diff.recurrence.again}
s_{m-1} 
=
\frac{1}{s_m} 
+
\underbrace{
\ep_m^{2\gamma}s_m
}_{
\leq C \ep_m^{2\gamma}
}
+
\underbrace{ 
\Bigl( 
\Bigl( \frac{\ep_m}{\ep_{m-1}^\q} \Bigr)^{\beta} - 1
\Bigr) 
}_{
\leq C \ep_m^{\nicefrac 1q} 
}
\underbrace{ 
s_m
\Bigl( \ep_m^{2\gamma} + \frac{1}{s_m^{2}} \Bigr) 
}_{
\leq C 
}
\,,
\end{equation}
where the bound for the three terms on the right side come from~\eqref{e.sm.bounded.yes} and~\eqref{e.supergeo}. We therefore obtain, for every~$m\in\{ 1,\ldots,M\}$, 
\begin{equation*}
\Bigl| 
s_{m-1} - \frac{1}{s_m} 
\Bigr| 
\leq 
C \ep_m^{2\gamma \wedge \frac{1}q} 
\,.
\end{equation*}
Iterating this inequality yields, for every~$n\in \{0,\ldots,M\}$,  
\begin{equation}
\label{e.s.m.see.saw}
\Bigl| 
s_{M-n}^{(-1)^n} - s_M
\Bigr|
\leq 
C \ep_{M-n+1}^{2\gamma\wedge \frac1q}
\,.
\end{equation}
If we enlarge~$\Lambda$ so that~$C_{\eqref{e.s.m.see.saw}} \ep_{1}^{2\gamma\wedge \frac1q} \leq \frac{1}{500}$, then we obtain,
\begin{equation*}
\max_{n \in \{0,\ldots,M\}} 
\Bigl| 
s_{M-n}^{(-1)^n} - s_M
\Bigr|
\leq 
C \ep_{M-n+1}^{2\gamma\wedge \frac1q}
\leq
\frac{1}{500}\,.
\end{equation*}
By enlarging~$\Lambda$, if necessary, we may use~\eqref{e.ratrat.prime} to obtain, 
for every~$n\in\{ 0,\ldots,M-1\}$,
\begin{equation*}
\max_{m\in \{n,\ldots,M-1\}}
\max\biggl\{ 
\frac{\kappa_{m}}{\kappa^\prime_{m}} -1
\,,
\frac{\kappa^\prime_{m}}{\kappa_{m}}-1
\biggr\} 
\leq 
C \ep_{n+1}^{\delta \wedge \gamma}
\leq
\frac{1}{500}
\,.
\end{equation*}
Combining these yields, some positive exponent~$\rho (\beta)>0$, 
\begin{equation}
\label{e.kappa.flip.flop}
\biggl| 
\frac{ \kappa_n  \cdot\tfrac43\sqrt{5} } {\ep_n^{\frac{2\beta}{q+1}} }
-
\tau^{(-1)^{M-n}}
\biggr| 
\leq C \ep_{n}^{\rho}
\leq
\frac{1}{100}\,.
\end{equation}
Finally, the existence of an infinite sequence satisfying~\eqref{e.kappa.flip.flop.infinite} is obtained by recursively defining, for each~$M \in\N$, a finite sequence~$\{ \kappa_{m}^{(M)}\}_{0\leq m \leq M}$ by 
\begin{equation}
\label{e.kappa.sequence.yet.again}
\left\{
\begin{aligned}
& 
\kappa_{m-1}^{(M)} = \Khom_m^{\kappa_m^{(M)}}\,,
\qquad m\in \{1,\ldots,M\}\,,
\\ & 
\kappa_M^{(M)} := \tau^{(-1)^M}(\tfrac43\sqrt{5})^{-1}   
\ep_M^{\frac{2\beta}{q+1}}
\,.
\end{aligned}
\right.
\end{equation}
Sending~$M\to \infty$, we find that~$\kappa^{(M)}_m$ converges, for each fixed~$m$, yielding the desired infinite sequence. 
\end{proof}

We next argue that the solutions of two  advection-diffusion equations with the same drift but different diffusivities must be different, at least for short times. In the following lemma,~$\mathbf{u}$ plays the role of~$\mathbf{b}_{m_*}$ and~$\kappa^{(i)}$ plays the role of~$\kappa^{(i)}_{m_*}$.

\begin{lemma} 
\label{l.solutions.no.same}
Suppose that~$0 < \kappa^{(1)} \leq \frac 45 \kappa^{(2)}$ and~$\mathbf{u}:\TT^2 \to \R^2$ is a smooth, incompressible vector field. Let~$\theta_0 \in H^2(\TT^2)$ and, for each~$i\in\{1,2\}$, let~$\theta^{(i)}$ be the solution of 
\begin{equation}
\label{e.pair.of.solutions}
\left\{
\begin{aligned}
& \partial_t \theta^{(i)} - \kappa^{(i)} \Delta \theta^{(i)} + \mathbf{u}  \cdot \nabla \theta^{(i)} 
= 0 
& \mbox{in} & \ (0,\infty) \times \R^2\,, 
\\ & 
\theta^{(i)} = \theta_0 
& \mbox{on} & \ \{ 0 \} \times \R^2\,.
\end{aligned}
\right.
\end{equation}
Then, for every time~$t$ satisfying
\begin{equation*}
0\leq t \leq \frac{1}{50} \min \biggl\{ \frac{1}{\|\nabla \mathbf{u}\|_{L^\infty([0,1]\times \TT^2)}}, \frac{\|\nabla\theta_0\|_{L^2(\TT^2)}^2}{\kappa^{(1)} \|\Delta \theta_0\|_{L^2(\TT^2)}^2} \biggr \}
\,,
\end{equation*}
we have the estimate
\begin{equation}
\label{e.different.L2s}
\|\theta^{(1)} (t,\cdot)\|_{L^2(\TT^2)}^2 
-
\|\theta^{(2)} (t,\cdot)\|_{L^2(\TT^2)}^2 
\geq 
\frac{11}{40} \kappa^{(1)}  t 
\|\nabla \theta_0\|_{L^2(\TT^2)}^2
\,.
\end{equation}
\end{lemma} 
\begin{proof}
With $\mathbf{u}$ smooth and incompressible, consider, for~$\kappa>0$, the drift-diffusion equation 
\begin{equation}
\partial_t \theta - \kappa \Delta \theta+ \mathbf{u} \cdot \nabla \theta = 0, \qquad 
\theta|_{t=0} = \theta_0,
\label{e.drift.and.diffusion.new}
\end{equation}
where $\theta_0 \in H^2(\TT^2)$ has zero mean. Denote $L:= \|\nabla \mathbf{u}\|_{L^\infty([0,1]\times \TT^2)}$. Standard energy estimates (in each of the space~$L^2$,~$H^1$, and~$H^2$) yields, for~$t\geq 0$, 
\begin{align*}
&\Bigl| 
\|\theta_0\|_{L^2(\TT^2)}^2 
-
\|\theta(t,\cdot)\|_{L^2(\TT^2)}^2
- 
2 \kappa t \|\nabla \theta_0\|_{L^2(\TT^2)}^2 
\Bigr| 
\notag\\
&\quad 
\leq 
2 \kappa t \|\nabla \theta_0\|_{L^2(\TT^2)}^2 (t L)^2 \exp(2 t L)
+
2 \kappa t \|\nabla \theta_0\|_{L^2(\TT^2)}^2  (t L) \exp(2 t L)
+
2 \kappa^2 t^2 \|\Delta \theta_0\|_{L^2(\TT^2)}^2  \exp(2 t L)
\,.
\end{align*}
As a consequence, we obtain
\begin{equation}
\label{e.taylor.at.time.zero}
\left|
\frac{  
\|\theta_0\|_{L^2(\TT^2)}^2 
-
\|\theta(t,\cdot)\|_{L^2}^2
 }
{
2 \kappa t \|\nabla \theta_0\|_{L^2}^2 
}
- 1 
\right|
\leq \biggl( tL + (t L)^2 + \kappa t  \frac{\|\Delta \theta_0\|_{L^2(\TT^2)}^2}{\|\nabla\theta_0\|_{L^2(\TT^2)}^2} \biggr)  \exp(2 t L).
\end{equation}
Therefore, for $t$ satisfying 
\begin{equation}
t \leq \frac{1}{50} \min \biggl\{ \frac{1}{L}, \frac{\|\nabla\theta_0\|_{L^2(\TT^2)}^2}{\kappa \|\Delta \theta_0\|_{L^2(\TT^2)}^2} \biggr \}
\,,
\label{e.t.small.inteval.*}
\end{equation}
we obtain
\begin{equation}
\frac{19}{10} \kappa t \|\nabla \theta_0\|_{L^2(\TT^2)}^2
\leq 
\|\theta_0\|_{L^2(\TT^2)}^2
-
\|\theta(t,\cdot)\|_{L^2(\TT^2)}^2 
\leq \frac{21}{10} \kappa t \|\nabla \theta_0\|_{L^2(\TT^2)}^2
\,.
\end{equation}
Now consider the solutions~$\theta^{(1)}$ and $\theta^{(2)}$  of~\eqref{e.pair.of.solutions} with diffusivity parameters~$\kappa^{(1)},\kappa^{(2)}>0$ satisfying~$\kappa^{(2)} \geq \frac 54 \kappa^{(1)}$. By~\eqref{e.taylor.at.time.zero} and the triangle inequality, for all~$t\geq 0$ satisfying~\eqref{e.t.small.inteval.*} with~$\kappa= \kappa^{(1)}$, we obtain 
\begin{equation*}
\|\theta^{(1)} (t,\cdot)\|_{L^2(\TT^2)}^2 
-
\|\theta^{(2)} (t,\cdot)\|_{L^2(\TT^2)}^2 
\geq 
\Bigl( \frac{19}{10} \kappa^{(2)} - \frac{21}{10} \kappa^{(1)} \Bigr) t \|\nabla \theta_0\|_{L^2}^2\,.
\end{equation*}
Since~$\kappa^{(2)} \geq \frac54 \kappa^{(1)}$, we obtain~\eqref{e.different.L2s}. 
The proof is complete. 
\end{proof}

We next argue that the previous two lemmas imply Proposition~\ref{p.no.selection.principle}. 

 \begin{proof}[Proof of  Proposition~\ref{p.no.selection.principle}] 
We apply Lemma~\ref{l.kappa.m.flippity.floppity} twice, with~$\tau = \tau_1:= \nicefrac12$ and~$\tau =\tau_2 := \nicefrac 34$, to find two sequences~$\{ \kappa_m^{(1)} \}_{m\in\N}$ and~$\{ \kappa_m^{(2)} \}_{m\in\N}$ satisfying the recurrence~\eqref{e.kappa.sequence.again} and, for each~$i\in\{1,2\}$, 
\begin{equation}
\label{e.kappa.flip.flop.infinite.i}
\biggl| 
\frac{ \kappa_n^{(i)}  \cdot\tfrac43\sqrt{5} } {\ep_n^{\frac{2\beta}{q+1}} }
-
\tau_i^{(-1)^{n}}
\biggr| 
\leq C \ep_{n}^{\rho}
\leq
\frac{1}{100}\,, \quad \forall n\in\N\,.
\end{equation}
This implies in particular that, for every~$n\in\N$, 
\begin{equation}
\label{e.kappas.gap}
\frac 54 \leq 
\max\biggl\{ \frac{ \kappa_n^{(1)}}{ \kappa_n^{(2)}} , \frac{ \kappa_n^{(2)}}{ \kappa_n^{(1)}} \biggr\} 
\leq 2
\,.
\end{equation}
According to the two sequences~$\{ \kappa_m^{(1)} \}_{m\in\N}$ and~$\{ \kappa_m^{(2)} \}_{m\in\N}$ constructed above, we denote by $\theta_m^{(1)}$ and~$\theta_m^{(2)}$ the solutions of the initial-value problems
\begin{equation}
\label{e.theta.m.yet.again}
\left\{
\begin{aligned}
& \partial_t \theta_m^{(i)} - \kappa_m^{(i)} \Delta \theta_m^{(i)} + \b_m \cdot \nabla \theta_m^{(i)} 
= 0 
& \mbox{in} & \ (0,\infty) \times \R^2\,, 
\\ & 
\theta_m^{(i)} = \theta_0 
& \mbox{on} & \ \{ 0 \} \times \R^2\,,
\end{aligned}
\right.
\end{equation}
where $\b_m$ is as in the rest of the paper (see~\eqref{e.psi.recursion}).

Our next goal is to show that if $\theta_0 \in H^2(\TT^2)$ has zero mean and satisfies certain conditions (see~\eqref{e.theta.0.condition.1} and~\eqref{e.theta.0.condition.2} below), then there exists a time $t\in (0,1]$ such that $ | \|\theta_m^{(1)} (t,\cdot)\|_{L^2(\TT^2)}^2 - \|\theta_m^{(2)} (t,\cdot)\|_{L^2(\TT^2)}^2  |$ is bounded from below by a positive quantity, uniformly in $m$ (for $m$ sufficiently large). In order to achieve this uniform in $m$ lower bound, as in the proof of Theorem~\ref{t.anomalous.diffusion} let us temporarily assume that $\theta_0$ is an analytic function, satisfying the quantitative estimate~\eqref{e.theta0.anal} for some $R_{\theta_0}>0$. (At the end of the argument, we approximate the given $H^2$ data with an analytic one, akin to~\eqref{e.analytic.approximation}.) For $\theta_0$ satisfying~\eqref{e.theta0.anal}, let $m_* = m_*(R_{\theta_0},\Lambda,\beta) \in \mathbb{N}$ be defined according to~\eqref{e.m*.def}; we also record the bound~\eqref{e.kappa*.def} which relates $\kappa_{m_*}$ and $\ep_{m_*}$ to $R_{\theta_0}$.

Applying~\eqref{e.different.L2s} to the problem~\eqref{e.theta.m.yet.again} with~$m= m_*\in\N$ (as defined in~\eqref{e.m*.def}), and appealing to the $\b_{m_*}$ estimate in~\eqref{e.bm.Dn}, to the bound for $\kappa_{m_*}^{(1)}$ implied by~\eqref{e.kappa.flip.flop.infinite.i} and~\eqref{e.kappas.gap}, we obtain for every $t\in (0,1]$ satisfying
\begin{equation}
\label{e.t.condition.selection}
0< t \leq \frac{1}{300} \min \left\{ 2^{-22} \ep_{m_*}^{2- \beta} \,,\,  \frac{\|\nabla\theta_0\|_{L^2(\TT^2)}^2}{\|\Delta \theta_0\|_{L^2(\TT^2)}^2}\ep_{m_*}^{-\frac{2\beta}{q+1}} \right \}\,,
\end{equation}
the estimate
\begin{equation*}
\Bigl| \|\theta_{m_*}^{(1)} (t,\cdot)\|_{L^2(\TT^2)}^2 
-
\|\theta_{m_*}^{(2)} (t,\cdot)\|_{L^2(\TT^2)}^2 \Bigr| 
\geq 
\frac{11}{80} \kappa_{m_*}^{(1)} t 
\|\nabla \theta_0\|_{L^2(\TT^2)}^2
\,.
\end{equation*}
In light of the energy identity for \eqref{e.theta.m.yet.again}, the above estimate yields a bound without squares
\begin{equation}
\label{e.different.L2s.no.squares}
\Bigl| 
\|\theta_{m_*}^{(1)} (t,\cdot)\|_{L^2(\TT^2)}
-
\|\theta_{m_*}^{(2)} (t,\cdot)\|_{L^2(\TT^2)}
\Bigr| 
\geq 
\frac{11}{160} \frac{ \kappa_{m_*}^{(1)} t 
\|\nabla \theta_0\|_{L^2(\TT^2)}^2}{\|\theta_0 \|_{L^2(\TT^2)}}
\,.
\end{equation}
In order to obtain an estimate which holds for all~$m \geq m_*$, we appeal to~\eqref{e.homogenization.m} (applicable since \eqref{e.m*.def} was designed to match \eqref{e.mtheta0.def}) and the triangle inequality, to arrive at:
\begin{align}
\label{e.different.L2s.no.squares.all.m}
\lefteqn{ 
\bigl| 
\|\theta_m^{(1)} (t,\cdot)\|_{L^2(\TT^2)}
-
\|\theta_m^{(2)} (t,\cdot)\|_{L^2(\TT^2)}
\bigr| 
} \qquad & 
\notag \\ & 
\geq 
\bigl| 
\|\theta_{m_*}^{(1)} (t,\cdot)\|_{L^2(\TT^2)}
-
\|\theta_{m_*}^{(2)} (t,\cdot)\|_{L^2(\TT^2)}
\bigr| 
-
2 \sup_{i\in\{1,2\}} \sup_{m\geq m_*} 
\|\theta_m^{(i)} - \theta_{m_*}^{(i)}\|_{L^\infty((0,1);L^2(\TT^2))}
\notag \\ & 
\geq 
\frac{11}{160} \frac{\kappa_{m_*}^{(1)} t 
\|\nabla \theta_0\|_{L^2(\TT^2)}^2}{\|\theta_0 \|_{L^2(\TT^2)}}
-
C  \ep_{m_*}^\delta 
\left\| \theta_{0} \right\|_{L^2( \TT^2)}
\,.
\end{align}
The constant $C_{\eqref{e.different.L2s.no.squares.all.m}}$ may be computed from $C_{\eqref{e.homogenization.m}}$, and the $C = C(\beta)$ which appears in the bound $\sum_{m\geq m_*} \ep_m^\delta \leq C \ep_{m_*}^{\delta}$. 
We emphasize that~\eqref{e.different.L2s.no.squares.all.m} is valid for every~$t$ satisfying~\eqref{e.t.condition.selection}. We need the right side of~\eqref{e.different.L2s.no.squares.all.m} to be positive, which together with the bound for $\kappa_{m_*}^{(1)}$ implied by~\eqref{e.kappa.flip.flop.infinite.i} and~\eqref{e.kappas.gap},  translates into the requirement that $t$ satisfies
\begin{equation}
\label{e.t.condition.selection.2}
t \geq  10^3 C_{\eqref{e.different.L2s.no.squares.all.m}} \ep_{m_*}^\delta  
\ep_{m_*}^{-\frac{2\beta}{q+1}}
\frac{\left\| \theta_{0} \right\|_{L^2( \TT^2)}^2}{ \|\nabla \theta_0\|_{L^2(\TT^2)}^2 } 
\,.
\end{equation}
Combining the above estimates, we arrive at the conclusion that if the initial datum $\theta_0$ and the parameter $\Lambda$ are chosen such that the inequality
\begin{equation}
\label{e.t.condition.selection.new}
10^3 C_{\eqref{e.different.L2s.no.squares.all.m}} 
\ep_{m_*}^\delta 
\ep_{m_*}^{-\frac{2\beta}{q+1}}
\frac{\left\| \theta_{0} \right\|_{L^2( \TT^2)}^2}{ \|\nabla \theta_0\|_{L^2(\TT^2)}^2 } 
\leq
\frac{1}{400} \min \left\{ 2^{-22} \ep_{m_*}^{2- \beta} \,,\,  \frac{\|\nabla\theta_0\|_{L^2(\TT^2)}^2}{\|\Delta \theta_0\|_{L^2(\TT^2)}^2}\ep_{m_*}^{-\frac{2\beta}{q+1}} \right \}
\,
\end{equation}
holds,
then there exists an interval of times $t\in[0,1]$ which satisfies both~\eqref{e.t.condition.selection} and~\eqref{e.t.condition.selection.2}, and moreover, for all such times $t$ by~\eqref{e.different.L2s.no.squares.all.m} we have the bound 
\begin{align}
\label{e.different.L2s.no.squares.all.m.new}
\bigl| 
\|\theta_m^{(1)} (t,\cdot)\|_{L^2(\TT^2)}
-
\|\theta_m^{(2)} (t,\cdot)\|_{L^2(\TT^2)}
\bigr| 
&\geq \frac{1}{16} \frac{\kappa_{m_*}^{(1)} t 
\|\nabla \theta_0\|_{L^2(\TT^2)}^2}{\|\theta_0 \|_{L^2(\TT^2)}} 
\notag\\
&\geq 
C_{\eqref{e.different.L2s.no.squares.all.m}}  \ep_{m_*}^\delta  
\|\theta_0 \|_{L^2(\TT^2)}
\notag\\
&\geq 
C_{\eqref{e.different.L2s.no.squares.all.m}} 
R_{\theta_0}^{\frac{2q \delta}{2+\gamma}}
\|\theta_0 \|_{L^2(\TT^2)}
\,,
\qquad \forall m\geq m_*.
\end{align}
In the second inequality we have appealed to~\eqref{e.t.condition.selection.2}, while in the third inequality we have used~\eqref{e.kappa*.def}, with $R_{\theta_0}$ being the analyticity radius  of the initial data $\theta_0$.

It thus remains to augment the condition~\eqref{e.t.condition.selection.new} and the inequality~\eqref{e.different.L2s.no.squares.all.m.new} with an approximation argument, which replaces an $H^2(\TT^2)$ function of zero mean with an analytic one. 
For a given $\theta_0 \in H^2(\TT^2)$, we proceed as in the proof of Theorem~\ref{t.anomalous.diffusion}. We define the length scale $L_{\theta_0}$ according to~\eqref{e.L.theta0}, and we mollify~$\theta_0$ with the standard heat kernel (with diffusion coefficient of unit size) at time~$\alpha^2 L_{\theta_0}^2$, with $\alpha>0$ to be determined later in the proof; the resulting analytic function $\tilde{\theta}_0$ defined in~\eqref{e.analytic.approximation} satisfies the analyticity condition~\eqref{e.theta0.anal} for~$R_{\tilde{\theta}_0} = c\alpha L_{\theta_0}$, for a universal constant $c  \in (0,1]$. For $m\in\mathbb{N}$ and $i\in\{1,2\}$, denote by $\theta_m^{(i)}$ the solution of \eqref{e.theta.m.yet.again} with initial $\theta_0$, and by $\tilde{\theta}_m^{(i)}$ the solution of \eqref{e.theta.m.yet.again} with initial $\tilde{\theta}_0$. Using linearity, the energy balance~\eqref{e.energy.balance.2}, the triangle inequality, and~\eqref{e.different.L2s.no.squares.all.m.new} we thus deduce that there exits $t \in [0,1]$ such that 
\begin{align}
\label{e.different.L2s.no.squares.all.m.mew}
\bigl| 
\|\theta_m^{(1)} (t,\cdot)\|_{L^2(\TT^2)}
-
\|\theta_m^{(2)} (t,\cdot)\|_{L^2(\TT^2)}
\bigr| 
&\geq 
\bigl| 
\|\tilde{\theta}_m^{(1)} (t,\cdot)\|_{L^2(\TT^2)}
-
\|\tilde{\theta}_m^{(2)} (t,\cdot)\|_{L^2(\TT^2)}
\bigr| 
-
2 \|\theta_0-\tilde{\theta}_0\|_{L^2(\TT^2)}
\notag\\
&\geq 
C_{\eqref{e.different.L2s.no.squares.all.m}} 
R_{\tilde{\theta}_0}^{\frac{2q \delta}{2+\gamma}}
\|\tilde{\theta}_0 \|_{L^2(\TT^2)}
-
2 \|\theta_0-\tilde{\theta}_0\|_{L^2(\TT^2)}
\,,
\end{align}
for all $m\geq m_*$, where $m_*$ is defined according to~\eqref{e.m*.def}, with $R_{\theta_0}$ replaced by $R_{\tilde{\theta}_0}$. Naturally, the existence of the time $t$ in \eqref{e.different.L2s.no.squares.all.m.mew} requires that \eqref{e.t.condition.selection.new} holds, with $\theta_0$ replaced by $\tilde{\theta}_0$. Since $\theta_0$ and its mollified version $\tilde{\theta}_0$ are close in $L^2(\TT^2)$, see estimate~\eqref{e.them.mollifiers}, and recalling that $R_{\tilde{\theta}_0} =  \alpha L_{\theta_0}/C$, we deduce from \eqref{e.different.L2s.no.squares.all.m.mew} that  for some universal constant $C\geq 1$, we have
\begin{equation}
\label{e.different.L2s.no.squares.all.m.mew.mew}
\bigl| 
\|\theta_m^{(1)} (t,\cdot)\|_{L^2(\TT^2)}
-
\|\theta_m^{(2)} (t,\cdot)\|_{L^2(\TT^2)}
\bigr| 
\geq 
\frac 12 C_{\eqref{e.different.L2s.no.squares.all.m}} 
\bigl(C^{-1} \alpha L_{\theta_0} \bigr)^{\frac{2q \delta}{2+\gamma}}
\| \theta_0 \|_{L^2(\TT^2)}
-
C \alpha \|\theta_0\|_{L^2(\TT^2)}
\,,
\end{equation}
for all $m\geq m_*$. Optimizing the above estimate with respect to $\alpha$, we arrive at 
\begin{equation}
\label{e.new.f-ing.alpha}
\alpha = C^{-1} L_{\theta_0}^{\frac{2q\delta}{2+\gamma - 2 q \delta}}
\end{equation}
where $C = C(\beta)\ge 1$ is a computable constant. Here we used that the parameters $q,\gamma,\delta$ are merely functions of $\beta$ (see~\eqref{e.q.def.0}--\eqref{e.gamma}), and they satisfy  $2+\gamma - 2 q \delta > 0 $ for all $\beta \in (1,4/3)$. Having chosen the parameter $\alpha$, we finally deduce from \eqref{e.different.L2s.no.squares.all.m.mew.mew} that for some $C=C(\beta)>0$, there exists $t \in [0,1]$ such that 
\begin{equation}
\label{e.different.L2s.no.squares.all.m.mew.mew.mew}
\bigl| 
\|\theta_m^{(1)} (t,\cdot)\|_{L^2(\TT^2)}
-
\|\theta_m^{(2)} (t,\cdot)\|_{L^2(\TT^2)}
\bigr| 
\geq 
C^{-1} L_{\theta_0}^{\frac{2q\delta}{2+\gamma - 2 q \delta}} \|\theta_0\|_{L^2(\TT^2)}
\,,
\end{equation}
for all $m\geq m_*$. Importantly, the above estimate is valid assuming that \eqref{e.t.condition.selection.new} holds with $\theta_0$ replaced by $\tilde{\theta}_0$.

In order to conclude, we identify for which $\theta_0 \in H^2(\TT^2)$  the condition 
\eqref{e.t.condition.selection.new} --- with $\theta_0$ replaced by $\tilde{\theta}_0$ --- 
is non-vacuous. Recall, by definition, that $\tilde{\theta}_0$ is obtained from $\theta_0$ by running the heat equation (at unit diffusivity) forward in time, up to the small time $\alpha^2 L_{\theta_0}^2$. Since $L_0 \leq (2\pi)^{-1}$ (recall \eqref{e.L.theta0}) and since the constant $C$ appearing in~\eqref{e.new.f-ing.alpha} may be increased by a factor of $10$ if needed, we thus have that $2^{-1} \|\theta_0\|_{\dot{H}^k(\TT^2)} \leq \|\tilde{\theta}_0\|_{\dot{H}^k(\TT^2)} \leq  \|\theta_0\|_{\dot{H}^k(\TT^2)}$, for $k\in \{0,1,2\}$. Therefore, if we are willing to give up a factor of $16$, the norms $\|\tilde{\theta}_0\|_{\dot{H}^k(\TT^2)}$ appearing in \eqref{e.t.condition.selection.new} for $k\in\{0,1,2\}$ may be replaced by the corresponding $\|{\theta}_0\|_{\dot{H}^k(\TT^2)}$ norms. With this in mind, we now summarize  condition \eqref{e.t.condition.selection.new} as requiring two bounds:  
\begin{align}
\label{e.crapola.1}
\bigl( C^\prime 
\ep_{m_*}^{\delta} \bigr)
\ep_{m_*}^{\beta-2  -\frac{2\beta}{q+1}}
&\leq
 \frac{ \|\nabla \theta_0\|_{L^2(\TT^2)}^2 } { \|\theta_{0}\|_{L^2( \TT^2)}^2}
\,,
\\
\label{e.crapola.2}
C^\prime
\ep_{m_*}^\delta 
&\leq
\frac{\|\nabla\theta_0\|_{L^2(\TT^2)}^4}{\|\theta_{0}  \|_{L^2( \TT^2)}^2 \|\Delta \theta_0\|_{L^2(\TT^2)}^2} 
\,,
\end{align}
for some sufficiently large $C^\prime = C^\prime(\beta)\geq 1$. 
In the above two inequalities, we recall from~\eqref{e.m*.def}, from the definition~$R_{\tilde{\theta}_0} =  \alpha L_{\theta_0}/C$, and from the choice of $\alpha$ in~\eqref{e.new.f-ing.alpha}, that $m_*\geq 1$ is defined as the integer which satisfies 
\begin{equation}
\ep_{m_*}  \leq R_{\tilde{\theta}_0}^{\frac{2}{2+\gamma}} < \ep_{m_* -1 }  \,,
\qquad 
\mbox{where}
\qquad
R_{\tilde{\theta}_0}  = (C^{\prime\prime})^{-1} L_{\theta_0}^{ \frac{2 + \gamma}{2+\gamma - 2 q \delta}}
\label{e.what.the.f.is.m*}
\end{equation}
and $C^{\prime\prime} = C^{\prime\prime}(\beta)\geq 1$ is sufficiently large.

At this stage, as in the statement of the Proposition, fix two parameters $A \in (0,1]$ and  $B>1$. Then, choose $\Lambda$ large enough, solely in terms of $A$, $B$, and $\beta$, to ensure that 
\begin{equation}
\label{e.take.Lambda.large.A.B}
C^\prime 
\Lambda^{-\delta} 
\leq A
\,,
\qquad
\bigl( C^\prime \Lambda^{-\delta} \bigr)
(B C^{\prime\prime})^{\frac{2}{2+\gamma}(2 -\beta +\frac{2\beta}{q+1})}
\leq 1
\,,
\end{equation}
where $C^\prime = C^\prime_{\eqref{e.crapola.2}}(\beta)\geq $ and $C^{\prime\prime} = C^{\prime\prime}_{\eqref{e.what.the.f.is.m*}}(\beta) \geq 1$ were already fixed. We recognize \eqref{e.take.Lambda.large.A.B} as the condition on $\Lambda$ present in Proposition~\ref{p.no.selection.principle}.

Then, since $m_*\geq 1$, we deduce from~\eqref{e.take.Lambda.large.A.B} that $C^\prime 
\ep_{m_*}^\delta  
\leq 
C^\prime 
\ep_{1}^\delta
\leq
C^\prime 
\Lambda^{-\delta} 
\leq A$. Therefore, if the initial datum $\theta_0$ satisfies 
\begin{equation}
\label{e.theta.0.condition.1}
\frac{\|\nabla\theta_0\|_{L^2(\TT^2)}^4}{\|\theta_{0}  \|_{L^2( \TT^2)}^2 \|\Delta \theta_0\|_{L^2(\TT^2)}^2} 
 \geq A
 \,, 
\end{equation}
then~\eqref{e.crapola.2} is satisfied. In order to obey condition~\eqref{e.crapola.1}, we additionally assume that $\theta_0$ satisfies 
\begin{equation}
\label{e.theta.0.condition.2}
\bigr(C^{\prime\prime}\bigl)^{\frac{2+\gamma - 2 q \delta}{2+\gamma}} \ep_{m_*}^{\frac{2+\gamma - 2 q \delta}{2}} 
\leq  \frac{\|\theta_0\|_{L^2(\TT^2)}}{\|\nabla \theta_0\|_{L^2(\TT^2)}} 
\leq \bigl(B C^{\prime\prime} \bigr)^{\frac{2+\gamma - 2 q \delta}{2+\gamma}} \ep_{m_*}^{\frac{2+\gamma - 2 q \delta}{2}} \,,
\end{equation}
where $C^{\prime\prime} = C^{\prime\prime}_{\eqref{e.what.the.f.is.m*}}(\beta) \geq 1$, and we recall that $L_{\theta_0} = \|\theta_0\|_{L^2(\TT^2)} \|\nabla \theta_0\|_{L^2(\TT^2)}^{-1}$. We remark that the lower bound in \eqref{e.theta.0.condition.2} is exactly a re-statement of the lower bound in~\eqref{e.what.the.f.is.m*}, so this lower bound does not constitute an assumption (this is just part of the definition of $m_*$); instead, it is the upper bound in \eqref{e.theta.0.condition.2} which is the assumption, and it states that this upper bound matches the lower bound up to some a-priori fixed constant (the upper bound in~\eqref{e.what.the.f.is.m*} does not guarantee this fact). With this extra assumption on $\theta_0$ imposed, we note that \eqref{e.take.Lambda.large.A.B} and \eqref{e.theta.0.condition.2} imply
\begin{align*}
\bigl( C^\prime \ep_{m_*}^{\delta} \bigr)
\ep_{m_*}^{\beta-2  -\frac{2\beta}{q+1}}
&\leq
\bigl( C^\prime \Lambda^{-\delta} \bigr)
(B C^{\prime\prime})^{\frac{2}{2+\gamma}(2 -\beta +\frac{2\beta}{q+1})} 
L_{\theta_0}^{-(1 + \frac{2q\delta}{2+\gamma - 2 q \delta})\frac{2}{2+\gamma}(2 -\beta +\frac{2\beta}{q+1})}
\notag\\
&\leq
L_{\theta_0}^{-(1 + \frac{2q\delta}{2+\gamma - 2 q \delta})\frac{2}{2+\gamma}(2 -\beta +\frac{2\beta}{q+1})}
\leq
L_{\theta_0}^{-2}
\,.
\end{align*}
The last inequality holds since $L_{\theta_0} \leq 1/(2\pi) \leq 1$, and upon unpacking the definitions of $q,\gamma,\delta$, as functions of $\beta$ (see~\eqref{e.q.def.0}--\eqref{e.gamma}), we observe that $(1 + \frac{2q\delta}{2+\gamma - 2 q \delta})\frac{2}{2+\gamma}(2 -\beta +\frac{2\beta}{q+1}) \leq 2$, whenever $\nicefrac{68}{67} \leq \beta< \nicefrac{4}{3}$. Since the right side of \eqref{e.crapola.2} is by definition equal to $L_{\theta_0}^{-2}$, the above estimate shows that condition~\eqref{e.crapola.2} is indeed satisfied. We recognize \eqref{e.theta.0.condition.1}--\eqref{e.theta.0.condition.2} as being precisely the conditions on $\theta_0$ present in the statement of Proposition~\ref{p.no.selection.principle}.

To summarize, we have proven that for $\Lambda$ sufficiently large to obey~\eqref{e.take.Lambda.large.A.B}, and all initial data $\theta_0\in \dot{H}^2(\TT)$ which satisfies \eqref{e.theta.0.condition.1}--\eqref{e.theta.0.condition.2}, there exists an integer $m_*\geq 1$ and a nonempty interval of times $t\in [0,1]$, such that the solutions $\{\theta_m^{(1)}\}_{m\geq m_*}$ and $\{\theta^{(2)}\}_{m\geq m_*}$ constructed in \eqref{e.theta.m.yet.again} satisfy \eqref{e.different.L2s.no.squares.all.m.mew.mew.mew}, i.e. 
\begin{equation*}
\bigl | \|\theta_m^{(1)} (t,\cdot)\|_{L^2(\TT^2)} - \|\theta_m^{(2)} (t,\cdot)\|_{L^2(\TT^2)} \bigr| 
\geq 
C(\beta,L_{\theta_0}) \|\theta_0\|_{L^2(\TT^2)} > 0 
\, .
\end{equation*}
Recall however that the functions $\theta_m^{(i)}$ defined in~\eqref{e.theta.m.yet.again}  are not the same as the functions $\theta^{\kappa_m^{(i)}}$, which are the solutions of the drift-diffusion equation with diffusivity parameter $\kappa_m^{(i)}$ and full vector field  $\b$ (instead of the truncated one $\b_m$). To estimate this difference we apply~\eqref{e.theta.thetaM.string}. In the context of the present argument, this estimate is to be applied with $\theta \mapsto \theta^{\kappa_m^{(i)}}$ and $\theta_M \mapsto \theta_m^{(i)}$, and $M \mapsto m$. We deduce that for all $m\geq m_*$, 
\begin{equation}
\|  \theta^{\kappa_m^{(i)}} - \theta_m^{(i)} \|_{L^\infty((0,1); L^2(\TT^2))}
\leq 
C \ep_m^{2\beta(q-\frac{2}{q+1})} \|\theta_0\|_{L^2(\TT^2)}
\,,
\label{e.closeness.estimate.1}
\end{equation}
for some $C = C (\beta)>0$. 
Since the exponent $2\beta(q-\frac{2}{q+1})$ is positive, for $m$ sufficiently large depending on $\beta$ and $\theta_0$, we can ensure that $C(\beta) \ep_m^{2\beta(q-\frac{2}{q+1})} \leq \frac 12 C(\beta,L_{\theta_0})$, and thus for some $t\in [0,1]$ and all $m$ sufficiently large,
\begin{equation}
\Bigl| \|\theta^{\kappa_m^{(1)}} (t,\cdot)\|_{L^2(\TT^2)} - \|\theta^{\kappa_m^{(2)}}  (t,\cdot)\|_{L^2(\TT^2)} \Bigr| 
\geq 
\frac 12 C(\beta,L_{\theta_0}) \|\theta_0\|_{L^2(\TT^2)} > 0 
\label{e.closeness.estimate.2}
\, .
\end{equation}
Moreover, estimate~\eqref{e.continuous.indeed} implies that each of the two sequences  $\{\theta_m^{(1)}\}_{m\geq m_*}$ and $\{\theta^{(2)}\}_{m\geq m_*}$ are Cauchy in $C^{0,\mu}([0,1];L^2(\TT^2))$ for some $\mu>0$, since the $\ep_m$'s decay super-geometrically as $m\to \infty$. By the triangle inequality and \eqref{e.closeness.estimate.1}, the sequences
$\{\theta^{\kappa_m^{(1)}}\}_{m\geq m_*}$ and $\{\theta^{\kappa_m^{(2)}}\}_{m\geq m_*}$ are Cauchy in $L^\infty([0,1];L^2(\TT^2))$; in fact, in light of \eqref{e.continuous.indeed.indeed} they are Cauchy sequences in $C^{0,\nicefrac{\mu}{2}}([0,1];L^2(\TT^2))$.
Clearly, these two sequences have different limits due to~\eqref{e.closeness.estimate.2}. Moreover, these limit points are weak solutions of the transport equation with velocity field $\b$, and belong to $C^{0,\nicefrac{\mu}{2}}([0,1];L^2(\TT^2))$.
This concludes the proof of Proposition~\ref{p.no.selection.principle}.
\end{proof}

\appendix

\section{Macroscopic mean drift destroys enhancement}
\label{a.nosweep}

The goal is to formalize the idea, mentioned in the introduction, that a slowly-varying background flow with large amplitude will destroy the enhancement generated by a mean-zero, time-independent microscopic flow with smaller amplitude. The arguments in this appendix are not used anywhere in the paper; we include these for informational purposes only.

\smallskip

To simplify the discussion, we consider only two scales and assume that the macroscopic background flow is constant. We therefore assume that~$\b:\TT^d \to \R^d$ is a periodic, incompressible, mean-zero, time-independent vector field and $v \in\R^d$ is a vector representing the constant background flow. 
We let~$\mathbf{m}$ be the stream matrix for~$\b$ and define for $0 < \eps \ll 1$
\begin{equation*}
\a^\eps:= I_d + \mathbf{m}(\tfrac{\cdot}{\eps}). 
\end{equation*}
We may then write 
\begin{equation*}
-\Delta + (\tfrac{1}{\eps}\b(\tfrac{x}{\eps})+\tfrac{1}{\eps}v) \cdot \nabla
=
-\nabla \cdot \a^\eps(x) \nabla + \tfrac{1}{\eps} v \cdot \nabla 
\,.
\end{equation*}
We wish to estimate the effective diffusivity matrix, denoted by~$\ahom_v$, which we obtain by homogenizing the operator
\begin{equation}
\label{e.operatorweareinterestedin}
\partial_t - \nabla \cdot \a^\eps \nabla + \tfrac{1}{\eps}  v \cdot \nabla \,.
\end{equation}
We can rewrite this operator, absorbing the constant vector~$v$ into the diffusion matrix, by changing variables in space-time: if~$\theta_\eps(t,x)$ is a solution of 
\begin{equation*}
\partial_t \theta_\eps
-\nabla \cdot \a^\eps(x) \nabla \theta_\eps + \tfrac{1}{\eps}  v \cdot \nabla \theta_\eps
= 0,
\end{equation*}
then defining $T_\eps(t,x):= \theta_\eps(t,x-\frac{1}{\eps}  tv)$, we find that 
\begin{equation*}
\partial_tT_\eps
-\nabla \cdot \a^\eps_v \nabla T_\eps
=0,
\qquad 
\mbox{where}
\qquad
\a^\eps_v(t,x):=\a^\eps(x-\tfrac{1}{\eps} tv) 
= I_d + \mathbf{m}(\tfrac{x}{\eps} - v \tfrac{t}{\eps^2})
=: \a_v(\tfrac{t}{\eps^2},\tfrac{x}{\eps}) 
\,.
\end{equation*}
We can therefore pose our problem as follows: we are interested in computing the effective diffusivity matrix for the parabolic operator~
$\partial_t - \nabla \cdot \a^\eps_v(t,x)\nabla$, where $\a^\eps_v(t,x)=\a_v(\frac{t}{\eps^2},\frac{x}{\eps})$ is periodic in~$x$ and \emph{quasiperiodic} in the time variable~$t$. By classical homogenization theory, this problem homogenizes to $\partial_t - \nabla\cdot \ahom_v \nabla$ for an effective matrix~$\ahom_v$ which depends on~$v$; we can invert the change of variables and see that~$\ahom_v$ is also the effective diffusivity of the original operator we were interested in~\eqref{e.operatorweareinterestedin}.

Let us assume that $v\in \R^d$ is a good Diophantine direction,\footnote{For example, let $v = (1,\sqrt{2})$ for $d=2$. Then, for any $k \in \ZZ^2_*$, $\left| \frac{k}{|k|} \cdot \frac{v}{|v|} \right| = \frac{|k_1 + k_2 \sqrt{2}|}{\sqrt{3} |k|} \geq \frac{1}{3\sqrt{3} |k|^2}$.} i.e., that there exist $A \in (0,1)$ and $\kappa>0$ with
\begin{equation*}
\left| \tfrac{k}{|k|} \cdot \tfrac{v}{|v|} \right|
\geq 
A |k|^{-\kappa}
\,,
\qquad 
\forall k \in \ZZ^d_* = \ZZ^d \setminus\{0\}
\,.
\end{equation*}
Note that the set of directions satisfying this condition has full Lebesgue measure. 
We can write the equation for the correctors~$\chi_e$ associated to the homogenization problem as 
\begin{equation*}
\partial_{t} \chi_e - \nabla \cdot \a_v \nabla \chi_e 
=
\b_v \cdot e 
\quad \mbox{in} \ \R \times \Rd\,,
\qquad 
\langle \chi_e\rangle = 0\,,
\end{equation*}
where $\b_v (t,x) = \nabla \cdot \a_v (t,x)= \b(x -t v)$. Here, and in what follows in this appendix, the brackets~$\langle\cdot\rangle$ denote the mean of a quasiperiodic function of time, which is $\TT^d$-periodic in space, i.e., for such $f$,
\begin{equation*}
\bigl\langle f \bigr\rangle = \lim_{T\to \infty} \fint_{-T}^T \fint_{\TT^d} f(t,x) dx dt
\,. 
\end{equation*}
The enhancement of diffusivity is related to the correctors' gradient field~$\nabla\chi_e$ by the formula
\begin{equation}
\label{e.enhance.para}
e\cdot \ahom_v e  - |e|^2
=
\bigl\langle 
\bigl| \nabla \chi_e \bigr|^2 
\bigr\rangle
\,.
\end{equation}
To obtain~\eqref{e.enhance.para}, we use the equation for $\chi_e=\chi_e(t,x)$ to get 
\begin{equation*}
0 =
-
\left\langle 
\tfrac12  \partial_t \chi_e^2
\right\rangle
=
\left\langle 
\nabla \chi_e
\cdot \a_v \big( e +  \nabla \chi_e\big) 
\right\rangle
\,.
\end{equation*}
Therefore, since~$\langle \nabla \chi_e \rangle =0$, and the symmetric part of $\a_v$ is $I_d$, we obtain from the above identity that 
\begin{equation*}
e\cdot \ahom_v e
=
e\cdot \left\langle 
\a_v \left( e + \nabla \chi_e \right)
\right\rangle
=
\left\langle 
\left( e + \nabla \chi_e \right)
\cdot \a_v \left( e + \nabla \chi_e \right)
\right\rangle
=
\bigl\langle 
\left| e + \nabla \chi_e \right|^2
\bigr\rangle
=
|e|^2 +
\bigl\langle 
\left| \nabla \chi_e \right|^2
\bigr\rangle
\,.
\end{equation*}
The identity~\eqref{e.enhance.para} says that the difference between the effective diffusivity~$\ahom_v$ and the molecular diffusivity matrix~$I_d$ is proportional to the size of the correctors. Next, we show that $\bigl\langle 
\left| \nabla \chi_e  \right|^2\bigr\rangle$ is small when $|v|\gg 1$. For this purpose, we use the identity 
\begin{equation*}
\bigl\langle 
\bigl| \nabla \chi_e \bigr|^2 
\bigr\rangle
=  
\bigl\langle (\b_v \cdot e) \chi_e \bigr\rangle
\end{equation*}
and estimate the right side.

\smallskip
Assume $f$ and $g$ are quasiperiodic in $t$ and $\TT^d$-periodic in $x$. Then,  denoting $\hat{f}_k(t)$ and $\hat{g}_k(t)$ the Fourier-series coefficients of $f$ and $g$ with respect to $x$ only, we  have by Plancherel that
\begin{equation*}
\langle f \, g \rangle = \lim_{T\to\infty}\fint_{-T}^T \sum_{k \in \ZZ^d_*}  \hat{f}_k(t) \overline{\hat{g}_k}(t) dt
\,.
\end{equation*}
We wish to apply this identity to $f = \chi_e$ and $g = e\cdot \b_v$, i.e., we wish to compute $\bigl\langle (\b_v \cdot e) \chi_e \bigr\rangle$. Since translation in physical space is modulation in Fourier space, we know that 
\begin{equation*}
\overline{\hat{g}_k}(t) 
= \widehat{(e\cdot \b)}_k e^{2\pi i k\cdot v t}
= \frac{d}{dt} \left(\frac{-i}{2\pi k \cdot v} \widehat{(e\cdot \b)}_k e^{2\pi i k\cdot v t}\right)
= \frac{-i}{2\pi k \cdot v}\frac{d}{dt}\overline{\hat{g}_k}(t) 
\,.
\end{equation*}
Now, quasiperiodic functions with mean zero gradients are sublinear at infinity, and so we have via integration by parts in time, 
\begin{equation}
\label{e.temp.temp.1}
\langle f \, g \rangle
=
\lim_{T\to\infty}\fint_{-T}^T \sum_{k \in \ZZ^d_*}  \hat{f}_k(t) \overline{\hat{g}_k}(t) dt
=
\lim_{T\to\infty}\fint_{-T}^T \sum_{k \in \ZZ^d_*}\frac{i}{2\pi k \cdot v} \overline{\hat{g}_k}(t) \frac{d}{dt} \hat{f}_k(t)  dt
\,.
\end{equation}
Using the equation satisfied by $f = \chi_e$, we have that 
\begin{equation*}
\frac{d}{dt} \hat{f}_k(t) = 2\pi i k \cdot \widehat{(\a_v \cdot \nabla \chi_e)}_k(t) + \hat{g}_k(t) 
\,.
\end{equation*}
Since $k\cdot v$ is odd in $k$, the second term in the above display does not contribute to the expression in \eqref{e.temp.temp.1}. We deduce finally that 
\begin{equation}
\label{e.temp.temp.2}
\bigl\langle (\b_v \cdot e) \chi_e \bigr\rangle
=
\frac{-1}{ |v|}
\lim_{T\to\infty}\fint_{-T}^T \sum_{k \in \ZZ^d_*}\frac{1}{\frac{k}{|k|} \cdot \frac{v}{|v|}}\widehat{(e\cdot \b)}_k e^{2\pi i k\cdot v t}   \frac{k}{|k|} \cdot \widehat{(\a_v \cdot \nabla \chi_e)}_k(t)  dt
\,.
\end{equation}
Now, since $v$ is a good Diophantine direction, we have the bound
\begin{equation*}
\left|\frac{1}{\frac{k}{|k|} \cdot \frac{v}{|v|}} \right| 
\leq 
\frac{|k|^{\kappa}}{A}
\,,
\qquad \mbox{for all} \qquad
k \in \ZZ^d_*
\,.
\end{equation*}
With this information, we return to~\eqref{e.temp.temp.2}, use Plancherel and Cauchy-Schwartz to deduce that 
\begin{align*}
\bigl\langle 
\bigl| \nabla \chi_e \bigr|^2 
\bigr\rangle
=
\left| \bigl\langle (\b_v \cdot e) \chi_e \bigr\rangle \right|
&\leq
\frac{1 + \|\mathbf{m}\|_{L^\infty(\TT^d)}}{A |v|}
\bigl\langle |\nabla \chi_e|^2 \bigr\rangle^{\nicefrac 12}
\left(\sum_{k \in \ZZ^d_*}|k|^{2\kappa} \bigl| \widehat{(e\cdot \b)}_k\bigr|^2 \right)^{\nicefrac 12}
\notag\\
&\leq
\frac{1 + \|\mathbf{m}\|_{L^\infty(\TT^d)}}{A |v|}
\bigl\langle |\nabla \chi_e|^2 \bigr\rangle^{\nicefrac 12}
\|e\cdot \b\|_{\dot{H}^\kappa(\TT^d)}
\,.
\end{align*}
As a consequence,
\begin{align*}
\bigl\langle 
\bigl| \nabla \chi_e \bigr|^2 
\bigr\rangle
\leq 
\frac{(1 + \|\mathbf{m}\|_{L^\infty(\TT^d)})^2 \|e\cdot \b\|_{\dot{H}^\kappa(\TT^d)}^2}{ A^2 |v|^2 }
\end{align*}
which becomes arbitrarily small for $|v|\gg 1$.  
Returning to \eqref{e.enhance.para}, we have thus shown that a mean drift of ``generic'' direction destroys the enhancement of diffusion when $|v|\gg 1$.

\section{Fa\'a di Bruno formula and its consequences}
\label{sec:chain:rule}

In order to show that the stream function~$\phi(t,x)$ constructed in Section~\ref{s.construction} has $C^{1,\beta}$ regularity as stated in Proposition~\ref{p.SAMS.regularity}, we require explicit estimates on the derivatives of the solutions~$X$ of the flow $\partial_t X = \f(t,X)$ in terms of those of~$\f$. A qualitative version of such an estimate (for instance, the statement that $\f\in C^k \implies X\in C^k$) is of course quite classical and can be found in most introductory textbooks on ODE theory. The difference here is that we need an explicit estimate which, while it must be known, is of a form we could not find written in the literature. Like the qualitative arguments, the proof boils down to differentiating the equation many times. The only difficulty is a bookkeeping one: we must keep track of all the terms arising out of repeatedly applying the chain rule; that is, we need to use the Fa\'a di Bruno formula. 

\subsection{Useful lemmas}

Recall from~\eqref{e.barf} that for any integer $n\geq 0$,  $C^n_x$ smooth function $f$, and $R>0$, we denote
\begin{align}
\snorm{f}_{n,R} =  \frac{(n+1)^2}{ n!R^{n}} \sup_{|\aa| = n} \left\| \partial^\aa f\right\|_{L^\infty_x} 
\,.
\label{eq.barf}
\end{align}
In the above definition, the  shift factor $(n+1)^2$ could be replaced by $(n+1)^r$ for any $r>1$, at the cost of introducing $r$-dependence on our constants. We shall frequently use the identity 
\begin{align}
\sum_{\abs{\bb} = k, \bb \leq \aa} \binom{\aa}{\bb} = \binom{n}{k}  
\label{eq.barf.1}
\end{align}
for $|\aa| =n$ and $0 \leq k \leq n$. By combining this identity with the Leibniz rule, we obtain:
\begin{lemma}[\bf Product estimate]
\label{l.product}
For $f , g \in C^n_x$ we have the bound
\begin{align}
\snorm{f g}_{n,R} 
\leq 
4 \Bigl(\max_{0\leq j \leq n} \snorm{f}_{j,R}\Bigr) \Bigl(\max_{0\leq j \leq n} \snorm{g}_{j,R}\Bigr)   
\,.
 \label{e.barf.2}
\end{align}
\end{lemma}
\begin{proof}[Proof of Lemma~\ref{l.product}]
The Leibniz rule and \eqref{eq.barf.1} give
\begin{align*}
 \snorm{f \, g}_{n,R} 
 &\leq 
 \frac{(n+1)^2}{n!R^n} \sum_{k=0}^n \sup_{|\aa|=n} \sum_{|\bb| =k,\bb \leq \aa} \binom{\aa}{\bb} \norm{\partial^\bb f}_{L^\infty_x} \norm{\partial^{\aa-\bb} g}_{L^\infty_x}
 \notag\\
 &\leq 
 \frac{(n+1)^2}{n!R^n} \sum_{k=0}^n  \binom{n}{k} \snorm{f}_{k,R} \frac{k! R^k}{(k+1)^2} \snorm{g}_{n-k,R} \frac{(n-k)! R^{n-k}}{(n-k+1)^2} 
  \notag\\
 &\leq \left(\max_{0\leq j \leq n} \snorm{f}_{j,R}\right) \left(\max_{0\leq j \leq n} \snorm{g}_{j,R}\right)    \sum_{k=0}^n   \frac{(n+1)^2}{(k+1)^2(n-k+1)^2}\,.
\end{align*}
The proof now follows since $\sum_{k=0}^n   (n+1)^2 (k+1)^{-2}(n-k+1)^{-2} \leq 4$ for all $n\geq 0$.
\end{proof}

The workhorse of this Appendix is a multivariable version of the Fa\'a di Bruno formula, with requires some additional notation. We denote by $\NN_0$ the set of all integers strictly larger than $-1$, and by $\NN_0^d$ the set of all multi-indices $\aa = (\alpha_1,\cdots,\alpha_d)$ with $\alpha_j \in \NN_0$. For a multi-index $\aa$, we write
$|\aa| = \alpha_1 + \ldots + \alpha_d$,  $\aa! =  (\alpha_1!) \cdot \ldots \cdot (\alpha_d !)$, $\partial^{\aa} = \partial_{x_1}^{\alpha_1} \ldots \partial_{x_d}^{\alpha_d}$, and $\yy^{\aa}= (y_1^{\alpha_1}) \cdot \ldots \cdot (y_d^{\alpha_d})$, where $\yy \in \RR^d$ is a point. The following notation shall be needed below.
Let $n\geq 1$, $\aa,\bb \in \NN_0^d$ be such that $\abs{\bb} = n$, and   $1 \leq |\aa| \leq n$.
For $1\leq s \leq n$ define the set\footnote{Here $\lll_{1} < \lll_{2}$ if either $\abs{\lll_{1}} < \abs{\lll_{2}}$, or $\abs{\lll_{1}}=\abs{\lll_{2}}$ and there exists $k\in \{1,\ldots,d\}$ such that $\lll_{1k'}=\lll_{2k'}$ for $k'< k$, and $\lll_{1k}<\lll_{2k}$.}
\begin{align}
p_s(\bb,\aa)
&= \Big\{ (\kk_1,\ldots,\kk_s; \lll_1,\ldots,\lll_s) \in (\NN_0^d \times\ldots\NN_0^d ;\NN_0^d\times \ldots \NN_0^d)\colon \notag \\
&\qquad 0< |\kk_j| , 0 < \lll_1 < \ldots < \lll_s, \sum_{j=1}^s \kk_j = \aa, \sum_{i=1}^s |\kk_j| \lll_j = \bb \Big \}.
\label{eq:partition:set}
\end{align}
With this notation in hand, we recall \cite[Theorem 2.1]{CS}.

\begin{proposition}[\bf Multivariate Fa\`a di Bruno Formula]
\label{prop:multi:Faa}
Let $h\colon \RR^d \to \RR$ be~$C^\infty$ in a neighborhood of ${\yy}_{0} := \gb(\xx_0)$ and $\gb\colon \RR^d \to \RR^d$ be~$C^\infty$ in a neighborhood of $\xx_0$. Denote their composition by~$f:= h\circ \gb$. 
Then, for every multiindex~$\bb$ with $n:=\abs{\bb}\geq 1$, 
\begin{equation*}
(\partial^{\bb} f)(\xx_0) = \bb! \sum_{1\leq |\aa|\leq n} (\partial^{\aa} h)(\gb(\xx_0)) \sum_{s=1}^n \sum_{p_s(\bb,\aa)} \prod_{j=1}^s \frac{\left( (\partial^{\lll_j} \gb)(\xx_0) \right)^{\kk_j}}{(\kk_j !) (\lll_j !)^{|\kk_j|}}.
\end{equation*}
Here we adopt the convention that $0^0 :=1$.
\end{proposition}

Proposition~\ref{prop:multi:Faa} expands the higher order chain rule into a complicated expression. In applications, we need to have tools which allow us to contract such complicated expressions into simple ones. The first re-summation lemma that we use in the paper is similar to \cite[Lemma 3.2]{CVW}:
\begin{lemma}
\label{lem:contract:Faa:1}
Fix the dimension $d\geq 1$. With the notation of Proposition~\ref{prop:multi:Faa}, we have 
\begin{equation*}
\bb! \sum_{1\leq |\aa|\leq n} (-d)^{-\abs{\aa}} \abs{\aa}! \sum_{s=1}^n \sum_{p_s(\bb,\aa)} \prod_{j=1}^s \frac{\left( \abs{\lll_j}! \binom{\nicefrac 12}{\abs{\lll_j}} \right)^{\abs{\kk_j}}}{(\kk_j !) (\lll_j !)^{|\kk_j|}} =  2 (n+1)! \binom{\nicefrac 12}{n+1}
\,.
\end{equation*} 
\end{lemma}
\begin{proof}[Proof of Lemma~\ref{lem:contract:Faa:1}]
Define the following functions
\begin{align*}
h(\yy) &= \bar h(y_1 + \ldots+ y_d), \qquad \bar h(z) =  \frac{1}{1 - \frac{z}{d}},
\\
g_1(\xx) = \ldots = g_d(\xx) &= \bar g(x_1 +\ldots+ x_d), \qquad \bar g(z) =  1 -  \sqrt{1-z}  \,,
\\
f(\xx) = h (g_1(\xx),\ldots,g_d(\xx)) &=  \bar f(x_1+\ldots+x_d), \qquad \bar f(z) = \bar h(d \bar g(z)) = \frac{1}{\sqrt{1-z}} \,.
\end{align*}
which are real-analytic functions in the neighborhood of $\xx = 0$. For any multi-index $\aa \in \NN_0^d$ we have that $(\partial^{\aa} f)(0) = (\partial^{ \abs{\aa}}\bar f)(0)$ and similarly for the functions $h, g_1, \ldots, g_d$. Moreover, we note the following identities\footnote{Recall that $\binom{\nicefrac 12}{n} n! = (\nicefrac 12) (\nicefrac 12 - 1) \ldots (\nicefrac 12 -n+1)$ for $n\geq 1$, and $\binom{\nicefrac 12}{0} := -1$.}
\begin{subequations}
\begin{align}
(\partial^n \bar h)(0) &= d^{-n} n! 
\label{eq:bar:h:derivatives}
\\
(\partial^n \bar g)(0) &= (-1)^{n-1} \binom{\nicefrac 12}{n} n! \geq 0 
\label{eq:bar:g:derivatives}
\\
(\partial^n \bar f)(0) &= 2 (-1)^{n} \binom{\nicefrac 12}{n+1} (n+1)! \geq 0
\label{eq:bar:f:derivatives}
\end{align}
\end{subequations}
which are valid for $n \geq 1$. Let $\bb$ be any multi-index of length $n$. We apply Proposition~\ref{prop:multi:Faa} to the function $f(x)$ defined above, and using \eqref{eq:bar:h:derivatives}--\eqref{eq:bar:g:derivatives} deduce that 
\begin{align}
(-1)^{-n} (\partial^{\bb} f)(0) 
&=(-1)^{-n}   \bb! \sum_{1\leq |\aa|\leq n} d^{-\abs{\aa}} \abs{\aa}! \sum_{s=1}^n \sum_{p_s(\bb,\aa)} \prod_{j=1}^s \frac{\left( (-1)^{\abs{\lll_j}-1} \binom{\nicefrac 12}{\abs{\lll_j}} \abs{\lll_j}! \right)^{\abs{\kk_j}}}{(\kk_j !) (\lll_j !)^{|\kk_j|}}
\notag\\
&=  \bb! \sum_{1\leq |\aa|\leq n} (-d)^{-\abs{\aa}} \abs{\aa}! \sum_{s=1}^n \sum_{p_s(\bb,\aa)} \prod_{j=1}^s \frac{\left( \binom{\nicefrac 12}{\abs{\lll_j}} \abs{\lll_j}! \right)^{\abs{\kk_j}}}{(\kk_j !) (\lll_j !)^{|\kk_j|}}
\label{eq:Faa:temp:1}
\end{align}
which is the expression that we wish to estimate. Here we have used that $\sum_j \abs{\kk_j} = \abs{\aa}$ and $\sum_j \abs{\kk_j} \abs{\lll_j} = \abs{\bb}=n$. On the other hand, we know that from \eqref{eq:bar:f:derivatives} that 
\begin{align}
(-1)^{-n} (\partial^{\bb} f)(0) = (-1)^{-n} (\partial^n \bar f)(0) = 2 \binom{\nicefrac 12}{n+1} (n+1)!\,.
\label{eq:Faa:temp:2}
\end{align}
Equating \eqref{eq:Faa:temp:1} and \eqref{eq:Faa:temp:2} concludes the proof of the lemma.
\end{proof}

Another re-summation lemma that we use is a multi-D version of~\cite[Lemma 1.4.1]{KP}:
\begin{lemma}
\label{lem:contract:Faa:2}
Let $R>0$ and $d\geq 1$. With the notation of Proposition~\ref{prop:multi:Faa}, we have that
\begin{equation*}
\bb! \sum_{1\leq |\aa|\leq n} R^{\abs{\aa}} \abs{\aa}! \sum_{s=1}^n \sum_{p_s(\bb,\aa)} \prod_{j=1}^s \frac{\left( \abs{\lll_j}!   \right)^{\abs{\kk_j}}}{(\kk_j !) (\lll_j !)^{|\kk_j|}} =  d R (1 +d R)^{n-1} n!
\end{equation*} 
\end{lemma}
\begin{proof}[Proof of Lemma~\ref{lem:contract:Faa:2}]
Similarly to the proof of Lemma~\ref{lem:contract:Faa:1}, we define the functions 
\begin{align*}
h(\yy) &= \bar h(y_1 + \ldots + y_d), \qquad   \bar h(z) =  \frac{1}{1 - d R (z-1) }\, ,   
\\
g_i(\xx)  &= \bar g(x_1 +\ldots+ x_d), \qquad  \bar g(z) = \frac{1}{d(1-z)}   \,,
\\
f(\xx) = h (g_1(\xx),\ldots,g_d(\xx)) &=  \bar f(x_1+ \ldots +x_d), \qquad   \bar f(z) = \bar h(d \bar g(z)) = \frac{1-z}{ 1-(d R+1)z}  \,,   
\end{align*}
so that for any $\bb \in \NN_0^2$ with $\abs{\bb} =n$, we have
\begin{align*}
(\partial^{\bb} h)(\nicefrac 1d, \ldots, \nicefrac 1d) &=   (\partial^n \bar h)(1) = (d R)^n n! \, ,\\
(\partial^{\bb} g_i)(0,\ldots,0) &= (\partial^n \bar g_i)(0) = \frac{n!}{d} \,, \\
(\partial^{\bb} f)(0,\ldots,0) &= (\partial^n \bar f)(0) = d R (1 +d R)^{n-1} n!\,.
\end{align*}
Using the above identities and Proposition~\ref{prop:multi:Faa}, we obtain that 
\begin{align*}
d R (1 +d R)^{n-1} n! &= (\partial^{\bb} f)(0,\ldots,0)  \notag\\
&=   \bb! \sum_{1\leq |\aa|\leq n} (dR)^{\abs{\aa}} \abs{\aa}! \sum_{s=1}^n \sum_{p_s(\bb,\aa)} \prod_{j=1}^s \frac{\left(  \frac 1d \abs{\lll_j}! \right)^{\abs{\kk_j}}}{(\kk_j !) (\lll_j !)^{|\kk_j|}}
\notag\\
&=  \bb! \sum_{1\leq |\aa|\leq n} R^{\abs{\aa}} \abs{\aa}! \sum_{s=1}^n \sum_{p_s(\bb,\aa)} \prod_{j=1}^s \frac{\left(   \abs{\lll_j}! \right)^{\abs{\kk_j}}}{(\kk_j !) (\lll_j !)^{|\kk_j|}}
\end{align*}
which concludes the proof.
\end{proof}

\begin{remark}
Due to the presence of the shift factor $(n+1)^2$ in \eqref{eq.barf}, in order for Lemmas~\ref{lem:contract:Faa:1} and~\ref{lem:contract:Faa:2} to be useful, we shall need the following inequality. 
Assume that $|\bb|=n$, $1\leq |\aa|\leq n$, $1\leq s \leq n$, and $(\kk_1,\ldots,\kk_s;\lll_1,\ldots,\lll_s) \in p_s(\bb,\aa)$, as defined in \eqref{eq:partition:set}. Then, we have that 
\begin{align}
\frac{(n+1)^2}{(|\aa|+1)^2}  \prod_{j=1}^s \frac{1}{(|\lll_j|+1)^{2 |\kk_j|}} \leq 1\,.
\label{e.vomit.again.3}
\end{align}
In order to prove \eqref{e.vomit.again.3}, we note that by the definition of the partition set $p_s(\bb,\aa)$ in \eqref{eq:partition:set}, we have $|\lll_s| \geq |\lll_j| \geq 1$ for all $1\leq j \leq s$. As such, we obtain that $n = |\bb| = \sum_{j=1}^s |\kk_j| |\lll_j| \leq |\lll_s| \sum_{j=1}^s |\kk_j| = |\lll_s| |\aa|$. Thus, the lower bound $|\lll_s| \geq n/|\aa|$ emerges, and since $|\kk_s|\geq 1$, we have
\begin{align*}
\frac{(n+1)^2}{(|\aa|+1)^2}  \prod_{j=1}^s \frac{1}{(|\lll_j|+1)^{2 |\kk_j|}} 
&\leq 
\frac{(n+1)^2}{(|\aa|+1)^2}  \frac{1}{(|\lll_s|+1)^{2 |\kk_s|}} \notag\\
&\leq 
\frac{(n+1)^2}{(|\aa|+1)^2}  \frac{1}{(\frac{n}{|\aa|}+1)^{2}} 
= 
\frac{|\aa|^2}{(|\aa|+1)^2}  \frac{(n+1)^2}{(n + |\aa|)^{2}} 
\leq 
1\,.
\end{align*}
In the last inequality we have used that $|\aa| \geq 1$.
This proves \eqref{e.vomit.again.3}.
\end{remark}

The following estimate is standard for $m=\infty$, as it give the radius of analyticity of the composition of two real-analytic functions. We recall it here for the sake of completeness, and note trivially that it holds for finite $m$.

\begin{proposition}[\bf Composition estimate]
\label{prop:compose:analytic}
Let $h \in L^\infty( C^m(\RR^d))$, $\gb \in L^\infty( C^m(\RR^d))^d$, and assume that there exist positive constants $C_h, C_g, R_h, R_g \in (0,\infty)$ such that 
\begin{equation*}
\snorm{h}_{n,R_h} \leq C_h
\qquad \mbox{and} \qquad 
\snorm{\gb}_{n,R_g} \leq C_g
\end{equation*}
for all $0 \leq n\leq m$ and respectively $1\leq n \leq m$. Then,   for every  $0 \leq n\leq m$, we have 
\begin{align}
\snorm{h \circ \gb }_{n,R} \leq C_h \, 
\qquad \mbox{where} \qquad 
R
= R_g (1 + d C_g R_h)\,.
\label{e.barf.3}
\end{align}
\end{proposition}
\begin{proof}[Proof of Proposition~\ref{prop:compose:analytic}]
Let $f( \xx) = h (\gb(\xx)) \in L^\infty(C^m(\RR^d))$. 
For $n=0$ the statement \eqref{e.barf.3} holds trivially. Let $n\geq 1$ and fix $\bb \in \NN_0^d$ with $\abs{\bb} = n$. From  Proposition~\ref{prop:multi:Faa}, the assumed bounds on $h$ and $\gb$,   the re-summation formula in Lemma~\ref{lem:contract:Faa:2}, and the bound \eqref{e.vomit.again.3}, we deduce that 
\begin{align}
\norm{\partial^{\bb} f}_{L^\infty} 
&\leq \bb! \sum_{1\leq |\aa|\leq n}  C_h R_h^{\abs{\aa}} \frac{\abs{\aa}!}{(\abs{\aa}+1)^2} \sum_{s=1}^n \sum_{p_s(\bb,\aa)} \prod_{j=1}^s \frac{\left( C_g R_g^{\abs{\lll_j}} \frac{\abs{\lll_j}!}{(\abs{\lll_j}+1)^2}  \right)^{\abs{\kk_j}}}{(\kk_j !) (\lll_j !)^{|\kk_j|}}
\notag\\
&\leq C_h  R_g^n \frac{\bb!}{(\abs{\bb}+1)^2} \sum_{1\leq |\aa|\leq n}   (C_g R_h)^{\abs{\aa}} \abs{\aa}! \sum_{s=1}^n \sum_{p_s(\bb,\aa)} \prod_{j=1}^s \frac{\left(  \abs{\lll_j}!  \right)^{\abs{\kk_j}}}{(\kk_j !) (\lll_j !)^{|\kk_j|}}
\notag\\
&= C_h  R_g^n (d C_g R_h) (1+ d C_g R_h)^{n-1} \frac{n!}{(n+1)^2}
\,.
\label{e.vomit.again}
\end{align}
Re-arranging the right side of the above, and recalling \eqref{eq.barf}, finishes the proof. 
\end{proof}

\subsection{Application to the transport equation}

For given smooth functions $\f,\g\colon \R \times \Rd \to \Rd$ we first consider the solution $Y$ 
of
\begin{equation}
\label{e.transport}
\left\{
\begin{aligned}
& \bigl( \partial_t + \f \cdot \nabla \bigr) Y = \g & \mbox{in} & \ \R \times \Rd, \\
& Y(0,\cdot) = 0  & \mbox{on} & \ \Rd\,.
\end{aligned}
\right.
\end{equation}

\begin{lemma} 
\label{l.transport.reg:shift}
Assume there exist~$C_{\f}, R_{\f},C_{\g},R_{\g}> 0$ with $R_{\g} \geq R_{\f}$ and $N\in\N$ such that 
\begin{align}
\label{e.bear.salmon.1}
\max_{1\leq n\leq N} \
\sup_{t\in \R} \
\snorm{\f(t,\cdot)}_{n,R_{\f}} \leq C_{\f}\,,
\qquad \mbox{and} \qquad  
\max_{1\leq n\leq N} \
\sup_{t\in \R} \ 
\snorm{\g(t,\cdot)}_{n,R_{\g}} \leq C_{\g}\,,
\end{align}
Then, the solution $Y$ of the transport equation \eqref{e.transport} satisfies
\begin{align}
\label{e.vomit.1}
\max_{1\leq n\leq N} \ 
\sup_{t  \in [-T,T]} \ 
\frac{1}{|t|} \snorm{Y(t,\cdot)}_{n,R_{Y}(t)}    \leq 8 d C_{\g} \,,
\end{align}
where 
\begin{equation}
R_Y(t) := R_{\g} + 4  |t| d C_{\f} R_{\f}^2 \leq R_{\g} (1+ 4  |t|   d C_{\f} R_{\f}) 
\qquad \mbox{and} \qquad 
T := \frac{1}{4 d C_{\f} R_{\f} } \,.
\label{e.bear.salmon.2}
\end{equation}
\end{lemma}
\begin{proof}[{Proof of Lemma~\ref{l.transport.reg:shift}}]
For $n=1$, upon differentiating \eqref{e.transport} we obtain
\begin{equation*}
(\partial_t + \f \cdot\nabla) \nabla Y = \nabla \g + \nabla \f \cdot \nabla Y
\,, \qquad 
\nabla Y(0,\cdot) = 0\,. 
\end{equation*}
Integrating this expression and appealing to \eqref{e.bear.salmon.1}, we obtain
\begin{equation*}
\snorm{Y(t,\cdot)}_{1,R_Y} \leq 4 R_Y^{-1} \norm{\nabla Y(t,\cdot)}_{L^\infty_x} \leq \frac{d C_{\g} R_\g  |t| e^{\frac 14 d T R_\f C_\f }}{R_Y} \leq 2 d C_\g
\end{equation*}
so that \eqref{e.vomit.1} holds when $n=1$.

\smallskip
We prove \eqref{e.vomit.1} inductively on $n$, and without loss of generality, we only prove it for $t\in(0,T]$. Let $|\aa|=n$.  Applying $\partial^\aa$ to both sides of \eqref{e.transport} yields:
\begin{align}
\label{e.transport.aa}
\left\{
\begin{aligned}
& \bigl( \partial_t  + \f \cdot \nabla \bigr)\partial^\aa Y 
= \partial^\aa\g 
+\sum_{\bb<\aa}
\binom{\aa}{\bb}\partial^{\aa-\bb}\f \cdot \nabla \partial^\bb Y
& \mbox{in} & \ \R \times \Rd, \\
& \partial^\aa Y(0,\cdot) = 0 & \mbox{on} & \ \Rd\,.
\end{aligned}
\right. 
\end{align}
Integrating \eqref{e.transport.aa},   using   assumption \eqref{e.bear.salmon.1}  and the inductive assumption \eqref{e.vomit.1}, we get
\begin{align*}
\sup_{|\aa| = n }\norm{\partial^\aa Y(t,\cdot)}_{L^\infty_x} 
&\leq t \sup_{|\aa| = n} \norm{\partial^\aa \g}_{L^\infty_{t,x}} 
  +  C_{\f}    \sup_{|\aa| = n } \sum_{|\bb| = n-1, \bb <\aa}  \sum_{\ell = 1}^d  \binom{\aa}{\bb}  \frac{R_{\f}}{2^2}  \int_0^t  \norm{\partial^{\bb + e_\ell} Y(s,\cdot)}_{L^\infty_x}  ds
\notag\\
&\quad   +16 d^2  C_{\f} C_{\g}  \sum_{k = 0}^{n-2} \sup_{|\aa| = n } \sum_{|\bb| = k, \bb <\aa}  \binom{\aa}{\bb}  \frac{R_{\f}^{n-k} (n-k)!}{(n-k+1)^2} \frac{(k+1)!}{(k+2)^2} \int_0^t s R_Y(s)^{k+1} ds 
\notag\\
&\leq t C_{\g}  R_{\g}^n \frac{n!}{(n+1)^2} 
+ \frac{d C_{\f}  R_{\f} n }{4}  \int_0^t \sup_{|\aa| = n }\norm{\partial^\aa Y(s,\cdot)}_{L^\infty_x} ds \notag\\
&\quad   +  16 d^2  C_{\f} C_{\g}  \frac{n!}{(n+1)^2} \sum_{k = 0}^{n-2}  \binom{n}{k}  \frac{R_{\f}^{n-k} (n+1)^2 (k+1)  }{ (n-k+1)^2(k+2)^{2}}   \int_0^t s R_Y(s)^{k+1} ds
\,.
\end{align*}
In the second inequality above we have   appealed to identity \eqref{eq.barf.1}. 
At this point we note that due to the definition of $R_Y(s)$ in \eqref{e.bear.salmon.2}, we have
\begin{align}
\int_0^t s R_Y(s)^j ds \leq \frac{t R_Y(t)^{j+1}}{4 d  (j+1) C_{\f} R_{\f}^2 } 
\label{eq.vomit.2}
\,.
\end{align}
It follows from the above two estimates, the   definition \eqref{eq.barf}, and the fact that by \eqref{e.bear.salmon.2} we have $R_Y(t) \geq R_{\g} \geq R_{\f}$, we obtain
\begin{align*}
\frac{\snorm{Y(\cdot,t)}_{n,R_Y(t)}}{t} 
&\leq \frac{C_{\g} R_{\g}^n}{R_Y(t)^n} + \frac{d C_{\f} R_{\f} n}{4 t R_Y(t)^n}  \int_0^t s R_Y(s)^{n} ds \sup_{s\in[0,t]}\frac{\snorm{Y(\cdot,s)}_{n,R_Y(s)}}{s} 
\notag\\
&\qquad +  \frac{16 d^2  C_{\f} C_{\g} }{t R_Y(t)^n} \sum_{k = 0}^{n-2}    \frac{R_{\f}^{n-k}(n+1)^2 (k+1) }{(n-k+1)^2(k+2)^{2}} \int_0^t s R_{Y}(s)^{k+1} ds
\notag\\
&\leq  C_{\g}  + \frac{1}{16}  \sup_{s\in[0,t]}\frac{\snorm{Y(\cdot,s)}_{n,R_Y(s)}}{s}  + 4 C_{\g} d \sum_{k = 0}^{n-2}    \frac{ (n+1)^2 }{(n-k+1)^2(k+2)^{2}}   
\,.
\end{align*}
Using the estimate $\sum_{k=0}^{n-2}   (n+1)^2 (k+1)^{-2}(n-k+1)^{-2} \leq 3/2$, and taking the supremum over $t\in [0,T]$, we obtain that 
\begin{equation*}
\sup_{t\in (0,T]} t^{-1} \snorm{Y(\cdot,t)}_{n,R_Y(t)}   \leq \frac{16}{15} \left( C_{\g} + 6 d C_{\g} \right) \leq 8  dC_{\g}\,.
\end{equation*}
The above estimate shows that \eqref{e.vomit.1} holds at level $n$, closing the inductive step.
\end{proof}

\begin{corollary}
\label{cor.transport.reg:shift} 
Under the assumptions of Lemma~\ref{l.transport.reg:shift}, for all $|t|\leq T$ and $0\leq n \leq N-1$, we have that 
\begin{equation}
\label{e.vomit.2} 
\snorm{\nabla Y(\cdot,t)}_{n,R_{\nabla Y}(t)} \leq \frac{4 d C_{\g} R_{\g}}{C_{\f} R_{\f}}
\end{equation}
where
\begin{equation}
\label{e.vomit.3} 
R_{\nabla Y}(t) = R_{\g} (1 + 4 |t| d C_{\f} R_{\f})^2 = R_Y(t) (1 + 4 |t| d C_{\f} R_{\f}) \,. 
\end{equation}
\end{corollary}
\begin{proof}[Proof of Corollary~\ref{cor.transport.reg:shift}]
Using the definition~\eqref{eq.barf}, the bound \eqref{e.vomit.1}, and the definitions of $R_Y(t)$ in \eqref{e.bear.salmon.2}, and $R_{\nabla Y}(t)$ in \eqref{e.vomit.3}, we obtain
\begin{align*}
\snorm{\nabla Y(\cdot,t)}_{n,R_{\nabla Y}(t)}
&\leq \frac{d (n+1)^2}{n! R_{\nabla Y}(t)^n} \frac{(n+1)! R_Y(t)^{n+1}}{(n+2)^2} \snorm{Y(\cdot,t)}_{n+1,R_{Y}(t)}
\notag\\
&\leq  \frac{d n R_Y(t)^{n+1}}{R_{\nabla Y}(t)^n}   8 t d C_{\g}
\leq  \frac{16 d^2 t n C_{\g} R_{\g}}{(1+ 4 t d C_{\f} R_{\f})^n} 
\leq   \frac{4 d C_{\g} R_{\g}}{C_{\f} R_{\f}} \,.
\end{align*}
In the last inequality we have used that $n a \leq (1+a)^n$ for any $a\geq 0$. 
\end{proof}

\begin{lemma}
\label{l.transport.time.analyticity}   
If in addition to \eqref{e.bear.salmon.1} we assume that 
\begin{align}
\snorm{\partial_t^m \f}_{n,R_{\f}} \leq C_{\f}  Q_{\f}^m m! \,,
\qquad \mbox{and} \qquad  
\snorm{\partial_t^m \g}_{n,R_{\g}} \leq C_{\g}  Q_{\g}^m m! \,,
\label{eq:f:assumption:shift:dt}
\end{align}
for all $0 \leq n \leq N$ and $0\leq m \leq M$, for some constants  $Q_\g, Q_\f >0$, then the solution $Y$ of the transport equation \eqref{e.transport} satisfies the space-time derivative bounds
\begin{align}
 \sup_{t\in[-T,T]} \snorm{\partial_t^m Y(t,\cdot)}_{n,R_Y(t)} \leq \frac{2  C_\g}{C_\f R_\f} \frac{(n+m)!m! }{n!} Q^{m}\,,
 \label{e.mad.1}
 \end{align}
for all $m\leq M+1$ and $n+m \leq N$,  where $R_Y(t)$ and $T$ are as defined in \eqref{e.bear.salmon.2}, and  
\begin{align}
Q = \max\left\{ Q_\f, Q_\g,  C_\f R_\f, 16 d C_\f R_\g \right\}
\,.
 \label{e.mad.2}
\end{align}
\end{lemma}
\begin{proof}[Proof of Lemma~\ref{l.transport.time.analyticity}]
The bound \eqref{e.mad.1} is already known to hold when $m=0$ due to \eqref{e.vomit.1} and the definition of $T$ in \eqref{e.bear.salmon.2}. We next proceed inductively, with respect to $m$.

\smallskip

Applying $\partial_t^m \partial^{\aa}$ to \eqref{e.transport}, with $|\aa| = n$, we obtain that 
\begin{equation*}
\partial_t^{m+1} \partial^{\aa} Y 
= \partial_t^m \partial^{\aa} \g 
- \sum_{j=0}^m  \sum_{k=0}^n \sum_{|\bb| = k, \bb \leq \aa} \binom{m}{j} \binom{\aa}{\bb} \partial_t^{m-j} \partial^{\aa-\bb} \f \cdot \nabla \partial_t^j \partial^{\bb} Y
\,.
\end{equation*}
Taking the supremum over $|\aa| = n$, using the definition \eqref{eq.barf} and the identity \eqref{eq.barf.1}, for any $R>0$ we deduce
\begin{align*}
\snorm{\partial_t^{m+1} Y}_{n,R} 
&\leq \snorm{\partial_t^{m} \g}_{n,R} 
+ \frac{(n+1)^2}{n! R^n} 
\sum_{j=0}^m  \sum_{k=0}^n \binom{m}{j} \binom{n}{k} 
\frac{(n-k)! R^{n-k}}{(n-k+1)^2} \snorm{\partial_t^{m-j} \f}_{n-k,R} 
\notag\\
&\qquad \qquad \qquad \qquad \qquad \times
\frac{d (k+1)! R^{k+1}}{(k+2)^2} \snorm{\partial_t^{j} Y}_{k+1,R} 
\notag\\
&\leq \snorm{\partial_t^{m} \g}_{n,R} 
+  
\sum_{j=0}^m  \sum_{k=0}^n \binom{m}{j}  
\frac{(n+1)^2 (k+1)d  R }{(n-k+1)^2(k+2)^2} \snorm{\partial_t^{m-j} \f}_{n-k,R}  \snorm{\partial_t^{j} Y}_{k+1,R} \,.
\end{align*}
By further appealing to \eqref{eq:f:assumption:shift:dt},   choosing $R = R_Y(t) \geq  R_{\g} \geq R_{\f}$, and using that $Q \geq Q_\g,  Q_\f$, we obtain
\begin{align*}
\frac{\snorm{\partial_t^{m+1} Y}_{n,R_Y}}{(m+1)!   Q^{m+1}}
&\leq \frac{C_{\g}  }{(m+1)Q}
+  
\sum_{j=0}^m  \sum_{k=0}^n    
\frac{d (n+1)^2 (k+1)  }{(n-k+1)^2(k+2)^2 (m+1)} \frac{C_{\f}R_{Y}}{Q }
\frac{\snorm{\partial_t^{j} Y}_{k+1,R_Y}}{j! Q^{j}} \,.
\end{align*}
Since $0\leq j \leq m$, we may appeal to \eqref{e.mad.1} inductively, and obtain
\begin{align*}
\frac{\snorm{\partial_t^{j} Y}_{k+1,R_{Y}}}{j! Q^{j}}
\leq \frac{2   C_\g (k+1+j)!}{  C_\f R_\f (k+1)!}
\,.
\end{align*}
Therefore, from the above two inequalities and the bound $R_Y(t) \leq 2 R_\g$, it follows that \eqref{e.mad.1} holds at level $m+1$, if we are able to establish the bound
\begin{align*}
 \frac{C_{\g}}{Q}
+   
\sum_{j=0}^m  \sum_{k=0}^n    
\frac{(n+1)^2  }{(n-k+1)^2(k+2)^2 (m+1)} \frac{2  d C_{\f} R_\g }{Q }
 \frac{2   C_\g (k+j+1)!}{ C_\f R_\f k!}
 \leq \frac{2   C_\g (n+m+1)!}{ C_\f R_\f n!}
 \,.
 \end{align*}
For this purpose, we first note that for $0\leq k \leq n$ and $0\leq j \leq m$, we have that  $ \frac{ (k+j+1)!}{k!} \leq \frac{(n+m+1)!}{n!}$. Second, we observe that $\sum_{j=0}^m  \sum_{k=0}^n    
\frac{(n+1)^2  }{(n-k+1)^2(k+2)^2 (m+1)} \leq 4$. Thus, it is left to verify the condition
\begin{align*}
\frac{C_\f R_\f }{Q} +  \frac{16 d C_{\f} R_\g  }{Q } \leq 2\,.
\end{align*}
In turn, this condition holds due to~\eqref{e.mad.2}. This concludes the inductive proof of \eqref{e.mad.1}.
\end{proof}

\subsection{Application to ODEs}
Suppose that  $\f\colon \R \times \Rd \to \Rd$ be sufficiently smooth, divergence-free, and let $\flow$ be the corresponding flow starting from~$x$, that is, $\flow(\cdot,x)$ satisfies
\begin{equation}
\label{e.ODE.flow}
\left\{
\begin{aligned}
& \partial_t \flow(t,x) = \f(t,\flow(t,x)) & \mbox{in} & \ (-\infty,\infty) \times \Rd, \\
& \flow(0,x) = x \, ,
\end{aligned}
\right.
\end{equation}
and denote by $\backflow(t,\cdot)$ the inverse flow, which thus solves
\begin{equation}
\label{e.PDE.backwards}
\left\{
\begin{aligned}
& \partial_t  \backflow + \f \cdot \nabla \backflow = 0 & \mbox{in} & \ (-\infty,\infty) \times \Rd, \\
& \backflow(0,x) = x \,.
\end{aligned}
\right.
\end{equation}
In particular, 
\begin{equation}
Y(t,x) = X^{-1}(t,x) -x
\label{e.backwards.minus.identity}
\end{equation}
solves \eqref{e.transport} with $\g = - \f$.

\begin{proposition} 
\label{p.ODE.flow}
Suppose that   $\f$ is divergence-free, satisfies the bound  \eqref{e.bear.salmon.1},  and let $\flow$ be the  solution of \eqref{e.ODE.flow}.
Then, for every $t \in [-T,T]$, where $T$ is as defined in \eqref{e.bear.salmon.2}, we have that
\begin{align}
\norm{\nabla \flow(t,\cdot) - \Id}_{L^\infty_x} &\leq  \abs{t} d C_\f R_\f  \leq \tfrac{1}{4} 
\,.
\label{eq:grad:psi} 
\end{align}
Moreover, for every $0 \leq n \leq N-1$, we have that
\begin{equation}
\label{e.ODE.flow.estimate}
\sup_{t\in [-T,T]} \snorm{\nabla X(t,X^{-1}(t,\cdot))}_{n,R_\f (1+ 4  |t|  d C_\f R_\f)^2} 
\leq (d-1)! (20d)^{d-1}
\, .
\end{equation}
\end{proposition}
\begin{proof}[{Proof of Proposition~\ref{p.ODE.flow}}]
The bound \eqref{eq:grad:psi} follows by directly differentiating \eqref{e.ODE.flow}, which yields
\begin{align*}
\partial_t (\partial_i X_j - \delta_{ij}) = \partial_k \f^j \circ X  \left( \partial_i X^k  - \delta_{ik} \right) + \partial_i \f^j \circ X\,, 
\qquad (\partial_i X_j - \delta_{ij})(0,x) = 0\,,
\end{align*}
and applying Gr\"onwall's inequality.  

\smallskip

In order to prove \eqref{e.ODE.flow.estimate}, we recall that 
\begin{equation}
\nabla X \circ X^{-1} = (\nabla X^{-1})^{-1} 
\end{equation}
as $d\times d$ matrices. Since $\f$ is divergence free, ${\rm det}(\nabla X^{-1}) = 1$, and so $\nabla X \circ X^{-1}$ is nothing by the transpose of the cofactor matrix associated to $\nabla X^{-1} = \Id + \nabla Y$. In turn, every entry of this cofactor matrix is a sum of $(d-1)!$ many homogenous monomials of degree $d-1$ in the entries of the matrix $\Id + \nabla Y$. Since when $\g = -\f$ the Corollary~\ref{cor.transport.reg:shift}  yields $\snorm{\nabla Y(\cdot,t)}_{n,R_\f(1+4 t d C_\f R_\f)^2} \leq 4 d$, it follows that $\snorm{\nabla X^{-1}(\cdot,t)}_{n,R_\f(1+4 t d C_\f R_\f)^2} \leq 5d $ for all $0 \leq n\leq N-1$ and $d\geq 1$. Applying Lemma~\ref{l.product} to a  sum of $(d-1)!$ homogenous monomials of degree $d-1$ in terms of functions which obey this bound, it follows that 
\begin{equation*}
 \snorm{\nabla X (t ,X^{-1}(t, \cdot)) }_{n,R_\f(1+4 t d C_\f R_\f)^2} 
 \leq (d-1)! 4^{d-2} (5d)^{d-1} \leq (d-1)! (20 d)^{d-1}
\end{equation*}
for all $0 \leq n\leq N-1$.
\end{proof}

In the course of our proof, we shall also require the regularity of the flow $\nabla X$ itself, not just its behavior under differentiation when we compose with $X^{-1}$.  
  
\begin{proposition} 
\label{p.ODE.flow.2}
Suppose that   $\f$ satisfies the bound  \eqref{e.bear.salmon.1},  and let $\flow$ be the  solution of \eqref{e.ODE.flow}.
Then, for every $t \in [-T,T]$, where $T$ is as defined in \eqref{e.bear.salmon.2}, and every $1 \leq n \leq N-1$, we have  
\begin{equation}
\label{e.ODE.flow.estimate:new}
\sup_{t\in [-T,T]} \snorm{\nabla \flow(t,x)  }_{n,  8 d R_\f (1+ 8 d C_\f R_\f |t|)}
 \leq   6d
\,.
\end{equation}
\end{proposition}
\begin{proof}[{Proof of Proposition~\ref{p.ODE.flow}}]
It turns out that \eqref{e.ODE.flow.estimate} is not convenient for setting up an induction scheme. Instead, we will inductively propagate 
\begin{align}
\sup_{t\in [-T,T]} \snorm{ \flow(t,x)  }_{n,  8 d R_\f (1+ B |t|)}
\leq  (-1)^{n-1} \binom{\nicefrac 12}{n}   \frac{1}{d R_\f}  \,, 
\quad \mbox{with} \quad   B= 8 d C_\f R_\f \,,
\label{eq:Dn:Psi:induction} 
\end{align}
uniformly for $t\in [-T,T]$, for all $1\leq n \leq N$.  Without loss of generality we prove \eqref{eq:Dn:Psi:induction} only for $0 \leq t \leq T$. The  bound claimed in  \eqref{e.ODE.flow.estimate}  follows from \eqref{eq:Dn:Psi:induction} since $4 (-1)^{n-1} \binom{\nicefrac 12}{n}  \leq \nicefrac 1n$ holds for all $n\geq 1$, and by the definition of $T$
we have $1 + B T \leq 3$. Indeed, 
\begin{align*}
\snorm{ \flow(t,x)  }_{n,  8 d R_\f (1+ B |t|)}
\leq  \frac{d (n+1)^3}{(n+2)^2} \cdot 24 d R_\f \cdot \frac{1}{4 n d R_\f} 
\leq 6d 
\end{align*}

\smallskip

For $n=1$, we verify \eqref{eq:Dn:Psi:induction} by using \eqref{eq:grad:psi}. This amounts to checking 
\begin{align}
\frac{4(1 + |t| d C_\f R_\f) }{8 d R_\f (1+ B |t|)}  \leq \frac{1}{2d R_\f}  
\end{align}
which holds since $B\geq d C_\f R_\f$.

\smallskip

Next, we fix $n \geq 2$ and consider a multi-index $\bb$ with $\abs{\bb} = n$.  By Proposition~\ref{prop:multi:Faa}, we have 
\begin{align}
\partial_t \partial^{\bb} \flow_j 
&= \partial^{\bb} \flow_k (\partial_k \f_j) \circ \flow
+ \bb! \sum_{2\leq |\aa|\leq n} (\partial^{\aa} \f_j)\circ \flow  \sum_{s=1}^n \sum_{p_s(\bb,\aa)} \prod_{j=1}^s \frac{\left( \partial^{\lll_j} \flow \right)^{\kk_j}}{(\kk_j !) (\lll_j !)^{|\kk_j|}}
\notag\\
&=: \partial^{\bb} \flow_k (\partial_k \f_j) \circ \flow
+ E_{\rm old}
\,. 
\label{eq:chain:rule:steroids}
\end{align}
In \eqref{eq:chain:rule:steroids} we have singled out the term $\partial^{\aa} \f_j$ when $\abs{\aa} = 1$; this is because only this term comes paired with a derivative of $\flow$ which is exactly of order $n$; and indeed one directly verify that for each summand in $E_{\rm old}$, we have $1 \leq \abs{\lll_j} \leq n-1$. This expression is thus suitable for an inductive estimate. From \eqref{e.bear.salmon.1},  \eqref{eq:Dn:Psi:induction},   the definition of the partition set in \eqref{eq:partition:set}, the identity provided by Lemma~\ref{lem:contract:Faa:1}, and the bound given in \eqref{e.vomit.again.3},  we deduce that 
\begin{align}
\abs{E_{\rm old}} 
&\leq \bb! \sum_{2\leq |\aa|\leq n}   \frac{C_\f  R_\f^{\abs{\aa}} \abs{\aa}!}{(|\aa|+1)^2} \sum_{s=1}^n \sum_{p_s(\bb,\aa)} \prod_{j=1}^s \frac{\left( (-1)^{\abs{\lll_j}-1} \binom{\nicefrac 12}{\abs{\lll_j}} \frac{\abs{\lll_j}!}{(\abs{\lll_j}+1)^2} \frac{1}{d R_\f} ((8 d R_\f)(1  + B t))^{\abs{\lll_j}}\right)^{\abs{\kk_j}}}{(\kk_j !) (\lll_j !)^{|\kk_j|}} 
\notag\\
&\leq C_\f \frac{(-1)^n (8 d R_\f)^n \left(1 +    B t\right)^n}{(n+1)^2} \bb! \sum_{1\leq |\aa|\leq n}  (- d)^{-\abs{\aa}} \abs{\aa}! \sum_{s=1}^n \sum_{p_s(\bb,\aa)} \prod_{j=1}^s \frac{\left( \binom{\nicefrac 12}{\abs{\lll_j}} \abs{\lll_j}!    \right)^{\abs{\kk_j}}}{(\kk_j !) (\lll_j !)^{|\kk_j|}} 
\notag\\
&\leq 2 C_\f \frac{(-1)^n (8 d R_\f)^n \left(1 +    B t\right)^n}{(n+1)^2}  (n+1)! \binom{\nicefrac 12}{n+1}
\label{eq:E:old:bound}
\,.
\end{align}
With \eqref{eq:E:old:bound} in hand, we return to \eqref{eq:chain:rule:steroids}, to which we apply Gr\"onwall's inequality (recall that $\partial^{\bb} \flow|_{t=0} =0$ since $\abs{\bb}\geq 2$) and deduce that
\begin{align*}
\norm{\partial^{\bb}\flow (t,\cdot) }_{L^\infty_x}
&\leq \exp\left(T \norm{\nabla \f}_{L^\infty_{t,x}}\right)  \int_0^t \abs{E_{\rm old}(s)}   ds
\notag\\
&\leq e^{\frac{1}{16}} 2 C_\f \frac{(-1)^n (8 d R_\f)^n}{(n+1)^2}  (n+1)! \binom{\nicefrac 12}{n+1}  \int_0^t  \left( 1 +   B s\right)^n ds
\notag\\
&= e^{\frac{1}{16}} 2 C_\f \frac{(-1)^n (8 d R_\f)^n}{(n+1)^2}  (n+1)! \binom{\nicefrac 12}{n+1}  \frac{\left( 1 +    B t\right)^{n+1}}{  B (n+1)} 
\,.
\end{align*}
The   bound \eqref{eq:Dn:Psi:induction} at level $n$ now follows once we establish 
\begin{align*}
&\quad e^{\frac{1}{16}} 2 C_\f (-1)^n  \binom{\nicefrac 12}{n+1} \frac{1+B T}{B} \leq   (-1)^{n-1} \binom{\nicefrac 12}{n}  \frac{1}{d R_\f} 
\,.
\end{align*}
Observing that $ (-1) \binom{\nicefrac 12}{n+1} \binom{\nicefrac 12}{n}^{-1} = \frac{n-1/2}{n+1} \leq 1$, and appealing to the definition of $T$ in \eqref{e.bear.salmon.2}, which gives $1 + B T \leq 1 + B/(4d C_\f R_\f)$, it follows that the above inequality is implied by the bound
\begin{align*}
\frac{e^{\frac{1}{16}}}{2} \left( 1 + \frac{B}{4 d C_\f R_\f} \right) \leq \frac{B}{4 d C_\f R_\f} \,.
\end{align*}
The above estimate now is clearly true by the definition of $B$ in \eqref{eq:Dn:Psi:induction}. This concludes the proof of the inductive step for~\eqref{eq:Dn:Psi:induction}, and thus of the proposition. 
\end{proof}

\section{Ergodic lemmas for periodic functions}
\label{a.ergodic}

\subsection{Some basic ergodic lemmas in one dimension}

\begin{lemma}[Basic $L^1$ ergodic lemma]
\label{l.averages}
Assume the following:
\begin{itemize}

\item $f:\Rd \to \R$ is a $\Zd$--periodic function satisfying, for given constants~$C_f>0$ and~$r\in (0,1]$, the quantitative analyticity condition
\begin{align}
\label{e.ass.f.anal}
\big\langle \bigl| \nabla^nf \bigr| \big\rangle 
\leq 
\frac{C_f n!}{r^n}
\,, \qquad \forall n\in\N\,.
\end{align}

\item
$g\in L^1_{\mathrm{loc}}(\Rd)$ is a~$N^{-1} \Zd$--periodic function for some given integer~$N \in\N$ with~$N^{-1} \leq r$.
\end{itemize}
Then there exists a constant~$C(d)<\infty$ such that
\begin{align}
\label{e.expdecay}
\bigl| \bigl\langle fg \bigr\rangle - \bigl\langle f \bigr\rangle\, \bigl\langle g \bigr\rangle \bigr|
\leq
C C_f \bigl\langle |g|\bigr\rangle   \exp\biggl( {-} \frac{Nr}{C} \biggr). 
\end{align}
\end{lemma}
\begin{proof}[Proof of Lemma~\ref{l.averages}]
We have that 
\begin{align*}
\langle fg \rangle - \langle f \rangle\langle g \rangle = \big\langle f \bigl( g - \langle g \rangle \bigr)  \big\rangle = \big\langle \bigl( f - \langle f \rangle \bigr) \bigl( g - \langle g \rangle \bigr)  \big\rangle
\,,
\end{align*} 
where, recall that $\average{\cdot}$ represents the average on $[0,1]^d$. Using Parseval's identity, and denoting by $\widehat{f}_k$ and $\widehat{g}_k$ the $k^{th}$ Fourier series coefficients of~$f$ and~$g$, respectively, we have 
\begin{align*}
 \big\langle \bigl( f - \langle f \rangle \bigr) \bigl( g - \langle g \rangle \bigr)  \big\rangle
 = C \sum_{k\in \Zd_*} \widehat f_k \; \overline{\widehat g_k}
\end{align*}
where $C$ is a dimensional constant (taking into account factors of $(2\pi)^d$), and $\Zd_* = \Zd \setminus \{0\}$. Moreover, since $g - \langle g \rangle$ is a zero mean $N^{-1} \Zd$--periodic function, all its {\em nontrivial} Fourier series coefficients $\widehat{g}_k$ have the property that $k$ is a nonzero integer multiple of $N$. Using this information, and letting $\tau >0$, we thus obtain that 
\begin{equation}
|\langle fg \rangle - \langle f \rangle\langle g \rangle|
\leq C \sum_{k = N \ell, \ell \in \Zd_*} e^{\tau |k|} |\widehat f_k| e^{-\tau |k|} |\widehat g_k| \,.
\label{eq:averages:temp:1}
\end{equation}
A simple exercise  shows that there exists a constant $C>0$ (which only depends on the dimension $d$) such that for $\tau = \nicefrac{r}{C}$, the condition \eqref{e.ass.f.anal} implies
\begin{equation*}
 e^{(\nicefrac{r}{C}) |k|} |\widehat f_k| \leq C C_f 
\end{equation*}
for all $k \in \Zd$. Moreover, we trivially have $|\widehat{g}_k|\leq C \average{|g|}$ uniformly in $k$. Thus, by \eqref{eq:averages:temp:1}  
\begin{equation*}
|\langle fg \rangle - \langle f \rangle\langle g \rangle|
\leq C C_f \sum_{k = N \ell, \ell \in \Zd_*}   e^{- (\nicefrac{r}{C})|k|} 
\leq C C_f \sum_{|\ell| =1}^\infty |\ell|^{d-1}  e^{- (\nicefrac{r N}{C})|\ell|} 
\end{equation*}
from which \eqref{e.expdecay} follows since $Nr \geq 1$.
\end{proof}

\begin{remark}[$L^2$ ergodic estimate]
\label{r.averages.square}
We also use the following variant of Lemma~\ref{l.averages}:
there exists a constant~$C(d)<\infty$ such that if~$f:\Rd\to \R$ is $\Zd$--periodic and satisfies
\begin{align}
\label{e.ass.f.anal.2}
\big\langle \left| \nabla^nf \right|^2 \big\rangle^{\nicefrac 12}
\leq 
\frac{C_f n!}{r^n}
\,,
\end{align}
and $g:\Rd\to \R$ is $N^{-1}\Zd$ periodic with $\langle |g|^2 \rangle <\infty$ and $Nr \geq 1$, then we have the bound
\begin{align}
\label{e.averages.square}
\left| 
\langle |f|^2 |g|^2 \rangle  
- 
\langle |f|^2 \rangle \langle |g|^2\rangle \right| 
\leq 
C C_f^2  \langle |g|^2\rangle \exp\left( -\frac{Nr}{C} \right). 
\end{align}
Indeed, we may just apply Lemma~\ref{l.averages} with $f \mapsto f^2$, $g\mapsto g^2$ and $r\mapsto \nicefrac{r}{2}$, because the Leibniz rule and assumption \eqref{e.ass.f.anal} on $f$ give
\begin{align*}
\big\langle \left| \nabla^n (f^2) \right| \big\rangle 
\leq 
\sum_{j=0}^n \binom{n}{j} 
\big\langle \left|  \nabla^j f \right|^2 \big\rangle^{\nicefrac12} 
\big\langle \left|  \nabla^{n-j} f \right|^2 \big\rangle^{\nicefrac12}
\leq C_f^2 \frac{(n+1)!}{r^n}\leq C_f^2 \frac{n!}{(\nicefrac{r}{2})^n}
 \,.
\end{align*}
\end{remark}

\begin{lemma}[$L^1$ and $L^2$ ergodic lemma with flows]
\label{l.flow.averages}
Assume the following: 

\begin{itemize}

\item $X \colon \Rd\to\Rd$ is a~$\Zd$-periodic, volume-preserving diffeomorphism
satisfying, for given constants~$C_X>0$ and~$R > 0$, the quantitative analyticity condition 
\begin{equation}
\label{e.ass.X.anal}
\norm{\nabla^n X}_{L^\infty(\Rd)} 
\leq   
C_X n! R^n  \,,
\quad \forall n\in\N\,.
\end{equation}

\item $f:\Rd \to \R$ is a $\Zd$--periodic function satisfying, for given constants~$C_f>0$ and~$r\in (0,1]$, the quantitative analyticity condition
\begin{equation}
\label{e.ass.f.anal.flows}
\big\langle \bigl| \nabla^nf \bigr| \big\rangle 
\leq 
\frac{C_f n!}{r^n}
\,, \qquad \forall n\in\N\,.
\end{equation}

\item
$g\in L^1_{\mathrm{loc}}(\Rd)$ is a~$N^{-1} \Zd$--periodic function for some given integer~$N \in\N$ satisfying
\begin{equation}
\label{e.N.ergodic.constraint}
N r \geq R (r +  d C_X)\,.
\end{equation}
\end{itemize}
Then, there exists  $C(d)<\infty$ such that
\begin{equation}
\label{e.new.expdecay}
\bigl| \bigl\langle f \,\bigl( g\circ X^{-1} \bigr) \bigr\rangle - \bigl\langle f \bigr\rangle
\,\bigl\langle g \bigr \rangle \bigr|
\leq
C C_f \, \bigl\langle |g|\bigr\rangle   
\exp\biggl( - \frac{N r }{C R (r +  d C_X )} \biggr)
\,.
\end{equation}
Moreover, if we replace~\eqref{e.ass.f.anal.flows} by the stronger assumption
\begin{equation}
\label{e.ass.f.anal.2.again}
\big\langle \left| \nabla^nf \right|^2 \big\rangle^{\nicefrac 12}
\leq 
\frac{C_f n!}{r^n}
\,,
\end{equation}
and the assumption of~$g\in L^1_{\mathrm{loc}}(\Rd)$ by~$g\in L^2_{\mathrm{loc}}(\Rd)$, then we have the estimate
\begin{equation}
\label{e.useful.expdecay}
\bigl| 
\big\langle |f|^2 |g\circ X^{-1}|^2 \big\rangle  
- 
\bigl\langle |f|^2 \bigr\rangle\, \bigl\langle |g|^2\bigr\rangle \bigr| 
\leq 
C C_f^2 \, \bigl\langle |g|^2 \bigr\rangle 
\exp\left( - \frac{N r }{C R (r +  d C_X )} \right). 
\end{equation}
\end{lemma}
\begin{proof}[Proof of Lemma~\ref{l.flow.averages}]
Let $\tilde f = f \circ X$, which is thus also $\Zd$-periodic and real-analytic. By the assumption \eqref{e.ass.f.anal} for $f$, \eqref{e.ass.X.anal} for $X$, and the composition estimate in Proposition~\ref{prop:compose:analytic}, we have
\begin{align}
 \| \nabla^n \tilde f  \|_{L^\infty(\Rd)} 
\leq 
\frac{C_f n!}{\tilde r^n}, 
\qquad\mbox{where} \qquad
\tilde r =  \frac{r }{R(r +  d C_X)} \,.
\label{e.ass.tilde.f.anal} 
\end{align}
Now since $X$ and $X^{-1}$ are volume preserving, we have that 
\begin{align*}
\average{f g\circ X^{-1}} = \average{f\circ X \, g} = \langle \tilde f \, g \rangle
\,.
\end{align*}
Thus, by \eqref{e.ass.tilde.f.anal} we may apply Lemma~\ref{l.averages} to the functions $\tilde f$ and $g$, and deduce that 
\begin{align}
\label{e.new.expdecay.tilde}
\left| \langle \tilde f \,g  \rangle - \langle \tilde f \rangle\langle g \rangle \right|
\leq
C C_f \langle |g|\rangle   \exp\left( - \frac{N \tilde r}{C} \right)
\,,
\end{align}
for $C=C(d)$, as soon as $N \tilde r \geq 1$. The bound \eqref{e.new.expdecay} now follows by appealing again to the volume preserving nature of $X$, which gives $\langle \tilde f \rangle = \langle  f \circ X \rangle = \average{f}$, and spelling out the $\tilde r$ in \eqref{e.ass.tilde.f.anal}, and using that by assumption we have $N \tilde r \geq 1$. The proof of \eqref{e.useful.expdecay} is the same as the one outlined in Remark~\ref{r.averages.square}, applied to $g$ and $\tilde f$. 
\end{proof}

\begin{lemma}[$H^{-1}$ ergodic lemma]
\label{l.Hminusone.ergodic}
Assume that~$f$ and $g$ satisfying the hypotheses of Remark~\ref{r.averages.square}, and additionally that $\average{fg} = 0$ and $\average{g}= 0$. Then, exists a constant $C = C(d)<\infty$ such that
\begin{align}
\norm{ f g}_{\dot{H}^{-1}(\TT^d)} 
\leq \frac{C}{N}  \left\langle |f|^2 \right\rangle^{\nicefrac12}  \left\langle |g|^2 \right\rangle^{\nicefrac12}  +   C_f  \left\langle |g|^2 \right\rangle^{\nicefrac12} \exp\left( - \frac{Nr}{C} \right)
\,.
\label{e.Hminusone.ergodic}
\end{align}
\end{lemma}
\begin{proof}[Proof of Lemma~\ref{l.Hminusone.ergodic}]
Internally to this proof, for a $\TT^d$ periodic function $\varphi$ we write $\Proj_{\leq L} \varphi$ to denote the truncation of the Fourier series of $\varphi$ at frequencies $k \in \Zd$ such that $|k| \leq L$; accordingly, define $\Proj_{>L} \varphi = \varphi - \Proj_{\leq L} \varphi$. We use similar notations for $\Proj_{<L}$ and $\Proj_{\geq L}$.  We will use two useful identities. First, since $g$ has zero mean and is $N^{-1}\Zd$ periodic, then all its active frequencies are nonzero $\Zd$-multiples of $N$, and thus 
\begin{equation}
g = \Proj_{\geq N} g 
\,.
\label{eq:temp:use:1}
\end{equation}
Second, using the triangle inequality in the frequency domain combined with \eqref{eq:temp:use:1} we may write
$
\Proj_{\leq \nicefrac{N}{2}} f \; g 
= \Proj_{\leq \nicefrac{N}{2}} f \; \Proj_{\geq N} g
= \Proj_{\geq \nicefrac{N}{2}} \left( \Proj_{\leq \nicefrac{N}{2}} f \; \Proj_{\geq N} g \right)
$.
Thus, combining \eqref{eq:temp:use:1} and the previous identity, we may rewrite
\begin{equation}
f  g =  \Proj_{\geq \nicefrac{N}{2}} \left( \Proj_{\leq \nicefrac{N}{2}} f \; \Proj_{\geq N} g \right) + \Proj_{> \nicefrac{N}{2}} f \; g 
\,.
\label{eq:temp:use:3}
\end{equation}
In particular, $\average{fg}=0$ implies that $\average{ \Proj_{> \nicefrac{N}{2}} f \; g} = 0$.

For brevity of notation, let $\Lambda = (-\Delta)^{\nicefrac 12}$, so that by \eqref{eq:temp:use:3} we have
\begin{align}
 \norm{ f g}_{\dot{H}^{-1}(\TT^d)} 
 &= \norm{\Lambda^{-1}( f g)}_{L^2(\TT^d)}
  \notag\\
 &\leq \norm{\Lambda^{-1} \Proj_{\geq \nicefrac{N}{2}} \left( \Proj_{\leq \nicefrac{N}{2}} f \; \Proj_{\geq N} g \right)}_{L^2(\TT^d)} 
 + \norm{\Lambda^{-1}(\Proj_{> \nicefrac{N}{2}} f \; g )}_{L^2(\TT^d)} 
 \notag\\
 &\leq \frac{C}{N} \norm{ \Proj_{\leq \nicefrac{N}{2}} f \;  g  }_{L^2(\TT^d)} 
 + C \norm{ \Proj_{> \nicefrac{N}{2}} f \; g  }_{L^2(\TT^d)} \,,
 \label{eq:temp:use:4}
\end{align}
where $C$ is a universal constant (related to the $2\pi$ which we are not writing anywhere).
In the last inequality above we have used two bounds: $\norm{\Lambda^{-1} \Proj_{\geq L}}_{L^2 \to L^2} \leq C L^{-1}$, which is a consequence of Plancherel; and $\norm{\Lambda^{-1} \Proj_{>0}}_{L^2 \to L^2} \leq C $, which holds since for $\Zd$ periodic functions we have $\Proj_{>0} = \Proj_{\geq 1}$.  

\smallskip

Let us first inspect the first term on the right side of \eqref{eq:temp:use:4}. Since $\nabla^n$ commutes with $\Proj_{\leq L}$ and since $\norm{ \Proj_{\leq L}}_{L^2 \to L^2} \leq 1$, assumption \eqref{e.ass.f.anal.2} holds with $f$ replaced by $\Proj_{\leq \nicefrac{N}{2}} f$, with the same constants $C_f$ and $r$ (note, we can also replace $f$ by $\Proj_{> \nicefrac{N}{2}} f$, and this fact will be used later). Thus, we may apply the conclusion of Remark~\ref{r.averages.square} to the product of  $\Proj_{\leq \nicefrac{N}{2}} f$ and $g$, resulting in
\begin{equation*}
\norm{ \Proj_{\leq \nicefrac{N}{2}} f \;  g  }_{L^2(\TT^d)} 
\leq  \left(\norm{   f }_{L^2(\TT^d)} + C C_f \exp(-\nicefrac{N r}{C}) \right) \norm{ g}_{L^2(\TT^d)}
\end{equation*}
where $C = C(d)$. Similarly, we may apply the conclusion of Remark~\ref{r.averages.square} to the product of  $\Proj_{> \nicefrac{N}{2}} f$ and $g$ and deduce 
\begin{equation*}
\norm{ \Proj_{> \nicefrac{N}{2}} f \;  g  }_{L^2(\TT^d)} 
\leq  \left(\norm{\Proj_{> \nicefrac{N}{2}} f}_{L^2(\TT^d)} + C C_f \exp(-\nicefrac{N r}{C}) \right) \norm{ g}_{L^2(\TT^d)} 
\,.
\end{equation*}
Since $N\geq r^{-1} \geq 1$, combining the above two displays with \eqref{eq:temp:use:4} concludes the proof of \eqref{e.Hminusone.ergodic}, but only once we show that $\norm{\Proj_{> \nicefrac{N}{2}} f}_{L^2(\TT^d)}$  is exponential small in $- \nicefrac{Nr}{C}$.

\smallskip

For this purpose, we note that for any $\tau>0$, we have $\norm{\exp(- \tau \Lambda) \Proj_{> \nicefrac{N}{2}}}_{L^2 \to L^2} \leq \exp(- \nicefrac{N \tau}{C})$. On the other hand, by expanding the power series of~$\exp$ and appealing to \eqref{e.ass.f.anal.2}, we obtain, for every~$0< \tau < r$,
\begin{align*}
\norm{\exp(\tau \Lambda) f }_{L^2(\TT^d)}  
\leq \sum_{m\geq 0} \frac{\tau^m}{m!} \norm{\Lambda^m f}_{L^2(\TT^d)} 
\leq C_f \sum_{m\geq 0} \frac{\tau^m}{r^m}  = \frac{C_f}{1- \nicefrac{\tau}{r}}
\end{align*}
Thus, letting $\tau = \nicefrac{r}{C}$ for a suitable $C = C(d) \geq 2$, this paragraph concludes with
\begin{align*}
\norm{\Proj_{> \nicefrac{N}{2}} f}_{L^2(\TT^d)} \leq  C C_f \exp( - \nicefrac{N r}{C}) \,,
\end{align*}
thereby concluding the proof of the Lemma.
\end{proof}

\begin{remark}[$H^{-1}$ ergodic lemma with flows]
\label{r.Hminusone.flow}
Let~$f$ and $g$ be as in Lemma~\ref{l.Hminusone.ergodic}, and let $X$ be a periodic volume preserving analytic diffeomorphism as in Lemma~\ref{l.flow.averages}, with 
\begin{align*}
\norm{\nabla X - \Id}_{L^\infty(\TT^d)} + \norm{\nabla X^{-1} - \Id}_{L^\infty(\TT^d)} \leq \nicefrac 12
\,.
\end{align*}
If $\average{f \, g\circ X^{-1}} = \average{g} = 0$, 
then in analogy to how \eqref{e.averages.square} implies \eqref{e.useful.expdecay}, from \eqref{e.Hminusone.ergodic} we may deduce
\begin{align}
\norm{ f \, g\circ X^{-1}}_{\dot{H}^{-1}(\TT^d)} 
\leq \frac{C}{N}  \left\langle |f|^2 \right\rangle^{\nicefrac12}  \left\langle |g|^2 \right\rangle^{\nicefrac12}  +   C_f  \left\langle |g|^2 \right\rangle^{\nicefrac12} \exp\left( - \frac{Nr}{C R (r + d C_X)} \right)
\,.
\label{e.Hminusone.ergodic.flow}
\end{align} 
The argument goes as follows. As in the proof of Lemma~\ref{l.flow.averages}, define $\tilde f = f \circ X$, which is thus periodic and satisfies the quantitative analyticity estimates \eqref{e.ass.tilde.f.anal}. By duality, using that $X$ is volume preserving and $|\nabla X| \leq \nicefrac 32$, we have
\begin{align*}
\norm{ f \, g\circ X^{-1}}_{\dot{H}^{-1}(\TT^d)} 
&= \!\!\!  \sup_{\varphi \in C^\infty(\TT^d), \|\nabla \varphi\|_{L^2} \leq 1} \abs{\int f \, g \circ X^{-1} \, \varphi} 
\notag\\
&= \!\!\! \sup_{\varphi \in C^\infty(\TT^d), \|\nabla \varphi\|_{L^2} \leq 1} \abs{\int \tilde f  \, g   \, \varphi \circ X} 
\leq \sup_{\tilde \varphi \in C^\infty(\TT^d), \|\nabla \tilde \varphi\|_{L^2} \leq \nicefrac 32} \abs{\int \tilde f  \, g   \, \tilde \varphi } 
\leq C  \| \tilde f \, g\|_{\dot{H}^{-1}(\TT^d)} \,.
\end{align*}
Moreover, note that $\langle \tilde f \, g\rangle = \average{f \, g\circ X^{-1}} = 0$ by assumption, so we may directly apply Lemma~\ref{l.Hminusone.ergodic} to the pair $\tilde f$ and $g$, and deduce that \eqref{e.Hminusone.ergodic.flow} holds.
\end{remark}

\subsubsection*{\bf Acknowledgments}
S.A. was supported by NSF grants DMS-1954357 and DMS-2000200 and by  the Simons Programme at IHES during a sabbatical visit. 
V.V. was supported by the NSF CAREER
grant DMS-1911413. We thank an anonymous referee for pointing us to the references~\cite{Taylor,Aris} and the connection between the formula~\eqref{e.Khom.int} and Taylor dispersion.

{\small
\bibliographystyle{alpha}
\bibliography{enhance.bib}
}

\end{document}